\documentclass[11pt,a4paper,oneside,reqno]{amsart}
\usepackage{lmodern}
\usepackage[T1]{fontenc}
\usepackage[main=english, french]{babel}
\usepackage{amsopn}

\newtheorem{theorem}{Theorem}[section]
\newtheorem{conj}[theorem]{Conjecture}
\newtheorem{question}[theorem]{Question}
\newtheorem{corollary}[theorem]{Corollary}
\newtheorem{lemma}[theorem]{Lemma}
\newtheorem{proposition}[theorem]{Proposition}

\theoremstyle{definition}
\newtheorem{definition}[theorem]{Definition}

\newtheorem{example}[theorem]{Example}
\newtheorem{remark}[theorem]{Remark}
\newtheorem{notation}[theorem]{Notation}

\usepackage[square,sort,comma,numbers]{natbib}

\usepackage[all]{xy}
\usepackage{enumerate}
\usepackage{amsfonts,amssymb,amsmath,eucal,pinlabel,array,hhline}
\usepackage{slashed}
\usepackage{tabulary}
\usepackage{fancyhdr}
\usepackage{a4wide}
\usepackage{calrsfs,bbm,dsfont}
\usepackage[position=b]{subcaption}

\newcounter{notes}%

\newcommand{\low}[1]{\multirow{2}{*}{$#1$}}  
\newcommand{\push}[1]{\multicolumn{2}{c}{$#1$}}  
	\newcommand{\ignore}[1]{}  

\newcommand{\Li}{\mathrm{Li}_2}
\newcommand{\Vol}{\mathrm{Vol}}

	\usepackage{comment}
\usepackage{amsfonts}
\usepackage{multirow} 
\usepackage{kbordermatrix}
\usepackage{brunnian} 

\usetikzlibrary{knots}

\usetikzlibrary{matrix}
\usetikzlibrary{arrows}

\usetikzlibrary{patterns}
\usetikzlibrary{shapes}
\usetikzlibrary{arrows}
\usetikzlibrary{decorations.markings}
\usetikzlibrary{decorations.pathreplacing}

\newcommand{\Log}{\mathrm{Log}}

\usetikzlibrary{decorations.pathreplacing,decorations.markings}

\tikzset{
  on each segment/.style={
    decorate,
    decoration={
      show path construction,
      moveto code={},
      lineto code={
        \path [#1]
        (\tikzinputsegmentfirst) -- (\tikzinputsegmentlast);
      },
      curveto code={
        \path [#1] (\tikzinputsegmentfirst)
        .. controls
        (\tikzinputsegmentsupporta) and (\tikzinputsegmentsupportb)
        ..
        (\tikzinputsegmentlast);
      },
      closepath code={
        \path [#1]
        (\tikzinputsegmentfirst) -- (\tikzinputsegmentlast);
      },
    },
  },
  mid arrow/.style={postaction={decorate,decoration={
        markings,
        mark=at position .5 with {\arrow[#1]{>}}
      }}}
      ,
}

\tikzset{
  on each segment/.style={
    decorate,
    decoration={
      show path construction,
      moveto code={},
      lineto code={
        \path [#1]
        (\tikzinputsegmentfirst) -- (\tikzinputsegmentlast);
      },
      curveto code={
        \path [#1] (\tikzinputsegmentfirst)
        .. controls
        (\tikzinputsegmentsupporta) and (\tikzinputsegmentsupportb)
        ..
        (\tikzinputsegmentlast);
      },
      closepath code={
        \path [#1]
        (\tikzinputsegmentfirst) -- (\tikzinputsegmentlast);
      },
    },
  },
  mid arrow d/.style={postaction={decorate,decoration={
        markings,
        mark=at position .5 with {\arrow[#1]{>>}}
      }}}
      ,
}

\tikzset{
  on each segment/.style={
    decorate,
    decoration={
      show path construction,
      moveto code={},
      lineto code={
        \path [#1]
        (\tikzinputsegmentfirst) -- (\tikzinputsegmentlast);
      },
      curveto code={
        \path [#1] (\tikzinputsegmentfirst)
        .. controls
        (\tikzinputsegmentsupporta) and (\tikzinputsegmentsupportb)
        ..
        (\tikzinputsegmentlast);
      },
      closepath code={
        \path [#1]
        (\tikzinputsegmentfirst) -- (\tikzinputsegmentlast);
      },
    },
  },
  mid arrow s/.style={postaction={decorate,decoration={
        markings,
        mark=at position .5 with {\arrow[#1]{stealth}}
      }}}
      ,
}

\tikzset{
  on each segment/.style={
    decorate,
    decoration={
      show path construction,
      moveto code={},
      lineto code={
        \path [#1]
        (\tikzinputsegmentfirst) -- (\tikzinputsegmentlast);
      },
      curveto code={
        \path [#1] (\tikzinputsegmentfirst)
        .. controls
        (\tikzinputsegmentsupporta) and (\tikzinputsegmentsupportb)
        ..
        (\tikzinputsegmentlast);
      },
      closepath code={
        \path [#1]
        (\tikzinputsegmentfirst) -- (\tikzinputsegmentlast);
      },
    },
  },
  mid arrow l/.style={postaction={decorate,decoration={
        markings,
        mark=at position .5 with {\arrow[#1]{latex}}
      }}}
      ,
}

\usepackage{graphicx}

\newcommand{\C}{\mathbb{C}}
\newcommand{\R}{\mathbb{R}}

\newcommand{\Z}{\mathbb{Z}}
\newcommand{\N}{\mathbb{N}}

\newcommand{\B}{\mathsf{b}}
\newcommand{\stareq}{\stackrel{\star}{=}}

\newcommand{\sarrow}{\,\mathbin{\rotatebox[origin=c]{90}{$\rightarrow$}}}
\newcommand{\darrow}{\,\mathbin{\rotatebox[origin=c]{90}{$\twoheadrightarrow$}}}

\DeclareSymbolFont{EulerScript}{U}{eus}{m}{n}
\DeclareSymbolFontAlphabet\mathscr{EulerScript}

\title[Triangulations \& Teichm\"uller TQFT volume conjecture for twist knots]{Geometric triangulations and the Teichm\"uller TQFT volume conjecture for twist knots}
\author{Fathi Ben Aribi, Fran\c cois Gu\'eritaud and Eiichi Piguet-Nakazawa}

\address{
	UCLouvain, IRMP, Chemin du Cyclotron 2 \\
	1348 Louvain-la-Neuve \\
	Belgium}
\email{fathi.benaribi@uclouvain.be}

\address{
	CNRS and IRMA, Universit\'e de Strasbourg, UMR 7501 \\
	7 rue Ren\'e Descartes, 67084 Strasbourg Cedex \\
	France}
\email{francois.gueritaud@unistra.fr}

\address{Universit\'e de Gen\`eve, Section de math\'ematiques, 2-4 rue du Li\`evre, Case postale 64 1211 Gen\`eve 4, Suisse}
\email{eiichi.piguet@unige.ch}

\begin{document}


\maketitle

\begin{abstract}
We construct a new infinite family of ideal triangulations and H-triangulations for the complements of twist knots, using a method originating from Thurston. These triangulations provide a new upper bound for the Matveev complexity of twist knot complements.

We then prove that these ideal triangulations are geometric. The proof uses techniques of Futer and the second author, which consist in studying the volume functional on the polyhedron of angle structures.

Finally, we use these triangulations to compute explicitly the partition function of the Teichm\"uller TQFT and to prove the associated volume conjecture for all twist knots, using the saddle point method.
\end{abstract}

\maketitle

\tableofcontents

\section{Introduction}

Quantum topology began in 1984 with the definition of the Jones polynomial \cite{Jo}, a knot invariant that Witten later retrieved in the  Chern--Simons quantum field theory on the three-sphere with gauge group $SU(2)$ \cite{Wi}. Following Witten's intuitions from physics, several Topological Quantum Field Theories (or \textit{TQFT} for short, meaning certain functors from cobordisms to vector spaces  \cite{At}) were defined in the nineties and provided new invariants of knots and $3$-manifolds \cite{RT1, RT2, BHMV1, BHMV2, TV}.

The \textit{volume conjecture} of Kashaev and Murakami--Murakami is perhaps the most studied conjecture in quantum topology currently \cite{Ka95, MM, MuIntro, MY}; it states that the colored Jones polynomials of a given hyperbolic knot evaluated at a certain root of unity asymptotically grow with an exponential rate, which is the hyperbolic volume of this knot. As such, it hints at a deep connection between quantum topology and classical geometry. In the last twenty years, several variants of the volume conjecture have been put forward for other quantum invariants: for instance the Baseilhac--Benedetti generalisation in terms of quantum hyperbolic invariants \cite{BB}, or the Chen--Yang volume conjecture on the Turaev--Viro invariants for hyperbolic $3$-manifolds \cite{CY}. Some of these conjectures have been proven for several infinite families of examples, such as the fundamental shadow links \cite{Co}, the Whitehead chains \cite{VdV} and integral Dehn fillings on the figure-eight knot complement \cite{Oh}.
 See \cite{MuIntro, MM} for more examples.

In \cite{AK}, Andersen and Kashaev constructed the Teichm\"uller TQFT, a \textit{generalised} Topological Quantum Field Theory, in the sense that the operators of the theory act on \textit{infinite-dimensional} vector spaces. The partition function of the Teichm\"uller TQFT yields a quantum invariant $|\mathcal{Z}_{\hbar}(X,\alpha)| \in \R_{>0}$ (indexed by a quantum parameter $\hbar >0$) of a triangulated $3$-manifold $X$ endowed with a family of dihedral angles $\alpha$, up to certain moves on such triangulations with angles (see \cite{AK} for details).
Taking its roots in quantum Teichm\"uller theory and making use of Faddeev's quantum dilogarithm, this infinite-dimensional TQFT is constructed with state integrals on tempered distributions from the given triangulation with angles. The Teichm\"uller TQFT already admits several formulations and generalisations (see \cite{AK, AKnew, KaWB, AKicm}), and it is still not clear at the time of writing which formulation one should favor in order to best reduce the technical constraints in the definitions and computations. 

Nevertheless, two points remain clear regardless of the chosen formulation. Firstly, the Teichm\"uller TQFT is a promising lead for obtaining a mathematical model of quantum Chern--Simons theory with non-compact gauge group $SL(2,\C)$ \cite{AK, AKicm, Mi}. Secondly, the Teichm\"uller TQFT should also satisfy a \textit{volume conjecture}, stated as follows without details: 

\begin{conj}[Conjecture 1 of \cite{AK}, Conjecture \ref{conj:vol:BAGPN}]\label{conj:vol:intro}
Let $M$ be  a  closed  oriented  3-manifold  and $K\subset M$ a  knot  whose complement  is  hyperbolic.  Then  the  partition  function  of  the  Teichm\"uller  TQFT  associated  to $(M,K)$ follows an exponential decrease in the semi-classical limit $\hbar \to 0^+$, whose rate is the hyperbolic volume $\Vol(M \setminus K)$.
\end{conj}

Generally speaking, solving a volume conjecture requires  {one} to find connections between quantum topology and hyperbolic geometry hidden in the invariant, and to overcome technical difficulties (often analytical in nature). The payoff is worth the hassle, though: the previously mentioned connections can enrich both domains of mathematics and may provide new insights on how  {one can best} mathematically model physical quantum field theories. 
In the present paper, we solve the Teichm\"uller TQFT volume conjecture for the infinite family of hyperbolic twist knots in $S^3$ (see Figure \ref{fig:twist:knot} for a picture of these knots). Before, the conjecture was proven for the first two knots of this family \cite{AK} and numerically checked for the next nine \cite{AN, BAPNcras}. 
Moreover, since the first version of the present paper, the conjecture has also been proven for an infinite family of fibered knots in lens spaces \cite{PN}.
To the authors' knowledge, the twist knots are now the first  {infinite} family of hyperbolic knots in $S^3$ for which a volume conjecture is proven.  {Meanwhile, specific infinite families of links were tackled in \cite{Co, VdV} and the closed integral surgeries on the figure-eight knot were handled in \cite{Oh} for the Chen--Yang volume conjecture; this last result is comparable to the main result of our paper, as the twist knots can also be seen as a family of Dehn surgeries (on the Whitehead link).}
We hope that the techniques and results of this paper can provide valuable insights for further studies of this volume conjecture or its siblings that concern other quantum invariants \cite{Ka95, MM, CY}. Notably, it would be interesting to try to apply the techniques of this paper to prove other conjectures for the twist knots.

\

Let us now  {specify} the objects used and the results proven in this paper. Before all, we should clarify that the results split in two halves: Sections \ref{sec:trig} to \ref{sec:vol:conj} focus on the hyperbolic twist knots with an \textit{odd} number of crossings, while the even twist knots are studied in 
Section~\ref{sec:appendix}. Indeed, the constructions and proofs vary slightly  {depending on} whether the crossing number is odd or even. Hence, the reader interested in discovering for the first time our objects and techniques should focus on the odd twist knots in Sections \ref{sec:trig} to \ref{sec:vol:conj}. Likewise, Section \ref{sec:appendix} is for the experienced reader who wants to understand the difficulties in generalising our results from one infinite family of knots to another, and can be a starting point for future further proofs of the Teichm\"uller TQFT volume conjecture.

\

The first part of this paper deals with topological constructions of triangulations for twist knot complements (Sections \ref{sec:trig} and \ref{sub:even:trig}).

In the seventies, Thurston showed that hyperbolic geometry was deeply related to low-dimensional topology. He notably conjectured that  every compact, oriented, irreducible, atoroidal $3$-manifold $M$ with (empty or)  toroidal boundary and infinite fundamental group admits a complete hyperbolic metric \cite{Th2}. This \textit{Hyperbolization Conjecture} was then proved by Thurston for Haken manifolds \cite{Mor} and later by Perelman in the general case \cite{MT}.  For $3$-manifolds with toroidal boundary, such as complements of knots in the three-sphere, this hyperbolic metric is unique up to isometry, by the Mostow--Prasad rigidity theorem \cite{Mo, Pr}. Hyperbolic geometry can thus provide topological invariants, such as the hyperbolic volume of a knot complement.

Several knot invariants  can be computed from an \textit{ideal triangulation} $X=(T_1,\ldots,T_N,\sim)$ of the knot complement $S^3 \setminus K$, that is to say a gluing of $N$ ideal (i.e.\ without their vertices) tetrahedra $T_1, \ldots, T_N$ along  {with} a pairing of faces $\sim$. As a given knot complement admits an infinite number of triangulations, it is therefore natural to look for convenient triangulations with as few tetrahedra as possible.

The twist knots $K_n$ of Figure \ref{fig:twist:knot} form the simplest infinite family of hyperbolic knots (when $n\geqslant 2$, starting at the figure-eight knot). Recall that a knot is \textit{hyperbolic}  if its complement admits a complete hyperbolic structure of finite volume. In order to study the Teichm\"uller TQFT for the family of twist knots, we thus constructed particularly convenient ideal triangulations of their complements.

An intermediate step was to construct \textit{H-triangulations} of $(S^3,K_n)$, which are  triangulations of $S^3$ by compact tetrahedra, where the knot $K_n$ is represented by  {a single} edge. We now state the first result of this paper.

\begin{theorem}[Theorem \ref{thm:trig}]\label{thm:intro:trig}
For every $n\geqslant 2$, there exist an ideal triangulation $X_n$ of the twist knot complement $S^3 \setminus K_n$ with $\lfloor \frac{n+4}{2} \rfloor$ tetrahedra and a H-triangulation $Y_n$ of the pair $(S^3,K_n)$ with $\lfloor \frac{n+6}{2} \rfloor$ tetrahedra. Moreover, the edges of all these triangulations admit orientations for which no triangle is a cycle.
\end{theorem}

The condition on edge orientations implies that every tetrahedron comes with a full order on its vertices: such a property is needed
to define the Teichm\"uller TQFT, see Section~\ref{sec:prelim}. Note that in \cite{BB}, this property is called a \textit{branching} on the triangulation (the first of several similarities between the Teichm\"uller TQFT and the Baseilhac--Benedetti quantum hyperbolic invariants).

To prove Theorem \ref{thm:intro:trig}, we study the cases ``$n$ odd'' and ``$n$ even'' separately. In both cases, we use a method introduced by Thurston
\cite{Th2}
 and later developed by Menasco and Kashaev--Luo--Vartanov \cite{Me, KLV}: we start from a diagram of the knot $K_n$ and we obtain a combinatorial description of $S^3$ as a polyhedron glued to itself, where $K_n$ is one particular edge. We then apply a combinatorial trick to reduce the number of edges in the polyhedron, and finally we triangulate it. This yields an H-triangulation $Y_n$ of $(S^3,K_n)$, which then gives the ideal triangulation $X_n$ of $S^3 \setminus K_n$ by  {collapsing} the  {single} tetrahedron containing the edge $K_n$.

The numbers $\lfloor \frac{n+4}{2} \rfloor$ in Theorem \ref{thm:intro:trig} give new upper bounds for the Matveev complexities of the manifolds $S^3 \setminus K_n$, and experimental tests on the software \textit{SnapPy} lead us to conjecture that these numbers are actually equal to the Matveev complexities for this family (see Conjecture \ref{conj:matveev} and Remark \ref{rem:snappy}).

\

In the second part of this paper (Sections \ref{sec:geom} and \ref{sub:even:geom}), we prove the \textit{geometricity} of these new ideal triangulations, which means that their tetrahedra can be endowed with positive dihedral angles corresponding to the complete hyperbolic structure on the underlying hyperbolic $3$-manifold. 

In \cite{Th}, Thurston provided a method to study geometricity of a given triangulation, which is a system of \textit{gluing equations} on complex parameters associated to the tetrahedra; if this system admits a solution, then this solution is unique and corresponds to the complete hyperbolic metric on the triangulated manifold.

However, this system of equations is difficult to solve in practice. In the nineties, Casson and Rivin devised a technique to prove geometricity (see the survey \cite{FG}). The idea is to focus on the argument part of the system of complex gluing equations (this part can be seen as a linear system) and use properties of the volume functional. Futer and the second author applied such a method for particular triangulations of once-punctured torus bundles and two-bridge link complements \cite{Gf}.

In this vein, we prove that the ideal triangulations $X_n$ of Theorem \ref{thm:intro:trig} are geometric.

\begin{theorem}[Theorems \ref{thm:geometric} and \ref{thm:appendix:geom:even}]\label{thm:intro:geom}
For every $n\geqslant 2$, $X_n$ is geometric.
\end{theorem}

To prove Theorem \ref{thm:intro:geom}, we use techniques of Futer and the second author (see \cite{FG, Gf}). We first prove that the space of angle structures on $X_n$ is non-empty (Lemma \ref{lem:non:empty} for the odd case), and then that the volume functional cannot attain its maximum on the boundary of this space (Lemma \ref{lem:interior} for the odd case). Then Theorem \ref{thm:intro:geom}  follows from a result of Casson and Rivin (see Theorem \ref{thm:casson:rivin}).

\

In the third part of this paper (Sections \ref{sec:part:odd}, \ref{sec:part:H:odd} and \ref{sub:even:tqft}), we compute the partition functions of the Teichm\"uller TQFT for the triangulations $X_n$ and $Y_n$, and we notably prove that they satisfy the properties expected in Conjecture \ref{conj:vol:BAGPN}. Without going into details, we can summarise these properties as:

\begin{theorem}[Theorems \ref{thm:part:func}, \ref{thm:even:part:func}, \ref{thm:part:func:Htrig:odd} and \ref{thm:part:func:Htrig:even}]\label{thm:intro:partition}
	For every $n\geqslant 2$ and every $\hbar>0$, the partition function $\mathcal{Z}_{\hbar}(X_n,\alpha)$ of the ideal triangulation $X_n$ 
	(resp.\ $\mathcal{Z}_{\hbar}(Y_n,\alpha)$ of the H-triangulation $Y_n$) is computed explicitly for every angle structure $\alpha$ of $X_n$ (resp.\ of $Y_n$). 
	
	Moreover, the value $| \mathcal{Z}_{\hbar}(X_n,\alpha)| $ depends only 
	on three entities:
	 two linear combinations of angles $\mu_{X_n}(\alpha)$ and $\lambda_{X_n}(\alpha)$ ({which are the angular holonomies of} the meridian and longitude of the knot $K_n$),
	and a function $(x \mapsto J_{X_n}(\hbar,x))$, defined on some open subset of $\C$, and independent of the angle structure $\alpha$.
	
Furthermore, the value $|J_{X_n}(\hbar,0)|$ can be retrieved in a certain asymptotic of the partition function $\mathcal{Z}_{\hbar}(Y_n,\alpha)$ of the H-triangulation $Y_n$.
\end{theorem}

The function $(\hbar \mapsto J_{X_n}(\hbar,0))$ should be seen as an analogue of the Kashaev invariant $\langle \cdot \rangle_{N}$ of \cite{Ka94,Ka95}, or of the colored Jones polynomials evaluated at a certain root of unity $J_{\cdot}(N,e^{2i \pi/N})$, where $\hbar$ behaves as the inverse of the color $N$. It is not clear at the time of writing that $(\hbar \mapsto J_{X_n}(\hbar,0))$ always yields a proper knot invariant independent of the triangulation. However,
Theorem \ref{thm:intro:partition} states that we can attain this function in at least two ways (as anticipated in the volume conjecture of \cite{AK}), which increases the number of available tools for proving such an invariance.
Theorem \ref{thm:intro:partition} is also of interest for studying the \textit{AJ-conjecture} for the Teichm\"uller TQFT, as stated in \cite{AM}.

To prove Theorem \ref{thm:intro:partition}, we compute the aforementioned partition functions, and especially their parts that encode how the faces of the triangulation are glued to one another (such a part is called the \textit{kinematical kernel}). We then show a connection between this kinematical kernel and the gluing equations on angles for the same triangulation, which allows us to prove that the partition function only depends on the angle structure $\alpha$ via the weight of $\alpha$ on each edge (which is constant equal to $2 \pi$) and via two angular holonomies $\mu_{X_n}(\alpha)$ and $\lambda_{X_n}(\alpha)$ related to the meridian and longitude of the twist knot $K_n$. Finally, we need to establish some uniform bounds on the quantum dilogarithm in order to apply the dominated convergence theorem in the computation of the asymptotic of $\mathcal{Z}_{\hbar}(Y_n,\alpha)$.

At the time of writing, whether or not the partition function always contains such topological information (the meridian and longitude of the knot) is an open question. Nevertheless, we hope that the patterns noticed for this infinite family of examples can illuminate the path.

\

In the fourth and final part of this paper (Sections \ref{sec:vol:conj} and \ref{sub:even:vol:conj}), we prove that the function $(\hbar \mapsto J_{X_n}(\hbar,0))$ (extracted from the partition functions of the Teichm\"uller TQFT in Theorem \ref{thm:intro:partition}) exponentially decreases in the semi-classical limit $ \hbar \to 0^+$, with decrease rate the hyperbolic volume.  {More precisely:}

\begin{theorem}[Theorems \ref{thm:vol:conj} and \ref{thm:even:vol:conj}]\label{thm:intro:vol:conj}
	For every $n\geqslant 2$, we have the following limit:
$$
\lim_{\hbar \to 0^+} 2\pi \hbar \log \vert J_{X_n}(\hbar,0) \vert
= -\emph{Vol}(S^3\backslash K_n).$$	
\end{theorem}

To prove Theorem \ref{thm:intro:vol:conj}, we apply the saddle point method on the semi-classical approximation of $\vert J_{X_n}(\hbar,0) \vert$ (expressed with classical dilogarithms $\Li$),  {and then} bound the remaining error terms with respect to $\hbar$. 

More precisely, the \textit{saddle point method} is a common designation of various theorems that state that an integral $\int_\gamma \exp(\lambda S(z)) dz$ behaves mostly as $\exp\left (\lambda \max_\gamma(\Re(S))\right )$ when $\lambda \to \infty$ (see Theorem \ref{thm:SPM} for the version we used, and \cite{Wo} for a survey). In order to apply this method, we must check technical conditions such as the fact that the maximum of $\Re(S)$ on $\gamma$ is unique and a simple critical point. Fortunately, in the present paper, these conditions are consequences of the \textit{geometricity} of the ideal triangulations $X_n$ (Theorem \ref{thm:intro:geom}); indeed, the equations $\nabla S = 0$ here correspond exactly to the complex gluing equations, and their unique solution (the complete hyperbolic angle structure) provides the expected saddle point. Geometricity was the main ingredient we needed, in order to go from a finite number of numerical checks of the Teichm\"uller TQFT volume conjecture \cite{BAPNcras} to an exact proof for an infinite family.

Note that thanks to Theorem \ref{thm:intro:geom}, we did not need to compute the exact value of the complete hyperbolic structure or of the hyperbolic volume, although such computations would be doable in the manner of \cite{CMY} with our triangulations $X_n$.

The previously mentioned error bounds follow from the fact that $J_{X_n}(\hbar,0)$ does not depend exactly on the potential function $S$ made of classical dilogarithms, but on a quantum deformation $S'_{\hbar}$ using quantum dilogarithms. An additional difficulty stems from the fact that we must bound the error uniformly on a \textit{non-compact} contour, when $\hbar \to 0^+$. To the authors' knowledge, this difficulty never happened in studies of volume conjectures for other quantum invariants, since asymptotics of these invariants (such as the colored Jones polynomials) involve integrals on \textit{compact} contours. Hence we hope that the analytical techniques we developed in this paper ({which} are not specific to the twist knots) can be of use for future studies of volume conjectures with unbounded contours. More precisely, the parity trick in Lemma \ref{lem:parity} and its application in the bound for the whole non-compact contour (Lemma \ref{lem:unif:bound}) are our main additions from the previous techniques of \cite{AH}.

\

It is natural to wonder if the main result of this paper can be extended to any knot complement admitting a geometric triangulation. For such a manifold, we expect the analytical results of Section \ref{sec:vol:conj} to hold similarly (as their proofs did not use the fact that we studied the twist knots). However, it is yet unclear how one can generalise the computation and simplification of the partition functions (see Section \ref{sec:part:odd}) and its relation with the gluing equations (Lemma \ref{lem:grad:thurston}). We expect that combinatorial techniques on triangulations, such as those used to define the Neumann-Zagier datum \cite{DG, NZ}, will be needed.

\

Part of the results in this paper (Theorems \ref{thm:trig}, \ref{thm:part:func}, \ref{thm:even:part:func}, \ref{thm:part:func:Htrig:odd} and \ref{thm:part:func:Htrig:even}) were announced in~\cite{BAPNcras}. Sections \ref{sec:trig}, \ref{sec:geom} and \ref{sub:even:trig} appeared in the arXiv preprint \cite{BAPN2}. 

The paper is organised as follows: in Section \ref{sec:prelim}, we review preliminaries and  {notation}; in Section \ref{sec:trig} we construct the triangulations for odd twist knots; in Section \ref{sec:geom}, we prove geometricity of these triangulations for odd twist knots; in Section \ref{sec:part:odd} (resp.\ \ref{sec:part:H:odd}) we compute the partition function of the Teichm\"uller TQFT for the ideal triangulations (resp.\ H-triangulations), still for odd twist knots; in Section \ref{sec:vol:conj}, we prove the volume conjecture for odd twist knots (readers eager to arrive at Section  \ref{sec:vol:conj} can skip Section \ref{sec:part:H:odd} after reading Section \ref{sec:part:odd}); finally, in Section \ref{sec:appendix}, we explain how the proofs of the previous sections differ for the even twist knots.

\section*{Acknowledgements}

The first and third authors were supported by the Swiss National Science Foundation at the
University of Geneva, with subsidy $200021\_162431$. The first author was moreover supported by the FNRS in his "Research Fellow" position at UCLouvain, under Grant no. 1B03320F. 
The second author acknowledges support from  the ANR under the
grant DynGeo (ANR-16-CE40-0025-01) and through the Labex Cempi (ANR-11-LABX0007-01).
We thank Rinat Kashaev for helpful discussions, Renaud Detcherry for his proof of Lemma \ref{lem:complex:sym}, and the University of Geneva and UCLouvain for their hospitality.  We thank the anonymous referees for their valuable corrections and suggestions, notably Remark \ref{rem:snappy}.

\section{Preliminaries and  {notation}} \label{sec:prelim}

\subsection{Triangulations}

In this section we follow \cite{AK, KaWB}.
A tetrahedron $T$ with faces $\mathsf{A},\mathsf{B},\mathsf{C},\mathsf{D}$ will be denoted as in Figure \ref{fig:tetrahedron}, where the face outside the circle represents the back face and the center of the circle is the opposite vertex pointing towards the reader. We always choose an \textit{order} on the four vertices of $T$ and we call them $0_T,1_T,2_T,3_T$ (or $0,1,2,3$ if the context makes it obvious). Consequently, if we rotate $T$ such that $0$ is in the center and $1$ at the top, then there are two possible places for vertices $2$ and $3$; we call $T$ a \textit{positive} tetrahedron if they are as in Figure \ref{fig:tetrahedron}, and \textit{negative} otherwise. We denote $\varepsilon(T) \in \{ \pm 1\}$ the corresponding \textit{sign} of $T$. We \textit{orient the edges} of $T$  {according} to the order on vertices, and we endow each edge with a parametrisation by $[0,1]$ respecting the orientation. Note that such a structure was called a \textit{branching} in \cite{BB}.

Thus, up to isotopies fixing the $1$-skeleton pointwise, there is only one way of \textit{gluing} two triangular faces together while \textit{respecting the order of the vertices} and the edge parametrisations, and that is the only type of face gluing we consider in this paper.

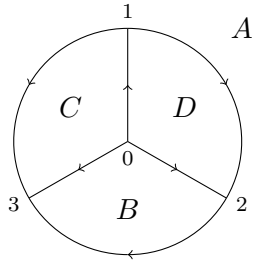
\begin{figure}[h]
\centering
\begin{tikzpicture}

\begin{scope}[xshift=0cm,yshift=0cm,rotate=0,scale=1.5]

\draw (0,-0.15) node{\scriptsize $0$} ;
\draw (0,1.15) node{\scriptsize $1$} ;
\draw (1,-0.55) node{\scriptsize $2$} ;
\draw (-1,-0.55) node{\scriptsize $3$} ;
\draw (1,1) node{$\mathsf{A}$} ;
\draw (0,-0.6) node{$\mathsf{B}$} ;
\draw (-0.5,0.3) node{$\mathsf{C}$} ;
\draw (0.5,0.3) node{$\mathsf{D}$} ;

\path [draw=black,postaction={on each segment={mid arrow =black}}]
(0,0)--(-1.732/2,-0.5);

\path [draw=black,postaction={on each segment={mid arrow =black}}]
(0,0)--(0,1);

\path [draw=black,postaction={on each segment={mid arrow =black}}]
(0,0)--(1.732/2,-0.5);

\draw[->](1.732/2,-0.5) arc (-30:-90:1);
\draw (0,-1) arc (-90:-150:1);

\draw[->](0,1) arc (90:30:1);
\draw (1.732/2,0.5) arc (30:-30:1);

\draw[->](0,1) arc (90:150:1);
\draw (-1.732/2,0.5) arc (150:210:1);

\end{scope}

\end{tikzpicture}
\caption{The positive tetrahedron $T$} 
\label{fig:tetrahedron}
\end{figure}

Note that a tetrahedron $T$ like in Figure \ref{fig:tetrahedron} will either represent a \textit{compact} tetrahedron homeomorphic to a $3$-ball $B^3$ (notably when considering \textit{H-triangulations}) or an \textit{ideal} tetrahedron homeomorphic to a $3$-ball minus $4$ points in the boundary (when considering \textit{ideal triangulations}).

A \textit{triangulation} $X=(T_1,\ldots,T_N,\sim)$ is the data of $N$ distinct tetrahedra $T_1, \ldots, T_N$ and an equivalence relation $\sim$ first defined on the faces by pairing and the only gluing that respects vertex order, and also induced on edges  {and} vertices by the combined identifications. We call $M_X$ the (pseudo-)$3$-manifold 
$ M_X = T_1 \sqcup \cdots \sqcup T_N / \sim$ obtained by quotient. Note that $M_X$ may fail to be a manifold only at  (the image by the quotient map of) a vertex  of the triangulation, whose regular  {neighbourhood} might not be a $3$-ball (but for instance a cone over a torus for exteriors of links).

We denote $X^{k}$ (for $k=0, \ldots, 3$) the set of $k$-cells of $X$ after identification by $\sim$. In this paper we always  {assume} that \textit{no face is left unpaired  by $\sim$}, thus $X^{2}$ is always of  {cardinality} $2N$. By a slight abuse of notation we also call $T_j$ the $3$-cell inside the tetrahedron $T_j$, so that $X^{3} = \{T_1, \ldots, T_N\}$. Elements of $X^{1}$ are usually represented by distinct types of arrows, which are drawn on the corresponding preimage edges, see Figure \ref{fig:id:tri:41:complement} for an example.

An \textit{ideal triangulation} $X$ contains ideal tetrahedra, and in this case the quotient space minus its vertices $M_X \setminus X^0$ is an open manifold. In this case we will denote $M=M_X \setminus X^0$ and say that the open manifold $M$ admits the ideal triangulation $X$.

A (one-vertex) \textit{H-triangulation} is a triangulation $Y$ with compact tetrahedra so that $M=M_Y$ is a closed manifold and $Y^0$ is a singleton, with one distinguished edge in $Y^1$;  this edge will represent a knot $K$ (up to ambient isotopy) in the closed manifold $M$, and we will say that $Y$ is an \textit{H-triangulation for $(M,K)$}.

Finally, for $X$ a triangulation and $k=0,1,2,3,$ we define $x_k\colon X^3 \to X^2$ the  {function} such that $x_k(T)$ is the equivalence class of the face of $T$ opposed to its vertex $k$.

\begin{example}\label{ex:41}
Figure \ref{fig:id:tri:41:complement} displays two possible ways of representing the same ideal triangulation of the complement of the figure-eight knot $M=S^3 \setminus 4_1$, with one positive and one negative tetrahedron. Here $X^3 =\{T_+, T_-\}$, $X^2 =\{\mathsf{A},\mathsf{B},\mathsf{C},\mathsf{D}\}$, $X^1 =\{ \sarrow , \darrow \}$ and $X^0$ is a singleton. On the left the tetrahedra are drawn as usual and all the cells are named; on the right we represent each tetrahedron by a ``comb'' \begin{tikzpicture}[style=very thick] 
\begin{scope}[scale=0.3]
\draw(0,0)--(3,0);
\draw(0,0)--(0,1/2);
\draw(1,0)--(1,1/2);
\draw(2,0)--(2,1/2);
\draw(3,0)--(3,1/2);
\end{scope}
\end{tikzpicture}
 with four spikes numbered $0,1,2,3,$ from left to right. We join the spike $j$ of $T$ to the spike $k$ of $T'$ if $x_j(T)=x_k(T')$, and we add a $+$ or $-$ next to each tetrahedron according to its sign.

\begin{figure}[h]
\centering
\begin{tikzpicture}

\begin{scope}[xshift=0cm,yshift=0cm,rotate=0,scale=1.5]

\draw (0,-0.15) node{\scriptsize $0$} ;
\draw (0,1.15) node{\scriptsize $1$} ;
\draw (1,-0.55) node{\scriptsize $2$} ;
\draw (-1,-0.55) node{\scriptsize $3$} ;

\draw (1,1) node{$\mathsf{B}$} ;
\draw (0,-0.6) node{$\mathsf{A}$} ;
\draw (-0.5,0.3) node{$\mathsf{C}$} ;
\draw (0.5,0.3) node{$\mathsf{D}$} ;

\draw (0,-1.4) node{\large $T_+$} ;

\path [draw=black,postaction={on each segment={mid arrow =black}}]
(0,0)--(-1.732/2,-0.5);

\path [draw=black,postaction={on each segment={mid arrow =black}}]
(0,0)--(0,1);

\path [draw=black,postaction={on each segment={mid arrow d =black}}]
(0,0)--(1.732/2,-0.5);

\draw[->](1.732/2,-0.5) arc (-30:-90:1);
\draw (0,-1) arc (-90:-150:1);

\draw[->>](0,1) arc (90:30:1);
\draw (1.732/2,0.5) arc (30:-30:1);

\draw[->>](0,1) arc (90:150:1);
\draw (-1.732/2,0.5) arc (150:210:1);

\end{scope}

\begin{scope}[xshift=4cm,yshift=0cm,rotate=0,scale=1.5]

\draw (0,-0.15) node{\scriptsize $0$} ;
\draw (0,1.15) node{\scriptsize $1$} ;
\draw (1,-0.55) node{\scriptsize $3$} ;
\draw (-1,-0.55) node{\scriptsize $2$} ;

\draw (1,1) node{$\mathsf{C}$} ;
\draw (0,-0.6) node{$\mathsf{D}$} ;
\draw (-0.5,0.3) node{$\mathsf{A}$} ;
\draw (0.5,0.3) node{$\mathsf{B}$} ;

\draw (0,-1.4) node{\large $T_-$} ;

\path [draw=black,postaction={on each segment={mid arrow =black}}]
(0,0)--(-1.732/2,-0.5);

\path [draw=black,postaction={on each segment={mid arrow d =black}}]
(0,0)--(0,1);

\path [draw=black,postaction={on each segment={mid arrow d =black}}]
(0,0)--(1.732/2,-0.5);

\draw(1.732/2,-0.5) arc (-30:-90:1);
\draw[<<-] (0,-1) arc (-90:-150:1);

\draw[->](0,1) arc (90:30:1);
\draw (1.732/2,0.5) arc (30:-30:1);

\draw[->](0,1) arc (90:150:1);
\draw (-1.732/2,0.5) arc (150:210:1);

\end{scope}


\begin{scope}[xshift=9.5cm,yshift=0cm,rotate=0]
\node[inner sep=0pt] (russell) at (0,0)
    {\includegraphics[scale=0.8]{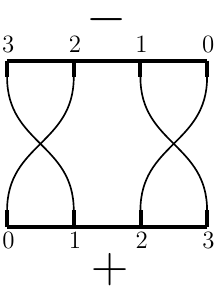}};
\end{scope}

\end{tikzpicture}
\caption{Two representations of an ideal triangulation of the knot complement $S^3 \setminus 4_1$.}
 \label{fig:id:tri:41:complement}
\end{figure}
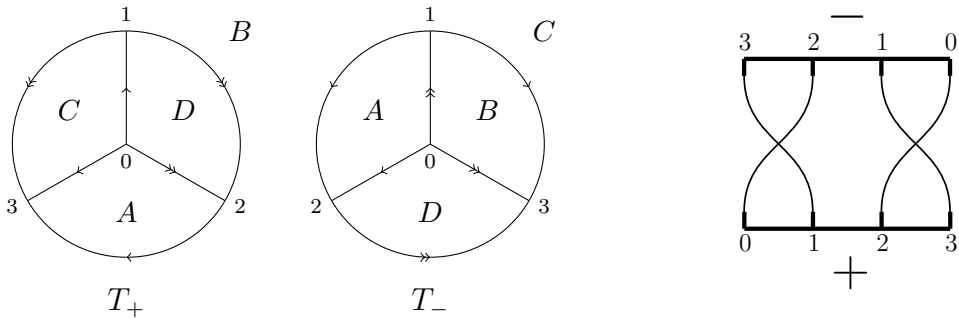

\end{example}

\subsection{Angle structures}

For a given triangulation $X=(T_1,\ldots,T_N,\sim)$ we denote $\mathcal{S}_X$ the set of \textit{shape structures on $X$}, defined as
$$
\mathcal{S}_X 
 = 
 \left \{ 
\alpha = \left (a_1,b_1,c_1,\ldots,
a_N,b_N,c_N\right ) \in (0,\pi)^{3N} \ \big | \
\forall k \in \{1,\ldots,N\}, \
a_k+b_k+c_k = \pi
\right \}.
$$
An angle $a_k$ (respectively $b_k,c_k$) represents the value of a dihedral angle on the edge $\overrightarrow{01}$ (respectively $\overrightarrow{02}$, $\overrightarrow{03}$) and its opposite edge in the tetrahedron $T_k$. If a particular shape structure $\alpha=(a_1,\ldots,c_N)\in \mathcal{S}_X$ is fixed, we define three associated  {function}s $\alpha_j\colon X^3 \to (0,\pi)$ (for $j=1,2,3$) that send $T_k$ to the $j$-th element of $\{a_k,b_k,c_k\}$ for each $k \in \{1,\ldots,N\}$.

Let $(X,\alpha)$ be a triangulation with a shape structure as before. We denote $\omega_{X,\alpha}\colon X^1 \to \mathbb{R}$ the associated \textit{weight function}, which sends an edge $e\in X^1$ to the sum of angles $\alpha_j(T_k)$ corresponding to tetrahedral edges that are preimages of $e$ by $\sim$. 
For example, if we denote 
$\alpha=(a_+,b_+,c_+,a_-,b_-,c_-)$ a shape structure on the triangulation $X$ of Figure \ref{fig:id:tri:41:complement}, then $\omega_{X,\alpha}(\sarrow) =
2 a_+ + c_+ + 2 b_- + c_-.$

One can also consider the closure $\overline{\mathcal{S}_X}$ (sometimes called the space of \textit{extended shape structures}) where the $a_k,b_k,c_k$ are taken in $[0,\pi]$ instead. The definitions of the  {function}s $\alpha_j$ and  $\omega_{X,\alpha}$ can immediately be extended.

We finally define $
\mathcal{A}_X := \left \{
\alpha \in \mathcal{S}_X \ \big | \
\forall e \in X^1, \ \omega_{X,\alpha}(e)=2\pi
\right \}
$
 the set of \textit{balanced shape structures on $X$}, or \textit{angle structures on $X$}, and
 $
\overline{\mathcal{A}_X} := \left \{
\alpha \in \overline{\mathcal{S}_X} \ \big | \
\forall e \in X^1, \ \omega_{X,\alpha}(e)=2\pi
\right \}
$ the set of \textit{extended angle structures on $X$}.

\subsection{The volume functional} \label{sub:volume}

In this section we recall some known facts about the volume functional on the space of angle structures. See for example the survey \cite{FG}  for details.

One can understand a shape structure $(a,b,c)$ on an ideal tetrahedron $T$ as a way of realising $T$ in the hyperbolic space $\mathbb{H}^3$, with its four vertices at infinity. In this hyperbolic ideal tetrahedron, the angles $a,b,c$ will represent dihedral angles between two  faces.

The \textit{Lobachevsky function} $\Lambda\colon \R \to \R$ given by:
$$
\Lambda(x) =  - \int_0^x \log \vert 2 \sin (t) \vert \, dt$$
is well defined, continuous on $\R$, and periodic with period $\pi$. Furthermore, if $T$ is a hyperbolic ideal tetrahedron with dihedral angles $a,b,c$, its
volume satisfies
$$
\Vol(T)=\Lambda(a)+\Lambda(b)+\Lambda(c).$$

Let $X=(T_1,\ldots,T_N,\sim)$ be an ideal triangulation and $\mathcal{A}_X$ its space of angle structures, which is a (possibly empty) convex polytope in $\R^{3 N}$.
Then we define a volume functional $\mathcal{V} \colon \overline{\mathcal{A}_X} \to \R$, by assigning to an (extended) angle structure $\alpha= (a_1,b_1,c_1,\ldots,a_N,b_N,c_N)$ the real number
$$\mathcal{V}(\alpha)= \Lambda(a_1)+\Lambda(b_1)+\Lambda(c_1) + \cdots + \Lambda(a_N)+\Lambda(b_N)+\Lambda(c_N).$$

By \cite[Propositions 6.1 and 6.6]{Gf} and \cite[Lemma 5.3]{FG}, the volume functional $\mathcal{V}$ is strictly concave on $\mathcal{A}_X$ and concave on $\overline{\mathcal{A}_X}$. The maximum of the volume functional is actually related to the complete hyperbolic structure, see for example \cite[Theorem 1.2]{FG} that we re-state below.

\begin{theorem}[Casson--Rivin] \label{thm:casson:rivin}
Let $M$ be an orientable $3$-manifold with boundary consisting of tori, and let $X$ be an ideal triangulation of $M$. Then an angle structure $\alpha \in \mathcal{A}_X$  corresponds
to a complete hyperbolic metric on the interior of $M$ (which is unique) if and only if $\alpha$ is a critical
point of the functional $\mathcal{V}\colon \mathcal{A}_X\to \R$.
\end{theorem}

In this last case, we say that the ideal triangulation $X$ of the $3$-manifold $M$ is \textit{geometric}.

\subsection{Thurston's complex gluing equations}\label{sub:thurston}

To a shape structure $(a,b,c)$ on an ordered tetrahedron $T$ (i.e.\ an element of $(0,\pi)^3$ of coordinate sum $\pi$) we can associate bijectively a \textit{complex shape structure} $z \in \R+i\R_{>0}$, as well as two companion complex numbers of positive imaginary part
$$z':=\frac{1}{1-z} \text{\ and \ } z'':=\frac{z-1}{z}.$$
Each of the $z, z', z''$ is associated to an edge, in a slightly different way according to $\varepsilon(T)$:
\begin{itemize}
\item In all cases, $z$ corresponds to the same two edges as the angle $a$.
\item If $\varepsilon(T)=1$, then $z'$ corresponds to $c$ and $z''$ to $b$.
\item If $\varepsilon(T)=-1$, then $z'$ corresponds to $b$ and $z''$ to $c$.
\end{itemize}
Another way of phrasing it is that $z, z', z''$ are always in a counterclockwise order around a vertex, whereas $a,b,c$ need to follow the specific vertex ordering of $T$.

In this article we will use the following definition of the complex logarithm:
\[
\Log(z) := \log\vert z \vert + i\arg(z)  \ \textrm{for} \ z \in \C^{*},
\]
where $\arg(z) \in (-\pi,\pi]$.

We now introduce a third way of describing the shape associated to a tetrahedron, by the complex number
$$y := \varepsilon(T)(\Log(z)-i \pi) \in \R  + i \varepsilon(T)(-\pi,0),$$
 which lives in a horizontal strip of the complex plane. 

We now list the equations relating $(a,b,c), (z,z',z'')$ and $y$ for both possible signs of~$T$:
\begin{align*}
\text{\underline{Positive \ tetrahedron:} \ }
&
y+i\pi = \Log(z) = \log\left (\dfrac{\sin(c)}{\sin(b)}\right ) + i a.
\\
& -\Log(1+e^y) = \Log(z') = \log\left (\dfrac{\sin(b)}{\sin(a)}\right ) + i c. \\
& \Log(1+e^{-y}) = \Log(z'') = \log\left (\dfrac{\sin(a)}{\sin(c)}\right ) + i b.\\
& y= \log\left (\dfrac{\sin(c)}{\sin(b)}\right ) -i(\pi-a) \in \R -i(\pi-a).\\
& z = -e^y \in \R + i\R_{>0}.
\end{align*}
\begin{align*}
\text{\underline{Negative \ tetrahedron:} \ }
&
-y+i\pi = \Log(z) = \log\left (\dfrac{\sin(b)}{\sin(c)}\right ) + i a.
\\
& -\Log(1+e^{-y}) = \Log(z') = \log\left (\dfrac{\sin(c)}{\sin(a)}\right ) + i b. \\
& \Log(1+e^{y}) = \Log(z'') = \log\left (\dfrac{\sin(a)}{\sin(b)}\right ) + i c.\\
& y= \log\left (\dfrac{\sin(c)}{\sin(b)}\right ) + i(\pi-a) \in \R +i (\pi-a).\\
& z = -e^{-y}\in \R + i\R_{>0}.
\end{align*}

For clarity, let us define the diffeomorphism
$$\psi_T\colon \R+i\R_{>0} \to \R - i \varepsilon(T)(0,\pi), \ z \mapsto \varepsilon(T)(\Log(z)-i \pi),$$
and its inverse
$$\psi^{-1}_T\colon \R - i \varepsilon(T)(0,\pi) \to \R+i\R_{>0}, \ y \mapsto -\exp\left (\varepsilon(T) y\right ).$$

We can now define the \textit{complex weight function}
$\omega^{\C}_{X,\alpha}\colon X^1 \to \C$ associated to a triangulation $X$ and an angle structure $\alpha \in \mathcal{A}_X$, which sends an edge $e \in X^1$ to the sum of logarithms of complex shapes associated to preimages of $e$ by $\sim$. For example, for the triangulation $X$ of Figure \ref{fig:id:tri:41:complement} and an angle structure $\alpha=(a_+,b_+,c_+,a_-,b_-,c_-)$, we have:
\begin{align*}
\omega^{\C}_{X,\alpha}(\sarrow) &=
2 \Log(z_+) + \Log(z'_+) + 2 \Log(z'_-) + \Log(z''_-)\\
&= \log\left (
\dfrac{\sin(c_+)^2 \sin(b_+) \sin(c_-)^2 \sin(a_-)}
{\sin(b_+)^2 \sin(a_+) \sin(a_-)^2 \sin(b_-)}
\right ) +i \omega_{X,\alpha}(\sarrow).
\end{align*}

Let $S$ denote one toroidal boundary component of a $3$-manifold $M$ ideally triangulated by $X=(T_1,\ldots,T_N,\sim)$, and $\sigma$ an oriented normal closed curve in $S$. 
Truncating the tetrahedra $T_j$ at each vertex yields a triangulation of $S$ by triangles coming from vertices of $X$ (called the \textit{cusp triangulation}).
If the curve $\sigma$ intersects these triangles transversely (without back-tracking), then $\sigma$  cuts off corners of each such encountered triangle. Let us then denote $(z_1,\ldots,z_l)$ the sequence of (abstract) complex shape variables associated to these corners (each such $z_k$ is of the form $z_{T_{j_k}}, z'_{T_{j_k}}$ or $z''_{T_{j_k}}$).
Following \cite{FG}, we define the
 \textit{complex holonomy} $H^\C(\sigma)$  as 
$H^\C(\sigma):= \sum_{k=1}^l \epsilon_k \Log(z_k),$
where $\epsilon_k$ is $1$ if the $k$-th cut corner lies on the left of $\sigma$ and $-1$ if it lies on the right. The \textit{angular holonomy} $H^\R(\sigma)$ of $\sigma$ is similarly defined, replacing the term $\Log(z_k)$ by the (abstract) angle  $d_k=\arg(z_k)=\Im (\Log (z_k))$  (which is of the form $a_{T_{j_k}}$, $b_{T_{j_k}}$ or $c_{T_{j_k}}$) lying in the $i$-th corner. For example, in the triangulation of  Figure \ref{fig:trig:cusp:odd}, we have 
$$H^\C(m_{X_n})=\Log(z_U)-\Log(z_V) \text{ \ \ and \ \ } H^\R (m_{X_n}) = a_U-a_V.$$

The \textit{complex gluing edge equations} associated to $X$ consist in asking that the holonomies of each closed curve in $\partial M$ circling a vertex of the induced boundary triangulation are all equal to $2i\pi$, or in other words that
$$\forall e\in X^1, \omega_{X,\alpha}^\C(e) = 2i \pi.$$
The \textit{complex completeness equations} require that the complex holonomies of all curves generating the first homology $H_1(\partial M)$ vanish (when $M$ is of toroidal boundary).

If $M$ is an orientable $3$-manifold with boundary consisting of tori, and  ideally triangulated by $X$, then an angle structure $\alpha \in \mathcal{A}_X$  corresponds to the complete hyperbolic metric on the interior of $M$ (which is unique) if and only if $\alpha$ satisfies the complex gluing edge equations and the complex completeness equations.

\subsection{The classical dilogarithm}

For the dilogarithm function, we will use the definition:
$$ \Li(z) := - \int_0^z \Log(1-u) \frac{du}{u} \ \ \ \textrm{for} \ z \in \C \setminus [1,\infty)$$
(see for example \cite{Za}).
For $z$ in the unit disk, $\Li(z)=\sum_{n\geq 1} n^{-2} z^n$.
We will use the following properties of the dilogarithm function, referring for example to \cite[Appendix A]{AH} for the proofs.

\begin{proposition}[Some properties of $\Li$]\label{prop:dilog}
\

\begin{enumerate}
\item (inversion relation) $$ \forall z \in \C \setminus [1,\infty), \
 \Li\left (\frac{1}{z}\right ) = - \Li(z) - \frac{\pi^2}{6} - \frac{1}{2}\Log(-z)^2.
$$
\item (integral form) For all $y \in \R +i(-\pi,\pi)$,
$$ \frac{-i}{2 \pi} \Li(-e^y) =
\int_{v \in \R + i 0^+}
\dfrac{\exp\left (-i \frac{y v}{\pi}\right )}{4 v^2 \sinh(v)} \,  dv.
$$
 In the previous formula and in the remainder of the paper, $\R + i 0^+$ denotes a contour in $\C$ that is deformed from the horizontal line $\R \subset \C$ by avoiding $0$ via the upper half-plane (with a small half-circle for example).
\end{enumerate}
\end{proposition}

\subsection{The Bloch--Wigner function}

The \emph{Bloch--Wigner function} $D:\C \rightarrow \R $ defined by
\[
D(z) := \Im(\Li(z)) + \arg(1-z)\log \vert z \vert \quad \text{ if $z\in \C \smallsetminus \R$, and $0$ otherwise}
\]
is continuous on $\C$, and real-analytic on $\mathbb{C} \backslash \{0,1\}$ (see \cite[Section 3]{Za} for details). 
The Bloch-Wigner function plays a central role in hyperbolic geometry.  
The following result will be important for us (for a proof, see \cite{NZ}).

\begin{proposition}
Let $T$ be an ideal tetrahedron in $\mathbb{H}^3$ with complex shape structure $z$. Then, its volume is given by
\[
\Vol(T)= D(z) = D \left( \frac{z-1}{z} \right) = D \left( \frac{1}{1-z} \right).
\]
\end{proposition}

\subsection{Twist knots}

We denote by $K_n$ the unoriented twist knot with $n$ half-twists and $n+2$ crossings, according to Figure \ref{fig:twist:knot}.

\begin{figure}[h]
\begin{center}
\includegraphics[scale=0.5]{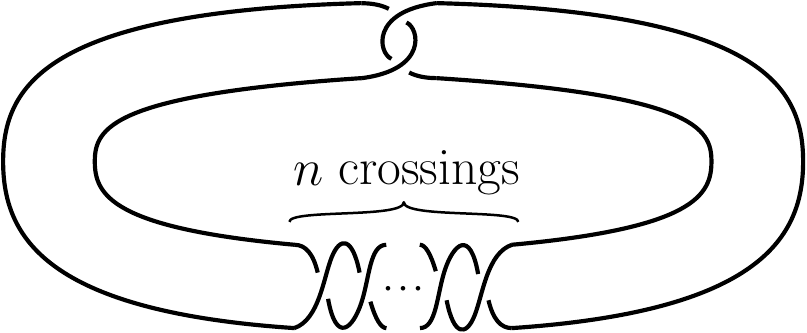}
\end{center}
\caption{The twist knot $K_n$} \label{fig:twist:knot}
\end{figure}

For clarity, we list the names of the  {first $13$}  twist knots in the table of Figure \ref{fig:table:twist:knot}, along with their hyperbolic volume and the coefficient of the Dehn filling one must apply on the Whitehead link to obtain 
(up to mirror image) the considered knot. The Dehn coefficient
is useful for studying $K_n$ for large $n$ on the software \textit{SnapPy} without having to draw a huge knot diagram by hand.

\begin{figure}[!h]
\begin{tabular}{|c|c|c|c|}
  \hline
$n$& $K_n$ & $\begin{matrix}
\text{Dehn Surgery coefficient} \\
\text{from the Whitehead link}
\end{matrix}$  & Hyperbolic volume  \\
  \hline
$0$& $0_1$ & $(1,0)$ & not hyperbolic  \\
  \hline
  $1$& $3_1$ & $(1,-1)$ & not hyperbolic \\
  \hline
  $2$& $4_1$ & $(1,1)$ & $2.02988321... $  \\
  \hline
  $3$& $5_2$ & $(1,-2)$ & $2.82812208... $  \\
  \hline
    $4$& $6_1$ & $(1,2)$ & $3.16396322... $  \\
  \hline
    $5$& $7_2$ & $(1,-3)$ & $3.33174423... $  \\
  \hline
    $6$& $8_1$ & $(1,3)$ & $3.42720524... $  \\
  \hline
    $7$& $9_2$ & $(1,-4)$ & $ 3.48666014...$  \\
  \hline
    $8$& $10_1$ & $(1,4)$ & $3.52619599... $  \\
  \hline
    $9$& $11_{a_{247}}$ & $(1,-5)$ & $3.55381991... $  \\
  \hline
    $10$& $12_{a_{803}}$ & $(1,5)$ & $ 3.57388254...$  \\
  \hline
    $11$& $13_{a_{3143}}$ & $(1,-6)$ & $ 3.588913917...$  \\
  \hline
    $12$& $14_{a_{12741}}$ & $(1,6)$ & $3.600467262... $  \\
  \hline
  
\end{tabular}
\caption{The first twist knots} \label{fig:table:twist:knot}
\end{figure}

The twist knots form, in a sense, the simplest infinite family of hyperbolic knots (for $n \geqslant 2$). This is why our initial motivation was to study the volume conjecture for the Teichm\"uller TQFT for this particular family (see \cite{BAPNcras}).

\begin{remark} The twist knots $K_{2n-1}$ and $K_{2n}$ are obtained, up to mirror image, by Dehn filling on one component of the Whitehead link with respective coefficients $(1,-n)$ and 
 $(1,n)$.
Replacing a twist knot by its inverse or mirror image has no effect on the modulus of the associated partition function of the Teichmüller TQFT (see Remark \ref{rem:mirror}).
	
	 Furthermore,  as a consequence of the J{\o}rgensen--Thurston theorem \cite{Th,NZ}, the hyperbolic volume of $K_n$ tends to 
	$3.6638623767088...$ 
	(the volume of the Whitehead link) as $n \to + \infty$.
\end{remark}

\subsection{Faddeev's quantum dilogarithm}

Recall \cite{AK} that for $\hbar >0$ and $\B >0$ such that $$(\B+\B^{-1}) \sqrt{\hbar} = 1,$$
 \emph{Faddeev's quantum dilogarithm} $\Phi_\B$ is the holomorphic function on $\R + i \left (\frac{-1}{2 \sqrt{\hbar}}, \frac{1}{2 \sqrt{\hbar}}\right )$ given by
$$
\Phi_\B(z) = \exp\left (
\frac{1}{4} \int_{w \in \R + i 0^+}
\dfrac{e^{-2 i z w} dw}{\sinh(\B w) \sinh({\B}^{-1}w) w}
\right ) \ \ \ \ \text{for} \ z \in \R + i \left (\frac{-1}{2 \sqrt{\hbar}}, \frac{1}{2 \sqrt{\hbar}}\right ),
$$
and extended to a meromorphic function for $z\in \C$ via the functional equation 
$$\Phi_\B\left (z-i \frac{\B^{\pm  1}}{2}\right )= \left (1+e^{2\pi \B^{\pm 1} z}\right )
\Phi_\B\left (z + i \frac{\B^{\pm 1}}{2}\right ).
$$
 Recall that $\R + i 0^+$ denotes a contour in $\C$ that is deformed from the horizontal line $\R \subset \C$ by avoiding $0$ by above.

Note that $\Phi_\B$ depends only on $\hbar = \frac{1}{(\B+\B^{-1})^2}$. Furthermore, as a consquence of the functional equation, the poles of $\Phi_\B$  lie on $ i \left [\frac{1}{2 \sqrt{\hbar}}, \infty\right ) $ and the zeroes lie symmetrically on $i \left (-\infty, \frac{-1}{2 \sqrt{\hbar}}\right ]$.  We stress the fact that in this paper we always assume that $\B$ is a real positive number, which simplifies several formulas in \cite[Appendix A]{AK}; notably the poles and zeroes live in the imaginary line instead of in sectors.

We now list several useful properties of Faddeev's quantum dilogarithm. We refer to \cite[Appendix A]{AK} for these properties (and several more), and to \cite[Lemma 3]{AH} for an alternate proof of the semi-classical limit property.

\begin{proposition}[Some properties of $\Phi_\B$]\label{prop:quant:dilog}
\

\begin{enumerate}
\item (inversion relation) For any $\B \in \R_{>0}$ and any  $z \in \R + i \left (\frac{-1}{2 \sqrt{\hbar}}, \frac{1}{2 \sqrt{\hbar}}\right )$, 
$$\Phi_\B(z) \Phi_\B(-z) = e^{i\frac{\pi}{12}(\B^2 + \B^{-2})} e^{i \pi z^2}.$$
\item (unitarity) For any $\B \in \R_{>0}$ and any  $z \in \R + i \left (\frac{-1}{2 \sqrt{\hbar}}, \frac{1}{2 \sqrt{\hbar}}\right )$, 
$$\overline{\Phi_\B(z)} = \frac{1}{\Phi_\B(\overline{z})}.$$
\item (semi-classical limit) For any $z \in \R + i \left (-\pi,\pi \right )$,
$$\Phi_\B\left (\frac{z}{2 \pi \B}\right ) = \exp\left (\frac{-i}{2 \pi \B^2} \Li (- e^z)\right ) \left ( 1 + O_{\B \to 0^+}(\B^2)\right ).$$
\item (behavior at infinity) For any  $\B \in \R_{>0}$,
\begin{align*}
 \Phi_\B(z) \ \ \underset{\Re(z)\to -\infty}{\sim} & \ \ 1, \\
 \Phi_\B(z) \ \  \underset{\Re(z)\to \infty}{\sim} & \ \ e^{i\frac{\pi}{12}(\B^2 + \B^{-2})} e^{i \pi z^2}.
\end{align*}
In particular, for any  $\B \in \R_{>0}$ and any $d \in \left (\frac{-1}{2 \sqrt{\hbar}}, \frac{1}{2 \sqrt{\hbar}}\right )$,
\begin{align*}
|\Phi_\B(x+id) | \ \ \underset{\R \ni x \to -\infty}{\sim} & \ \ 1, \\
|\Phi_\B(x+id) | \ \  \underset{\R \ni x \to +\infty}{\sim} & \ \ e^{-2 \pi x d}.
\end{align*}
\end{enumerate}
\end{proposition}

\subsection{The Teichm\"uller TQFT of Andersen--Kashaev}

In this section we follow \cite{AK, KaWB, Kan}. Let $\mathcal{S}(\R^d)$ denote the Schwartz space of smooth 
functions from $\R^d$ to $\C$ that are rapidly decaying (in the sense that any derivative decays faster than any negative power of the norm of the input). 
Its continuous dual $\mathcal{S}'(\R^d)$ is the space of tempered distributions.

Recall that the \emph{Dirac delta function} is the tempered distribution $\mathcal{S}(\R) \to \C$ denoted by $\delta(x)$ or $\delta$ and defined by
$
\delta(x) \cdot f:= \int_{x \in \R} \delta(x) f(x) dx =
f(0)
$ for all $f \in \mathcal{S}(\R)$ (where $x \in \R$ denotes the argument of $f\in \mathcal{S}(\R)$).
Furthermore, we have the equality of tempered distributions
\[
\delta(x)=\int_{w \in \R} e^{-2 \pi i x w} \,dw,
\] 
in the sense that for all $f \in \mathcal{S}(\R)$, 
$$
\left (\int_{w \in \R} e^{-2 \pi i x w} \,dw\right ) (f) =
\int_{x \in \R} \int_{w \in \R} e^{-2 \pi i x w} f(x)  \,dw \, dx \ = f(0) = \delta(x) \cdot f.
$$
The second equality follows from applying the Fourier transform $\mathcal{F}$ twice and using the fact that $\mathcal{F}(\mathcal{F}(f))(x) = f(-x)$ for $f\in \mathcal{S}(\R), x \in \R$. Recall also that the definition of the Dirac delta function and the previous argument have multi-dimensional analogues (see for example \cite{Kan} for details).

Given a triangulation $X$, 
writing $X^k$ for its collection of $k$-cells ($k\in \{0,1,2,3\}$), we assign to the tetrahedra $T_1, \ldots,T_N \in X^3$ formal real variables $t_1, \ldots, t_N$. 
We name $\mathsf{t}\colon T_j \mapsto t_{j}$ the corresponding bijection, and $\mathbf{t} = (t_{1},\ldots,t_{N})$ the corresponding
formal vector in $\R^{X^{3}}$.

Recall the notation $x_i(T) \in X^2$ for the $i$-th face ($i\in \{0,1,2,3\}$) of the tetrahedron $T\in X^3$.

We now define the kinematical kernel of $X$, which is a tempered distribution. Note that in many cases of interest (Lemma \ref{lem:kin:odd} and the proof of Theorem \ref{thm:even:part:func}), a distribution-free formula holds (Lemma~\ref{lem:dirac} below)  and might be used as an alternate definition. However this is not always the case: in the proofs of Theorems \ref{thm:part:func:Htrig:odd} and \ref{thm:part:func:Htrig:even}, each of the kinematical kernels associated to an \textit{H-triangulation} $Y_n$ is a distribution supported on a codimension-$2$ hyperplane. It is yet unclear whether these two types of kinematical kernels are the only ones that can appear.

\begin{definition}
Let $X$ be a triangulation such that $H_2(M_X\smallsetminus X^0,\Z)=0$. The \textit{kinematical kernel of $X$} is a tempered distribution $\mathcal{K}_X \in \mathcal{S}'\left (\R^{X^{3}}\right )$ defined by the integral
$$\mathcal{K}_X(\mathbf{t}) = \int_{\mathbf{x} \in \R^{X^{2}}} d\mathbf{x} \prod_{T \in X^3} e^{ 2 i \pi \varepsilon(T) x_0(T) \mathsf{t}(T)}
\delta\left ( x_0(T)- x_1(T)+ x_2(T)\right )
\delta\left ( x_2(T)- x_3(T)+ \mathsf{t}(T)\right )$$
where, with a slight abuse of notation, $x_i(T)$ 
refers to the $x_i(T)$-th component of $\mathbf{x}\in \R^{X^2}$.
(This convention, of denoting by $x_i(T)$ both a $2$-cell and the formal variable associated to it, is taken from~\cite{AK}: it will help keep our formulas short.)
\end{definition}

Essentially, if $\pi : \R^{X^2 \cup X^3} \rightarrow \R^{X^3}$ denotes the canonical projection, then
 $\mathcal{K}_X(\mathbf{t})$ associates to a Schwartz function $f:\R^{X^3} \rightarrow \R$ the (normalized) integral,
 over the affine subspace of $\R^{X^2 \cup X^3}$ where the arguments of the $\delta$'s vanish,
  of the product $(f \circ \pi) \cdot g$, where $g:\R^{X^2 \cup X^3} \rightarrow \R$ is the exponential of a certain quadratic form.
  
More formally, one
should understand the integral of the previous formula as the following equality of tempered distributions, similarly as above ( ${\!\top}$ denoting the transpose):
$$
\mathcal{K}_X(\mathbf{t}) =
\int_{\mathbf{x} \in \R^{X^{2}}} d\mathbf{x} 
\int_{\mathbf{w} \in \R^{2 N}} d\mathbf{w} \
e^{ 2 i \pi \mathbf{t}^{\!\top} R \mathbf{x}}
e^{ -2 i \pi \mathbf{w}^{\!\top} A \mathbf{x}}
e^{ -2 i \pi \mathbf{w}^{\!\top} B \mathbf{t}} \
\in \mathcal{S}'\left (\R^{X^{3}}\right ),
$$
where
$\mathbf{w}=(w_1,\ldots,w_N,w'_1, \ldots,w'_N)$ is a vector of $2N$ new real variables, such that $w_j,w'_j$ are associated to 
$\delta\left ( x_0(T_j)- x_1(T_j)+ x_2(T_j)\right )$ and
$\delta\left ( x_2(T_j)- x_3(T_j)+ \mathsf{t}(T_j)\right )$, and where
 $R,A,B$ are matrices with integer coefficients depending on the values $x_k(T_j)$, i.e.\ on the combinatorics of the face gluings. More precisely, the rows (resp.\ columns) of $R$ are indexed by the vector of tetrahedron variables $\mathbf{t}$ (resp.\ of face variables $\mathbf{x}$) and $R$ has a coefficient $\varepsilon(T_j)=\pm 1$ at coordinate $(t_j,x_0(T_j))$ and zero everywhere else; $B$ is indexed by  $\mathbf{w}$ (rows) and  $\mathbf{t}$ (columns) and has a $1$ at the coordinate $(w'_j,t_j)$; finally, $A$ is such that $A \mathbf{x} + B  \mathbf{t}$ is a column vector indexed by  $\mathbf{w}$ containing the values 
 $\left (x_0(T_j)- x_1(T_j)+ x_2(T_j)\right )_{1\leq j \leq N}$ followed by $\left (x_2(T_j)- x_3(T_j)+ t_j \right)_{1\leq j \leq N}$.

 \begin{example}\label{ex:41:RAB}
For the triangulation of $S^3 \setminus 4_1$ in Example \ref{ex:41}, if we denote $\mathbf{x} =(\mathsf{A},\mathsf{B},\mathsf{C},\mathsf{D})$, $\mathbf{t}=(t_{T_+}, t_{T_-})$ and $\mathbf{w}=(w_{T_+},w_{T_-},w'_{T_+},w'_{T_-})$, then we can calculate the three matrices
$$
R=\kbordermatrix{
	\mbox{} 	& \mathsf{A} & \mathsf{B} & \mathsf{C} & \mathsf{D} \\
	t_{T_+} 		& 0 & 1 & 0 & 0 \\
	t_{T_-} 		& 0 & 0 & 1 & 0 \\
},
\
A=\kbordermatrix{
	\mbox{} 	& \mathsf{A} & \mathsf{B} & \mathsf{C} & \mathsf{D} \\
	w_{T_+} 		& -1 & 1 & 1 & 0 \\
	w_{T_-} 		& 0 & 1 & 1 & -1 \\
	w'_{T_+} 		& 0 & 0 & 1 & -1 \\
	w'_{T_-} 		& -1 & 1 & 0 & 0 \\	
},
 \
 B=\kbordermatrix{
	\mbox{} 	& t_{T_+} & t_{T_-}\\
	w_{T_+} 		& 0 & 0 \\
	w_{T_-} 		& 0 & 0 \\
	w'_{T_+} 		& 1 & 0 \\
	w'_{T_-} 		& 0 & 1 \\
},
$$
with the unfortunate clash in notation regarding the letters $A, B$ and $\mathsf{A}, \mathsf{B}$.
One can see that the involution
$\mathsf{A} \leftrightarrow \mathsf{D},
\mathsf{B} \leftrightarrow \mathsf{C},
T_+ \leftrightarrow T_-$
preserves each matrix $R,A,B$;
this involution also acts as a reflection in Figure \ref{fig:id:tri:41:complement}: it swaps the edges while respecting their orientations (and reverses 3-dimensional orientation).
 \end{example}

\begin{lemma}\label{lem:dirac}
If the $2N\times 2N$ matrix $A$ in the previous formula is invertible, then the kinematical kernel is simply a bounded function given by:
$$ \mathcal{K}_X(\mathbf{t}) = \frac{1}{| \det(A) |} e^{ 2 i \pi \mathbf{t}^{\!\top} (-R A^{-1} B) \mathbf{t}}.
$$\end{lemma}

\begin{proof}
The lemma follows from the same argument as above (swapping integration symbols and applying the  Fourier transform $\mathcal{F}$ twice), this time for the multi-dimensional function
$f_{\mathbf{t}}:= \left (\mathbf{x} \mapsto e^{ 2 i \pi \mathbf{t}^{\!\top} R \mathbf{x}}\right ).$ More precisely:
\begin{align*}
 \mathcal{K}_X(\mathbf{t}) &=
\int_{\mathbf{x} \in \R^{X^{2}}} d\mathbf{x} 
\int_{\mathbf{w} \in \R^{2 N}} d\mathbf{w} \
e^{ 2 i \pi \mathbf{t}^{\!\top} R \mathbf{x}}
e^{ -2 i \pi \mathbf{w}^{\!\top} A \mathbf{x}}
e^{ -2 i \pi \mathbf{w}^{\!\top} B \mathbf{t}} \\
&= 
\int_{\mathbf{w} \in \R^{2N}} d\mathbf{w} \
e^{ -2 i \pi \mathbf{w}^{\!\top} B \mathbf{t}}
\int_{\mathbf{x} \in \R^{2N}} d\mathbf{x} \
f_{\mathbf{t}}(\mathbf{x})
e^{ -2 i \pi \mathbf{w}^{\!\top} A \mathbf{x}}\\
&= \int_{\mathbf{w} \in \R^{2N}} d\mathbf{w} \
e^{ -2 i \pi \mathbf{w}^{\!\top} B \mathbf{t}} \
\mathcal{F}\left (f_{\mathbf{t}}\right ) (A^{\!\top} \mathbf{w})\\
&=\frac{1}{| \det(A)|}
\int_{\mathbf{v} \in \R^{2N}} d\mathbf{v} \
e^{ -2 i \pi \mathbf{v}^{\!\top} A^{-1} B \mathbf{t}} \
\mathcal{F}\left (f_{\mathbf{t}}\right ) (\mathbf{v})\\
&=\frac{1}{| \det(A)|}
\mathcal{F}\left (\mathcal{F}\left (f_{\mathbf{t}}\right )\right ) (A^{-1} B \mathbf{t}) = \frac{1}{| \det(A)|}
f_{\mathbf{t}} (-A^{-1} B \mathbf{t})= \frac{1}{| \det(A) |} e^{ 2 i \pi \mathbf{t}^{\!\top} (-R A^{-1} B) \mathbf{t}}.
\end{align*}
\end{proof}

The product of several Dirac delta functions might not be a tempered distribution in general. However the kinematical kernels in this paper will always be, thanks to the assumption that $H_2(M_X\setminus X^0,\Z)=0$ (satisfied by  {any} knot complement). See \cite{AK} for more details, via the theory of wave fronts. The key property to notice is the linear independence of the terms $x_0(T_j)- x_1(T_j)+ x_2(T_j), \ x_2(T_j)- x_3(T_j)+ t_j$.

\begin{definition}
Let $X$ be a triangulation. Its \textit{dynamical content} associated to $\hbar>0$ is a function $\mathcal{D}_{\hbar,X}\colon \mathcal{A}_X \to  \mathcal{S}\left (\R^{X^{3}}\right )$ defined on each set of angles $\alpha \in \mathcal{A}_X$ by
$$\mathcal{D}_{\hbar,X}(\mathbf{t},\alpha)= \prod_{T\in X^{3}} 
\dfrac{\exp \left( \hbar^{-1/2} \alpha_3(T) \mathsf{t}(T) \right )}
{\Phi_\B\left (\mathsf{t}(T) - \dfrac{i}{2 \pi \sqrt{\hbar}}\varepsilon(T) (\pi-\alpha_1(T))\right )^{\varepsilon(T)}}.
$$
\end{definition}

Note that $\mathcal{D}_{\hbar,X}(\cdot,\alpha)$ is in $\mathcal{S}\left (\R^{X^{3}}\right )$ thanks to the properties of $\Phi_\B$ and the positivity of the dihedral angles in $\alpha$ (see \cite{AK} for details).

More precisely, each term in the dynamical content has exponential decrease as described in the following lemma.

\begin{lemma}\label{lem:dec:exp}
	Let $\B \in \R_{>0}$ and $a,b,c \in (0,\pi)$ such that $a+b+c=\pi$. Then
	$$
	\left |
	\dfrac{e^{\frac{1}{ \sqrt{\hbar}} c x}}{\Phi_\B\left (x-\frac{i}{ 2 \pi \sqrt{\hbar}}(b+c)\right )}
	\right | \underset{\R \ni x \to \pm \infty}{\sim} \left |
	e^{\frac{1}{ \sqrt{\hbar}} c x} \Phi_\B\left (x+\frac{i}{ 2 \pi \sqrt{\hbar}}(b+c)\right )
	\right | \ \ \left \{
	\begin{matrix}
	\underset{\R \ni x \to -\infty}{\sim} e^{\frac{1}{ \sqrt{\hbar}} c x}. \\
	\ \\
	\underset{\R \ni x \to +\infty}{\sim} e^{-\frac{1}{ \sqrt{\hbar}} b x}.
	\end{matrix} \right .
	$$
\end{lemma}

\begin{proof}
The lemma immediately follows from Proposition \ref{prop:quant:dilog} (4).
\end{proof}

Lemma \ref{lem:dec:exp} illustrates why we need the three angles $a,b,c$ to be in $(0,\pi)$: $b$ and $c$ must be positive in order to have exponential decrease in both directions, and $a$ must be  {positive} as well so that $b+c < \pi$ and $\Phi_\B\left (x \pm \frac{i}{ 2 \pi \sqrt{\hbar}}(b+c)\right )$ is always defined.

Now, for $X$ a triangulation such that $H_2(M_X\setminus X_0,\Z)=0$, $\hbar>0$ and $\alpha \in \mathcal{A}_X$ an angle structure,  the associated \textit{partition function of the Teichm\"uller TQFT} is the complex number:
$$\mathcal{Z}_{\hbar}(X,\alpha)= \int_{\mathbf{t} \in \R^{X^3}}  \mathcal{K}_X(\mathbf{t}) \mathcal{D}_{\hbar,X}(\mathbf{t},\alpha) d\mathbf{t} \ \ \ \in \C. $$

Andersen and Kashaev proved in \cite{AK} that the  {modulus} $\left |\mathcal{Z}_{\hbar}(X,\alpha) \right | \in \R_{>0}$ is invariant under Pachner moves with positive angles, and then generalised this property to a larger class of moves and triangulations with angles, using analytic continuation in complex-valued $\alpha$ \cite{AKicm}.

\begin{remark}\label{rem:mirror}
If we denote $X^\sharp$ the \textit{mirror image} of the triangulation $X$ (obtained by applying a reflection to each tetrahedron), then all tetrahedron signs $\varepsilon(T_j)$ are multiplied by $-1$. Therefore, it follows from the definition of the Teichm\"uller TQFT and Proposition \ref{prop:quant:dilog} (2) that 
$\mathcal{Z}_{\hbar}(X^\sharp,\alpha) = \overline{\mathcal{Z}_{\hbar}(X,\alpha)},$
and thus $\left |\mathcal{Z}_{\hbar}(X^\sharp,\alpha) \right | = \left |\mathcal{Z}_{\hbar}(X,\alpha) \right |$. Consequently, the following results will stand for the twist knots $K_n$ of Figure \ref{fig:twist:knot} and their mirror images $K^\sharp_n$.
\end{remark}

We can now state our version of the \textit{volume conjecture} for the Teichm\"uller TQFT, in a slightly different (and less powerful) way  {than the one in} Andersen--Kashaev in \cite[Conjecture 1]{AK}. Notably, we make the statements depend on specific chosen triangulations $X$ and $Y$; thus we will not be interested in the present paper in how the following properties change under Pachner moves or depend on the triangulations. For some insights on these points, see \cite{AK}. We also introduced a new combination of angles $\mu_X$, which has  {an} interesting topological origin.

\begin{conj}[see \cite{AK}, Conjecture 1] \label{conj:vol:BAGPN}
	Let $M$ be a connected closed oriented $3$-manifold and let $K \subset M$ be a hyperbolic knot.
There exist an ideal triangulation $X$ of $M \setminus K$ and a one-vertex H-triangulation $Y$ of $(M,K)$ such that $K$ is represented by an edge $\overrightarrow{K}$ in a single tetrahedron $Z$ of $Y$, and $\overrightarrow{K}$ has only one pre-image. Moreover,
there exists a function $J_X\colon \R_{>0} \times \C \to \C$ such that 
 the following properties hold:
\begin{enumerate}
\item   There exist $\mu_X, \lambda_X$
 linear combinations  of dihedral angles in $X$ such that
for all angle structures $\alpha \in \mathcal{A}_{X}$ and all $\hbar>0$, we have:
\begin{equation*}
\left |\mathcal{Z}_{\hbar}(X,\alpha) \right |
= \left | 
\int_{\mathbb{R}+i \frac{\mu_{X}(\alpha) }{2\pi \sqrt{\hbar}}  } 
J_{X}(\hbar,x)
e^{\frac{1}{2 \sqrt{\hbar}}  x  \lambda_{X}(\alpha)} 
dx \right |.
\end{equation*}
Moreover, if $M=S^3$, then $J_X$ can be chosen such that $\mu_X, \lambda_X$ are angular holonomies associated to a meridian and a preferred longitude of $K$.
\item For every $\B>0$, and for every 
$\tau\in \mathcal{S}_{Y \setminus Z} \times \overline{\mathcal{S}_Z}$
such that $\omega_{Y,\tau}$ vanishes on the edge $\overrightarrow{K}$ and is equal to $2\pi$ on every other edge, one has, denoting $\hbar = \frac{1}{(\B+\B^{-1})^2}$:
\begin{equation*}
\underset{\tiny \begin{matrix}\alpha \to \tau \\ \alpha \in \mathcal{S}_{Y} \end{matrix}}{\lim} \left |
\Phi_{\B}\left( \frac{\pi-\omega_{Y,\alpha}\left (\overrightarrow{K}\right )}{2\pi i \sqrt{\hbar}} \right)  \mathcal{Z}_{\hbar}(Y,\alpha)\right | = \left | J_{X}(\hbar,0)\right |,
\end{equation*}
\item In the semi-classical limit $\hbar \to 0^+$, we retrieve the hyperbolic volume of $K$ as:
$$
\lim_{\hbar \to 0^+} 2\pi \hbar  \log \vert J_{X}(\hbar,0) \vert
= -\Vol(M\backslash K).$$
\end{enumerate}
\end{conj}

The rest of the paper consists in proving Conjecture \ref{conj:vol:BAGPN} for the infinite family of hyperbolic twist knots (in Theorems
\ref{thm:trig},  
\ref{thm:part:func}, 
\ref{thm:part:func:Htrig:odd}, 
\ref{thm:vol:conj},
\ref{thm:even:part:func}, 
\ref{thm:part:func:Htrig:even}
 and 
\ref{thm:even:vol:conj}). Several remarks are in order concerning Conjecture \ref{conj:vol:BAGPN}.

\begin{remark}
In Conjecture \ref{conj:vol:BAGPN} (1), one may notice that $J_X,  \mu_X$ and $\lambda_X$ are not unique, since one can for example replace $(J_X(\hbar,x),x,\mu_X,\lambda_X)$ by
\begin{itemize}
\item either $(J_X(\hbar,x)e^{-\frac{1}{2 \sqrt{\hbar}}C x},x,\mu_X,\lambda_X+C)$ for any constant $C \in \R$,
\item or $(D J_X(\hbar,D x'),x',\mu_X/D,D \lambda_X)$ for any constant $D \in \R^*$ (via the change of variable $x'=x/D$).
\end{itemize}
Note however that in both cases, the expected limit $\lim_{\hbar \to 0^+} 2\pi \hbar \log \vert J_{X}(\hbar,0) \vert$ does not change. When $M=S^3$, a promising way to reduce ambiguity in the definition of $J_X$ is to impose that $\mu_X(\alpha)$ and $\lambda_X(\alpha)$  are uniquely determined as the angular holonomies of a meridian and a preferred longitude  of the knot $K$. In proving Conjecture \ref{conj:vol:BAGPN} (1) for the twist knots in Theorems \ref{thm:part:func} and \ref{thm:even:part:func}, we find such properties for $\mu_X$ and $\lambda_X$.
\end{remark}

\begin{remark}
The function $(\hbar \mapsto J_X(\hbar,0))$ should play the role of the Kashaev invariant in the comparison with the Kashaev--Murakami--Murakami volume conjecture \cite{Ka95,MM}. Notably, the statement of Conjecture \ref{conj:vol:BAGPN} (2) has a similar form as the definition of the Kashaev invariant in \cite{Ka94} and Conjecture \ref{conj:vol:BAGPN} (3) resembles the volume conjecture stated in \cite{Ka95}, where $\hbar$ corresponds to the inverse of the color $N$.	
\end{remark}

\begin{remark}
The final form of the Teichm\"uller TQFT volume conjecture is not yet set in stone, notably because of the suboptimal definitions of the function $(\hbar \mapsto J_X(\hbar,0))$ (in Conjecture \ref{conj:vol:BAGPN} (1) and (2))
and the uncertain invariance of the variables and statements under (ordered) Pachner moves. Nevertheless, we hope Conjecture \ref{conj:vol:BAGPN} as stated here and its resolution can help us understand better how to solve these difficulties in the future.	
\end{remark}

\subsection{Saddle point method}

Let $n\geqslant 1$ be an integer.
Recall \cite{Kr} that a complex-valued function $(z_1,\ldots,z_n) \mapsto S(z_1,\ldots,z_n)$ defined on an open subset of $\C^n$ is called \textit{analytic} (or \textit{holomorphic}) if it is analytic in every variable (as a function of one complex variable).
Moreover, its \textit{holomorphic gradient} $\nabla S$  is the function valued in $\C^n$ whose coordinates are the partial derivatives $\dfrac{\partial S}{\partial z_j}$, and its \textit{holomorphic hessian} $\mathrm{Hess}(S)$ is the $n\times n$ matrix with  {coefficients} the second partial derivatives $\dfrac{\partial^2 S}{\partial z_j z_k}$; in both of these cases, the \textit{holomorphic}  {designation} comes from the absence of partial derivatives of the form $\dfrac{\partial }{\partial \overline{z_j}}$.

The \textit{saddle point method} is a general name for studying asymptotics of integrals of the form $\int f e^{\lambda S}$ when $\lambda\to +\infty$. The main contribution is expected to be the value of the integrand at a saddle point of $S$ maximizing $\Re S$. For an overview of such methods, see \cite[Chapter II]{Wo}.

Before going in detail in the saddle point method, let us recall the notion of asymptotic expansion.

\begin{definition}
Let $f:\Omega \to \C$ be a function where $\Omega \subset \C$ is unbounded. A complex power series $\sum_{n=0}^{\infty} a_n z^{-n}$ (either convergent or divergent) is called an \emph{asymptotic expansion} of $f$ if, for every fixed integer $N \geq 0$, one has
\[
f(z) = \sum_{n=0}^{N} a_nz^{-n} + O\left (z^{-(N+1)}\right )
\]
when $z \to \infty$. 
In this case, one denotes
\[
f(z) 
\underset{z \to \infty}{\cong}
 \sum_{n=0}^{\infty} a_nz^{-n}.
\]
\end{definition}

For various properties of asymptotic expansions, see \cite{Wo}.

The following theorem is due to Fedoryuk and can be found in \cite[Section 2.4.5]{Fe2} (for the statement) and in \cite[Chapter 5]{Fe1} (for the details and proofs, in Russian).
 {Compare also with
\cite[Theorem 4.2]{PV} (in English).}
 To our knowledge, this is the only version of the saddle point method in the literature for $f,S$ analytic functions in several complex variables.

\begin{theorem}[Fedoryuk]\label{thm:SPM}
Let $m\geqslant 1$ be an integer, and
$ \gamma^m $ an $m$-dimensional smooth compact real sub-manifold of $\C^m$ with connected boundary.
We denote $z=(z_1,\ldots,z_m) \in \C^m$ and $dz= dz_1\cdots dz_m$. Let $z\mapsto f(z)$ and $z\mapsto S(z)$ be two complex-valued functions analytic on a domain $D$ such that $\gamma^m \subset D \subset \C^m$. We consider the integral
$$F(\lambda) = \int_{\gamma^m} f(z) \exp(\lambda  S(z)) \, dz,$$
with parameter $\lambda \in \R$.

Assume that $\max_{z \in \gamma^m} \Re S(z)$ is attained only at a point $z^0$, which is an interior point of $\gamma^m$ and  a simple saddle point of $S$ (i.e.\ $\nabla S(z^0)=0$ and $\det \mathrm{Hess}(S)(z^0) \neq 0$).

Then as $\lambda \to + \infty$, there is the asymptotic expansion
$$F(\lambda) 
\underset{\lambda \to \infty}{\cong}
 \left (\dfrac{2\pi}{\lambda}\right )^{m/2} \dfrac{\exp \left ( \lambda  S(z^0)\right )}{\sqrt{\det \mathrm{Hess}(S)(z^0)}}
\left [
f(z^0) + \sum_{k=1}^\infty c_k \lambda^{-k}
\right ],
$$
where the $c_k$ are complex numbers and the choice of branch for the root $\sqrt{\det  \mathrm{Hess}(S)(z^0)}$ depends on the orientation of the contour $\gamma^m$.

In particular, $\lim_{\lambda \to + \infty}
\frac{1}{\lambda} \log \vert F(\lambda) \vert = \Re S(z^0)$.
\end{theorem}

\subsection{ {Notation} and conventions}
Let $p \in \N$. In the various following sections, we will use the following recurring conventions:
\begin{itemize}
\item A roman letter in \textbf{bold} will denote a vector of $p+2$ variables (often integration variables), which are the aforementioned letter indexed by $1, \ldots, p,U, W$. For example,
$ \mathbf{y} = (y_1,\ldots,y_p,y_U,y_W)$.
\item A roman letter in \textbf{bold} and with a tilde $\widetilde{ \ }$ will have $p+3$ variables indexed by $1, \ldots, p, U,V,W$. For example,
$ \widetilde{\mathbf{y}}' = (y'_1,\ldots,y'_p,y'_U,y'_V,y'_W)$. 
\item Matrices and other vectors of size $p+3$ will also wear a tilde but will not necessarily be in bold, for example $\widetilde{C}(\alpha)=(c_1,\ldots,c_p,c_U,c_V,c_W)$.
\item A roman letter in \textbf{bold} and with a hat $\widehat{ \ }$ will have $p+4$ variables indexed by $1, \ldots, p, U,V,W, Z$. For example,
$ \widehat{\mathbf{t}} = (t_1,\ldots,t_p,t_U,t_V,t_W,t_Z)$.
\end{itemize}
For $j \in \{1,\ldots,p,U,V,W,Z\}$, we will also use the conventions that:
\begin{itemize}
\item the symbols $e_j, f_j$ are faces of a triangulation (for $j \in \{1,\ldots,p\}$),
\item the symbol $\overrightarrow{\eta_j}$ is an edge of a triangulation (for $j \in \{1,\ldots,p\}$),
\item the integration variable $t_j$ lives in $\R$,
\item the symbols $a_j,b_j,c_j$ are angles in $(0,\pi)$ (sometimes $[0,\pi]$) with sum $\pi$,
\item the integration variable $y'_j$ lives in $\R \pm \frac{i(\pi-a_j)}{2 \pi \sqrt{\hbar}}$,
\item the integration variable $y_j$ live{}s in $\R \pm i(\pi-a_j)$,
\item the symbols $x_j, d_j$ are the real and imaginary part of $y_j$,
\item  the symbol $z_j$ lives in $\R + i \R_{>0}$,
\end{itemize}
and are (each time) naturally associated to the tetrahedron $T_j$. Moreover, we will simply  {write} $U,V,W,Z$ for the tetrahedra $T_U,T_V,T_W,T_Z$.

\section{New triangulations for the twist knots}\label{sec:trig}

We describe the construction of new triangulations for the twist knots, starting from a knot diagram and using an algorithm introduced by Thurston in \cite{Th2} and refined in \cite{Me, KLV}. For the odd twist knots the details are in this section, and for the even twist knots they are in Section~\ref{sec:appendix}.

\subsection{Statement of results}

\begin{theorem}\label{thm:trig}
For every $n\geqslant 3$ odd (respectively for every $n\geqslant 2$ even), the triangulations $X_n$ and $Y_n$ represented in Figure \ref{fig:trig:odd} (respectively in Figure \ref{fig:trig:even}) are an ideal triangulation of $S^3 \setminus K_n$ and an H-triangulation of $(S^3,K_n)$ respectively.
\end{theorem}

\begin{figure}[!h]
\begin{center}
\includegraphics[scale=0.4,angle=90]{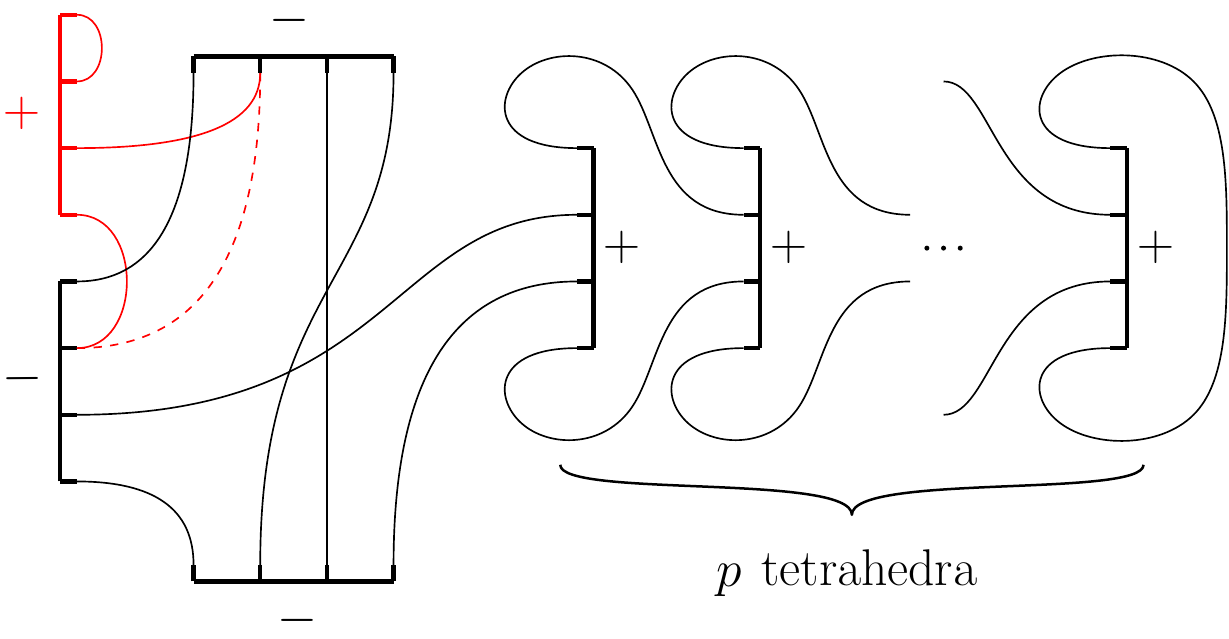}
\end{center}
\caption{An H-triangulation $Y_n$ of $(S^3,K_n)$ (full red part) and an ideal triangulation $X_n$ of $S^3 \setminus K_n$ (dotted red part), for odd $n\geqslant 3$, with $p=\frac{n-3}{2}$.} \label{fig:trig:odd}
\end{figure}

\begin{figure}[!h]
\begin{center}
\includegraphics[scale=0.4,angle=90]{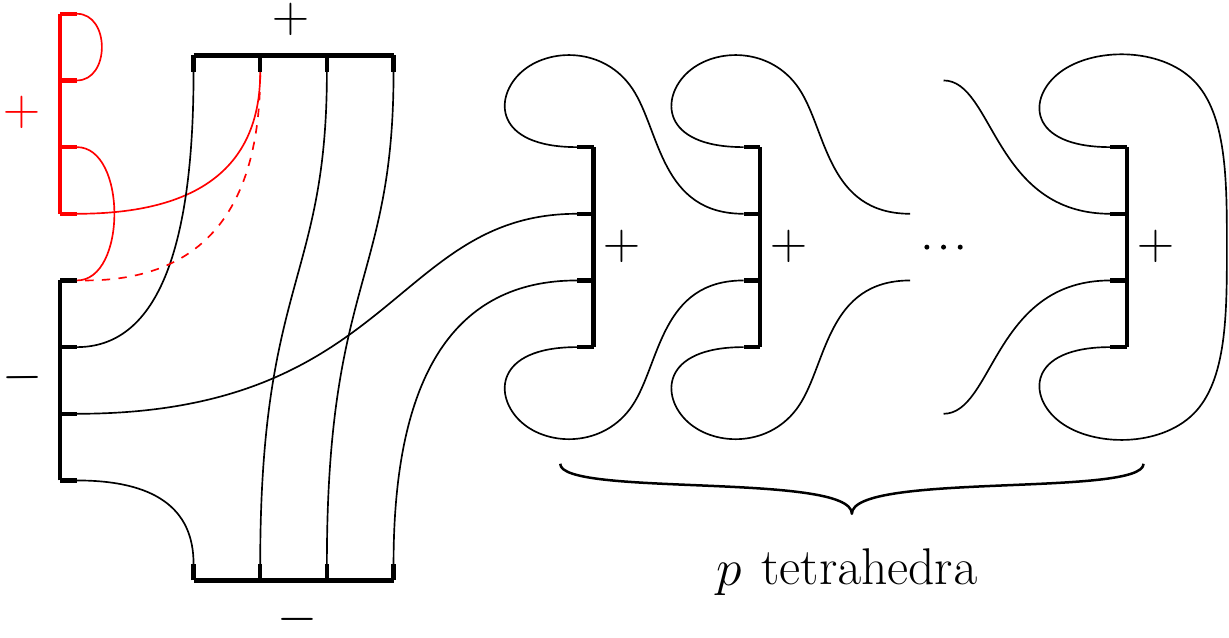}
\end{center}
\caption{An H-triangulation $Y_n$ of $(S^3,K_n)$ (full red part) and an ideal triangulation $X_n$ of $S^3 \setminus K_n$ (dotted red part), for even $n\geqslant 2$, with $p=\frac{n-2}{2}$.} \label{fig:trig:even}
\end{figure}

Figures \ref{fig:trig:odd} and \ref{fig:trig:even} display an H-triangulation $Y_n$ of $(S^3,K_n)$, and the corresponding ideal triangulation $X_n$ of $S^3 \setminus K_n$ is obtained by replacing the upper left red tetrahedron (partially glued to itself) by the dotted line (note that we omitted the numbers $0,1,2,3$ of the vertices for simplicity). Theorem \ref{thm:trig} is proven by applying an algorithm due to Thurston (later refined by Menasco and Kashaev--Luo--Vartanov) to construct a polyhedral decomposition of $S^3$ where the knot $K_n$ is one of the edges, starting from a diagram of $K_n$; along the way we apply a combinatorial trick to reduce the number of edges and we finish by choosing a convenient triangulation of the polyhedron. Once we have the H-triangulation of $(S^3,K_n)$, we can collapse both the edge representing the knot $K_n$ and its  {one} underlying tetrahedron to obtain an ideal triangulation of $S^3 \setminus K_n$. This is detailed in Section \ref{sub:trig:odd} (for odd $n$) and in Section \ref{sub:even:trig} (for even $n$).

\subsection{Consequences  {for} Matveev complexity}

An immediate consequence of Theorem \ref{thm:trig} is a new upper bound for the Matveev complexity of a general twist knot complement. Recall that the Matveev complexity $\mathfrak{c}(S^3\setminus K)$ of a knot complement is equal to the minimal number of tetrahedra in an ideal triangulation of this knot complement $S^3\setminus K$ (see \cite{Ma} for this definition and the original wider definition using simple spines).

\begin{corollary}\label{cor:complexity}
Let $n\geqslant 2$. Then the Matveev complexity $\mathfrak{c}\left (S^3\setminus K_n\right )$ of the $n$-th twist knot complement satisfies:
$$ \mathfrak{c}\left (S^3\setminus K_n\right ) \leqslant \left \lfloor \frac{n+4}{2} \right  \rfloor.$$
\end{corollary}

Corollary \ref{cor:complexity} follows immediately from Theorem \ref{thm:trig} and is of double interest.

 Firstly, this new upper bound, which is roughly half the crossing number of the knot, is stricly better  {than} the upper bounds currently  {in the literature.} Indeed, the usual upper bound for $\mathfrak{c}\left (S^3\setminus K_n\right )$ is roughly $4$ times the crossing number (see for example \cite[Proposition 2.1.11]{Ma}); a better upper bound for two-bridge knots is given in \cite[Theorem 1.1]{IN}, and is equal to $n$ for the $n$-th twist knot $K_n$.
 
 Secondly, experiments on the software \textit{SnapPy} lead us to conjecture that the bound of Corollary \ref{cor:complexity} is actually an exact value. 

\begin{conj}\label{conj:matveev}
Let $n\geqslant 3$. Then the Matveev complexity $\mathfrak{c}\left (S^3\setminus K_n\right )$ of the $n$-th twist knot complement satisfies:
$$ \mathfrak{c}\left (S^3\setminus K_n\right ) = \left \lfloor \frac{n+4}{2} \right  \rfloor.$$
\end{conj}

\begin{remark}\label{rem:snappy}
The statement of Conjecture \ref{conj:matveev} holds at least for $3 \leq n \leq 17$, i.e.\ when the triangulation $X_n$ has up to $10$ tetrahedra. Indeed, Burton \cite{BB} has systematically listed all ideal manifolds of up to $9$ tetrahedra (and homeomorphisms between pairs thereof), which makes a search possible --- using certified volume computations under Sage  to ensure a manifold is \emph{not} in the list. The computations can be done on SnapPy, with code such as :\\
\\
\texttt{
In[1]: W = Manifold('5\^}\texttt{2\_1') \\
In[2]: W.dehn\_fill((1,6),0) \\
In[3]: W.identify() \\
Out[3]: [t00017(0,0), K8\_1(0,0), K14a12741(0,0)] \\
}
\\
\texttt{
sage: W = snappy.Manifold('5\_1\^}\texttt{2')\\
sage: W.dehn\_fill((1,8),0)\\
sage: W.volume(verified=True)\\
3.627534484691?\\
sage: L = snappy.OrientableCuspedCensus[3.6275:3.6276]\\
sage: len(L)\\
0}\\

The first block of the above code determines if the manifold obtained by Dehn filling with coefficient $(1,6)$ on the Whitehead link complement (i.e.\ the twist knot complement $S^3 \setminus K_{12}$, recall Figure \ref{fig:table:twist:knot}) lies in the current census. This manifold is identified as \texttt{K8\_1}, namely the first
knot complement in the census that is triangulated with $8$ tetrahedra.  Since it cannot be triangulated with fewer tetrahedra because of the completeness of the enumeration, its Matveev complexity is indeed equal to $\left \lfloor \frac{12+4}{2} \right  \rfloor = 8$ as expected in Conjecture \ref{conj:matveev}.  The same conclusion holds if we replace $12$ with any value in $\{3,...,15\}$.

For $n \in \{16, 17\}$, Corollary \ref{cor:complexity} yields the upper bound $10$ for the Matveev complexity of $S^3 \setminus K_n$, but as the second block of code shows (for $n=16$ and the $(1,8)$ Dehn filling coefficient), $S^3 \setminus K_n$ is not in the census of ideal manifolds with Matveev complexity $9$ or less.

For general $n$, no lower bounds for  $\mathfrak{c}\left (S^3\setminus K_n\right )$ have yet been found, to our knowledge.
\end{remark}

In the rest of this section, we present one last lead that gives credence to Conjecture \ref{conj:matveev}, via the notion of complexity of \textit{pairs}. 

As defined in \cite{PP}, the Matveev complexity $\mathfrak{c}\left (S^3, K_n\right )$ of the knot $K_n$ in $S^3$ is the minimal number of tetrahedra in a triangulation of $S^3$ where $K_n$ is the union of some quotient edges. Since H-triangulations (as defined in this article) are such triangulations, we deduce from Theorem \ref{thm:trig} the following corollary:

\begin{corollary}\label{cor:complexity:pair}
	Let $n\geqslant 2$. Then the Matveev complexity $\mathfrak{c}\left (S^3, K_n\right )$ of the $n$-th twist knot in $S^3$ satisfies:
	$$ \mathfrak{c}\left (S^3, K_n\right ) \leqslant \left \lfloor \frac{n+6}{2} \right  \rfloor.$$
\end{corollary}

 {The upper bound of $\left \lfloor \frac{n+6}{2} \right  \rfloor$ for the knots $K_n$ in Corollary \ref{cor:complexity:pair} is better than the upper bound of $4n+10$ in \cite[Propostion 5.1]{PP}.}
For these same knots $K_n$, the best lower bound to date seems to be linear in $\log(n)$, see  \cite[Theorem 5.4]{PP}. Still, we offer the following conjecture:

\begin{conj}\label{conj:matveev:pair}
	Let $n\geqslant 3$. Then the Matveev complexity $\mathfrak{c}\left (S^3, K_n\right )$ of the $n$-th twist knot in $S^3$ satisfies:
	$$ \mathfrak{c}\left (S^3, K_n\right ) = \left \lfloor \frac{n+6}{2} \right  \rfloor.$$
\end{conj}

If true, Conjecture \ref{conj:matveev:pair} would  {have the surprising consequence} that the H-triangulation $Y_n$ of cardinality $\left \lfloor \frac{n+6}{2} \right  \rfloor$ would be minimal although it has the double restriction that the knot $K_n$ lies in only one edge of the triangulation of $S^3$ and that $Y_n$ admits a vertex ordering.

Conjectures \ref{conj:matveev} and \ref{conj:matveev:pair} are equivalent if and only if the following question admits a positive answer:

\begin{question}\label{quest:matveev}
	Let $n\geqslant 2$. Do the respective Matveev complexities of the $n$-th twist knot complement  and of the $n$-th twist knot in $S^3$  {always} differ by $1$, i.e.\  
	 do we always have
	$$ \mathfrak{c}\left (S^3, K_n\right ) = \mathfrak{c}\left (S^3\setminus K_n\right )+1 \ \ ?$$
\end{question}

Question \ref{quest:matveev} looks far from easy to solve, though. On one hand, it is not clear that the minimal triangulation for the pair $(S^3,K_n)$ can always yield an ideal triangulation for $S^3\setminus K_n$ by collapsing exactly one tetrahedron (which is the case for $X_n$ and $Y_n$ as we will see in the following section). On the other hand, it is not clear that one can always construct an H-triangulation of  $(S^3,K_n)$ from an ideal triangulation of $S^3\setminus K_n$ by adding only one tetrahedron. 

The previously mentioned lower bound linear in $\log(n)$ for $\mathfrak{c}\left (S^3, K_n\right )$ comes from the general property that
 $$\frac{1}{2} \mathfrak{c}(M_n) \leqslant \mathfrak{c}\left (S^3, K_n\right )$$
  where $M_n$ is the double branched cover of $(S^3,K_n)$ \cite[Proposition 5.2]{PP}. Here $M_n$ happens to be the lens space $L(2n+1,n)$ (see for example \cite[Section 12]{BZH}), whose Matveev complexity is not yet known but conjectured to be $n-1$ through a general conjecture on the complexity of lens spaces \cite[Section 2.3.3 page 77]{Ma}.
However, the  current best lower bounds for complexities of lens spaces such as $L(2n+1,n)$ are linear in $\log(n)$, as a consequence of \cite[Section 5.2]{PP}.

 Hence, if the lens space complexity conjecture holds, then we would have from Corollary \ref{cor:complexity:pair} the double bound
 $$ \left \lceil \dfrac{n-1}{2} \right  \rceil \leqslant  \mathfrak{c}\left (S^3, K_n\right ) \leqslant \left \lfloor \frac{n+6}{2} \right  \rfloor,$$
 which would imply that $\mathfrak{c}\left (S^3, K_n\right )$ can only take  four possible values. All this makes Conjecture \ref{conj:matveev:pair} sound more plausible, and Conjecture \ref{conj:matveev} as well by extension.

\subsection{Construction for odd twist knots}\label{sub:trig:odd}

We first consider a general twist knot $K_n$ for $n\geqslant 3$, $n$ odd. We will construct an H-triangulation of $(S^3,K_n)$ and an ideal triangulation of $S^3 \setminus K_n$ starting from a knot diagram of $K_n$. The method dates back to Thurston \cite{Th} and was also described in more detail in \cite{KLV, Me}.

\begin{figure}[!h]
\centering
\begin{tikzpicture}[every path/.style={string ,black} , every node/.style={transform shape , knot crossing , inner sep=1.5 pt } ]

\begin{scope}[scale=0.7]

\begin{scope}[dashed,decoration={
    markings,
    mark=at position 0.5 with {\arrow{>}}}
    ] 
    \draw[postaction={decorate}] (-2,-2)--(-4,-4);
    \draw[postaction={decorate}] (-7,1)--(-4,-4);
    \draw[postaction={decorate}] (-2,-2)--(-4,0);
    \draw[postaction={decorate}] (-7,1)--(-4,0);
    \draw[postaction={decorate}] (-7,1)--(-2,6);
    \draw[postaction={decorate}] (-2,2)--(-4,0);
    \draw[postaction={decorate}] (-2,2)--(-2,6);
    \draw[postaction={decorate}] (6,3)--(2,2);
    \draw[postaction={decorate}] (4,0)--(2,2);
    \draw[postaction={decorate}] (6,3)--(7,0);
    \draw[postaction={decorate}] (4,0)--(7,0);
    \draw[postaction={decorate}] (3,-4)--(7,0);
    \draw[postaction={decorate}] (3,-4)--(2,-2);
    \draw[postaction={decorate}] (4,0)--(2,-2);
\end{scope}

\begin{scope}[dashed,decoration={
    markings,
    mark=at position 0.5 with {\arrow{>>}}}
    ] 
    \draw[postaction={decorate}] (-2,2)--(2,2);
    \draw[postaction={decorate}] (-2,2)--(-1,4);
    \draw[postaction={decorate}] (1,4)--(2,2);
    \draw[postaction={decorate}] (1,4)--(-1,4);
    \draw[postaction={decorate}] (1,4)--(3,6);
    \draw[postaction={decorate}] (-2,6)--(-1,4);
    \draw[postaction={decorate}] (-2,6)--(3,6);
\end{scope}

\draw[style=dashed] (1,4) -- (6,3);

\begin{scope}[xshift=3.5cm, yshift=3.5cm, rotate=-100, scale=0.2]
\draw (1,0) -- (0,1) -- (-1,0) -- (1,0);
\end{scope}

\draw[color=blue, line width=0.5mm] (3,6) -- (6,3); 
\draw[color=blue, line width=0.5mm] (2,2) -- (7,0);
\draw[color=blue, line width=0.5mm] (4,0) -- (3,-4);
\draw[color=blue, line width=0.5mm] (-7,1) -- (-2,-2);
\draw[color=blue, line width=0.5mm] (-4,0) -- (-2,6);
\draw[color=blue, line width=0.5mm] (3,6) -- (-1,4);
\draw[color=blue, line width=0.5mm] (-2,2) -- (1,4);

\draw[color=blue, line width=0.5mm] (2,-2) -- (3.3,-1.5);
\draw[color=blue, line width=0.5mm] (4,-1.25) -- (7,0);

\draw[color=blue, line width=0.5mm] (4,0) -- (4.5,0.7);
\draw[color=blue, line width=0.5mm] (4.8,1.2) -- (6,3);

\draw[color=blue, line width=0.5mm] (-1,4) -- (-0.2,3.5);
\draw[color=blue, line width=0.5mm] (0.2,3.2) -- (2,2);

\draw[color=blue, line width=0.5mm] (1,4) -- (0.3,4.45);
\draw[color=blue, line width=0.5mm] (-0.1,4.7) -- (-2,6);

\draw[color=blue, line width=0.5mm] (-7,1) -- (-3.7,1.6);
\draw[color=blue, line width=0.5mm] (-3.1,1.75) -- (-2,2);

\draw[color=blue, line width=0.5mm] (-4,0) -- (-4,-0.6);
\draw[color=blue, line width=0.5mm] (-4,-1.1) -- (-4,-4);

\draw[scale=4,color=blue] (0,-3/4) node {$\ldots$};

\draw[scale=2] (0,0) node {$D$};
\draw[scale=2] (3/2,4.5/2) node {$m$};
\draw[scale=2] (2.5/2,3/2) node {$r$};
\draw[scale=2] (-1.5/2,4/2) node {$s$};
\draw[scale=2] (-6/2,4/2) node {$E$};

\end{scope}
\end{tikzpicture}
\caption{Building an H-triangulation from a diagram of $K_n$} \label{fig:diagram:htriang}
\end{figure}
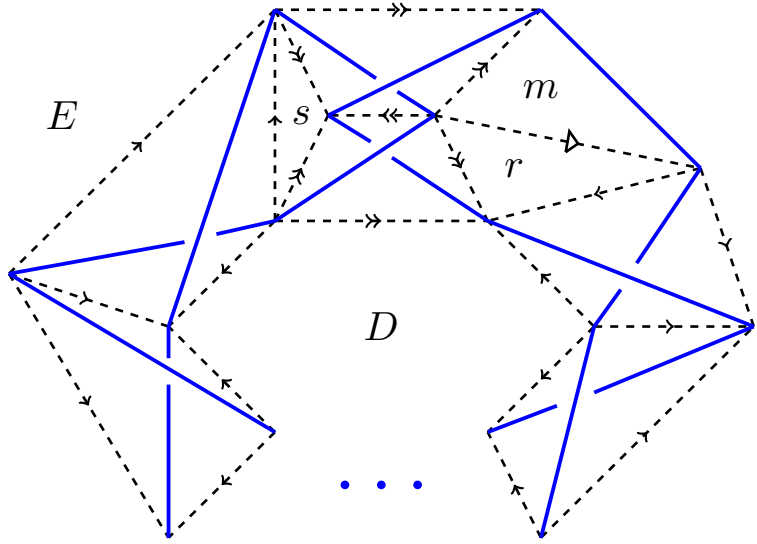

For the first step, as in Figure \ref{fig:diagram:htriang}, we choose a middle point for each arc of the diagram, except for one arc where we choose two (the  upper right one on the figure), and we draw quadrilaterals around the crossings with the chosen points as vertices (in dotted lines in Figure \ref{fig:diagram:htriang}).

 {
Observe that dotted edges in the same quadrilateral are isotopic through $(S^3,K_n)$.
We consider the equivalence relation on dotted edges generated by ``being part of the same quadrilateral'', and we choose an arrow type and an orientation for each equivalence class. In Figure \ref{fig:diagram:htriang} there are two such classes, one with a simple arrow and one with a double arrow. The arrows on the dotted edges are oriented so that isotopy through $(S^3,K_n)$ preserves the oriented labels, which makes
the directions keep alternating when one goes around any quadrilateral.}

There remains one quadrilateral with three dotted edges and one edge from the knot $K_n$. We cut this one into two triangles $m$ and $r$, introducing a third arrow type, the ``white triangle'' one (see Figure \ref{fig:diagram:htriang}).

Here $m,r,s,D,E$ are the polygonal $2$-cells that  {(together with the quadrilaterals)} decompose the equatorial plane around the knot; note that $m,r,s$ are triangles, $D$ is an $(n+1)$-gon and $E$ is an $(n+2)$-gon.

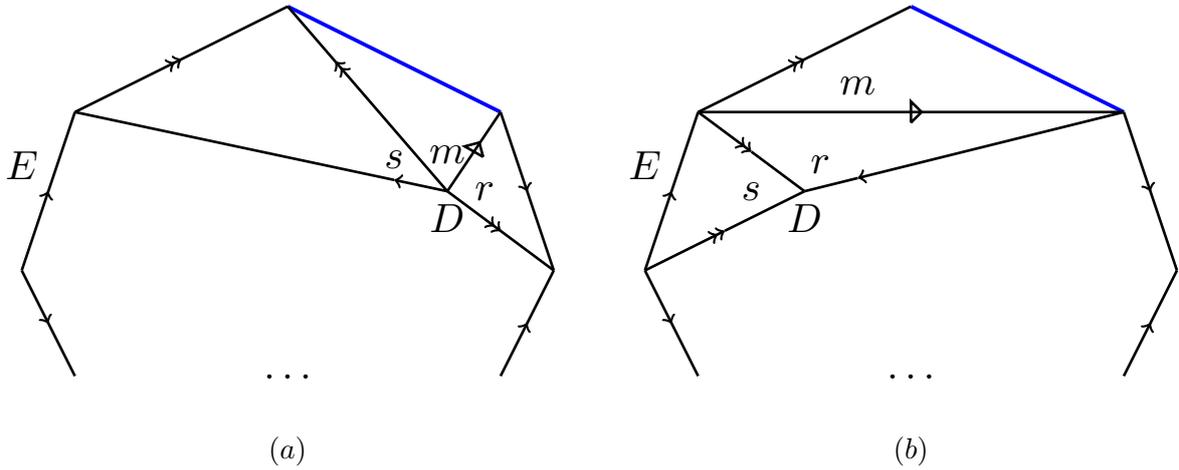
\begin{figure}[!h]

\centering
\begin{tikzpicture}[every path/.style={string ,black} , every node/.style={transform shape , knot crossing , inner sep=1.5 pt } ]

\draw[scale=1] (0,-1) node {$(a)$};
\draw[scale=1] (8.2,-1) node {$(b)$};


\begin{scope}[scale=0.7]

\begin{scope}[decoration={
    markings,
    mark=at position 0.5 with {\arrow{>}}}
    ] 
    \draw[postaction={decorate}] (-5,2)--(-4,0);
    \draw[postaction={decorate}] (-5,2)--(-4,5);
    \draw[postaction={decorate}] (4,5) -- (5,2);
    \draw[postaction={decorate}] (4,0) -- (5,2);
\end{scope}

\begin{scope}[decoration={
    markings,
    mark=at position 0.5 with {\arrow{>>}}}
    ] 
    \draw[postaction={decorate}] (-4,5)--(0,7);
\end{scope}

\draw[color=blue, line width=0.5mm] (0,7) -- (4,5);

\draw[scale=2] (0,0) node {$\ldots$};

\draw[->] (3,3.5) -- (2,3.5 +1.5/7);
\draw (2,3.5 +1.5/7) -- (-4,5);

\begin{scope}[xshift=3.5cm, yshift=4.25cm, rotate=-30, scale=0.2]
\draw (1,0) -- (0,1) -- (-1,0) -- (1,0);
\end{scope}
\draw (3,3.5) -- (4,5);

\begin{scope}[decoration={
    markings,
    mark=at position 0.5 with {\arrow{>>}}}
    ] 
    \draw[postaction={decorate}] (3,3.5) -- (5,2);
\end{scope}

\begin{scope}[decoration={
    markings,
    mark=at position 0.7 with {\arrow{>>}}}
    ] 
    \draw[postaction={decorate}] (3,3.5) -- (0,7);
\end{scope}

\draw[scale=2] (3/2,3/2) node {$D$};
\draw[scale=2] (3/2,4.2/2) node {$m$};
\draw[scale=2] (2/2,4.1/2) node {$s$};
\draw[scale=2] (3.7/2,3.5/2) node {$r$};

\draw[scale=2] (-5/2,4/2) node {$E$};

\end{scope}


\begin{scope}[xshift=8.2cm,scale=0.7]

\begin{scope}[decoration={
    markings,
    mark=at position 0.5 with {\arrow{>}}}
    ] 
    \draw[postaction={decorate}] (-5,2)--(-4,0);
    \draw[postaction={decorate}] (-5,2)--(-4,5);
    \draw[postaction={decorate}] (4,5) -- (5,2);
    \draw[postaction={decorate}] (4,0) -- (5,2);
\end{scope}

\begin{scope}[decoration={
    markings,
    mark=at position 0.5 with {\arrow{>>}}}
    ] 
    \draw[postaction={decorate}] (-4,5)--(0,7);
\end{scope}

\draw[color=blue, line width=0.5mm] (0,7) -- (4,5);

\draw[scale=2] (0,0) node {$\ldots$};

\begin{scope}[xshift=0cm, yshift=5cm, rotate=-90, scale=0.2]
\draw (1,0) -- (0,1) -- (-1,0) -- (1,0);
\end{scope}
\draw (-4,5) -- (4,5);

\draw[->>] (-5,2) -- (-3.5,2.75);
\draw (-3.5,2.75) -- (-2,3.5);

\draw[->>] (-4,5) -- (-3,4.25);
\draw (-3,4.25) -- (-2,3.5);

\draw[->] (4,5) -- (-1,5-1.5*5/6);
\draw (-1,5-1.5*5/6) -- (-2,3.5);
\draw[scale=2] (-2/2,3/2) node {$D$};
\draw[scale=2] (-1/2,5.5/2) node {$m$};
\draw[scale=2] (-3/2,3.5/2) node {$s$};
\draw[scale=2] (-1.7/2,4/2) node {$r$};

\draw[scale=2] (-5/2,4/2) node {$E$};

\end{scope}
\end{tikzpicture}
\caption{Boundaries of $B_+$ and $B_-$}\label{fig:boundary:balls}
\end{figure}

In Figure \ref{fig:diagram:htriang} we can see that around each crossing of the diagram, there are six edges (two in blue from the knot, four dotted with arrows) that  {define} an embedded tetrahedron. We will now collapse each of these tetrahedra into one segment, so that each of the two ``knot edges''  {is} collapsed to an extremal point of the segment and all four dotted edges fuse into a single one, with natural orientation.
The homeomorphism type of $(S^3,K_n)$ does not change if we collapse every tetrahedron in such a way, and that is what we do next.

After such a collapse, the ambient space (that we will call again $S^3$) decomposes as one $0$-cell (the collapsed point), four edges (simple arrow, double arrow, arrow with a triangle and blue edge coming from $K_n$), five polygonal $2$-cells still denoted $m,r,s,D,E$, and two $3$-balls $B_+$ and $B_-$, respectively from  {above} and below the figure. The boundaries of $B_+$ and $B_-$ are given in Figure \ref{fig:boundary:balls}. Note that the boundary of $B_+$ is obtained from Figure \ref{fig:diagram:htriang} by collapsing the upper strands of $K_n$, and $B_+$ is implicitly residing above Figure \ref{fig:boundary:balls} (a). Similarly, $B_-$ resides  {below} Figure \ref{fig:boundary:balls} (b). Note that the boundary of $D$, read clockwise, is the sequence of $n+1$ arrows $\twoheadrightarrow, \leftarrow, \rightarrow, \ldots, \leftarrow$ with the simple arrows alternating directions.

\begin{figure}[!h]
\centering
\begin{tikzpicture}[every path/.style={string ,black} , every node/.style={transform shape , knot crossing , inner sep=1.5 pt } ]

\draw[scale=1] (0,-2) node {$(a)$};
\draw[scale=1] (8,-2) node {$(b)$};


\begin{scope}[scale=0.7]

\begin{scope}[decoration={
    markings,
    mark=at position 0.5 with {\arrow{>}}}
    ] 
    \draw[postaction={decorate}] (-5,2)--(-4,0);
    \draw[postaction={decorate}] (-5,2)--(-4,5);
    \draw[postaction={decorate}] (4,5) -- (5,2);
    \draw[postaction={decorate}] (4,0) -- (5,2);
\end{scope}

\begin{scope}[decoration={
    markings,
    mark=at position 0.5 with {\arrow{>>}}}
    ] 
    \draw[postaction={decorate}] (-4,5)--(0,7);
\end{scope}

\draw[color=blue, line width=0.5mm] (0,7) -- (4,5);

\draw[scale=2] (0,0) node {$\ldots$};

\begin{scope}[xshift=0cm, yshift=5cm, rotate=-90, scale=0.2]
\draw (1,0) -- (0,1) -- (-1,0) -- (1,0);
\end{scope}
\draw (-4,5) -- (4,5);

\draw[->>] (-5,2) -- (-3.5,2.75);
\draw (-3.5,2.75) -- (-2,3.5);

\draw[->>] (-4,5) -- (-3,4.25);
\draw (-3,4.25) -- (-2,3.5);

\draw[->] (4,5) -- (-1,5-1.5*5/6);
\draw (-1,5-1.5*5/6) -- (-2,3.5);
\draw[scale=2] (-2/2,3/2) node {$D$};
\draw[scale=2] (-1/2,5.5/2) node {$m$};
\draw[scale=2] (-3/2,3.5/2) node {$s$};
\draw[scale=2] (-1.7/2,4/2) node {$r$};


\begin{scope}[xshift=3.5cm, yshift=4.25cm, rotate=-30, scale=0.2]
\draw[color=red] (1,0) -- (0,1) -- (-1,0) -- (1,0);
\end{scope}

\begin{scope}[style=dashed]

\draw[color=red][->] (3,3.5) -- (2,3.5 +1.5/7);
\draw[color=red] (2,3.5 +1.5/7) -- (-4,5);

\draw[color=red] (3,3.5) -- (4,5);

\begin{scope}[decoration={
    markings,
    mark=at position 0.5 with {\arrow{>>}}}
    ] 
    \draw[postaction={decorate}][color=red] (3,3.5) -- (5,2);
\end{scope}

\begin{scope}[decoration={
    markings,
    mark=at position 0.7 with {\arrow{>>}}}
    ] 
    \draw[postaction={decorate}][color=red] (3,3.5) -- (0,7);
\end{scope}

\draw[scale=2][color=red] (3/2,3/2) node {$D$};
\draw[scale=2][color=red] (3/2,4.2/2) node {$m$};
\draw[scale=2][color=red] (2/2,4.1/2) node {$s$};
\draw[scale=2][color=red] (3.7/2,3.5/2) node {$r$};

\end{scope}

\end{scope}

\begin{scope}[xshift=8cm,yshift=-1.5cm,scale=0.5]

\begin{scope}[decoration={
    markings,
    mark=at position 0.5 with {\arrow{>}}}
    ] 
    \draw[postaction={decorate}] (0,4)--(0,0);
    \draw[postaction={decorate}] (0,4)--(0,6);
    \draw[postaction={decorate}] (0,10) -- (0,12);
    \draw[postaction={decorate}] (-8,16) -- (0,12);
    \draw[postaction={decorate},color=red] (4,14) -- (0,0);
    \draw[postaction={decorate}] (-8,16) -- (-1,5);
\end{scope}

\begin{scope}[decoration={
    markings,
    mark=at position 0.5 with {\arrow{>>}}}
    ] 
    \draw[postaction={decorate}] (0,4)--(-1,5);
    \draw[postaction={decorate}] (0,0)--(-1,5);
    \draw[postaction={decorate},color=red] (4,14) -- (8,16);
    \draw[postaction={decorate}] (0,0) -- (8,16);
    \draw[postaction={decorate},color=red] (4,14) -- (0,12);
\end{scope}

\draw (0,0) -- (-8,16);
\begin{scope}[xshift=-4cm, yshift=8cm, rotate=30, scale=0.2]
\draw (1,0) -- (0,1) -- (-1,0) -- (1,0);
\end{scope}

\draw[color=red] (4,14) -- (-8,16);
\begin{scope}[xshift=-2cm, yshift=15cm, rotate=75, scale=0.2]
\draw[color=red] (1,0) -- (0,1) -- (-1,0) -- (1,0);
\end{scope}

\draw[color=blue, line width=0.5mm] (-8,16) -- (8,16);

\draw[scale=2] (0,8/2) node {$\vdots$};

\draw[scale=2, color=red] (4/2,15/2) node {$m$};
\draw[scale=2, color=red] (2.8/2,13.8/2) node {$r$};
\draw[scale=2, color=red] (5/2,14/2) node {$s$};
\draw[scale=2, color=red] (3/2,12.5/2) node {$D$};

\draw[scale=2] (-4/2,5/2) node {$m$};
\draw[scale=2] (-1.5/2,4.5/2) node {$r$};
\draw[scale=2] (-0.3/2,3.5/2) node {$s$};
\draw[scale=2] (-0.8/2,6/2) node {$D$};

\end{scope}

\end{tikzpicture}
\caption{A cellular decomposition of $(S^3,K_n)$ as a polyhedron glued to itself}\label{fig:Htriang:polyhedron}
\end{figure}
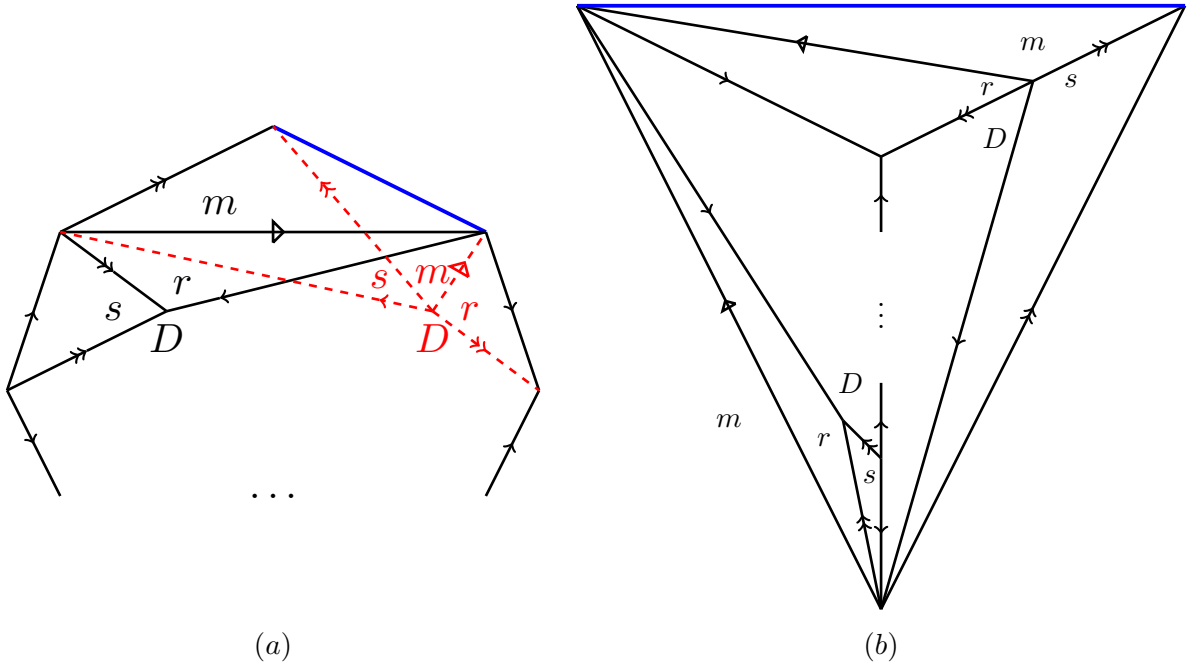

 {By gluing the balls $B_+$ and $B_-$ along the face $E$, we obtain a new ball whose face pairings still describe $S^3$ (see Figure \ref{fig:Htriang:polyhedron} (a)).}
Indeed, since $B_-$ is behind Figure \ref{fig:boundary:balls} (b) and $B_+$ in front of Figure \ref{fig:boundary:balls} (a), we can picture the gluing along $E$ in the following way, from front to back:
\begin{itemize}
\item the faces $D,m,r,s$ of $B_-$,
\item the $3$-cell $B_-$,
\item the face $E$ of $B_-$,
\item the face $E$ of $B_+$,
\item the $3$-cell $B_+$,
\item the faces $D,m,r,s$ of $B_+$.
\end{itemize}

Note that in Figure \ref{fig:Htriang:polyhedron} (a) the red dashed faces lie on the back of the figure, and the only $3$-cell now lives inside the polyhedron. Finally we can rotate this polyhedron and obtain the cellular decomposition of $S^3$ in Figure \ref{fig:Htriang:polyhedron} (b), where one  face $m$ is in the back and the seven other faces lie in front.

\begin{figure}[!h]
\centering
\begin{tikzpicture}
[every path/.style={string ,black}]

\begin{scope}[xshift=0cm,yshift=0cm,scale=1]

\draw[scale=1] (0,-1.5) node {\large $(a)$};

\draw (0,0) node[shape=circle,fill=black,scale=0.3] {};
\draw (0,1) node[shape=circle,fill=black,scale=0.3] {};
\draw (0,2) node[shape=circle,fill=black,scale=0.3] {};
\draw (0,3) node[shape=circle,fill=black,scale=0.3] {};
\draw (0,4) node[shape=circle,fill=black,scale=0.3] {};

\begin{scope}[decoration={markings,mark=at position 0.5 with {\arrow{>}}}] 
    \draw[postaction={decorate}] (0,0)--(0,1);
    \draw[postaction={decorate}] (0,2)--(0,1);
    \draw[postaction={decorate}] (0,2)--(0,3);
    \draw[postaction={decorate}] (0,4)--(0,3);   
\end{scope}

\draw[->,>=latex] (0,0) .. controls +(-0.5,0) and +(0,-0.5) .. (-1,1);
\draw (-1,1) .. controls +(0,0.5) and +(-0.5,0) .. (0,2);

\draw[scale=1] (-0.5,1) node {$u$};
\draw[scale=1] (1,2) node {$F$};
\draw[scale=1] (-1,3) node {$F$};
\draw[scale=1] (0,-0.5) node {$\vdots$};
\draw[scale=1] (0,4.5) node {$\vdots$};

\end{scope}

\begin{scope}[xshift=4.5cm,yshift=0cm,scale=1]

\draw[scale=1] (0,-1.5) node {\large $(b)$};

\draw (0,0) node[shape=circle,fill=black,scale=0.3] {};
\draw (0,1) node[shape=circle,fill=black,scale=0.3] {};
\draw (0,2) node[shape=circle,fill=black,scale=0.3] {};
\draw (0,3) node[shape=circle,fill=black,scale=0.3] {};
\draw (0,4) node[shape=circle,fill=black,scale=0.3] {};

\begin{scope}[decoration={markings,mark=at position 0.5 with {\arrow{>}}}] 
    \draw[postaction={decorate}] (0,0)--(0,1);
    \draw[postaction={decorate}] (0,2)--(0,1);
    \draw[postaction={decorate}] (0,2)--(0,3);
    \draw[postaction={decorate}] (0,4)--(0,3);   
\end{scope}

\draw[->,>=latex] (0,0) .. controls +(-0.5,0) and +(0,-0.5) .. (-1,1);
\draw (-1,1) .. controls +(0,0.5) and +(-0.5,0) .. (0,2);

\draw[->>,>=latex] (0,2) .. controls +(-0.5,0) and +(0,-0.5) .. (-1,3);
\draw (-1,3) .. controls +(0,0.5) and +(-0.5,0) .. (0,4);

\draw[->>,>=latex] (0,0) .. controls +(0.5,0) and +(0,-0.5) .. (1,1);
\draw (1,1) .. controls +(0,0.5) and +(0.5,0) .. (0,2);

\draw[scale=1] (-0.5,1) node {$u$};
\draw[scale=1] (-0.5,3) node {$v$};
\draw[scale=1] (0.5,1) node {$v$};
\draw[scale=1] (1,2) node {$F'$};
\draw[scale=1] (-1,2) node {$F'$};
\draw[scale=1] (0,-0.5) node {$\vdots$};
\draw[scale=1] (0,4.5) node {$\vdots$};

\end{scope}

\begin{scope}[xshift=9cm,yshift=0cm,scale=1]

\draw[scale=1] (2,-1.5) node {\large $(c)$};

\draw (0,0) node[shape=circle,fill=black,scale=0.3] {};
\draw (0,2) node[shape=circle,fill=black,scale=0.3] {};
\draw (0,3) node[shape=circle,fill=black,scale=0.3] {};
\draw (0,4) node[shape=circle,fill=black,scale=0.3] {};

\begin{scope}[decoration={markings,mark=at position 0.5 with {\arrow{>}}}] 
    \draw[postaction={decorate}] (0,2)--(0,3);
    \draw[postaction={decorate}] (0,4)--(0,3);   
\end{scope}

\draw[->,>=latex] (0,0) .. controls +(-0.5,0) and +(0,-0.5) .. (-1,1);
\draw (-1,1) .. controls +(0,0.5) and +(-0.5,0) .. (0,2);

\draw[->>,>=latex] (0,2) .. controls +(-0.5,0) and +(0,-0.5) .. (-1,3);
\draw (-1,3) .. controls +(0,0.5) and +(-0.5,0) .. (0,4);

\draw[->>,>=latex] (0,0) .. controls +(0.5,0) and +(0,-0.5) .. (1,1);
\draw (1,1) .. controls +(0,0.5) and +(0.5,0) .. (0,2);

\draw[scale=1] (0,1) node {$w$};
\draw[scale=1] (-0.5,3) node {$v$};
\draw[scale=1] (1,2) node {$F'$};
\draw[scale=1] (-1,2) node {$F'$};
\draw[scale=1] (0,-0.5) node {$\vdots$};
\draw[scale=1] (0,4.5) node {$\vdots$};

\draw[scale=1] (2,2) node {$\sqcup$};

\begin{scope}[xshift=4cm,yshift=1cm]
\draw[->,>=latex] (0,0) .. controls +(-0.5,0) and +(0,-0.5) .. (-1,1);
\draw (-1,1) .. controls +(0,0.5) and +(-0.5,0) .. (0,2);

\draw (0,0) node[shape=circle,fill=black,scale=0.3] {};
\draw (0,1) node[shape=circle,fill=black,scale=0.3] {};
\draw (0,2) node[shape=circle,fill=black,scale=0.3] {};

\begin{scope}[decoration={markings,mark=at position 0.5 with {\arrow{>}}}] 
    \draw[postaction={decorate}] (0,0)--(0,1);
    \draw[postaction={decorate}] (0,2)--(0,1);
\end{scope}

\draw[->>,>=latex] (0,0) .. controls +(0.5,0) and +(0,-0.5) .. (1,1);
\draw (1,1) .. controls +(0,0.5) and +(0.5,0) .. (0,2);

\draw[scale=1] (-0.5,1) node {$u$};
\draw[scale=1] (0.5,1) node {$v$};
\draw[scale=1] (0,2.3) node {$w$};

\end{scope}

\end{scope}

\begin{scope}[xshift=0cm,yshift=-7cm,scale=1]

\draw[scale=1] (2,-1.5) node {\large $(d)$};

\draw (0,0) node[shape=circle,fill=black,scale=0.3] {};
\draw (0,2) node[shape=circle,fill=black,scale=0.3] {};
\draw (0,3) node[shape=circle,fill=black,scale=0.3] {};
\draw (0,4) node[shape=circle,fill=black,scale=0.3] {};

\begin{scope}[decoration={markings,mark=at position 0.5 with {\arrow{>}}}] 
    \draw[postaction={decorate}] (0,2)--(0,3);
    \draw[postaction={decorate}] (0,4)--(0,3);   
\end{scope}

\draw[->,>=latex] (0,0) .. controls +(-0.5,0) and +(0,-0.5) .. (-1,1);
\draw (-1,1) .. controls +(0,0.5) and +(-0.5,0) .. (0,2);

\draw[->>,>=latex] (0,2) .. controls +(-0.5,0) and +(0,-0.5) .. (-1,3);
\draw (-1,3) .. controls +(0,0.5) and +(-0.5,0) .. (0,4);

\draw[->>,>=latex] (0,0) .. controls +(0.5,0) and +(0,-0.5) .. (1,1);
\draw (1,1) .. controls +(0,0.5) and +(0.5,0) .. (0,2);

\draw[scale=1] (0,1) node {$w$};
\draw[scale=1] (-0.5,3) node {$v$};
\draw[scale=1] (1,2) node {$F'$};
\draw[scale=1] (-1,2) node {$F'$};
\draw[scale=1] (0,-0.5) node {$\vdots$};
\draw[scale=1] (0,4.5) node {$\vdots$};

\draw[scale=1] (2,2) node {$\sqcup$};

\begin{scope}[xshift=4cm,yshift=1cm]
\draw[->>,>=latex] (0,0) .. controls +(-0.5,0) and +(0,-0.5) .. (-1,1);
\draw (-1,1) .. controls +(0,0.5) and +(-0.5,0) .. (0,2);

\draw (0,0) node[shape=circle,fill=black,scale=0.3] {};
\draw (1,1) node[shape=circle,fill=black,scale=0.3] {};
\draw (0,2) node[shape=circle,fill=black,scale=0.3] {};

\begin{scope}[decoration={markings,mark=at position 0.5 with {\arrow{>}}}] 
    \draw[postaction={decorate}] (0,0)--(1,1);
    \draw[postaction={decorate}] (0,2)--(1,1);
\end{scope}

\draw[->,>=latex] (0,0) -- (0,1);
\draw (0,1) -- (0,2);

\draw[scale=1] (-0.5,1) node {$w$};
\draw[scale=1] (0.5,1) node {$u$};
\draw[scale=1] (0,2.3) node {$v$};

\end{scope}

\end{scope}

\begin{scope}[xshift=8.5cm,yshift=-7cm,scale=1]

\draw[scale=1] (0,-1.5) node {\large $(e)$};

\draw (0,0) node[shape=circle,fill=black,scale=0.3] {};
\draw (0,2) node[shape=circle,fill=black,scale=0.3] {};
\draw (0,3) node[shape=circle,fill=black,scale=0.3] {};
\draw (0,4) node[shape=circle,fill=black,scale=0.3] {};

\begin{scope}[decoration={markings,mark=at position 0.5 with {\arrow{>}}}] 
    \draw[postaction={decorate}] (0,2)--(0,3);
    \draw[postaction={decorate}] (0,4)--(0,3);   
\end{scope}

\draw[->,>=latex] (0,0) .. controls +(-0.5,0) and +(0,-0.5) .. (-1,1);
\draw (-1,1) .. controls +(0,0.5) and +(-0.5,0) .. (0,2);

\draw[->>,>=latex] (0,2) .. controls +(-0.5,0) and +(0,-0.5) .. (-1,3);
\draw (-1,3) .. controls +(0,0.5) and +(-0.5,0) .. (0,4);

\draw[->>,>=latex] (0,0) .. controls +(0.5,0) and +(0,-0.5) .. (1,1);
\draw (1,1) .. controls +(0,0.5) and +(0.5,0) .. (0,2);

\draw[->,>=latex] (0,2) .. controls +(-0.25,0) and +(0,-0.5) .. (-0.5,3);
\draw (-0.5,3) .. controls +(0,0.5) and +(-0.25,0) .. (0,4);

\draw[scale=1] (0,1) node {$w$};
\draw[scale=1] (-0.25,3) node {$u$};
\draw[scale=1] (-0.75,3) node {$w$};
\draw[scale=1] (1,2) node {$F'$};
\draw[scale=1] (-1,2) node {$F'$};
\draw[scale=1] (0,-0.5) node {$\vdots$};
\draw[scale=1] (0,4.5) node {$\vdots$};

\end{scope}

\begin{scope}[xshift=13cm,yshift=-7cm,scale=1]

\draw[scale=1] (0,-1.5) node {\large $(f)$};

\draw (0,0) node[shape=circle,fill=black,scale=0.3] {};
\draw (0,2) node[shape=circle,fill=black,scale=0.3] {};
\draw (0,3) node[shape=circle,fill=black,scale=0.3] {};
\draw (0,4) node[shape=circle,fill=black,scale=0.3] {};

\begin{scope}[decoration={markings,mark=at position 0.5 with {\arrow{>}}}] 
    \draw[postaction={decorate}] (0,2)--(0,3);
    \draw[postaction={decorate}] (0,4)--(0,3);   
\end{scope}

\draw[->,>=latex] (0,0) -- (0,1);
\draw (0,1) -- (0,2);

\draw[->,>=latex] (0,2) .. controls +(-0.5,0) and +(0,-0.5) .. (-1,3);
\draw (-1,3) .. controls +(0,0.5) and +(-0.5,0) .. (0,4);

\draw[scale=1] (-0.5,3) node {$u$};
\draw[scale=1] (1,2) node {$F''$};
\draw[scale=1] (-1,2) node {$F''$};
\draw[scale=1] (0,-0.5) node {$\vdots$};
\draw[scale=1] (0,4.5) node {$\vdots$};

\end{scope}

\end{tikzpicture}
\caption{The bigon trick}
\label{fig:bigon:trick}
\end{figure}
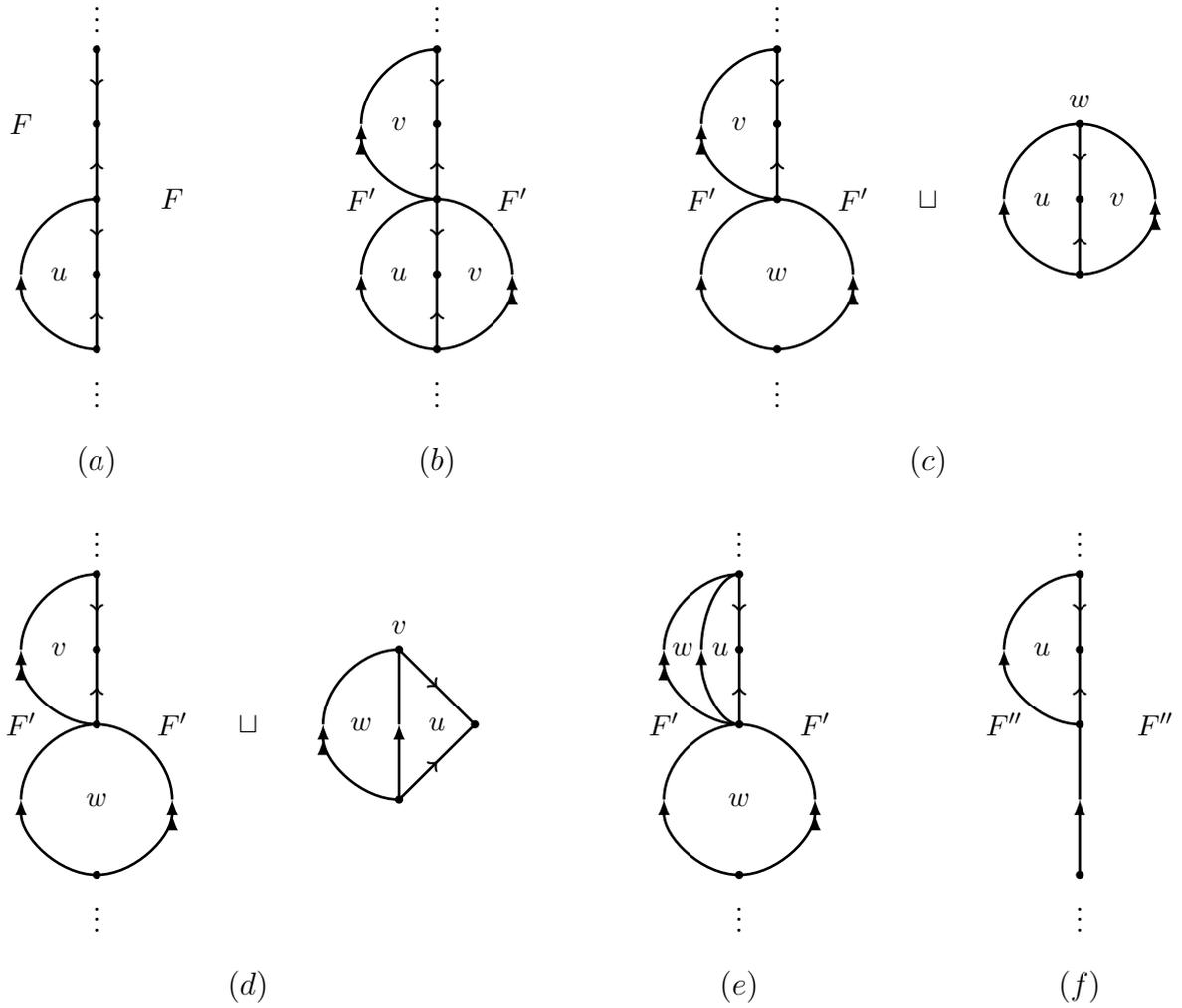

We will now use the \textit{bigon trick} to find another polyhedral description of $(S^3,K_n)$ with many fewer edges. The bigon trick is described in Figure \ref{fig:bigon:trick} (a) to (f). We start at (a), with the two faces $F$ having several edges in common, and a triangle $u$ adjacent to $F$ (note that there is a second face $u$ adjacent to the other $F$ somewhere else). Then we go to (b) by cutting $F$ along a new edge (with double full arrow) into $F'$ and a triangle $v$. The CW-complex described in (b) is the same as the one in (c), where the right part is a $3$-ball whose boundary is cut into the triangles $u$ and $v$ and the bigon $w$. The picture in (d) is simply the one from (c) with the ball rotated so that $v$ lies in the back instead of $w$. Then we obtain (e) by gluing the two parts of (d) along the face $v$, and finally (f) by fusing $F'$ and $w$ into a new face $F''$. As a result, we replaced two simple arrows by one longer different  {\textit{full} arrow (where \textit{full} means that the arrow is marked by a solid triangle)} and we  {slid} the face $u$ up.

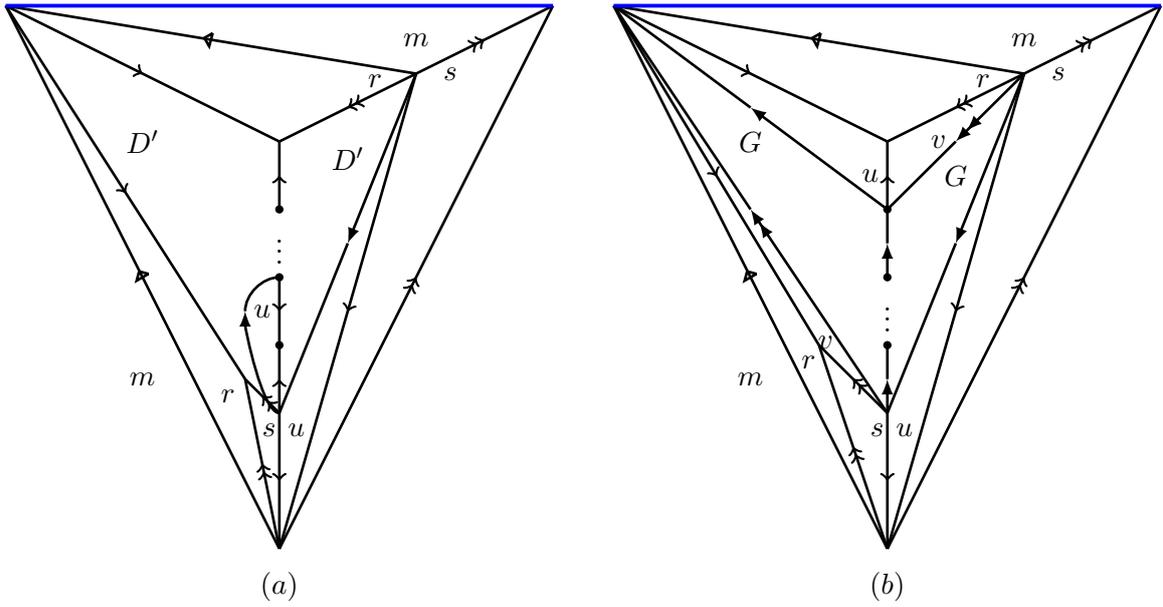
\begin{figure}[!h]
\centering
\begin{tikzpicture}[every path/.style={string ,black}]

\draw[scale=1] (0,-2) node {$(a)$};
\draw[scale=1] (8,-2) node {$(b)$};

\begin{scope}[xshift=0cm,yshift=-1.5cm,scale=0.45]

\draw (0,6) node[shape=circle,fill=black,scale=0.3] {};
\draw (0,8) node[shape=circle,fill=black,scale=0.3] {};
\draw (0,10) node[shape=circle,fill=black,scale=0.3] {};

\draw[->,>=latex] (0,4) .. controls +(-0.5,0) and +(0,-0.5) .. (-1,7);
\draw (-1,7) .. controls +(0,0.5) and +(-0.5,0) .. (0,8);

\draw[scale=2] (-0.5/2,7/2) node {$u$};

\draw[->,>=latex] (4,14) -- (2,9);
\draw (2,9) -- (0,4);

\draw[scale=2] (0.5/2,3.5/2) node {$u$};

\begin{scope}[decoration={
    markings,
    mark=at position 0.5 with {\arrow{>}}}
    ] 
    \draw[postaction={decorate}] (0,4)--(0,0);
    \draw[postaction={decorate}] (0,4)--(0,6);
    \draw[postaction={decorate}] (0,8)--(0,6);
    \draw[postaction={decorate}] (0,10) -- (0,12);
    \draw[postaction={decorate}] (-8,16) -- (0,12);
    \draw[postaction={decorate}] (4,14) -- (0,0);
    \draw[postaction={decorate}] (-8,16) -- (-1,5);
\end{scope}

\begin{scope}[decoration={
    markings,
    mark=at position 0.5 with {\arrow{>>}}}
    ] 
    \draw[postaction={decorate}] (0,4)--(-1,5);
    \draw[postaction={decorate}] (0,0)--(-1,5);
    \draw[postaction={decorate}] (4,14) -- (8,16);
    \draw[postaction={decorate}] (0,0) -- (8,16);
    \draw[postaction={decorate}] (4,14) -- (0,12);
\end{scope}

\draw (0,0) -- (-8,16);
\begin{scope}[xshift=-4cm, yshift=8cm, rotate=30, scale=0.2]
\draw (1,0) -- (0,1) -- (-1,0) -- (1,0);
\end{scope}

\draw (4,14) -- (-8,16);
\begin{scope}[xshift=-2cm, yshift=15cm, rotate=75, scale=0.2]
\draw (1,0) -- (0,1) -- (-1,0) -- (1,0);
\end{scope}

\draw[color=blue, line width=0.5mm] (-8,16) -- (8,16);

\draw[scale=2] (0,9/2) node {$\vdots$};

\draw[scale=2] (4/2,15/2) node {$m$};
\draw[scale=2] (2.8/2,13.8/2) node {$r$};
\draw[scale=2] (5/2,14/2) node {$s$};
\draw[scale=2] (2/2,11.5/2) node {$D'$};

\draw[scale=2] (-4/2,5/2) node {$m$};
\draw[scale=2] (-1.5/2,4.5/2) node {$r$};
\draw[scale=2] (-0.3/2,3.5/2) node {$s$};
\draw[scale=2] (-4/2,12/2) node {$D'$};

\end{scope}

\begin{scope}[xshift=8cm,yshift=-1.5cm,scale=0.45]

\draw (0,6) node[shape=circle,fill=black,scale=0.3] {};
\draw (0,8) node[shape=circle,fill=black,scale=0.3] {};
\draw (0,10) node[shape=circle,fill=black,scale=0.3] {};

\draw[->>,>=latex] (4,14) -- (2,12);
\draw (2,12) -- (0,10);

\draw[scale=2] (1.5/2,12/2) node {$v$};

\draw[->>,>=latex] (0,4) -- (-4,10);
\draw (-4,10) -- (-8,16);

\draw[scale=2] (-1.8/2,6.1/2) node {$v$};

\draw[->,>=latex] (0,10) -- (-4,13);
\draw (-4,13) -- (-8,16);

\draw[scale=2] (-0.5/2,11/2) node {$u$};

\draw[->,>=latex] (4,14) -- (2,9);
\draw (2,9) -- (0,4);

\draw[scale=2] (0.5/2,3.5/2) node {$u$};

\begin{scope}[decoration={
    markings,
    mark=at position 0.5 with {\arrow{>}}}
    ] 
    \draw[postaction={decorate}] (0,4)--(0,0);
    \draw[postaction={decorate}] (0,10) -- (0,12);
    \draw[postaction={decorate}] (-8,16) -- (0,12);
    \draw[postaction={decorate}] (4,14) -- (0,0);
    \draw[postaction={decorate}] (-8,16) -- (-2,6);
\end{scope}

\draw[->,>=latex] (0,4) -- (0,5);
\draw (0,5) -- (0,6);

\draw[->,>=latex] (0,8) -- (0,9);
\draw (0,9) -- (0,10);

\begin{scope}[decoration={
    markings,
    mark=at position 0.5 with {\arrow{>>}}}
    ] 
    \draw[postaction={decorate}] (0,4)--(-2,6);
    \draw[postaction={decorate}] (0,0)--(-2,6);
    \draw[postaction={decorate}] (4,14) -- (8,16);
    \draw[postaction={decorate}] (0,0) -- (8,16);
    \draw[postaction={decorate}] (4,14) -- (0,12);
\end{scope}

\draw (0,0) -- (-8,16);
\begin{scope}[xshift=-4cm, yshift=8cm, rotate=30, scale=0.2]
\draw (1,0) -- (0,1) -- (-1,0) -- (1,0);
\end{scope}

\draw (4,14) -- (-8,16);
\begin{scope}[xshift=-2cm, yshift=15cm, rotate=75, scale=0.2]
\draw (1,0) -- (0,1) -- (-1,0) -- (1,0);
\end{scope}

\draw[color=blue, line width=0.5mm] (-8,16) -- (8,16);

\draw[scale=2] (0,7/2) node {$\vdots$};

\draw[scale=2] (4/2,15/2) node {$m$};
\draw[scale=2] (2.8/2,13.8/2) node {$r$};
\draw[scale=2] (5/2,14/2) node {$s$};
\draw[scale=2] (2/2,11/2) node {$G$};

\draw[scale=2] (-4/2,5/2) node {$m$};
\draw[scale=2] (-2.3/2,5.5/2) node {$r$};
\draw[scale=2] (-0.3/2,3.5/2) node {$s$};
\draw[scale=2] (-4/2,12/2) node {$G$};

\end{scope}

\end{tikzpicture}
\caption{A cellular decomposition of $(S^3,K_n)$ before and after the bigon trick}\label{fig:Htriang:bigon:trick}
\end{figure}

Let us now go back to our cellular decomposition of $(S^3,K_n)$.
We start from Figure \ref{fig:Htriang:polyhedron} (b) and cut $D$ into new faces $u$ and $D'$ as in Figure \ref{fig:Htriang:bigon:trick} (a). Then we apply the bigon trick $p$ times, where $p:=\tfrac{n-3}{2}$, to slide the cell $u$ on the left $D'$, and finally we cut the face obtained from $D'$ a final time into a $p+2$-gon $G$ and a triangle $v$ by adding a double full arrow. See Figure \ref{fig:Htriang:bigon:trick} (b).

Note that if $n=3$, i.e.\ $p=0$, we do not use the bigon trick, and simply denote $D'$ by $v$. In this case, $G$ is empty and the double full arrow should be identified with the simple full arrow.

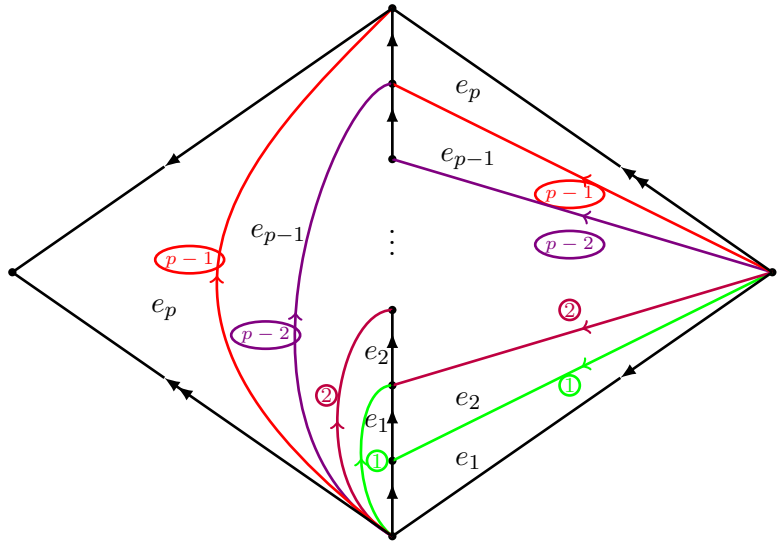
\begin{figure}[!h]
\centering
\begin{tikzpicture}
[every path/.style={string ,black}]

\draw (0,0) node[shape=circle,fill=black,scale=0.3] {};
\draw (0,1) node[shape=circle,fill=black,scale=0.3] {};
\draw (0,2) node[shape=circle,fill=black,scale=0.3] {};
\draw (0,3) node[shape=circle,fill=black,scale=0.3] {};
\draw (0,5) node[shape=circle,fill=black,scale=0.3] {};
\draw (0,6) node[shape=circle,fill=black,scale=0.3] {};
\draw (0,7) node[shape=circle,fill=black,scale=0.3] {};
\draw (5,3.5) node[shape=circle,fill=black,scale=0.3] {};
\draw (-5,3.5) node[shape=circle,fill=black,scale=0.3] {};

\draw[scale=1] (0,4) node {$\vdots$};

\draw[scale=1,color=black] (1,1) node {$e_1$};
\draw[scale=1,color=black] (-0.2,1.5) node {$e_1$};
\draw[scale=1,color=black] (1,1.8) node {$e_2$};
\draw[scale=1,color=black] (-0.2,2.4) node {$e_2$};
\draw[scale=1,color=black] (1,5) node {$e_{p-1}$};
\draw[scale=1,color=black] (-1.5,4) node {$e_{p-1}$};
\draw[scale=1,color=black] (1,5.9) node {$e_p$};
\draw[scale=1,color=black] (-3,3) node {$e_p$};

\begin{scope}[xshift=0cm,yshift=0cm,rotate=0,scale=1/1.5]
\draw[color=green] (3.5,3) circle (0.2) node{\scriptsize $1$};
\end{scope}

\begin{scope}[xshift=0cm,yshift=0cm,rotate=0,scale=1/1.5]
\draw[color=green] (-0.3,1.5) circle (0.2) node{\scriptsize $1$};
\end{scope}

\begin{scope}[xshift=0cm,yshift=0cm,rotate=0,scale=1/1.5,color=purple]
\draw[color=purple] (3.5,4.5) circle (0.2) node{\scriptsize $2$};
\end{scope}

\begin{scope}[xshift=0cm,yshift=0cm,rotate=0,scale=1/1.5,color=purple]
\draw[color=purple] (-1.3,2.8) circle (0.2) node{\scriptsize $2$};
\end{scope}

\begin{scope}[xshift=0cm,yshift=0cm,rotate=0,scale=1/1.5]
\draw[color=violet] (3.5,5.8)  node{\tiny $p-2$};
\node[draw,ellipse,color=violet] (S) at(3.5,5.8) {\ \ \ \ };
\end{scope}

\begin{scope}[xshift=0cm,yshift=0cm,rotate=0,scale=1/1.5]
\draw[color=violet] (-2.5,4)  node{\tiny $p-2$};
\node[draw,ellipse,color=violet] (S) at(-2.5,4) {\ \ \ \ };
\end{scope}

\begin{scope}[xshift=0cm,yshift=0cm,rotate=0,scale=1/1.5]
\draw[color=red] (3.5,6.8)  node{\tiny $p-1$};
\node[draw,ellipse,color=red] (S) at(3.5,6.8) {\ \ \ \ };
\end{scope}

\begin{scope}[xshift=0cm,yshift=0cm,rotate=0,scale=1/1.5]
\draw[color=red] (-4,5.5)  node{\tiny $p-1$};
\node[draw,ellipse,color=red] (S) at(-4,5.5) {\ \ \ \ };
\end{scope}

\begin{scope}[decoration={markings,mark=at position 0.5 with {\arrow{>}}}] 
    \draw[postaction={decorate},color=green] (5,3.5)--(0,1);
    \draw[postaction={decorate},color=purple] (5,3.5)--(0,2);
    \draw[postaction={decorate},color=violet] (5,3.5)--(0,5);
    \draw[postaction={decorate},color=red] (5,3.5)--(0,6); 
   	\draw[postaction={decorate},color=green] (0,0) .. controls +(-0.6,0.5) and +(-0.5,0) .. (0,2);  
   	\draw[postaction={decorate},color=purple] (0,0) .. controls +(-1.2,0.7) and +(-0.7,0) .. (0,3);  
   	\draw[postaction={decorate},color=violet] (0,0) .. controls +(-2.5,1.5) and +(-0.7,0) .. (0,6);  
   	\draw[postaction={decorate},color=red] (0,0) .. controls +(-4,3) and +(-2,-2) .. (0,7); 
\end{scope}

\draw[->>,>=latex] (0,0) -- (-3,2.1);
\draw (-3,2.1) -- (-5,3.5);
\draw[->,>=latex] (0,7) -- (-3,7-2.1);
\draw (-3,7-2.1) -- (-5,3.5);

\draw (0,0) -- (3,2.1);
\draw[<-,>=latex] (3,2.1) -- (5,3.5);
\draw (0,7) -- (3,7-2.1);
\draw[<<-,>=latex] (3,7-2.1) -- (5,3.5);

\begin{scope}[yshift=0cm]
\draw[->,>=latex] (0,0) -- (0,0.7);
\draw (0,0.6)--(0,1);
\end{scope}

\begin{scope}[yshift=1cm]
\draw[->,>=latex] (0,0) -- (0,0.7);
\draw (0,0.6)--(0,1);
\end{scope}

\begin{scope}[yshift=2cm]
\draw[->,>=latex] (0,0) -- (0,0.7);
\draw (0,0.6)--(0,1);
\end{scope}

\begin{scope}[yshift=5cm]
\draw[->,>=latex] (0,0) -- (0,0.7);
\draw (0,0.6)--(0,1);
\end{scope}

\begin{scope}[yshift=6cm]
\draw[->,>=latex] (0,0) -- (0,0.7);
\draw (0,0.6)--(0,1);
\end{scope}

\end{tikzpicture}
\caption{Decomposing the two faces $G$ in a tower of tetrahedra}
\label{fig:GG:tower}
\end{figure}

Then, if $p\geqslant 1$,  we triangulate the two faces $G$ as in Figure \ref{fig:GG:tower}: we add $p-1$ new edges drawn with simple arrows and circled $k$ for $k=1, \ldots, p-1$ (and drawn in different colors in Figure \ref{fig:GG:tower} but not in the following pictures), and $G$ is cut into $p$ triangles $e_1, \ldots, e_p$. This still makes sense if $p=1$, in  {which} case we have $G=e_p=e_1$ and no new edges.

Now, by combining Figures \ref{fig:Htriang:bigon:trick} (b) and \ref{fig:GG:tower}, we obtain a decomposition of $S^3$ as a polyhedron with only triangular faces glued to one another, and $K_n$ still represents the blue edge after identifications. 
In order to harmonize the  {notation} with the small cases ($p=0,1$), we do the following arrow replacements:
\begin{itemize}
\item  {Replace} full black simple arrow by simple arrow with circled $0$,
\item  {replace} full black double arrow by simple arrow with circled $p$,
\item  {replace} white triangle simple arrow by simple arrow with circled $p+1$.
\end{itemize}
Moreover, we cut the previous  {polyhe}dron of Figures \ref{fig:Htriang:bigon:trick} (b) and \ref{fig:GG:tower} into $p+4$ tetrahedra, introducing new triangular faces $e_{p+1}$ (behind $r,u,v$), $g$ (behind $r,s,v$), $s'$ (completing $m,m,s$), $f_p$ (completing $g,s',u$) and $f_1, \ldots, f_{p-1}$ at each of the $p-1$ ``floors'' of the tower of Figure \ref{fig:GG:tower} (from front to back of the figure).
 {More precisely, for $k \in \{1, \ldots, p-1\}$, the new face $f_k$ is made of three edges: \begin{itemize}
\item an edge with circled $k$ (curved and going up in the left half of Figure \ref{fig:GG:tower}),
\item an edge with circled $k+1$ (or full black double arrow for $k=p-1$) in Figure \ref{fig:GG:tower}, going from the  rightmost vertex to the  endpoint of the previous edge, 
\item the edge with  full black simple arrow in Figure \ref{fig:GG:tower} going from the  rightmost vertex  to the bottom vertex.	\end{itemize} }
 We add the convention $f_0=e_1$ to account for the case $p=0$.
We also choose an orientation for the blue edge and thus a sign for the tetrahedron that contains it (this choice will not have any influence on the ideal triangulation, though).

Finally, we obtain the H-triangulation for $(S^3,K_n)$ described in Figure \ref{fig:H:trig:odd}, for any $p \geqslant 0$ (recalling the convention $f_0=e_1$ if need be).

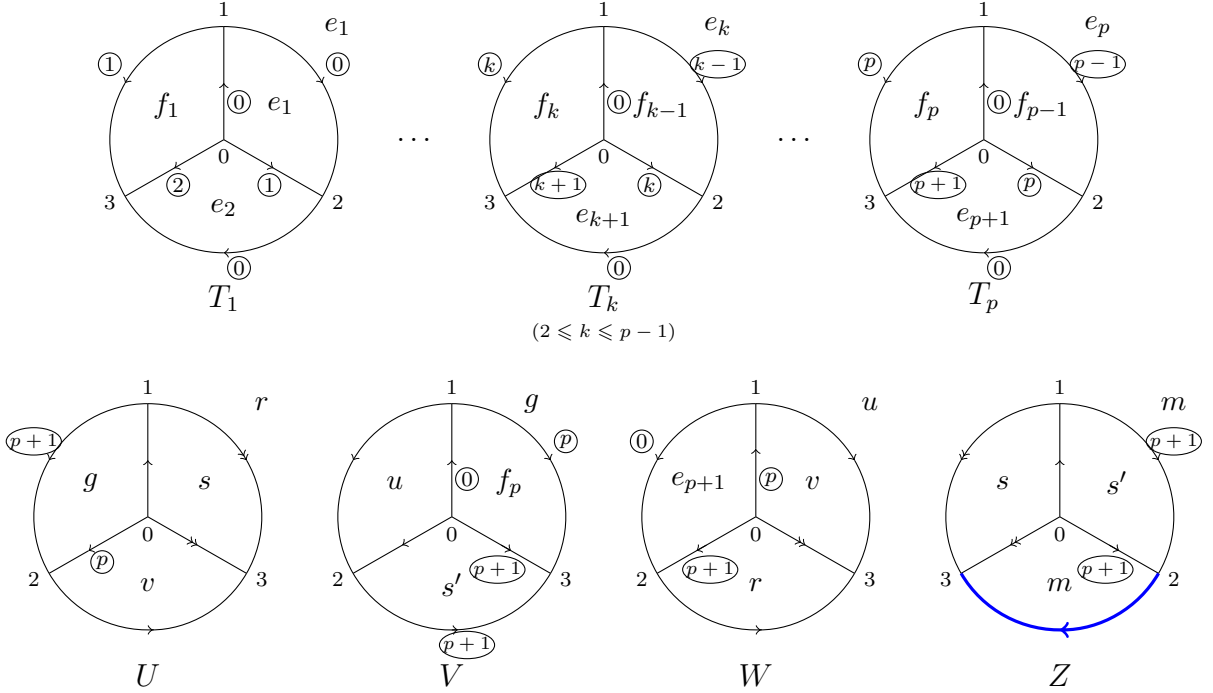
\begin{figure}[!h]
\begin{tikzpicture}
%

\begin{scope}[xshift=1cm,yshift=0cm,rotate=0,scale=1.5]

\draw (0,-0.15) node{\scriptsize $0$} ;
\draw (0,1.15) node{\scriptsize $1$} ;
\draw (1,-0.55) node{\scriptsize $2$} ;
\draw (-1,-0.55) node{\scriptsize $3$} ;
\draw (1,1) node{$e_1$} ;
\draw (0,-0.6) node{$e_2$} ;
\draw (-0.5,0.3) node{$f_1$} ;
\draw (0.5,0.3) node{$e_1$} ;

\draw (0,-1.4) node{\large $T_1$} ;

\path [draw=black,postaction={on each segment={mid arrow=black}}]
(0,0)--(-1.732/2,-0.5);

\path [draw=black,postaction={on each segment={mid arrow =black}}]
(0,0)--(0,1);

\path [draw=black,postaction={on each segment={mid arrow =black}}]
(0,0)--(1.732/2,-0.5);

\draw[->](1.732/2,-0.5) arc (-30:-90:1);
\draw (0,-1) arc (-90:-150:1);

\draw[->](0,1) arc (90:30:1);
\draw (1.732/2,0.5) arc (30:-30:1);

\draw[->](0,1) arc (90:150:1);
\draw (-1.732/2,0.5) arc (150:210:1);

\begin{scope}[xshift=0cm,yshift=0cm,rotate=0,scale=1/1.5]
\draw (-1.5,1) circle (0.15) node{\scriptsize $1$};
\end{scope}

%
\begin{scope}[xshift=0cm,yshift=0cm,rotate=0,scale=1/1.5]
\draw (1.5,1) circle (0.15) node{\scriptsize $0$};
\end{scope}

%
\begin{scope}[xshift=0cm,yshift=0cm,rotate=0,scale=1/1.5]
\draw (0.2,0.5) circle (0.15) node{\scriptsize $0$};
\end{scope}

%
\begin{scope}[xshift=0cm,yshift=0cm,rotate=0,scale=1/1.5]
\draw (0.2,-1.7) circle (0.15) node{\scriptsize $0$};
\end{scope}

\begin{scope}[xshift=0cm,yshift=0cm,rotate=0,scale=1/1.5]
\draw (-0.6,-0.6) circle (0.15) node{\scriptsize $2$};
\end{scope}

\begin{scope}[xshift=0cm,yshift=0cm,rotate=0,scale=1/1.5]
\draw (0.6,-0.6) circle (0.15) node{\scriptsize $1$};
\end{scope}

\end{scope}

\draw (3.5,0) node{$\ldots$} ;
\draw (8.5,0) node{$\ldots$} ;

\begin{scope}[xshift=6cm,yshift=0cm,rotate=0,scale=1.5]

\draw (0,-0.15) node{\scriptsize $0$} ;
\draw (0,1.15) node{\scriptsize $1$} ;
\draw (1,-0.55) node{\scriptsize $2$} ;
\draw (-1,-0.55) node{\scriptsize $3$} ;
\draw (1,1) node{$e_k$} ;
\draw (0,-0.7) node{$e_{k+1}$} ;
\draw (-0.5,0.3) node{$f_k$} ;
\draw (0.5,0.3) node{$f_{k-1}$} ;

\draw (0,-1.4) node{\large $T_k$} ;
\draw (0,-1.7) node{\tiny $(2\leqslant k \leqslant p-1)$} ;

\path [draw=black,postaction={on each segment={mid arrow=black}}]
(0,0)--(-1.732/2,-0.5);

\path [draw=black,postaction={on each segment={mid arrow =black}}]
(0,0)--(0,1);

\path [draw=black,postaction={on each segment={mid arrow =black}}]
(0,0)--(1.732/2,-0.5);

\draw[->](1.732/2,-0.5) arc (-30:-90:1);
\draw (0,-1) arc (-90:-150:1);

\draw[->](0,1) arc (90:30:1);
\draw (1.732/2,0.5) arc (30:-30:1);

\draw[->](0,1) arc (90:150:1);
\draw (-1.732/2,0.5) arc (150:210:1);

%
\begin{scope}[xshift=0cm,yshift=0cm,rotate=0,scale=1/1.5]
\draw (0.2,0.5) circle (0.15) node{\scriptsize $0$};
\end{scope}

%
\begin{scope}[xshift=0cm,yshift=0cm,rotate=0,scale=1/1.5]
\draw (0.2,-1.7) circle (0.15) node{\scriptsize $0$};
\end{scope}

\begin{scope}[xshift=0cm,yshift=0cm,rotate=0,scale=1/1.5]
\draw (-1.5,1) circle (0.15) node{\scriptsize $k$};
\end{scope}

\begin{scope}[xshift=0cm,yshift=0cm,rotate=0,scale=1/1.5]
\draw (1.5,1)  node{\tiny $k-1$};
\node[draw,ellipse] (S) at(1.5,1) {\ \ \ };
\end{scope}

\begin{scope}[xshift=0cm,yshift=0cm,rotate=0,scale=1/1.5]
\draw (-0.6,-0.6)  node{\tiny $k+1$};
\node[draw,ellipse] (S) at(-0.6,-0.6) {\ \ \ };
\end{scope}

\begin{scope}[xshift=0cm,yshift=0cm,rotate=0,scale=1/1.5]
\draw (0.6,-0.6) circle (0.15) node{\scriptsize $k$};
\end{scope}

\end{scope}

\begin{scope}[xshift=11cm,yshift=0cm,rotate=0,scale=1.5]

\draw (0,-0.15) node{\scriptsize $0$} ;
\draw (0,1.15) node{\scriptsize $1$} ;
\draw (1,-0.55) node{\scriptsize $2$} ;
\draw (-1,-0.55) node{\scriptsize $3$} ;
\draw (1,1) node{$e_p$} ;
\draw (0,-0.7) node{$e_{p+1}$} ;
\draw (-0.5,0.3) node{$f_p$} ;
\draw (0.5,0.3) node{$f_{p-1}$} ;

\draw (0,-1.4) node{\large $T_p$} ;

\path [draw=black,postaction={on each segment={mid arrow=black}}]
(0,0)--(-1.732/2,-0.5);

\path [draw=black,postaction={on each segment={mid arrow =black}}]
(0,0)--(0,1);

\path [draw=black,postaction={on each segment={mid arrow =black}}]
(0,0)--(1.732/2,-0.5);

\draw[->](1.732/2,-0.5) arc (-30:-90:1);
\draw (0,-1) arc (-90:-150:1);

\draw[->](0,1) arc (90:30:1);
\draw (1.732/2,0.5) arc (30:-30:1);

\draw[->](0,1) arc (90:150:1);
\draw (-1.732/2,0.5) arc (150:210:1);

%
\begin{scope}[xshift=0cm,yshift=0cm,rotate=0,scale=1/1.5]
\draw (0.2,0.5) circle (0.15) node{\scriptsize $0$};
\end{scope}

\begin{scope}[xshift=0cm,yshift=0cm,rotate=0,scale=1/1.5]
\draw (0.2,-1.7) circle (0.15) node{\scriptsize $0$};
\end{scope}

\begin{scope}[xshift=0cm,yshift=0cm,rotate=0,scale=1/1.5]
\draw (-1.5,1) circle (0.15) node{\scriptsize $p$};
\end{scope}

\begin{scope}[xshift=0cm,yshift=0cm,rotate=0,scale=1/1.5]
\draw (1.5,1)  node{\tiny $p-1$};
\node[draw,ellipse] (S) at(1.5,1) {\ \ \ };
\end{scope}

\begin{scope}[xshift=0cm,yshift=0cm,rotate=0,scale=1/1.5]
\draw (-0.6,-0.6)  node{\tiny $p+1$};
\node[draw,ellipse] (S) at(-0.6,-0.6) {\ \ \ };
\end{scope}

\begin{scope}[xshift=0cm,yshift=0cm,rotate=0,scale=1/1.5]
\draw (0.6,-0.6) circle (0.15) node{\scriptsize $p$};
\end{scope}

\end{scope}

\begin{scope}[xshift=0cm,yshift=-5cm,rotate=0,scale=1.5]

\draw (0,-0.15) node{\scriptsize $0$} ;
\draw (0,1.15) node{\scriptsize $1$} ;
\draw (1,-0.55) node{\scriptsize $3$} ;
\draw (-1,-0.55) node{\scriptsize $2$} ;
\draw (1,1) node{$r$} ;
\draw (0,-0.6) node{$v$} ;
\draw (-0.5,0.3) node{$g$} ;
\draw (0.5,0.3) node{$s$} ;

\draw (0,-1.4) node{\large $U$} ;

\path [draw=black,postaction={on each segment={mid arrow =black}}]
(0,0)--(0,1);

\path [draw=black,postaction={on each segment={mid arrow d=black}}]
(0,0)--(1.732/2,-0.5);

\draw[->] (0,0)--(-1.732*0.3,-1*0.3);
\draw (-1.732*0.3,-1*0.3)--(-1.732/2,-1/2);

\draw[color=black][-<](1.732/2,-0.5) arc (-30:-90:1);
\draw[color=black] (-1.732/2,-0.5) arc (-150:-87:1);

\draw[color=black][->>](0,1) arc (90:30:1);
\draw[color=black] (1.732/2,0.5) arc (30:-30:1);

\draw[color=black][->](0,1) arc (90:150:1);
\draw[color=black] (-1.732/2,0.5) arc (150:210:1);
\begin{scope}[xshift=0cm,yshift=0cm,rotate=0,scale=1/1.5]
\draw (-1.5,1)  node{\tiny $p+1$};
\node[draw,ellipse] (S) at(-1.5,1) {\ \ \ };
\end{scope}

\begin{scope}[xshift=0cm,yshift=0cm,rotate=0,scale=1/1.5]
\draw (-0.6,-0.6) circle (0.15) node{\scriptsize $p$};
\end{scope}

\end{scope}

\begin{scope}[xshift=4cm,yshift=-5cm,rotate=0,scale=1.5]

\draw (0,-0.15) node{\scriptsize $0$} ;
\draw (0,1.15) node{\scriptsize $1$} ;
\draw (1,-0.55) node{\scriptsize $3$} ;
\draw (-1,-0.55) node{\scriptsize $2$} ;
\draw (0.7,1) node{$g$} ;
\draw (0,-0.6) node{$s'$} ;
\draw (-0.5,0.3) node{$u$} ;
\draw (0.5,0.3) node{$f_p$} ;

\draw (0,-1.4) node{\large $V$} ;

\path [draw=black,postaction={on each segment={mid arrow =black}}]
(0,0)--(0,1);

\path [draw=black,postaction={on each segment={mid arrow =black}}]
(0,0)--(-1.732/2,-0.5);

\draw[color=black,<-](1.732/2*0.6,-0.5*0.6) -- (0,0);
\draw[color=black](1.732/4,-0.25) -- (1.732/2,-0.5);

\draw[color=black](1.732/2,-0.5) arc (-30:-90:1);
\draw[color=black,->] (-1.732/2,-0.5) arc (-150:-87:1);

\draw[color=black][->](0,1) arc (90:30:1);
\draw[color=black] (1.732/2,0.5) arc (30:-30:1);

\draw[color=black][->](0,1) arc (90:150:1);
\draw[color=black] (-1.732/2,0.5) arc (150:210:1);

\begin{scope}[xshift=0cm,yshift=0cm,rotate=0,scale=1/1.5]
\draw (0.2,-1.7)  node{\tiny $p+1$};
\node[draw,ellipse] (S) at(0.2,-1.7) {\ \ \ };
\end{scope}

\begin{scope}[xshift=0cm,yshift=0cm,rotate=0,scale=1/1.5]
\draw (0.6,-0.7)  node{\tiny $p+1$};
\node[draw,ellipse] (S) at(0.6,-0.7) {\ \ \ };
\end{scope}

\begin{scope}[xshift=0cm,yshift=0cm,rotate=0,scale=1/1.5]
\draw (1.5,1) circle (0.15) node{\scriptsize $p$};
\end{scope}

%
\begin{scope}[xshift=0cm,yshift=0cm,rotate=0,scale=1/1.5]
\draw (0.2,0.5) circle (0.15) node{\scriptsize $0$};
\end{scope}

\end{scope}

\begin{scope}[xshift=8cm,yshift=-5cm,rotate=0,scale=1.5]

\draw (0,-0.15) node{\scriptsize $0$} ;
\draw (0,1.15) node{\scriptsize $1$} ;
\draw (1,-0.55) node{\scriptsize $3$} ;
\draw (-1,-0.55) node{\scriptsize $2$} ;
\draw (1,1) node{$u$} ;
\draw (0,-0.6) node{$r$} ;
\draw (-0.5,0.3) node{$e_{p+1}$} ;
\draw (0.5,0.3) node{$v$} ;

\draw (0,-1.4) node{\large $W$} ;

\draw[->](0,0)--(0,0.6);
\draw(0,0.6)--(0,1);

\path [draw=black,postaction={on each segment={mid arrow d=black}}]
(0,0)--(1.732/2,-0.5);

\draw[color=black,<-](-1.732/2*0.6,-0.5*0.6) -- (0,0);
\draw[color=black](-1.732/4,-0.25) -- (-1.732/2,-0.5);

\draw[color=black][-<](1.732/2,-0.5) arc (-30:-90:1);
\draw[color=black] (-1.732/2,-0.5) arc (-150:-87:1);

\draw[color=black][->](0,1) arc (90:30:1);
\draw[color=black] (1.732/2,0.5) arc (30:-30:1);

\draw[color=black][->](0,1) arc (90:150:1);
\draw[color=black] (-1.732/2,0.5) arc (150:210:1);

\begin{scope}[xshift=0cm,yshift=0cm,rotate=0,scale=1/1.5]
\draw (-0.6,-0.7)  node{\tiny $p+1$};
\node[draw,ellipse] (S) at(-0.6,-0.7) {\ \ \ };
\end{scope}

\begin{scope}[xshift=0cm,yshift=0cm,rotate=0,scale=1/1.5]
\draw (-1.5,1) circle (0.15) node{\scriptsize $0$};
\end{scope}

%
\begin{scope}[xshift=0cm,yshift=0cm,rotate=0,scale=1/1.5]
\draw (0.2,0.5) circle (0.15) node{\scriptsize $p$};
\end{scope}

\end{scope}

\begin{scope}[xshift=12cm,yshift=-5cm,rotate=0,scale=1.5]

\draw (0,-0.15) node{\scriptsize $0$} ;
\draw (0,1.15) node{\scriptsize $1$} ;
\draw (1,-0.55) node{\scriptsize $2$} ;
\draw (-1,-0.55) node{\scriptsize $3$} ;
\draw (1,1) node{$m$} ;
\draw (0,-0.6) node{$m$} ;
\draw (-0.5,0.3) node{$s$} ;
\draw (0.5,0.3) node{$s'$} ;

\draw (0,-1.4) node{\large $Z$} ;

\path [draw=black,postaction={on each segment={mid arrow =black}}]
(0,0)--(0,1);

\draw[color=black,<-](1.732/2*0.6,-0.5*0.6) -- (0,0);
\draw[color=black](1.732/4,-0.25) -- (1.732/2,-0.5);

\path [draw=black,postaction={on each segment={mid arrow d=black}}]
(0,0)--(-1.732/2,-0.5);

\draw[very thick,color=blue][->](1.732/2,-0.5) arc (-30:-90:1);
\draw[very thick,color=blue] (-1.732/2,-0.5) arc (-150:-87:1);

\draw[color=black,->](0,1) arc (90:30:1);
\draw[color=black] (1.732/2,0.5) arc (30:-30:1);

\draw[color=black][->>](0,1) arc (90:150:1);
\draw[color=black] (-1.732/2,0.5) arc (150:210:1);

\begin{scope}[xshift=0cm,yshift=0cm,rotate=0,scale=1/1.5]
\draw (1.5,1)  node{\tiny $p+1$};
\node[draw,ellipse] (S) at(1.5,1) {\ \ \ };
\end{scope}

\begin{scope}[xshift=0cm,yshift=0cm,rotate=0,scale=1/1.5]
\draw (0.6,-0.7)  node{\tiny $p+1$};
\node[draw,ellipse] (S) at(0.6,-0.7) {\ \ \ };
\end{scope}

\end{scope}
\end{tikzpicture}
\caption{The H-triangulation $Y_n$ for $(S^3,K_n)$, $n$ odd, $n \geqslant 3$, with $p=\frac{n-3}{2}$} \label{fig:H:trig:odd}
\end{figure}

In the H-triangulation of Figure \ref{fig:H:trig:odd} there are
\begin{itemize}
\item $1$ common vertex,
\item $p+5 = \frac{n+7}{2}$ edges (simple arrow $\overrightarrow{\eta_s}$, double arrow $\overrightarrow{\eta_d}$, blue simple arrow $\overrightarrow{K_n}$, and the simple arrows $\overrightarrow{\eta_0}, \ldots, \overrightarrow{\eta_{p+1}}$ indexed by $0, \ldots p+1$ in circles),
\item $2p+8 = n+5$ faces ($e_1, \ldots, e_{p+1}, f_1, \ldots, f_{p},g, m,r,s,s',u,v  $),
\item $p+4 = \frac{n+5}{2}$ tetrahedra ($T_1, \ldots, T_{p},  U,  V, W, Z$) .
\end{itemize}

We are now ready to obtain an ideal triangulation of $S^3 \setminus K_n$.
From the H-triangulation of $(S^3,K_n)$ of Figure \ref{fig:H:trig:odd}, let us collapse  the whole tetrahedron $Z$ into a triangle: this transforms the blue edge (corresponding to $K_n$) into a point, collapses the two faces $m$, and identifies the faces $s$ and $s'$ in a new face also called $s$, and the double arrow edge to the arrow with circled $p+1$.

Hence we get an ideal triangulation of the knot complement $S^3 \setminus K_n$, detailed in Figure \ref{fig:id:trig:odd}.

\begin{figure}[!h]
\begin{tikzpicture}
%

\begin{scope}[xshift=1cm,yshift=0cm,rotate=0,scale=1.5]

\draw (0,-0.15) node{\scriptsize $0$} ;
\draw (0,1.15) node{\scriptsize $1$} ;
\draw (1,-0.55) node{\scriptsize $2$} ;
\draw (-1,-0.55) node{\scriptsize $3$} ;
\draw (1,1) node{$e_1$} ;
\draw (0,-0.6) node{$e_2$} ;
\draw (-0.5,0.3) node{$f_1$} ;
\draw (0.5,0.3) node{$e_1$} ;

\draw (0,-1.4) node{\large $T_1$} ;

\path [draw=black,postaction={on each segment={mid arrow=black}}]
(0,0)--(-1.732/2,-0.5);

\path [draw=black,postaction={on each segment={mid arrow =black}}]
(0,0)--(0,1);

\path [draw=black,postaction={on each segment={mid arrow =black}}]
(0,0)--(1.732/2,-0.5);

\draw[->](1.732/2,-0.5) arc (-30:-90:1);
\draw (0,-1) arc (-90:-150:1);

\draw[->](0,1) arc (90:30:1);
\draw (1.732/2,0.5) arc (30:-30:1);

\draw[->](0,1) arc (90:150:1);
\draw (-1.732/2,0.5) arc (150:210:1);

\begin{scope}[xshift=0cm,yshift=0cm,rotate=0,scale=1/1.5]
\draw (-1.5,1) circle (0.15) node{\scriptsize $1$};
\end{scope}

%
\begin{scope}[xshift=0cm,yshift=0cm,rotate=0,scale=1/1.5]
\draw (1.5,1) circle (0.15) node{\scriptsize $0$};
\end{scope}

%
\begin{scope}[xshift=0cm,yshift=0cm,rotate=0,scale=1/1.5]
\draw (0.2,0.5) circle (0.15) node{\scriptsize $0$};
\end{scope}

%
\begin{scope}[xshift=0cm,yshift=0cm,rotate=0,scale=1/1.5]
\draw (0.2,-1.7) circle (0.15) node{\scriptsize $0$};
\end{scope}

\begin{scope}[xshift=0cm,yshift=0cm,rotate=0,scale=1/1.5]
\draw (-0.6,-0.6) circle (0.15) node{\scriptsize $2$};
\end{scope}

\begin{scope}[xshift=0cm,yshift=0cm,rotate=0,scale=1/1.5]
\draw (0.6,-0.6) circle (0.15) node{\scriptsize $1$};
\end{scope}

\end{scope}

\draw (3.5,0) node{$\ldots$} ;
\draw (8.5,0) node{$\ldots$} ;

\begin{scope}[xshift=6cm,yshift=0cm,rotate=0,scale=1.5]

\draw (0,-0.15) node{\scriptsize $0$} ;
\draw (0,1.15) node{\scriptsize $1$} ;
\draw (1,-0.55) node{\scriptsize $2$} ;
\draw (-1,-0.55) node{\scriptsize $3$} ;
\draw (1,1) node{$e_k$} ;
\draw (0,-0.7) node{$e_{k+1}$} ;
\draw (-0.5,0.3) node{$f_k$} ;
\draw (0.5,0.3) node{$f_{k-1}$} ;

\draw (0,-1.4) node{\large $T_k$} ;
\draw (0,-1.7) node{\tiny $(2\leqslant k \leqslant p-1)$} ;

\path [draw=black,postaction={on each segment={mid arrow=black}}]
(0,0)--(-1.732/2,-0.5);

\path [draw=black,postaction={on each segment={mid arrow =black}}]
(0,0)--(0,1);

\path [draw=black,postaction={on each segment={mid arrow =black}}]
(0,0)--(1.732/2,-0.5);

\draw[->](1.732/2,-0.5) arc (-30:-90:1);
\draw (0,-1) arc (-90:-150:1);

\draw[->](0,1) arc (90:30:1);
\draw (1.732/2,0.5) arc (30:-30:1);

\draw[->](0,1) arc (90:150:1);
\draw (-1.732/2,0.5) arc (150:210:1);

%
\begin{scope}[xshift=0cm,yshift=0cm,rotate=0,scale=1/1.5]
\draw (0.2,0.5) circle (0.15) node{\scriptsize $0$};
\end{scope}

%
\begin{scope}[xshift=0cm,yshift=0cm,rotate=0,scale=1/1.5]
\draw (0.2,-1.7) circle (0.15) node{\scriptsize $0$};
\end{scope}

\begin{scope}[xshift=0cm,yshift=0cm,rotate=0,scale=1/1.5]
\draw (-1.5,1) circle (0.15) node{\scriptsize $k$};
\end{scope}

\begin{scope}[xshift=0cm,yshift=0cm,rotate=0,scale=1/1.5]
\draw (1.5,1)  node{\tiny $k-1$};
\node[draw,ellipse] (S) at(1.5,1) {\ \ \ };
\end{scope}

\begin{scope}[xshift=0cm,yshift=0cm,rotate=0,scale=1/1.5]
\draw (-0.6,-0.6)  node{\tiny $k+1$};
\node[draw,ellipse] (S) at(-0.6,-0.6) {\ \ \ };
\end{scope}

\begin{scope}[xshift=0cm,yshift=0cm,rotate=0,scale=1/1.5]
\draw (0.6,-0.6) circle (0.15) node{\scriptsize $k$};
\end{scope}

\end{scope}

\begin{scope}[xshift=11cm,yshift=0cm,rotate=0,scale=1.5]

\draw (0,-0.15) node{\scriptsize $0$} ;
\draw (0,1.15) node{\scriptsize $1$} ;
\draw (1,-0.55) node{\scriptsize $2$} ;
\draw (-1,-0.55) node{\scriptsize $3$} ;
\draw (1,1) node{$e_p$} ;
\draw (0,-0.7) node{$e_{p+1}$} ;
\draw (-0.5,0.3) node{$f_p$} ;
\draw (0.5,0.3) node{$f_{p-1}$} ;

\draw (0,-1.4) node{\large $T_p$} ;

\path [draw=black,postaction={on each segment={mid arrow=black}}]
(0,0)--(-1.732/2,-0.5);

\path [draw=black,postaction={on each segment={mid arrow =black}}]
(0,0)--(0,1);

\path [draw=black,postaction={on each segment={mid arrow =black}}]
(0,0)--(1.732/2,-0.5);

\draw[->](1.732/2,-0.5) arc (-30:-90:1);
\draw (0,-1) arc (-90:-150:1);

\draw[->](0,1) arc (90:30:1);
\draw (1.732/2,0.5) arc (30:-30:1);

\draw[->](0,1) arc (90:150:1);
\draw (-1.732/2,0.5) arc (150:210:1);

%
\begin{scope}[xshift=0cm,yshift=0cm,rotate=0,scale=1/1.5]
\draw (0.2,0.5) circle (0.15) node{\scriptsize $0$};
\end{scope}

\begin{scope}[xshift=0cm,yshift=0cm,rotate=0,scale=1/1.5]
\draw (0.2,-1.7) circle (0.15) node{\scriptsize $0$};
\end{scope}

\begin{scope}[xshift=0cm,yshift=0cm,rotate=0,scale=1/1.5]
\draw (-1.5,1) circle (0.15) node{\scriptsize $p$};
\end{scope}

\begin{scope}[xshift=0cm,yshift=0cm,rotate=0,scale=1/1.5]
\draw (1.5,1)  node{\tiny $p-1$};
\node[draw,ellipse] (S) at(1.5,1) {\ \ \ };
\end{scope}

\begin{scope}[xshift=0cm,yshift=0cm,rotate=0,scale=1/1.5]
\draw (-0.6,-0.6)  node{\tiny $p+1$};
\node[draw,ellipse] (S) at(-0.6,-0.6) {\ \ \ };
\end{scope}

\begin{scope}[xshift=0cm,yshift=0cm,rotate=0,scale=1/1.5]
\draw (0.6,-0.6) circle (0.15) node{\scriptsize $p$};
\end{scope}

\end{scope}

\begin{scope}[xshift=1cm,yshift=-5cm,rotate=0,scale=1.5]

\draw (0,-0.15) node{\scriptsize $0$} ;
\draw (0,1.15) node{\scriptsize $1$} ;
\draw (1,-0.55) node{\scriptsize $3$} ;
\draw (-1,-0.55) node{\scriptsize $2$} ;
\draw (1,1) node{$r$} ;
\draw (0,-0.6) node{$v$} ;
\draw (-0.5,0.3) node{$g$} ;
\draw (0.5,0.3) node{$s$} ;

\draw (0,-1.4) node{\large $U$} ;

\path [draw=black,postaction={on each segment={mid arrow =black}}]
(0,0)--(0,1);

\path [draw=black,postaction={on each segment={mid arrow =black}}]
(0,0)--(1.732/2,-0.5);

\draw[->] (0,0)--(-1.732*0.3,-1*0.3);
\draw (-1.732*0.3,-1*0.3)--(-1.732/2,-1/2);

\draw[color=black][-<](1.732/2,-0.5) arc (-30:-90:1);
\draw[color=black] (-1.732/2,-0.5) arc (-150:-87:1);

\draw[color=black][->](0,1) arc (90:30:1);
\draw[color=black] (1.732/2,0.5) arc (30:-30:1);

\draw[color=black][->](0,1) arc (90:150:1);
\draw[color=black] (-1.732/2,0.5) arc (150:210:1);
\begin{scope}[xshift=0cm,yshift=0cm,rotate=0,scale=1/1.5]
\draw (-1.5,1)  node{\tiny $p+1$};
\node[draw,ellipse] (S) at(-1.5,1) {\ \ \ };
\end{scope}

\begin{scope}[xshift=0cm,yshift=0cm,rotate=0,scale=1/1.5]
\draw (1.5,1)  node{\tiny $p+1$};
\node[draw,ellipse] (S) at(1.5,1) {\ \ \ };
\end{scope}

\begin{scope}[xshift=0cm,yshift=0cm,rotate=0,scale=1/1.5]
\draw (0.6,-0.7)  node{\tiny $p+1$};
\node[draw,ellipse] (S) at(0.6,-0.7) {\ \ \ };
\end{scope}

\begin{scope}[xshift=0cm,yshift=0cm,rotate=0,scale=1/1.5]
\draw (-0.6,-0.6) circle (0.15) node{\scriptsize $p$};
\end{scope}

\end{scope}

\begin{scope}[xshift=6cm,yshift=-5cm,rotate=0,scale=1.5]

\draw (0,-0.15) node{\scriptsize $0$} ;
\draw (0,1.15) node{\scriptsize $1$} ;
\draw (1,-0.55) node{\scriptsize $3$} ;
\draw (-1,-0.55) node{\scriptsize $2$} ;
\draw (0.7,1) node{$g$} ;
\draw (0,-0.6) node{$s$} ;
\draw (-0.5,0.3) node{$u$} ;
\draw (0.5,0.3) node{$f_p$} ;

\draw (0,-1.4) node{\large $V$} ;

\path [draw=black,postaction={on each segment={mid arrow =black}}]
(0,0)--(0,1);

\path [draw=black,postaction={on each segment={mid arrow =black}}]
(0,0)--(-1.732/2,-0.5);

\draw[color=black,<-](1.732/2*0.6,-0.5*0.6) -- (0,0);
\draw[color=black](1.732/4,-0.25) -- (1.732/2,-0.5);

\draw[color=black](1.732/2,-0.5) arc (-30:-90:1);
\draw[color=black,->] (-1.732/2,-0.5) arc (-150:-87:1);

\draw[color=black][->](0,1) arc (90:30:1);
\draw[color=black] (1.732/2,0.5) arc (30:-30:1);

\draw[color=black][->](0,1) arc (90:150:1);
\draw[color=black] (-1.732/2,0.5) arc (150:210:1);

\begin{scope}[xshift=0cm,yshift=0cm,rotate=0,scale=1/1.5]
\draw (0.2,-1.7)  node{\tiny $p+1$};
\node[draw,ellipse] (S) at(0.2,-1.7) {\ \ \ };
\end{scope}

\begin{scope}[xshift=0cm,yshift=0cm,rotate=0,scale=1/1.5]
\draw (0.6,-0.7)  node{\tiny $p+1$};
\node[draw,ellipse] (S) at(0.6,-0.7) {\ \ \ };
\end{scope}

\begin{scope}[xshift=0cm,yshift=0cm,rotate=0,scale=1/1.5]
\draw (1.5,1) circle (0.15) node{\scriptsize $p$};
\end{scope}

%
\begin{scope}[xshift=0cm,yshift=0cm,rotate=0,scale=1/1.5]
\draw (0.2,0.5) circle (0.15) node{\scriptsize $0$};
\end{scope}

\end{scope}

\begin{scope}[xshift=11cm,yshift=-5cm,rotate=0,scale=1.5]

\draw (0,-0.15) node{\scriptsize $0$} ;
\draw (0,1.15) node{\scriptsize $1$} ;
\draw (1,-0.55) node{\scriptsize $3$} ;
\draw (-1,-0.55) node{\scriptsize $2$} ;
\draw (1,1) node{$u$} ;
\draw (0,-0.6) node{$r$} ;
\draw (-0.5,0.3) node{$e_{p+1}$} ;
\draw (0.5,0.3) node{$v$} ;

\draw (0,-1.4) node{\large $W$} ;

\draw[->](0,0)--(0,0.6);
\draw(0,0.6)--(0,1);

\path [draw=black,postaction={on each segment={mid arrow =black}}]
(0,0)--(1.732/2,-0.5);

\draw[color=black,<-](-1.732/2*0.6,-0.5*0.6) -- (0,0);
\draw[color=black](-1.732/4,-0.25) -- (-1.732/2,-0.5);

\draw[color=black][-<](1.732/2,-0.5) arc (-30:-90:1);
\draw[color=black] (-1.732/2,-0.5) arc (-150:-87:1);

\draw[color=black][->](0,1) arc (90:30:1);
\draw[color=black] (1.732/2,0.5) arc (30:-30:1);

\draw[color=black][->](0,1) arc (90:150:1);
\draw[color=black] (-1.732/2,0.5) arc (150:210:1);

\begin{scope}[xshift=0cm,yshift=0cm,rotate=0,scale=1/1.5]
\draw (-0.6,-0.7)  node{\tiny $p+1$};
\node[draw,ellipse] (S) at(-0.6,-0.7) {\ \ \ };
\end{scope}

\begin{scope}[xshift=0cm,yshift=0cm,rotate=0,scale=1/1.5]
\draw (0.6,-0.7)  node{\tiny $p+1$};
\node[draw,ellipse] (S) at(0.6,-0.7) {\ \ \ };
\end{scope}

\begin{scope}[xshift=0cm,yshift=0cm,rotate=0,scale=1/1.5]
\draw (-1.5,1) circle (0.15) node{\scriptsize $0$};
\end{scope}

%
\begin{scope}[xshift=0cm,yshift=0cm,rotate=0,scale=1/1.5]
\draw (0.2,0.5) circle (0.15) node{\scriptsize $p$};
\end{scope}

\end{scope}

\end{tikzpicture}
\caption{The ideal triangulation $X_n$ for $S^3\setminus K_n$, $n$ odd, $n \geqslant 3$, with $p=\frac{n-3}{2}$} \label{fig:id:trig:odd}
\end{figure}
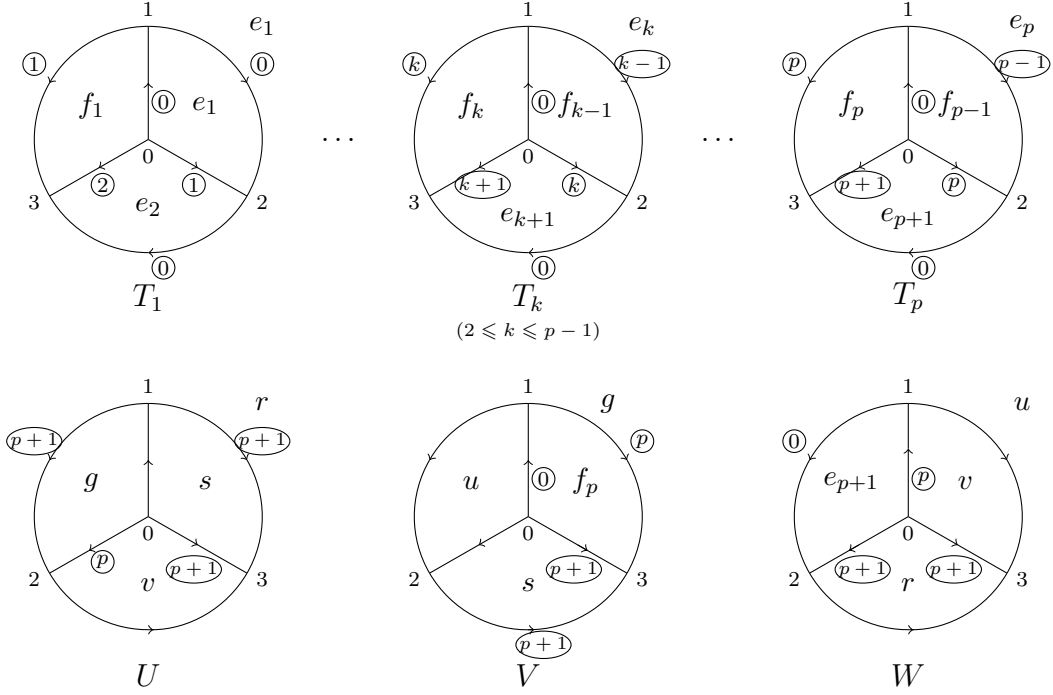

In Figure \ref{fig:id:trig:odd} there are
\begin{itemize}
\item $1$ common vertex,
\item $p+3 = \frac{n+3}{2}$ edges (simple arrow $\overrightarrow{\eta_s}$ and the simple arrows $\overrightarrow{\eta_0}, \ldots, \overrightarrow{\eta_{p+1}}$ indexed by $0, \ldots p+1$ in circles),
\item $2p+6 = n+3$ faces ($e_1, \ldots, e_{p+1}, f_1, \ldots, f_{p}, g,r,s,u,v  $),
\item $p+3 = \frac{n+3}{2}$ tetrahedra ($T_1, \ldots, T_{p}, U,  V, W$).
\end{itemize}

\subsection{Proof of Theorem \ref{thm:trig}}

We can now conclude with the proof of Theorem \ref{thm:trig}.

\begin{proof}[Proof of Theorem \ref{thm:trig}]
The triangulations of Figures \ref{fig:H:trig:odd} and \ref{fig:id:trig:odd} correspond to the common ``comb representation'' of Figure \ref{fig:trig:odd}.

Similarly, the triangulations of Figures \ref{fig:H:trig:even} and \ref{fig:id:trig:even} (constructed in Section \ref{sub:even:trig}) correspond to the common ``comb representation'' of Figure \ref{fig:trig:even}.
\end{proof}

\section{Angle structures and geometricity (odd case)}\label{sec:geom}

In this section, $n$ will be an odd integer greater than or equal to $3$. 

\subsection{Geometricity of the ideal triangulations}

Here we will compute the balanced angle relations for the ideal triangulations $X_n$ and their spaces of angle structures $\mathcal{A}_{X_n}$. We will then prove that the $X_n$ are \textit{geometric}.

\begin{theorem}\label{thm:geometric}
For every odd $n \geqslant 3$, the ideal triangulation $X_n$ of the $n$-th twist knot complement $S^3 \setminus K_n$ is geometric.
\end{theorem}

To prove Theorem \ref{thm:geometric}, we follow  Futer--Gu\'eritaud \cite{FG}: we first prove that the space of angle structures $\mathcal{A}_{X_n}$ is non-empty (Lemma \ref{lem:non:empty}); then we prove by contradiction that the volume functional cannot attain its maximum on the boundary $\overline{\mathcal{A}_{X_n}} \setminus \mathcal{A}_{X_n}$ (Lemma \ref{lem:interior}).

For the remainder of this section, $n$ will be a fixed odd integer, $n \geqslant 7$. Recall that $p=\frac{n-3}{2}$. The cases $n=3, 5$ (i.e.\ $p=0, 1$) are similar and simpler than the general following $n \geqslant 7$ case, and will be discussed at the end of this section (Remark
\ref{rem:p1}).

Recall that we denoted $\overrightarrow{\eta_0}, \ldots, \overrightarrow{\eta_{p+1}}, \overrightarrow{\eta_s} \in (X_n)^1$ the $p+3$ edges in $X_n$ respectively represented in Figure \ref{fig:id:trig:odd} by arrows with circled $0$, \ldots, circled $p+1$ and simple arrow.

For $\alpha=(a_1,b_1,c_1,\ldots,a_p,b_p,c_p,a_U,b_U,c_U,a_V,b_V,c_V,a_W,b_W,c_W) \in \mathcal{S}_{X_n}$ a shape structure on $X_n$, we compute the weights of each edge:
\begin{itemize}
\item $\omega_s(\alpha):= \omega_{X_n,\alpha}(\overrightarrow{\eta_s})=
2 a_U+b_V+c_V+a_W+b_W
$
\item $\omega_0(\alpha):= \omega_{X_n,\alpha}(\overrightarrow{\eta_0})=
2 a_1 + c_1 + 2 a_2 + \ldots + 2 a_p + a_V+c_W
$
\item $\omega_1(\alpha):= \omega_{X_n,\alpha}(\overrightarrow{\eta_1})=
2b_1+c_2
$
\\

\item $\omega_k(\alpha):= \omega_{X_n,\alpha}(\overrightarrow{\eta_k})=
c_{k-1}+2b_k+c_{k+1}
$ \ \
(for $2\leqslant k \leqslant p-1$)
\\
\vspace*{-2mm}
\item $\omega_p(\alpha):= \omega_{X_n,\alpha}(\overrightarrow{\eta_p})=
c_{p-1}+2b_p+b_U+b_V+a_W$
\item $\omega_{p+1}(\alpha):= \omega_{X_n,\alpha}(\overrightarrow{\eta_{p+1}})=
c_p+b_U+2c_U+a_V+c_V+b_W+c_W
$
\end{itemize}
The space of angle structures $\mathcal{A}_{X_n}$ is made of shape structures $\alpha \in \mathcal{S}_{X_n}$ satisfying $\omega_j(\alpha)=2\pi$ for all $j\in\{s,0,\ldots,p+1\}$. 
The sum of all these equations says that all the angles add up to $(p+3)\pi$, which is true in any shape structure, therefore we can drop $\omega_0(\alpha)$ as redundant.
Using the properties of shape structures, $\mathcal{A}_{X_n}$ is thus defined by the $p+2$ following equations on $\alpha$:

\begin{itemize}
\item $E_s(\alpha): \ 2 a_U =a_V+c_W $
\item $E_1(\alpha): \ 2b_1+c_2 = 2 \pi$ 
\\ \vspace*{-2mm}
\item $E_k(\alpha): \ c_{k-1}+2b_k+c_{k+1} = 2 \pi $
\ \ (for $2\leqslant k \leqslant p-1$)
\\ \vspace*{-2mm}
\item $E_p(\alpha): \    c_{p-1}+2b_p+ (b_U+b_V+a_W)=2\pi$
\item $E_{p+1}(\alpha): \ 3c_p + (a_U+a_V+c_W) + 3(c_U+c_V+b_W) = 3\pi~;$
\end{itemize}
the last line was obtained as $3B_{p+1}+2 B_s - 3F_U - 2F_V - 2F_W$,
where $F_j$ is the relationship $a_j+b_j+c_j=\pi$ and $B_j$ is the relationship $\omega_{j}(\alpha) = 2 \pi$.
In other words, 
$$\mathcal{A}_{X_n} = \{ \alpha \in \mathcal{S}_{X_n} \ | \ \forall j \in \{s,1,\ldots,p+1\}, \ E_j(\alpha)\}.$$

\begin{lemma}\label{lem:non:empty}
The set  $\mathcal{A}_{X_n}$ is non-empty.
\end{lemma}

\begin{proof}
For small $\epsilon>0$, define:
$$
\begin{pmatrix}a_j\\b_j\\c_j \end{pmatrix}:=
\begin{pmatrix}\epsilon\\ \pi - \epsilon(j^2+1) \\ \epsilon j^2 \end{pmatrix}
\text{for } 1\leqslant j \leqslant p-1
, \
\begin{pmatrix}a_p\\b_p\\c_p \end{pmatrix}:=
\begin{pmatrix} \pi/2 - \epsilon(p^2+2p-1)/2 \\ \pi/2 - \epsilon(p^2-2p+1)/2 \\ \epsilon p^2 \end{pmatrix},
$$
$$
\begin{pmatrix}a_U\\b_U\\c_U \end{pmatrix} = 
\begin{pmatrix}a_V\\b_V\\c_V \end{pmatrix} = 
\begin{pmatrix}c_W\\a_W\\b_W \end{pmatrix}:=
\begin{pmatrix}  \pi/2 + \epsilon p^2/2  \\ \pi/3   \\ \pi/6 - \epsilon p^2/2  \end{pmatrix}.
$$
By direct computation, we can check that this $\alpha$  is a shape structure (the angles are in $(0,\pi)$ if $\epsilon$ is small enough), and that  the equations $E_j(\alpha)$ are satisfied for $j\in\{s,1,\ldots,p+1\}$.
\end{proof}

We will say that a tetrahedron $T$ of a triangulation $X$ endowed with an extended shape structure $\alpha \in \overline{\mathcal{S}_X}$ is \textit{flat for $\alpha$} if  {at least} one of the three angles of $T$ is zero, and \textit{taut for $\alpha$} if two angles are zero and the third is $\pi$. In both cases, $T$ has a volume equal to zero. 

\begin{lemma}\label{lem:flat:taut}
Suppose $\alpha \in \overline{\mathcal{A}_{X_n}} \setminus \mathcal{A}_{X_n}$ is such that the volume functional on $\overline{\mathcal{A}_{X_n}}$ is maximal at $\alpha$. If an angle of $\alpha$ equals $0$, then the other two angles for the same tetrahedron are $0$ and~$\pi$. In other words, if a tetrahedron is flat for $\alpha$, then it is taut for $\alpha$.
\end{lemma}

\begin{proof}
The proof is a computation, for which we refer to \cite[Proposition 7.1]{Gf}.
The basic idea is that unflattening a single flat-but-not-taut tetrahedron will make its volume increase with unbounded derivative.
\end{proof}

Next, we claim that among the volume maximizers, there is one such that 
$(a_U,b_U,c_U)=(a_V,b_V,c_V)=(c_W,a_W,b_W)$.
The involution $(a_V, b_V, c_V) \leftrightarrow (c_W,a_W,b_W)$ preserves all equations $E_j(\alpha)$, so by concavity of the volume function, there is a maximizer such that $(a_V, b_V, c_V)=(c_W,a_W,b_W)$. 
By $E_s(\alpha)$ this implies $a_U=a_V=c_W$. 
The order-3 substitution of variables $$(a_U, b_U, c_U) \rightarrow  (a_V, b_V, c_V) \rightarrow (c_W, a_W, b_W) \rightarrow (a_U, b_U, c_U)$$
then clearly leaves $E_p$ and $E_{p+1}$ unchanged, so by concavity we may average out and find a maximizer such that 
$U,V,W$ have the same angles, as desired.

These identifications make $E_s(\alpha)$ redundant. Moreover, dropping the angles of $V$ and $W$ as variables, we may now rewrite the system of constraints as 
\begin{itemize}
\item $E_1 : \ 2b_1+c_2 = 2 \pi$
\vspace{2mm}
\item $E_k : \ c_{k-1}+2b_k+c_{k+1} = 2 \pi $ \quad (for $2\leqslant k \leqslant p-1$)
\vspace{2mm} 
\item $E'_p : \ c_{p-1}+2b_p + 3b_U  = 2 \pi$
\item $E'_{p+1} : \ c_p + a_U + 3c_U  = \pi$  \quad (not $2\pi$!).
\end{itemize}

\begin{lemma}\label{lem:interior}
Suppose  that the volume functional on $\overline{\mathcal{A}_{X_n}}$ is maximal at $\alpha$. Then $\alpha$ cannot be on the boundary $\overline{\mathcal{A}_{X_n}} \setminus \mathcal{A}_{X_n}$, and is necessarily in the interior $\mathcal{A}_{X_n}$.
\end{lemma}

\begin{proof}
By Lemma~\ref{lem:flat:taut}, it is enough enough to show that there are no taut tetrahedra, i.e.\ that each tetrahedron has at least one angle that is not in $\{0,\pi\}$.

First, the tetrahedron $T_p$ is not taut. 
Indeed, on one hand $c_p=\pi$ would by $E'_{p+1}$ entail $a_U=c_U=0$, hence $b_U=\pi$, incompatible with $E'_p$.
On the other hand,  suppose by contradiction that $c_p=0$. Observe that the non-negative sequence $(0, c_1, \dots, c_p)$ is convex, because $E_k$ can be rewritten $c_{k-1} - 2c_k + c_{k+1} = 2 a_k \geq 0$ (agreeing that ``$c_0$'' stands for $0$). 
Hence $c_1=\dots=c_p=0$, and $b_p\in\{0,\pi \}$ by Lemma~\ref{lem:flat:taut}. 
If $b_p=0$ then $(E'_p, E'_{p+1})$ yield $(a_U, b_U, c_U)=(0,2\pi/3, \pi/3)$. 
If $b_p=\pi$ they yield $(a_U, b_U, c_U)=(\pi,0,0)$.
In either case, all tetrahedra are flat so the volume vanishes and cannot be  {maximal. This} contradiction shows $c_p>0$.

Next, we show that $U$ is not 
taut. 
We cannot have $c_U=\pi$ or $b_U=\pi$, by $E'_{p+1}$ and $E'_p$.
But $a_U=\pi$ is also impossible, since by $E'_{p+1}$ it would imply $c_p=0$, ruled out above.

We can see by induction that $b_1,\dots, b_{p-1}>0$: the initialisation is given by $E_1$, written as $b_1=\pi-c_2/2\geq \pi/2$. 
For the induction step, suppose $b_{k-1}>0$ for some $1<  k \leq p-1$: then $c_{k-1}<\pi$, hence $E_k$ implies $b_k>0$.

Finally, $b_1,\dots, b_{p-1}<\pi$: we show this by \emph{descending} induction. 
Initialisation: by $E_{p-1}$, we have $b_{p-1}\leq \pi-c_p/2 < \pi$ since $T_p$ is not flat. 
For the induction step, suppose $b_{k+1}<\pi$ for some $1\leq  k < p-1$: then $0< b_{k+1}<\pi$ by the previous induction, hence $c_{k+1}>0$ by Lemma~\ref{lem:flat:taut}, hence $E_k$ implies $b_k<\pi$.
\end{proof}

\begin{remark}[Cases $p=0,1$]\label{rem:p1}
The above discussion is valid for $p\geq 2$.
If $p=1$, we have only the weights $\omega_s$, $\omega_{p+1}$ and $\omega_p$, the latter taking the form $2b_p+b_U+b_V+a_W$ (i.e.\ the variable ``$c_{p-1}$'' disappears from equation $E'_p$). The argument is otherwise unchanged --- the inductions in the proof of Lemma~\ref{lem:interior} being empty.

If $p=0$, we find only one equation $E'_{p+1}: \ a_U+3c_U=\pi$ (i.e.\ the variable ``$c_p$'' disappears). The volume maximizer $(a_U, b_U, c_U)$ on the segment from $(\pi,0,0)$ to $(0,2\pi/3, \pi/3)$ yields the complete hyperbolic metric.
\end{remark}

\begin{remark}
In \emph{most} boundary points of $\mathcal{A}_{X_n}$ (e.g.\ inside a top-dimensional face of $\mathcal{A}_{X_n}$), some tetrahedron is flat but not taut.
From such a point, it is easy to increase volume by moving inwards (see Lemma~\ref{lem:flat:taut}). 
The proof of Lemma~\ref{lem:interior} above essentially says that \emph{every} boundary point is either of that sort, or has zero volume. 
Ideal triangulations other than $X_n$ do not always have that convenient property: compare with Lemma~\ref{lem:girafe} below (for even twist knots) and the more involved discussion that follows it. 
\end{remark}

\begin{proof}[Proof of Theorem \ref{thm:geometric}]
In the case $n\geqslant 7$, we have proven in Lemma 
\ref{lem:non:empty} that $\mathcal{A}_{X_n}$ is non-
empty, thus the volume functional $\mathcal{V}\colon 
\overline{\mathcal{A}_{X_n}}\to \R$ admits a maximum at 
a certain point $\alpha \in \overline{\mathcal{A}
_{X_n}}$  as a continuous function on a non-empty 
compact set. We proved in Lemma \ref{lem:interior} that 
$\alpha \notin \overline{\mathcal{A}_{X_n}} \setminus \mathcal{A}_{X_n}$, therefore 
$\alpha \in  \mathcal{A}_{X_n}$. It follows from 
Theorem \ref{thm:casson:rivin} that $X_n$ is geometric.

For the cases $n=3$ and $n=5$, we follow the same reasoning, 
replacing Lemma~\ref{lem:interior} with Remark~\ref{rem:p1}.
\end{proof}

\subsection{The cusp triangulation} \label{sub:cusp:trig}

\begin{figure}[!h]
	\centering
	\begin{tikzpicture}[every path/.style={string ,black}]
	
	\begin{scope}[scale=0.65]
	
	\draw[color=blue] (11-21,0.5) node {$3_U$};
	\draw (10.5-21,4) node {$v$};
	\draw (11.7-21,5) node {$s$};
	\draw (-10.5,-1.7) node {$r$};
\draw[color=brown] (-11.7,-1.7) node {$a$};
\draw[color=brown] (-9.4,-1.7) node {$b$};
	\draw[color=brown] (11.8-21,10) node {$c$};
	
	\draw[color=blue] (-10.3,9) node {$3_W$};
	\draw (10.5-21,6) node {$v$};
	\draw (10.5-21,11.7) node {$r$};
	\draw (-11.7,4) node {$u$};

	\draw[color=brown] (-11.8,1) node {$b$};
	\draw[color=brown] (9.3-21,11.7) node {$a$};
	\draw[color=brown] (11.7-21,11.7) node {$c$};


	\draw[color=blue] (-7.9,3) node {$3_V$};
	\draw (-7.7,-0.8) node {$g$};
	\draw (-8.7,6) node {$s$};
	\draw (-7.6,4.8) node {$f_p$};
	\draw[color=brown] (-8.8,-1.5) node {$a$};
	\draw[color=brown] (-6.3,0.3) node {$b$};
	\draw[color=brown] (-8.8,10) node {$c$};
	
	\draw[color=blue] (-5,6.5) node {$3_p$};
	\draw (-7,6) node {$f_p$};
	\draw (-4,10) node {$e_{p+1}$};
	\draw (-3,5) node {$e_p$};
	\draw[color=brown] (-0.4,8.8) node {$a$};
	\draw[color=brown] (-5.9,0.6) node {$b$};
	\draw[color=brown] (-8.7,11.6) node {$c$};
	
	\draw[color=blue] (-2,11) node { $2_W$};
	\draw (-3.5,10.5) node { $e_{p+1}$};
	\draw (-0.3,10.5) node { $u$};
	\draw (-4.5,11.7) node { $r$};
	\draw[color=brown] (-0.3,11.7) node {$a$};
	\draw[color=brown] (-7.2,11.7) node {$b$};
	\draw[color=brown] (-0.3,9.3) node {$c$};

	\draw[color=blue] (-6,-1.3) node {$2_U$};
	\draw (-6.8,-0.8) node {$g$};
	\draw (-5,-1.7) node {$r$};
	\draw (-3.5,-1.2) node {$v$};
	\draw[color=brown] (-2,-1.7) node {$a$};
	\draw[color=brown] (-6,-0.3) node {$b$};
	\draw[color=brown] (-8,-1.7) node {$c$};
	
	\draw[color=blue] (-1.5,-0.5) node {$1_W$};
	\draw (-2.7,-0.8) node { $v$};
	\draw (-2.6,0.1) node { $e_{p+1}$};
	\draw (-0.3,-0.5) node { $u$};
	\draw[color=brown] (-5.1,-0.1) node {$a$};
	\draw[color=brown] (-0.3,-1.6) node {$b$};
	\draw[color=brown] (-0.3,0.6) node {$c$};

	
	\draw[color=blue] (11,0.5) node {$2_V$};
	\draw (10.5,6) node {$g$};
	\draw (11.7,5) node {$u$};
	\draw (10.5,11.7) node {$s$};
	\draw[color=brown] (9.3,11.7) node {$b$};
	\draw[color=brown] (11.8,10) node {$c$};
	\draw[color=brown] (11.7,11.7) node {$a$};
	
	\draw[color=blue] (10.5,9.5) node {$1_U$};
	\draw (10.5,4) node {$g$};
	\draw (21-11.7,4) node {$r$};
	\draw (21-10.5,-1.7) node {$s$};
	\draw[color=brown] (21-11.8,1) node {$c$};
	\draw[color=brown] (21-11.7,-1.7) node {$a$};
	\draw[color=brown] (21-9.4,-1.7) node {$b$};

	
	\draw[color=blue] (6,10+1.2) node {$0_U$};
	\draw (7,11) node {$v$};
	\draw (5,11.7) node {$s$};
	\draw (3.5,11.2) node {$g$};
	\draw[color=brown] (2,11.7) node {$a$};
	\draw[color=brown] (6,10.3) node {$b$};
	\draw[color=brown] (8,11.7) node {$c$};
	
	\draw[color=blue] (8,7) node {$0_W$};
	\draw (7.7,10.8) node {$v$};
	\draw (8.7,4) node {$r$};
	\draw (7.8,10-4.8) node {$e_{p+1}$};
	\draw[color=brown] (8.8,11.5) node {$c$};
	\draw[color=brown] (6.3,10-0.3) node {$a$};
	\draw[color=brown] (8.8,0) node {$b$};
	
	\draw[color=blue] (5,10-6.5) node {$0_p$};
	\draw (6.9,10-6) node {$e_{p+1}$};
	\draw (4,0) node {$f_{p}$};
	\draw (3.4,5) node {$f_{p-1}$};
	\draw[color=brown] (0.4,10-8.8) node {$a$};
	\draw[color=brown] (5.9,10-0.6) node {$b$};
	\draw[color=brown] (8.7,10-11.6) node {$c$};
	
	\draw[color=blue] (1.7,-0.8) node {$0_{V}$};
	\draw (0.3,-0.5) node { $u$};
	\draw (2.5,-0.3) node { $f_p$};
	\draw (4.5,-1.8) node { $s$};
	\draw[color=brown] (0.2,0.7) node { $a$};
	\draw[color=brown] (0.2,-1.7) node { $b$};
	\draw[color=brown] (7.7,-1.8) node { $c$};
	
	\draw[color=blue] (1.5,10.6) node {$1_{V}$};
	\draw (0.2,10.5) node { $u$};
	\draw (2,9.7) node { $f_p$};
	\draw (3,10.7) node { $g$};
	\draw[color=brown] (0.2,9.4) node {$a$};
	\draw[color=brown] (5,10.1) node {$b$};
	\draw[color=brown] (0.2,11.7) node {$c$};

	\draw (0,-2)--(10,-2)--(12,-2)--(12,12)--(9,12)--(0,12)--(-9,12)--(-12,12)--(-12,-2)--(-9,-2)--(0,-2);
	\draw (0,-2)--(0,1)--(9,-2)--(6,10)--(9,12)--(9,-2);
	\draw (9,-2)--(12,12);
	\draw (6,10)--(0,12)--(0,9)--(-9,12)--(-6,0)--(0,-2);
	\draw (-6,0)--(-9,-2)--(-9,12);
	\draw (-12,-2)--(-9,12);
	
	\draw[color=black] (-6,0)--(0,9)--(6,10)--(0,1)--(-6,0);
	\draw[color=black] (-6,0)--(-3.6,2)--(0,1)--(-2.4,3)--(-1.2,4)--(0,1)--(0,9)--(-1.2,4);
	\draw[color=black] (-3.6,2)--(0,9)--(-2.4,3);
	\draw[color=black] (0,9)--(1.2,6)--(0,1)--(2.4,7)--(0,9)--(3.6,8)--(0,1);
	\draw[color=black] (1.2,6)--(2.4,7);
	\draw[color=black] (3.6,8)--(6,10);

	\draw[color=violet,style=dashed,very thick] (10,-2)--(10,6);
	\draw[color=violet,style=dashed, very thick,<-] (10,6)--(10,12);
	\draw[color=violet] (9.6,7) node {\scriptsize $m_{X_n}$};

	\draw[color=teal,style=dashed, very thick,->] (-12,5.5)--(-11,5.5+3.25);	
		\draw[color=teal,style=dashed, very thick] (-11,5.5+3.25)--(-10,12);	
	\draw[color=teal,style=dashed, very thick,->] (-10,-2)--(-9,-1)--(-5,-1);	
		\draw[color=teal,style=dashed, very thick]	(-5,-1)--(0,-1)--(7.5,-2);	
	\draw[color=teal,style=dashed, very thick,->] (7.5,12)--(7.5+0.9,12-1.3)	;
		\draw[color=teal,style=dashed, very thick] (7.5+0.9,12-1.3)--(12,5.5)	;
		\draw[color=teal] (-11,10) node {\scriptsize $l_{X_n}$};
	
	\draw[color=blue] (3.3,8.5) node {$1_p$};
	\draw (4.5,9.1) node {$e_p$};
	\draw (3.4,9.3) node {$f_p$};
	\draw (2.1,8.8) node {$f_{p-1}$};
	\draw[color=brown] (1,8.9) node {$a$};
	
	\draw[color=blue] (4,7.7) node {\tiny $0_{p-1}$};
	\draw (4.5,8.45) node {\tiny $e_p$};
	\draw (3,6.5) node {\tiny $f_{p-2}$};
	\draw (4,7.2) node {\tiny $f_{p-1}$};
	\draw[color=brown] (1.8,4) node {\tiny $a$};
	
	\draw[color=blue] (-3.9,2.3) node {\tiny $3_{p-1}$};
	\draw (-3,4) node {\tiny $e_p$};
	\draw (-3.45,3) node {\tiny $e_{p-1}$};
	\draw (-4.35,1.7) node {\tiny $f_{p-1}$};
	\draw[color=brown] (-1.2,6.95) node {\tiny $a$};
	
	\draw[color=blue] (-3.4,1.2) node {$2_{p}$};
	\draw (-2,1.4) node {\tiny $e_p$};
	\draw (-3,0.65) node {\tiny $e_{p+1}$};
	\draw (-4.2,0.9) node {\tiny $f_{p-1}$};
	\draw[color=brown] (-1,1.05) node {\tiny $a$};
	
	\draw[color=blue] (-0.7,4) node {$2_{1}$};
	\draw (-0.5,5.9) node {\tiny $e_1$};
	\draw (-0.3,4.7) node {\tiny $e_1$};
	\draw (-0.5,3) node {\tiny $e_2$};
	\draw[color=brown] (-0.2,2) node {\tiny $a$};
	
	\draw[color=blue] (-1.4,4.5) node {$3_1$};
	\draw (-1,5.7) node {\tiny $e_1$};
	\draw (-1.35,5) node {\tiny $e_2$};
	\draw (-1.65,3.9) node {\tiny $f_{1}$};
	\draw[color=brown] (-0.65,7) node {\tiny $a$};
	
	\draw[color=blue] (-1.7,2.9) node {$2_2$};
	\draw (-1.3,2.3) node {\tiny $e_3$};
	\draw (-1,2.8) node {\tiny $e_2$};
	\draw (-1.5,3.45) node {\tiny $f_{1}$};
	\draw[color=brown] (-0.6,1.8) node {\tiny $a$};
	
	\draw[color=blue] (0.5,6) node {$1_{1}$};
	\draw (0.25,4.8) node {\tiny $e_1$};
	\draw (0.5,4) node {\tiny $e_1$};
	\draw (0.5,7) node {\tiny $f_1$};
	\draw[color=brown] (0.15,8.1) node {\tiny $a$};
	
	\draw[color=blue] (1.8,7.1) node {$1_{2}$};
	\draw (1.6,6.6) node {\tiny $e_2$};
	\draw (1.3,7.6) node {\tiny $f_2$};
	\draw (1.1,7) node {\tiny $f_1$};
	\draw[color=brown] (0.5,8.3) node {\tiny $a$};
	
	\draw[color=blue] (1.45,5.5) node {$0_1$};
	\draw (1.8,6.2) node {\tiny $e_2$};
	\draw (1.1,4.5) node {\tiny $e_1$};
	\draw (1.4,5) node {\tiny $f_1$};
	\draw[color=brown] (0.65,3) node {\tiny $a$};
	
	\draw (3-0.12,7.5-0.1) node[shape=circle,fill=black,scale=0.2] {};
	\draw (3,7.5) node[shape=circle,fill=black,scale=0.2] {};
	\draw (3+0.12,7.5+0.1) node[shape=circle,fill=black,scale=0.2] {};
	
	\draw (-3-0.12,2.5-0.1) node[shape=circle,fill=black,scale=0.2] {};
	\draw (-3,2.5) node[shape=circle,fill=black,scale=0.2] {};
	\draw (-3+0.12,2.5+0.1) node[shape=circle,fill=black,scale=0.2] {};
	
	
\draw[color=red, very thick,->] (-12,12)--(-11,12);
\draw[color=red, very thick] (-11,12)--(-9,12);
	\draw[color=red] (-10.7,12.5) node {(i)};
	
\draw[color=red, very thick,->] (-9,12)--(-9,4);
\draw[color=red, very thick] (-9,4)--(-9,-2);
\draw[color=red] (-8.4,4) node {(ii)};

\draw[color=red, very thick,->] (-9,-2)--(-4,-2);
\draw[color=red, very thick] (-4,-2)--(0,-2);
\draw[color=red] (-4,-2.5) node {(iii)};

\draw[color=red, very thick,->] (0,12)--(4,12);
\draw[color=red, very thick] (4,12)--(9,12);
\draw[color=red] (4,12.5) node {(iv)};

\draw[color=red, very thick,->] (9,12)--(9,6);
\draw[color=red, very thick] (9,6)--(9,-2);
\draw[color=red] (8.4,6) node {(v)};

\draw[color=red, very thick,->] (9,-2)--(11,-2);
\draw[color=red, very thick] (11,-2)--(12,-2);
\draw[color=red] (10.7,-2.5) node {(vi)};

	\end{scope}
	
	\end{tikzpicture}
	\caption{Triangulation of the boundary torus for the truncation of $X_n$, $n$ odd, with angles (brown), meridian curve $m_{X_n}$ (violet, dashed), longitude curve $l_{X_n}$ (green, dashed) and preferred longitude curve $l_{X_n}^0=$ (i)$\cup \ldots \cup$(vi) (red).}\label{fig:trig:cusp:odd}
\end{figure}
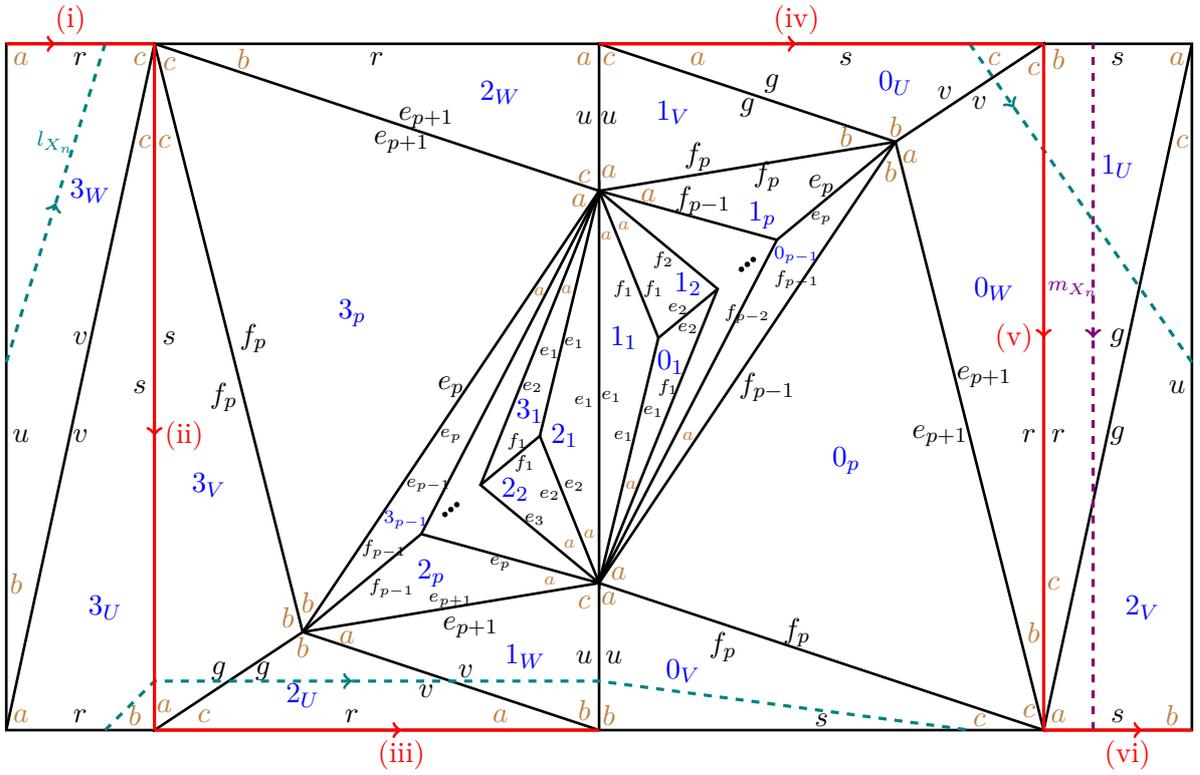

If we truncate the ideal triangulation $X_n$ of Figure  \ref{fig:id:trig:odd} by removing a small  {neighbourhood} of each vertex, then we obtain a cellular decomposition by compact truncated tetrahedra of the knot exterior $S^3 \setminus \nu(K_n)$ (where $\nu(K)$ is an open tubular  {neighbourhood} of $K$). This induces a triangulation on the boundary torus $\partial \nu(K_n)$, where each triangle comes from a pre-quotient vertex of a tetrahedron of $X$. See Figure \ref{fig:trig:cusp:odd} for the full description of the triangulation of this torus. 

The triangles are called (in blue) by the names of the corresponding truncated vertices (written $k_j$ for the $k$-th vertex in the $j$-th tetrahedron), the edges are called (in black) by the names of the truncated faces they are part of, and the angles $a,b,c$ at each corner of a triangle (in brown) obviously come from the corresponding truncated edges in $X_n$. Note that we did not put the indices on $a,b,c$ for readability, but it goes without saying that angles $a,b,c$ in the triangle $k_j$ are actually the coordinates $a_j,b_j,c_j$. Moreover, for some small faces, we only indicated the brown $a$ angle for readability; the $b$ and $c$ follow clockwise (since all the concerned tetrahedra have positive sign).

We drew three particular curves  in Figure \ref{fig:trig:cusp:odd}: $m_{X_n}$ in violet and dashed, $l_{X_n}$ in green and dashed, and finally the concatenation  {$l_{X_n}^0:=$} (i) $\cup \ldots \cup$ (vi) in red. These curves can be seen as generators of the first homology group of the torus. We call $m_{X_n}$ a \textit{meridian curve} since it actually comes from the projection to $\partial \nu(K_n)$ of a meridian curve in $S^3 \setminus K_n$, the one circling the knot and going through faces $s$ and $E$ on the upper left of Figure \ref{fig:diagram:htriang}, to be exact (we encourage the motivated reader to check this fact by following the curve on the several pictures from Figure \ref{fig:diagram:htriang} to \ref{fig:id:trig:odd}). Similarly, $l_{X_n}$ and  {$l_{X_n}^0=$} (i) $\cup \ldots \cup$ (vi) are two distinct  \textit{longitude curves}, and  {$l_{X_n}^0$} corresponds to a \textit{preferred longitude} of the knot $K_n$, i.e.\ a longitude with zero linking number with the knot. 

This last fact can be checked in Figure \ref{fig:longitude:odd}: on the bottom of the figure, the sub-curves (i) to (vi) are drawn on a truncated tetrahedron $U$; on the top of the figure, the corresponding full longitude curve (in red) is drawn in the exterior of the knot (in blue) before the collapsing of the knot into one point (compare with Figure \ref{fig:diagram:htriang}). We check that in each square on the left of the figure, the sum of the signs of crossings between blue and red strands is zero (the signs are marked in green circled $+$ and $-$), and thus the red longitude curve has zero linking number with the knot, i.e.\ is a preferred longitude.

\begin{figure}
	\includegraphics[scale=1.7]{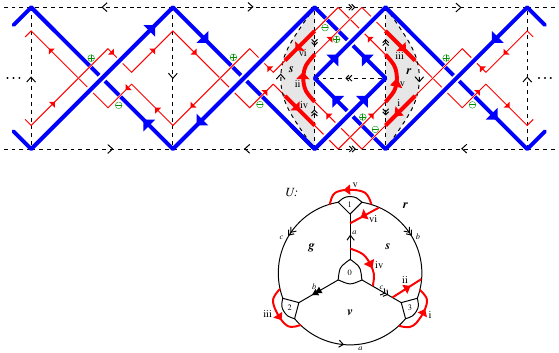}
	\caption{A preferred longitude  $l_{X_n}^0=$ (i)$\cup \ldots \cup$(vi) (in red) for the odd twist knot $K_n$, seen in $S^3 \setminus K_n$ (top) and on the truncated tetrahedron $U$ (bottom).}\label{fig:longitude:odd}
\end{figure}

To the curves $m_{X_n}$ and $l_{X_n}$ are associated combinations of angles (the \textit{angular holonomies})
$$m_{X_n}(\alpha):= H^\R(m_{X_n}) = a_U-a_V \ \ \text{and} \ \ l_{X_n}(\alpha):=H^\R(l_{X_n}) = 2(c_V-b_W),$$
following the convention that when the curve crosses a triangle, the lone angle among the three is counted positively if it lies on the left of the curve, and negatively if it lies on the right. Remark that this convention cannot rigorously be applied to the red curve  {$l_{X_n}^0=$} {(i) $\cup \ldots \cup$ (vi)} in Figure \ref{fig:trig:cusp:odd}, since it lies on edges and vertices. Nevertheless, one can see in Figure \ref{fig:trig:cusp:odd} that in the homology group of the boundary torus, we have the relation
$$  {l_{X_n}^0} = l_{X_n} + 2 m_{X_n}.$$

\subsection{The complex gluing equations}\label{sub:complete:odd}

Here seems to be an appropriate place to list the complex versions of the balancing and completeness equations for $X_n$, which will be useful in Section \ref{sec:vol:conj}.

For a complex shape structure $\widetilde{\mathbf{z}}=(z_1,\ldots,z_p,z_U,z_V,z_W) \in (\R+i\R_{>0})^{p+3}$, its complex weight functions are:

\begin{itemize}
\item $\omega^{\C}_s(\widetilde{\mathbf{z}}):= \omega^{\C}_{X_n,\alpha}(\overrightarrow{\eta_s})=
2\Log(z_U) + \Log(z'_V) + \Log(z''_V) + \Log(z_W) + \Log(z'_W)
$
\item $\omega^{\C}_0(\widetilde{\mathbf{z}}):= \omega^{\C}_{X_n,\alpha}(\overrightarrow{\eta_0})=
2\Log(z_1) + \Log(z'_1) + 2\Log(z_2) + \cdots + 2\Log(z_p) + \Log(z_V) + \Log(z''_W)
$
\item $\omega^{\C}_1(\widetilde{\mathbf{z}}):= \omega^{\C}_{X_n,\alpha}(\overrightarrow{\eta_1})=
2\Log(z''_1) + \Log(z'_2)
$
\\
\vspace*{-2mm}
\item $\omega^{\C}_k(\widetilde{\mathbf{z}}):= \omega^{\C}_{X_n,\alpha}(\overrightarrow{\eta_k})=
\Log(z'_{k-1}) + 2\Log(z''_k) + \Log(z'_{k+1})
$
\ \
(for $2\leqslant k \leqslant p-1$)
\\
\vspace*{-2mm}
\item $\omega^{\C}_p(\widetilde{\mathbf{z}}):= \omega^{\C}_{X_n,\alpha}(\overrightarrow{\eta_p})=
\Log(z'_{p-1}) + 2\Log(z''_p) + \Log(z'_U) + \Log(z'_V) + \Log(z_W)$
\item $\omega^{\C} _{p+1}(\widetilde{\mathbf{z}}):= \omega^{\C}_{X_n,\alpha}(\overrightarrow{\eta_{p+1}})=
\Log(z'_p) + \Log(z'_U) + 2\Log(z''_U) + \Log(z_V) + \Log(z''_V) + \Log(z'_W) + \Log(z''_W)
$
\end{itemize}

It follows from Theorem \ref{thm:geometric} that there exists exactly one complex angle structure $\widetilde{\mathbf{z^0}}=(z_1^0,\ldots,z_p^0,z_U^0,z_V^0,z_W^0)\in (\R+i\R_{>0})^{p+3}$ corresponding to the complete hyperbolic metric. This $\widetilde{\mathbf{z^0}}$ is the only $\widetilde{\mathbf{z}} \in (\R+i\R_{>0})^{p+3}$ satisfying 
$$ \omega^{\C}_s(\widetilde{\mathbf{z}}) = \omega^{\C}_0(\widetilde{\mathbf{z}}) = \ldots = \omega^{\C}_{p+1}(\widetilde{\mathbf{z}}) = 2 i \pi$$
as well as the complex completeness equation
$$\Log(z_U)-\Log(z_V)=0$$
coming from the meridian curve $m_{X_n}$.

These conditions are equivalent to the following system $\mathcal{E}^{co}_{X_n}(\widetilde{\mathbf{z}})$ of equations on $\widetilde{\mathbf{z}}$:
\begin{itemize}
\item $\mathcal{E}_{X_n,0}(\widetilde{\mathbf{z}}) \colon \Log(z'_1) + 2 \Log(z_1)+\cdots + 2\Log(z_p)+2\Log(z_U) = 2i\pi$
\item $\mathcal{E}_{X_n,1}(\widetilde{\mathbf{z}}) \colon 2\Log(z''_1)+\Log(z'_2)=2i\pi$\\
\vspace*{-2mm}
\item $\mathcal{E}_{X_n,k}(\widetilde{\mathbf{z}}) \colon \Log(z'_{k-1})+2\Log(z''_k)+\Log(z'_{k+1})=2i\pi$ \ \ (for $2 \leqslant k \leqslant p-1$)\\
\vspace*{-2mm}
\item $\mathcal{E}_{X_n,p+1}^{co}(\widetilde{\mathbf{z}}) \colon \Log(z'_{p}) +2 \Log(z''_{U})-\Log(z_{W})=0$
\item $\mathcal{E}_{X_n,s}^{co}(\widetilde{\mathbf{z}}) \colon \Log(z''_{W}) -\Log(z_{U})=0$
\item $z_V=z_U$
\end{itemize}
Indeed, notice that the equation $\omega^{\C}_p(\widetilde{\mathbf{z}})=2i\pi$ was redundant with the other complex balancing equation. Remark furthermore that the variable $z_V$ only appears in the equation $z_V=z_U$, which is why  we will allow a slight abuse of notation to use the equations 
$$\mathcal{E}_{X_n,0}(\mathbf{z}), \ldots, \mathcal{E}_{X_n,p-1}(\mathbf{z}),
\mathcal{E}^{co}_{X_n,p+1}(\mathbf{z}), \mathcal{E}^{co}_{X_n,s}(\mathbf{z})$$
 also for a variable $\mathbf{z}=(z_1,\ldots,z_p,z_U,z_W) \in (\R+i\R_{>0})^{p+2}$ without the coordinate $z_V$ (see Lemma \ref{lem:grad:thurston}).


\section{Partition function for the ideal triangulations (odd case)}\label{sec:part:odd}

\begin{notation}
From now, we will denote $\stareq$ the equality up to taking the complex  {modulus}.
\end{notation}

In this section, $n$ will be an odd integer greater than or equal to $3$, and $p=\frac{n-3}{2}$.
We will compute the partition functions of the Teichm\"uller TQFT for the ideal triangulations $X_n$ of the twist knot complements $S^3\setminus K_n$ constructed in Section \ref{sec:trig} and we will prove that they can be expressed in a simple way using a one-variable function independent of the angle structure, as well as only two linear combinations of angles, which are two independent angular holonomies in the cusp link triangulation.

This results in a slightly different version of the first statement in the Andersen--Kashaev volume conjecture of  \cite[Conjecture 1 (1)]{AK}.
Note that our partition functions are computed only for the specific ideal triangulations $X_n$. In order to generalise  Theorem \ref{thm:part:func} to any ideal triangulation of a twist knot complement, one would need further properties of invariance under change of triangulation (more general than the ones discussed in \cite{AK}).
A version for the even case is proved in Section \ref{sub:even:tqft} (see Theorem \ref{thm:even:part:func}).

\begin{theorem}\label{thm:part:func}
Let $n\geqslant 3$ be an odd integer and $p=\frac{n-3}{2}$. Consider the ideal triangulation $X_n$ of $S^3\setminus K_n$ described in Figure \ref{fig:id:trig:odd}. Then for all angle structures $\alpha=(a_1,\ldots,c_W) \in \mathcal{A}_{X_n}$ and all $\hbar>0$, we have:
\begin{equation*}
\mathcal{Z}_{\hbar}(X_n,\alpha) 
\stareq 
\int_{\mathbb{R}+i \frac{\mu_{X_n}(\alpha) }{2\pi \sqrt{\hbar}}  } 
J_{X_n}(\hbar,x)
e^{\frac{1}{2 \sqrt{\hbar}}  x  \lambda_{X_n}(\alpha)} 
dx,
\end{equation*}
with 
\begin{itemize}
\item the degree one angle polynomial $\mu_{X_n}\colon\alpha\mapsto  a_U- a_V$,
\item the degree one angle polynomial $\lambda_{X_n}\colon\alpha\mapsto 2(a_U-a_V+c_V-b_W)$,
\item the map $(\hbar,x) \mapsto$
\begin{equation*}
J_{X_n}(\hbar,x)=\int_{\mathcal{Y}'} d\mathbf{y'} \
e^{2 i \pi \left (\mathbf{y'}^{\!\top} Q_n \mathbf{y'}
	+ x(x- y'_U-y'_W)\right )}
e^{
	\frac{1}{\sqrt{\hbar}} \left (\mathbf{y'}^{\!\top} \mathcal{W}_n
	- \pi x\right )}
\dfrac{
	\Phi_\B\left (y'_U\right )
	\Phi_\B\left (y'_U+x\right )
	\Phi_\B\left (y'_W\right )
}{
	\Phi_\B\left (y'_1\right )
	\cdots
	\Phi_\B\left (y'_p\right )
}
,
\end{equation*}
where $\mathcal{Y}'=\mathcal{Y}'_{\hbar,\alpha}=
\prod_{k=1}^p\left (\R - \frac{i}{2 \pi \sqrt{\hbar}} (\pi-a_k)\right ) 
\times 
\prod_{l=U,W}
\left (\R + \frac{i}{2 \pi \sqrt{\hbar}} (\pi-a_l)\right ),$
\begin{equation*}
\mathbf{y'}=\begin{bmatrix}
y'_1 \\ \vdots \\ y'_p \\ y'_U \\y'_W
\end{bmatrix}, 
\quad 
\mathcal{W}_n=\begin{bmatrix}-2p\pi \\ \vdots \\ -2 \pi \left ( k p - \frac{k(k-1)}{2}\right ) \\ \vdots \\ -p(p+1)\pi \\ (p^2+p+1)\pi \\ \pi\end{bmatrix} 
\quad 
\text{ and }
\quad
Q_n=\begin{bmatrix}
1 & 1 & \cdots & 1 & -1 & 0 \\ 
1 & 2 & \cdots & 2 & -2 & 0 \\ 
\vdots & \vdots & \ddots & \vdots & \vdots & \vdots \\ 
1 & 2 & \cdots & p & -p & 0 \\ 
-1 & -2 & \cdots & -p & p & \frac{1}{2} \\ 
0 & 0 & \cdots & 0 & \frac{1}{2} & 0 
\end{bmatrix}.
\end{equation*}
\end{itemize}
\end{theorem}

The reader may notice that indices corresponding to $V$ are missing in the integration variables. This comes from the change of variables $x= y'_V-y'_U$, which makes $x$ replace the variable $y'_V$. Simply speaking, we chose to make $V$ disappear rather than $U$, because   $V$ appeared a lot less than $U$ in the defining gluing equations (see end of Section \ref{sec:geom}).

\begin{remark}
Note that, if you fix $\hbar>0$ and $x \in \R + i\left (-\frac{1}{2 \sqrt{\hbar}},\frac{1}{2 \sqrt{\hbar}}\right )$,
 the integration contour $\mathcal{Y}'$ in the definition of $J_{X_n}(\hbar,x)$ depends a priori on the angle structure $\alpha$; however, since the integrand in $J_{X_n}(\hbar,x)$ is a holomorphic function of the variables in $\mathbf{y'}$ on a {neighbourhood} of $\mathcal{Y}'$ in $\C^{p+2}$, it follows from the Bochner--Martinelli formula (that generalises the Cauchy theorem, see \cite{Kr}) and the fast decay properties of this integrand at infinity that $\mathcal{Y}'$ could be replaced with a different contour. In this sense, $J_{X_n}(\hbar,x)$ is  independent of the angle structure $\alpha$. Nevertheless, picking the particular contour $\mathcal{Y}'=\mathcal{Y}'(\hbar,\alpha)$ with the complete structure $\alpha=\alpha^0$ will help us  prove the volume conjecture in Section \ref{sec:vol:conj}.
\end{remark}

\begin{remark}
The quantities $\mu_{X_n}(\alpha)$ and $\lambda_{X_n}(\alpha)$ in Theorem \ref{thm:part:func} satisfy the following relations with the angular holonomies corresponding to the meridian and longitude curves $m_{X_n}(\alpha), l_{X_n}(\alpha)$ from Section \ref{sub:cusp:trig}:
$$ \mu_{X_n}(\alpha) = m_{X_n}(\alpha) \ \ \text{and} \ \ \lambda_{X_n}(\alpha) = l_{X_n}(\alpha) + 2 m_{X_n}(\alpha).$$
Hence, $\lambda_{X_n}(\alpha)$ is the angular holonomy of a curve on $\partial \nu (K_n)$ that is equal in homology to the curve {$l_{X_n}^0=$} (i) $\cup \ldots \cup$ (vi) (of Figures \ref{fig:trig:cusp:odd} and \ref{fig:longitude:odd}), thus $\lambda_{X_n}(\alpha)$ comes from a preferred longitude of the knot, as expected in Conjecture \ref{conj:vol:BAGPN} (1). Similarly, $\mu_{X_n}(\alpha)$ is associated to a meridian of the knot.
\end{remark}

We will need two lemmas to prove Theorem \ref{thm:part:func}.

\begin{lemma}\label{lem:kin:odd}
	Let $n\geqslant 3$ be an odd integer and $p=\frac{n-3}{2}$. For the ideal triangulation $X_n$ of $S^3\setminus K_n$ described in Figure \ref{fig:id:trig:odd}, the kinematical kernel is 
 $\mathcal{K}_{X_n}(\mathbf{\widetilde{t}})= \exp\left (2 i \pi \mathbf{\widetilde{t}}^{\top} \widetilde{Q}_n \mathbf{\widetilde{t}} \right ),$
 where $\mathbf{\widetilde{t}} = (t_1, \ldots, t_p, t_U, t_V, t_W)^{\!\top} \in \R^{X_n^{3}}$ and
 $\widetilde{Q}_n$  is the following symmetric matrix with half-integer coefficients:
$$\widetilde{Q}_n=\kbordermatrix{
	\mbox{}	&t_1		&t_2		&\cdots 	& t_{p-1} 	& t_p 	& \omit\vrule	& t_U 	& t_V 	& t_W 	\\
	t_1 		& 1 		& 1 		& \cdots 	& 1  		& 1 		& \omit\vrule	& -1 		& 0 		& 0 		\\
	t_2 		& 1 		& 2 		& \cdots 	& 2  		& 2 		& \omit\vrule	& -2 		& 0 		& 0 		\\
	\vdots 	& \vdots 	& \vdots 	& \ddots 	& \vdots  	& \vdots 	& \omit\vrule	& \vdots 	& \vdots 	& \vdots 	\\
	t_{p-1} 	& 1 		& 2 		& \cdots 	& p-1  	& p-1 	& \omit\vrule	& -(p-1) 	& 0 		& 0 		\\
	t_p 		& 1 		&  2 		& \cdots 	&  p-1 	& p 		& \omit\vrule	& -p 		& 0 		& 0 		\\
	\cline{1-1} \cline{2-10}
	t_U 		& -1 		& -2 		& \cdots 	& -(p-1)  	& -p  	& \omit\vrule	& p+2 	& -3/2 	& 1 		\\
	t_V 		& 0 		& 0 		& \cdots 	&  0 		& 0 		& \omit\vrule	& -3/2 	& 1 		& -1/2 	\\
	t_W 		& 0 		& 0 		& \cdots 	&  0 		& 0 		& \omit\vrule	& 1 		& -1/2 	& 0 }.
$$
\end{lemma}

\begin{proof}
Let  $n \geqslant 3$ be an odd integer and $p=\frac{n-3}{2}$. We will denote $$\mathbf{\widetilde{t}} = (\mathsf{t}(T_1), \ldots, \mathsf{t}(W))^{\!\top}= (t_1, \ldots, t_p, t_U, t_V, t_W)^{\!\top} \in \R^{X_n^{3}}$$ a vector whose coordinates are associated to the tetrahedra ($t_j$ for $T_j$).
The generic vector in $\R^{X_n^{2}}$ corresponding to the face variables will be denoted  
$$\mathbf{x}=(e_1, \ldots,e_p,e_{p+1}, f_1, \ldots, f_p,v,r,s,g,u)^{\!\top} \in \R^{X_n^{2}}.$$
By definition, the kinematical kernel is:
$$\mathcal{K}_{X_n}\left (\mathbf{\widetilde{t}}\right ) = \int_{\mathbf{x} \in \R^{X_n^{2}}} d\mathbf{x} \prod_{T \in X_n^3} e^{2 i \pi \varepsilon(T) x_0(T) \mathsf{t}(T)}
\delta( x_0(T)- x_1(T)+ x_2(T))
\delta( x_2(T)- x_3(T)+ \mathsf{t}(T)).
$$

Following  Lemma \ref{lem:dirac}
 we compute from Figure \ref{fig:id:trig:odd} that:
$$
\mathcal{K}_{X_n}\left (\mathbf{\widetilde{t}}\right ) =
\int_{\mathbf{x} \in \R^{X_n^{2}}} d\mathbf{x} 
\int_{\mathbf{w} \in \R^{2(p+3)}} d\mathbf{w} \
e^{ 2 i \pi \mathbf{\widetilde{t}}^{\!\top} R \mathbf{x}}
e^{ -2 i \pi \mathbf{w}^{\!\top} A \mathbf{x}}
e^{ -2 i \pi \mathbf{w}^{\!\top} B \mathbf{\widetilde{t}}},
$$
where
$\mathbf{w}=(w_1, \ldots,w_W, w'_1, \ldots, w'_W)^{\!\top} \in \R^{2(p+3)}$
and the matrices $R,A,B$ are given by: 
$$
R=\kbordermatrix{
	\mbox{} 	& e_1 	& \dots 	& e_p	&e_{p+1} 	& \omit\vrule	&f_1	& \ldots 	& f_p	&\omit\vrule	&  v 	& r 	& s 		& g 	& u \\
	t_1 		& 1	 	 & 		&\push{\low{0}} 	& \omit\vrule	& 	&  		&  	&\omit\vrule	&   	&  	&   		&  	&  \\
	\vdots 	& 	    	&\ddots 	& 		&	  	& \omit\vrule	& 	&  0		&  	&\omit\vrule	& 	&   	&  0		&	&  \\
	t_{p}    	&\push{0}		       	& 1        	& 		& \omit\vrule	&  	&  		&	&\omit\vrule	&	&	&		&	&  \\ 
	\cline{1-1} \cline{2-15}
	t_U	 	& 		& 		& 		&  		& \omit\vrule	&	&	 	&	&\omit\vrule	&0  	&-1 	&0  		&0 	&0 \\
	t_V	 	& 		&\push{0}	  		&  		& \omit\vrule	&	& 	0	& 	&\omit\vrule	&0  	&0 	&0  		&-1 	&0 \\
	t_W	 	& 		&  		& 		& 	 	& \omit\vrule	&	&	  	&  	&\omit\vrule	&0 	&0 	&0  		&0 	&-1
},
$$


$$
A=\kbordermatrix{
	\mbox{} 	& e_1 	& e_2 	& \dots 	&e_p		& e_{p+1}	& \omit\vrule	&f_1	&f_{2}	&  \ldots 	& f_p	&\omit\vrule	&  v 	& r 	& s 		& g 	& u \\
	w_1 		& 1		&   -1 	& 		&		&		& \omit\vrule	& 1	&		&  		&  	&\omit\vrule	&   	&  	&  	 	&  	&  \\
	\vdots 	&     		&   \ddots	& \ddots 	&	\push{0}		& \omit\vrule	& 	&\ddots	&\push{0}  	&\omit\vrule	& 	&   	&  \low{0}	&	&  \\
	\vdots 	&     	\push{0}   		& \ddots 	&\ddots	&		& \omit\vrule	& \push{0}  	&\ddots	&  	&\omit\vrule	& 	&   	&  		&	&  \\
	w_{p} 	&		&		&		&	1	& -1		& \omit\vrule 	&	&		&		&1	&\omit\vrule	&	&	&		&	&  \\
\cline{1-1} \cline{2-17}
	w_{U} 	&		&		&		&		&  		& \omit\vrule 	&	&		&		&0	&\omit\vrule	&-1	&1	&1		&0	&0  \\
	w_{V} 	&		&		& 0		&		&  		& \omit\vrule 	&	&		&		&1	&\omit\vrule	&0	&0	&-1		&1	&0  \\
	w_{W} 	&		&		&		&		&  		& \omit\vrule 	&	&		&		&0	&\omit\vrule	&1	&-1	&0		&0	&1  \\
\cline{1-1} \cline{2-17}
	w'_1 		& -1		&   	 	& 		&		&		& \omit\vrule	& 1	&  		&  		&	&\omit\vrule	&   	&  	&   		&  	&  \\
	\vdots 	&     		&   		& 	 	&		&		& \omit\vrule	& -1	&  \ddots	&  \push{0}	&\omit\vrule	& 	&   	&  \low{0}	&	&  \\
	\vdots 	&     		&   		& 	 	&		&		& \omit\vrule	& 	&  \ddots	& \ddots 	&	&\omit\vrule	& 	&   	&  		&	&  \\
	w'_{p} 	&		&		&		&		&  		& \omit\vrule 	&	\push{0}	&-1		&1	&\omit\vrule	&	&	&		&	&  \\
\cline{1-1} \cline{2-17}
	w'_{U} 	&		&		&		&		&  		& \omit\vrule 	&	&		&		&0	&\omit\vrule	&0	&0	&1		&-1	&0  \\
	w'_{V} 	&		&		&		&		&  		& \omit\vrule 	&	&		&		&1	&\omit\vrule	&0	&0	&0		&0	&-1 \\
	w'_{W} 	&		&		&		&		&  -1		& \omit\vrule 	&	&		&		&0	&\omit\vrule	&1	&0	&0		&0	&0 
},
$$

$$
B=\kbordermatrix{
	\mbox{} 	& t_1 	& \dots 	& t_{p} 	& \omit\vrule	& t_{U}	& t_{V}	& t_{W}	\\ 
	w_1 		& 		&   	 	& 		&			&		&		& 		\\
	\vdots 	&     		&   		& 	 	&			&		&		& 		\\
	w_{p} 	&		&		&	    \multicolumn{3}{c}{\low{0}}	& 		&		\\  \cline{1-1} 
	w_{U} 	&		&		&		&			&		&  		&		\\
	w_{V} 	&		&		& 		&			&		&  		&		\\
	w_{W} 	&		&		&		&			&		&  		&		\\
\cline{1-1} \cline{2-8}
	w'_1 		& 1		&   	 	& 		&			&		&		&  		\\
	\vdots 	&     		&\ddots   	& 	 	&			&		&	0	& 		\\
	w'_{p} 	&		&		& \multicolumn{3}{c}{\low{\ddots}}	&  		&		\\  \cline{1-1}  
	w'_{U} 	&		&		&		&			&		&  		&		\\
	w'_{V} 	&		&	0	&		&			&		&  \ddots	& 		\\
	w'_{W} 	&		&		&		&			&		&  		&	1	
}.
$$

Careful computation yields that $\det(A)=1$ and that the inverse $A^{-1}$ 
is equal to
$$
A^{-1}=\kbordermatrix{
	\mbox{} 	& w_{1}	& w_{2} 	& \ldots	& w_{p-1}	& w_{p}	& \omit\vrule	& w_{U}	& w_{V}	& w_{W}	& \omit\vrule	&  w'_{1}	& w'_{2}	& \ldots	& w'_{p-1}			& w'_{p}			& \omit\vrule 	& w'_{U}	& w'_{V}	& w'_{W}	\\ 
	e_{1} 	& 	0	&		& \cdots	& 		&	0	& \omit\vrule 	& 	0	& 	1	& 	0	& \omit\vrule	&  	-1	& 	-1	&	\push{\cdots}			& 	-1			& \omit\vrule 	& 	1	& 	0	& 	0	\\ 
	e_{2} 	& 	-1	& 	0	& 		& 		& 		& \omit\vrule 	& 	0	& 	2	& 	0	& \omit\vrule	&  	-1	& 	-2	&	\push{\cdots}			& 	-2			& \omit\vrule 	& 	2	& 	0	& 	0	\\ 
\low{\vdots} 	& 	-1	& 	-1	& \ddots	& 		& \vdots	& \omit\vrule 	& 		& 		& 		& \omit\vrule	& 	 	&		& \ddots	& 				& 	\vdots		& \omit\vrule 	& 		& 		& 		\\ 
		 	& \vdots	& 		& \ddots	& 	0	& 	0	& \omit\vrule 	& \vdots	& \vdots	& \vdots	& \omit\vrule	&  \vdots	& \vdots	&		&\text{\tiny {\(1-p\)}}	& \text{\tiny {\(1-p\)}}	& \omit\vrule 	& \vdots	& \vdots	& \vdots	\\ 
	e_{p} 	& 		& 		& 		& 	-1	& 	0	& \omit\vrule 	& 		& 		& 		& \omit\vrule	&  		& 		&		& \text{\tiny {\(1-p\)}}	& 	-p			& \omit\vrule 	& 		& 		& 		\\ 
	e_{p+1} 	& -1		& 		& \cdots	& 		& 	-1	& \omit\vrule 	& 	0	&  p+1	& 	0	& \omit\vrule	&  	-1	&	-2	&\cdots	& \text{\tiny {\(1-p\)}}	& 	-p			& \omit\vrule 	& 	p+1	& 	0	& 	0	\\ 
\cline{1-1} \cline{2-20}
	f_{1} 	& 		& 		& 		& 		& 		& \omit\vrule 	& 	0	& 	1	& 	0	& \omit\vrule	&  	0	& -1		& 	\push{\cdots}			&	-1			& \omit\vrule 	& 	1	& 	0	& 	0	\\ 
	f_{2} 	& 		& 		& 		& 		& 		& \omit\vrule 	& 		& 	1	& 		& \omit\vrule	&  	0	&0		& \ddots	&				&	\low{\vdots}	& \omit\vrule 	& 	1	& 		& 		\\ 
	\vdots 	& 		& 		& 	0	& 		& 		& \omit\vrule 	& \vdots	& \vdots	& \vdots	& \omit\vrule	& \vdots 	& 		& \ddots	&	-1			&	-1			& \omit\vrule 	& \vdots	& \vdots	& \vdots	\\ 
	f_{p-1} 	& 		& 		& 		& 		& 		& \omit\vrule 	& 		& 		& 		& \omit\vrule	&  	0	& 		& 		&	0			&	-1			& \omit\vrule 	& 		& 		& 		\\ 
	f_{p} 	& 		& 		& 		& 		& 		& \omit\vrule 	& 	0	& 	1	& 	0	& \omit\vrule	&  	0	& 		& \cdots	&				&	0			& \omit\vrule 	& 	1	& 	0	& 	0	\\ 
\cline{1-1} \cline{2-20}
	v	 	& 	-1	& 		& \cdots	& 		& 	-1	& \omit\vrule 	& 	0	& p+1	& 	0	& \omit\vrule	&  	-1	& 	-2	& \push{\cdots}				&	-p			& \omit\vrule 	& 	p+1	& 	0	& 	1	\\ 
	r	 	& 	-1	& 		& \cdots	& 		& 	-1	& \omit\vrule 	& 	0	& p+2	& 	-1	& \omit\vrule	&  	-1	& 	-2	& \push{\cdots}				&	-p			& \omit\vrule 	& 	p+2	& 	-1	& 	1	\\ 
	s	 	& 		& 		& 		& 		& 		& \omit\vrule 	& 	1	& 	-1	& 	1	& \omit\vrule	&  		& 		& 		&				&				& \omit\vrule 	& 	-1	& 	1	& 	0	\\ 
	g	 	& 		& 		& 	0	& 		& 		& \omit\vrule 	& 	1	& 	-1	& 	1	& \omit\vrule	&  		& 		& 	0	&				&				& \omit\vrule 	& 	-2	& 	1	& 	0	\\ 
	u	 	& 		& 		& 		& 		& 		& \omit\vrule 	& 	0	& 	1	& 	0	& \omit\vrule	&  		& 		& 		&				&				& \omit\vrule 	& 	1	& 	-1	& 	0	
}.
$$

Hence, following Lemma \ref{lem:dirac}, we have
$$
\mathcal{K}_{X_n}\left (\mathbf{\widetilde{t}}\right ) = \frac{1}{|\det(A)|} e^{ 2 i \pi \mathbf{\widetilde{t}}^{\!\top} (-R A^{-1} B) \mathbf{\widetilde{t}}}= e^{ 2 i \pi \mathbf{\widetilde{t}}^{\!\top} (-R A^{-1} B) \mathbf{\widetilde{t}}}.$$
The lemma finally follows from the identity $2 \widetilde{Q}_n = (-R A^{-1} B) + (-R A^{-1} B)^{\!\top}$, where $\widetilde{Q}_n$ is defined in the statement of the lemma.
\end{proof}

The following lemma relates the symmetric matrix $\widetilde{Q}_n$ to the gluing equations.

\begin{lemma}\label{lem:2QGamma+C}
		Let $n\geqslant 3$ be an odd integer and $p=\frac{n-3}{2}$.
		Let 		$\alpha = (a_1,b_1,c_1,\ldots,a_W,b_W,c_W) \in \mathcal{S}_{X_n}$ denote a shape structure. If we denote $\widetilde{Q}_n$  the symmetric matrix from Lemma \ref{lem:kin:odd}, $\widetilde{C}(\alpha) = (c_1,\ldots,c_W)^{\!\top}$, and $\widetilde{\Gamma}(\alpha) := (a_1-\pi,\ldots, a_p-\pi,\pi-a_U,\pi-a_V,\pi-a_W)^{\!\top}$, then
		(indexing entries by $k\in\{1,\ldots,p\}$ and $U,V,W$) we have the vector equality
		$
2\widetilde{Q}_n \widetilde{\Gamma}(\alpha) + \widetilde{C}(\alpha) =$ 

$$ 
\renewcommand{\kbldelim}{(}
\renewcommand{\kbrdelim}{)}
 \kbordermatrix{
	\mbox{}	& \\
	 k=1	& \vdots\\
\vdots	& \hspace{6mm} k(\omega_{s}(\alpha) -2(p+2) \pi) + \sum_{j=1}^{k}j \omega_{k-j}(\alpha) \\
	k=p 	& \vdots \\ 
\cline{1-1} \cline{2-2}
	 	& \omega_{p+1}(\alpha)- \omega_{s}(\alpha) - \left ( p(\omega_{s}(\alpha) -2(p+2) \pi) + \sum_{j=1}^{p}j \omega_{p-j}(\alpha) \right )  + 2 \pi - \frac{1}{2} \lambda_{X_n}(\alpha) \\
		& \frac{1}{2}\lambda_{X_n}(\alpha) + \omega_{s}(\alpha) - 3 \pi \\
		& 3 \pi - \omega_{s}(\alpha)
},
$$
where $\lambda_{X_n}(\alpha)=2(a_U-a_V+c_V-b_W)$.

In particular, for $\alpha \in \mathcal{A}_{X_n}$ an angle structure, the vector of angles

\renewcommand{\arraystretch}{1.2} $$
2\widetilde{Q}_n \widetilde{\Gamma}(\alpha) + \widetilde{C}(\alpha) = \:
\renewcommand{\kbldelim}{(}
\renewcommand{\kbrdelim}{)}
 \kbordermatrix{
\mbox{}	& \\
	 k=1	& \vdots\\
\vdots	& -2 \pi \left ( k p - \dfrac{k(k-1)}{2}\right ) \\
	k=p 	& \vdots \\ 
\cline{1-1} \cline{2-2}
		& (p^2+p+2)\pi - \frac{1}{2}\lambda_{X_n}(\alpha) \\
		& \frac{1}{2}\lambda_{X_n}(\alpha) - \pi \\
		& \pi }
$$ \renewcommand{\arraystretch}{1.0}
only depends on the linear combination $\lambda_{X_n}(\alpha)$.
\end{lemma}

\begin{proof}
The lemma follows from direct computations.
\end{proof}

We can now proceed with the proof of Theorem \ref{thm:part:func}.

\begin{proof}[Proof of Theorem \ref{thm:part:func}]
Let  $n \geqslant 3$ be an odd integer and $p=\frac{n-3}{2}$.
We want to compute the partition function associated to $X_n$ and prove that it is of the desired form. We know the form of the kinematical kernel from Lemma \ref{lem:kin:odd}. Let us now compute the dynamical content. Let $\alpha=(a_1,b_1,c_1,\ldots,a_W,b_W,c_W) \in \mathcal{A}_{X_n}$, $\hbar>0$ and $\mathbf{\widetilde{t}} = (t_1, \ldots, t_p, t_U, t_V, t_W)^{\!\top} \in \R^{X_n^{3}}$.

By definition, the dynamical content $\mathcal{D}_{\hbar,X_n}(\mathbf{\widetilde{t}},\alpha)$ is equal to:
$$e^{\frac{1}{\sqrt{\hbar}} \widetilde{C}(\alpha)^{\!\top} \mathbf{\widetilde{t}}}
\dfrac{
\Phi_\B\left (t_U + \frac{i}{2 \pi \sqrt{\hbar}} (\pi-a_U)\right )
\Phi_\B\left (t_V + \frac{i}{2 \pi \sqrt{\hbar}}(\pi-a_V)\right )
\Phi_\B\left (t_W + \frac{i}{2 \pi \sqrt{\hbar}} (\pi-a_W)\right )
}{
\Phi_\B\left (t_1 - \frac{i}{2 \pi \sqrt{\hbar}} (\pi-a_1)\right )
\cdots
\Phi_\B\left (t_p - \frac{i}{2 \pi \sqrt{\hbar}} (\pi-a_p)\right )
},$$
where $\widetilde{C}(\alpha) = (c_1, \ldots, c_p,c_U,c_V, c_W)^{\!\top}$ as in the statement of Lemma \ref{lem:2QGamma+C}.

Now we can compute the partition function of the Teichm\"uller TQFT. By definition:
$$\mathcal{Z}_{\hbar}(X_n,\alpha)= \int_{\mathbf{\widetilde{t}}\in\R^{X_n^{3}}} d\mathbf{\widetilde{t}} \mathcal{K}_{X_n}(\mathbf{\widetilde{t}}) \mathcal{D}_{\hbar,X_n}(\mathbf{\widetilde{t}},\alpha).$$ We do the following  change of variables:
\begin{itemize}
\item $y'_k = t_k - \frac{i}{2 \pi \sqrt{\hbar}} (\pi-a_k)$ for $1 \leqslant k \leqslant p$,
\item $y'_l = t_l + \frac{i}{2 \pi \sqrt{\hbar}} (\pi-a_l)$ for $l\in\{U,V,W\}$,
\end{itemize}
and we denote $\mathbf{\widetilde{y}'}=\left (y'_1, \ldots, y'_{p}, y'_U, y'_V, y'_W\right )^{\!\top}$. 
We also denote 
$$\widetilde{\mathcal{Y}}'_{\hbar,\alpha} :=
\prod_{k=1}^p\left (\R - \frac{i}{2 \pi \sqrt{\hbar}} (\pi-a_k)\right ) 
\times 
\prod_{l=U,V,W}
\left (\R + \frac{i}{2 \pi \sqrt{\hbar}} (\pi-a_l)\right ),$$
the subset of $\C^{p+3}$ on which the variables in $\mathbf{\widetilde{y}'}$ will reside. Finally we denote:
\begin{align*}
\widetilde{\Gamma}(\alpha) := \frac{2 \pi \sqrt{\hbar}}{i}(\mathbf{\widetilde{y}'}-\mathbf{\widetilde{t}}) 
= (a_1-\pi,\ldots, a_p-\pi,\pi-a_U,\pi-a_V,\pi-a_W)^{\!\top}.
\end{align*}
as in the statement of Lemma \ref{lem:2QGamma+C}.
We can now compute:

\begin{align*}
&\mathcal{Z}_{\hbar}(X_n,\alpha)
= \int_{\mathbf{\widetilde{t}}\in\R^{X_n^{3}}} d\mathbf{\widetilde{t}} \mathcal{K}_{X_n}(\mathbf{\widetilde{t}}) \mathcal{D}_{\hbar,X_n}(\mathbf{\widetilde{t}},\alpha) \\
&= \int_{\mathbf{\widetilde{y}'}\in \widetilde{\mathcal{Y}}'_{\hbar,\alpha}} d\mathbf{\widetilde{y}'} 
\mathcal{K}_{X_n}\left (\mathbf{\widetilde{y}'}-\frac{i}{2 \pi \sqrt{\hbar}} \widetilde{\Gamma}(\alpha) \right ) \mathcal{D}_{\hbar,X_n}\left (\mathbf{\widetilde{y}'}-\frac{i}{2 \pi \sqrt{\hbar}} \widetilde{\Gamma}(\alpha),\alpha\right ) \\
&=  \int_{\mathbf{\widetilde{y}'} \in \widetilde{\mathcal{Y}}'_{\hbar,\alpha}}
\hspace*{-0.3cm} d\mathbf{\widetilde{y}'} 
e^{
2 i \pi \mathbf{\widetilde{y}}^{\prime T} \widetilde{Q}_n  \mathbf{\widetilde{y}'} 
+\frac{2}{\sqrt{\hbar}} \widetilde{\Gamma}(\alpha)^{\!\top} \widetilde{Q}_n  \mathbf{\widetilde{y}'} 
- \frac{i}{2 \pi \hbar}  \widetilde{\Gamma}(\alpha)^{\!\top} \widetilde{Q}_n  \widetilde{\Gamma}(\alpha) 
+ \frac{1}{\sqrt{\hbar}} \widetilde{C}(\alpha)^{\!\top} \mathbf{\widetilde{y}'} 
- \frac{i}{2 \pi \hbar} \widetilde{C}(\alpha)^{\!\top} \widetilde{\Gamma}(\alpha)
}
\frac{
\Phi_\B\left (y'_U\right )
\Phi_\B\left (y'_V\right )
\Phi_\B\left (y'_W\right )
}{
\Phi_\B\left (y'_1\right )
\cdots
\Phi_\B\left (y'_p\right )
}   \\
&\stareq  
\int_{\mathbf{\widetilde{y}'} \in \widetilde{\mathcal{Y}}'_{\hbar,\alpha}} d\mathbf{\widetilde{y}'} 
e^{
2 i \pi \mathbf{\widetilde{y}}^{\prime T} \widetilde{Q}_n  \mathbf{\widetilde{y}'} 
+\frac{2}{\sqrt{\hbar}} \widetilde{\Gamma}(\alpha)^{\!\top} \widetilde{Q}_n  \mathbf{\widetilde{y}'} 
+ \frac{1}{\sqrt{\hbar}} \widetilde{C}(\alpha)^{\!\top} \mathbf{\widetilde{y}'} 
}
\dfrac{
\Phi_\B\left (y'_U\right )
\Phi_\B\left (y'_V\right )
\Phi_\B\left (y'_W\right )
}{
\Phi_\B\left (y'_1\right )
\cdots
\Phi_\B\left (y'_p\right )
} \\
&=
\int_{\mathbf{\widetilde{y}'} \in \widetilde{\mathcal{Y}}'_{\hbar,\alpha}} d\mathbf{\widetilde{y}'}
e^{
2 i \pi \mathbf{\widetilde{y}}^{\prime T} \widetilde{Q}_n  \mathbf{\widetilde{y}'} + \frac{1}{\sqrt{\hbar}} \widetilde{\mathcal{W}}(\alpha)^{\!\top} \mathbf{\widetilde{y}'} 
}
\dfrac{
\Phi_\B\left (y'_U\right )
\Phi_\B\left (y'_V\right )
\Phi_\B\left (y'_W\right )
}{
\Phi_\B\left (y'_1\right )
\cdots
\Phi_\B\left (y'_p\right )
},
\end{align*}
where $\widetilde{\mathcal{W}}(\alpha):= 2 \widetilde{Q}_n  \widetilde{\Gamma}(\alpha)+\widetilde{C}(\alpha)$.
Now, from Lemma \ref{lem:2QGamma+C}, we have
$$\widetilde{\mathcal{W}}(\alpha) = \begin{pmatrix}-2p\pi\\
\vdots \\
-2 \pi \left ( k p - \dfrac{k(k-1)}{2}\right ) \\
\vdots \\
-p(p+1)\pi  \\
(p^2+p+2)\pi - \frac{1}{2}\lambda_{X_n}(\alpha) \\
\frac{1}{2}\lambda_{X_n}(\alpha) - \pi \\
\pi 
\end{pmatrix}.$$

We define a new variable $x:= y'_V-y'_U$ living in the set 
$$\mathcal{Y}'^0_{\hbar,\alpha}:=\R + 
\frac{i}{2 \pi \sqrt{\hbar}} (a_U-a_V),$$ and we also define $
\mathbf{y'}$ (respectively  $\mathcal{Y}'
_{\hbar,\alpha}$) exactly like $
\widetilde{\mathbf{y}'}$ (respectively $
\widetilde{\mathcal{Y}}'_{\hbar,\alpha}$) but with the second-to-last coordinate (corresponding to the tetrahedron $V$) taken out. 
We finally define 
\begin{equation} \label{eqn:form:W:and:Q}
\mathcal{W}_{n}=
\begin{bmatrix}\mathcal{W}_{n,1} \\ \vdots \\ \mathcal{W}_{n,k} \\ \vdots \\ \mathcal{W}_{n,p} \\ \mathcal{W}_{n,U} \\ \mathcal{W}_{n,W} \end{bmatrix}
:=
\begin{bmatrix}-2p\pi \\ \vdots \\ -2 \pi \left ( k p - \frac{k(k-1)}{2}\right ) \\ \vdots \\ -p(p+1)\pi \\ (p^2+p+1)\pi \\ \pi\end{bmatrix}
\qquad 
\text{ and }
\qquad
Q_n:=\begin{bmatrix}
1 & 1 & \cdots & 1 & -1 & 0 \\ 
1 & 2 & \cdots & 2 & -2 & 0 \\ 
\vdots & \vdots & \ddots & \vdots & \vdots & \vdots \\ 
1 & 2 & \cdots & p & -p & 0 \\ 
-1 & -2 & \cdots & -p & p & \frac{1}{2} \\ 
0 & 0 & \cdots & 0 & \frac{1}{2} & 0 
\end{bmatrix}.
\end{equation}
Notice that $Q_n$ is obtained from $\widetilde{Q}_n$ by the following operations:
\begin{itemize}
	\item add the $V$-row to the $U$-row,
	\item add the $V$-column to the $U$-column,
	\item delete the $V$-row and the $V$-column,
\end{itemize}
and $\mathcal{W}_{n}$ is obtained  from $\widetilde{\mathcal{W}}(\alpha)$ by the same operations on rows.

We can now use the substitution $y'_V = y'_U+x$ to compute:
\begin{align*}
2 i \pi \widetilde{\mathbf{y}}^{\prime T} \widetilde{Q}_n \widetilde{\mathbf{y}'} 
&= 
2 i \pi \left ( 
(\mathbf{y'}^{\!\top} Q_n \mathbf{y'} 
- p {y'_U}^2 - y'_U y'_W) 
+ (p+2){y'_U}^2 - 3 y'_U y'_V + 2 y'_U y'_W
 + {y'_V}^2 - y'_V y'_W
\right ) \\
&= 
2 i \pi \left (\mathbf{y'}^{\!\top} Q_n \mathbf{y'} -xy'_U -x y'_W+x^2\right ),
\end{align*}
and 
$\frac{1}{\sqrt{\hbar}} \widetilde{\mathcal{W}}(\alpha)^{\!\top} \widetilde{\mathbf{y}'}
= \frac{1}{\sqrt{\hbar}} \left (\mathcal{W}_n^{\!\top} \mathbf{y'}
+x (\frac{1}{2}\lambda_{X_n}(\alpha)-\pi)\right )$,
thus

\begin{align*}
&\mathcal{Z}_{\hbar}(X_n,\alpha) \stareq
\int_{\mathbf{\widetilde{y}'} \in \widetilde{\mathcal{Y}}'_{\hbar,\alpha}} d\mathbf{\widetilde{y}'}
e^{
2 i \pi \mathbf{\widetilde{y}}^{\prime T} \widetilde{Q}_n  \mathbf{\widetilde{y}'} + \frac{1}{\sqrt{\hbar}}  \widetilde{\mathcal{W}}(\alpha)^{\!\top} \mathbf{\widetilde{y}'} 
}
\dfrac{
\Phi_\B\left (y'_U\right )
\Phi_\B\left (y'_V\right )
\Phi_\B\left (y'_W\right )
}{
\Phi_\B\left (y'_1\right )
\cdots
\Phi_\B\left (y'_p\right )
}
\\
&\stareq
\int dx d\mathbf{y'} 
e^{2 i \pi \left (\mathbf{y'}^{\!\top} Q_n \mathbf{y'}+x(x-y'_U -y'_W)\right )+
\frac{1}{\sqrt{\hbar}} \left (\mathcal{W}_n^{\!\top} \mathbf{y'}
+x (\frac{1}{2}\lambda_{X_n}(\alpha)-\pi)\right )
}
\dfrac{
\Phi_\B\left (y'_U\right )
\Phi_\B\left (y'_U+x\right )
\Phi_\B\left (y'_W\right )
}{
\Phi_\B\left (y'_1\right )
\cdots
\Phi_\B\left (y'_p\right )
}  ,
\end{align*}
where the variables $(\mathbf{y'},x)$ in the last integral lie in $\mathcal{Y}'_{\hbar,\alpha} \times \mathcal{Y}'^0_{\hbar,\alpha}$.

Finally 
we obtain that
\begin{equation*}
\mathcal{Z}_{\hbar}(X_n,\alpha) \stareq  \int_{x \in \R + \frac{i}{2 \pi \sqrt{\hbar}} \mu_{X_n}(\alpha)} J_{X_n}(\hbar,x)e^{\frac{1}{2 \sqrt{\hbar}}  x  \lambda_{X_n}(\alpha)} dx,
\end{equation*}
where 
\begin{equation*} 
J_{X_n}(\hbar,x)=\int_{\mathcal{Y}'} d\mathbf{y'} \
e^{2 i \pi \left (\mathbf{y'}^{\!\top} Q_n \mathbf{y'}
+ x(x- y'_U-y'_W)\right )}
e^{
\frac{1}{\sqrt{\hbar}} \left (\mathbf{y'}^{\!\top} \mathcal{W}_n
- \pi x\right )}
\dfrac{
\Phi_\B\left (y'_U\right )
\Phi_\B\left (y'_U+x\right )
\Phi_\B\left (y'_W\right )
}{
\Phi_\B\left (y'_1\right )
\cdots
\Phi_\B\left (y'_p\right )
},
\end{equation*}
$\mathcal{Y}' = \mathcal{Y}'_{\hbar,\alpha}$,
and $\mu_{X_n}(\alpha)=a_U-a_V$,
which concludes the proof.
\end{proof}

We conclude this section with a slight rephrasing of Theorem \ref{thm:part:func}, in the following Corollary \ref{cor:part:func}.
Although the expression in Theorem \ref{thm:part:func} was the closest to the statement of \cite[Conjecture 1 (1)]{AK}, we find that the following re-formulation has additional benefits: the integration multi-contour  is now independent of $\hbar$ and the integrand is closer to the form $e^{\frac{1}{2 \pi \hbar} S(\mathbf{y})}$  that we need in order to apply the saddle point method (see Theorem \ref{thm:SPM}, where $\lambda \to \infty$ should be thought of as $2 \pi \hbar \to 0^+$).

\begin{corollary}\label{cor:part:func}
Let $n\geqslant 3$ be an odd integer and $p=\frac{n-3}{2}$. Consider the ideal triangulation $X_n$ of $S^3\setminus K_n$ from Figure \ref{fig:id:trig:odd}. Then for all angle structures $\alpha \in \mathcal{A}_{X_n}$ and all $\hbar>0$, we have:
\begin{equation*}
\mathcal{Z}_{\hbar}(X_n,\alpha) 
\stareq 
\int_{\mathbb{R}+i \mu_{X_n}(\alpha)  } 
\mathfrak{J}_{X_n}(\hbar,\mathsf{x})
e^{\frac{1}{4 \pi \hbar}  \mathsf{x}  \lambda_{X_n}(\alpha)} 
d\mathsf{x},
\end{equation*}
with the  map
\begin{equation*}
\mathfrak{J}_{X_n}\colon(\hbar,\mathsf{x})\mapsto
\left (
\frac{1}{2 \pi \sqrt{\hbar}}
\right )^{p+3}
\int_{\mathcal{Y}_\alpha} 
d\mathbf{y} \
e^{\frac
{
i \mathbf{y}^{\!\top} Q_n \mathbf{y} + 
 i \mathsf{x}(\mathsf{x}- y_U-y_W) + 
  \mathbf{y}^{\!\top} \mathcal{W}_n - \pi \mathsf{x}
}
 {2 \pi \hbar} 
}
\dfrac{
\Phi_\B\left ( \frac{y_U}{2 \pi \sqrt{\hbar}}  \right )
\Phi_\B\left ( \frac{y_U+\mathsf{x}}{2 \pi \sqrt{\hbar}} \right )
\Phi_\B\left ( \frac{y_W}{2 \pi \sqrt{\hbar}} \right )
}{
\Phi_\B\left (\frac{y_1}{2 \pi \sqrt{\hbar}}\right )
\cdots 
\Phi_\B\left (\frac{y_p}{2 \pi \sqrt{\hbar}}\right )
} 
,
\end{equation*}
where $\mu_{X_n},\lambda_{X_n}, \mathcal{W}_n, Q_n$ are the same as in Theorem \ref{thm:part:func}, and 
$$\mathcal{Y}_\alpha =
\prod_{k=1}^p\left (\R - i (\pi - a_k)\right ) 
\times 
\prod_{l=U,W}
\left (\R + i (\pi - a_l)\right ).$$
\end{corollary}

\begin{proof}
We start from the expressions in Theorem \ref{thm:part:func}, and, with $\hbar >0$ fixed, we do the change of variables
$y_j = (2\pi \sqrt{\hbar}) y'_j$ for $j \in \{1, \ldots, p,U,W\}$
and $\mathsf{x} = (2\pi \sqrt{\hbar}) x$.
\end{proof}


\section{Partition function for the H-triangulations (odd case)}\label{sec:part:H:odd}

As stated in the introduction, this section is not essential for understanding the proof of the volume conjecture in Section \ref{sec:vol:conj}, and thus may be skipped at first read. However similar this section looks to the previous Section \ref{sec:part:odd}, subtle differences remain in the equations and calculations, and details should thus be read carefully.

Before stating Theorem \ref{thm:part:func:Htrig:odd}, we compute the weights on each edge of the H-triangulation $Y_n$ given in Figure \ref{fig:H:trig:odd} (for $n \geqslant 3$ odd).  

Recall that we denoted $\overrightarrow{\eta_0}, \ldots, \overrightarrow{\eta_{p+1}}, \overrightarrow{\eta_s}, \overrightarrow{\eta_d}, \overrightarrow{K_n} \in (Y_n)^1$ the $p+5$ edges in $Y_n$ respectively represented in Figure \ref{fig:H:trig:odd} by arrows with circled $0$, \ldots, circled $p+1$, simple arrow, double arrow and blue simple arrow.

For $\alpha=(a_1,b_1,c_1,\ldots,a_p,b_p,c_p,a_U,b_U,c_U,a_V,b_V,c_V,a_W,b_W,c_W,a_Z,b_Z,c_Z) \in \mathcal{S}_{Y_n}$ a shape structure on $Y_n$, the weights of each edge are given by:
\begin{itemize}
\item $\widehat{\omega}_s(\alpha):= \omega_{Y_n,\alpha}(\overrightarrow{\eta_s})=
2 a_U+b_V+c_V+a_W+b_W+a_Z
$
\item $\widehat{\omega}_d(\alpha):= \omega_{Y_n,\alpha}(\overrightarrow{\eta_d})=
b_U+c_U+c_W+b_Z+c_Z
$
\item $\omega_0(\alpha):= \omega_{Y_n,\alpha}(\overrightarrow{\eta_0})=
2 a_1 + c_1 + 2 a_2 + \ldots + 2 a_p + a_V+c_W
$
\item $\omega_1(\alpha):= \omega_{Y_n,\alpha}(\overrightarrow{\eta_1})=
2b_1+c_2
$
\\
\vspace*{-2mm}
\item $\omega_k(\alpha):= \omega_{Y_n,\alpha}(\overrightarrow{\eta_k})=
c_{k-1}+2b_k+c_{k+1}
$
\ \
(for $2\leqslant k \leqslant p-1$)
\\
\vspace*{-2mm}
\item $\omega_p(\alpha):= \omega_{Y_n,\alpha}(\overrightarrow{\eta_p})=
c_{p-1}+2b_p+b_U+b_V+a_W$
\item $\widehat{\omega}_{p+1}(\alpha):= \omega_{Y_n,\alpha}(\overrightarrow{\eta_{p+1}})=
c_p+c_U+a_V+c_V+b_W+b_Z+c_Z
$
\item $\widehat{\omega}_{\overrightarrow{K_n}}(\alpha):= \omega_{Y_n,\alpha}(\overrightarrow{K_n})=
a_Z
$
\end{itemize}

Note that some of these weights have the same value as the ones for $X_n$ listed in Section \ref{sec:geom} (and are thus also denoted $\omega_j(\alpha)$), and some are specific to $Y_n$ (and are written with a hat).

We can now compute the partition function of the Teichm\"uller TQFT for the H-triangulations $Y_n$, and prove the following theorem. We will denote $\mathcal{S}_{Y_n \setminus Z}$ the space of shape structures on every tetrahedron of $Y_n$ except for $Z$.

\begin{theorem}\label{thm:part:func:Htrig:odd}
Let $n \geqslant 3$ be an odd integer, $p=\frac{n-3}{2}$ and $Y_n$ the one-vertex H-triangulation of the pair $(S^3,K_n)$ from Figure \ref{fig:H:trig:odd}. Then for every $\hbar>0$ and for every 
$\tau\in \mathcal{S}_{Y_n \setminus Z} \times  \overline{\mathcal{S}_Z}$
 such that $\omega_{Y_n,\tau}$ vanishes on $\overrightarrow{K_n}$ and is equal to $2\pi$ on every other edge, one has
\begin{equation*}
\underset{\tiny \begin{matrix}\alpha \to \tau \\ \alpha \in \mathcal{S}_{Y_n} \end{matrix}}{\lim}
	 \Phi_{\B}\left( \frac{\pi-\omega_{Y_n,\alpha}\left (\overrightarrow{K_n}\right )}{2\pi i \sqrt{\hbar}} \right)  \mathcal{Z}_{\hbar}(Y_n,\alpha) \stareq J_{X_n}(\hbar,0),
\end{equation*}
where $J_{X_n}$ is defined in Theorem \ref{thm:part:func}.
\end{theorem}

Before proving Theorem \ref{thm:part:func:Htrig:odd}, let us mention a useful result: the fact that $\Phi_\B$ is bounded on compact horizontal bands.

\begin{lemma}\label{lem:PhiB:bounded}
Let $\hbar>0$ and $\delta \in (0,\pi/2)$. Then
$\displaystyle M_{\delta,\hbar} := \underset{z \in \R+i[\delta,\pi-\delta]}{\max} |\Phi_\B(z)|$ is finite.
\end{lemma}
\begin{proof}
Let $\hbar>0$ and $\delta \in (0,\pi/2)$. By contradiction, let us assume that $ M_{\delta,\hbar} = \infty$. Then there exists a sequence $(z_n)_{n\in \N} \in \left (\R+i[\delta,\pi-\delta]\right )^\N$ such that $|\Phi_\B(z_n)| \underset{n \to \infty}{\to} \infty$.

If $(\Re(z_n))_{n\in \N}$ is bounded, then $(z_n)_{n\in \N}$ lives in a compact set, which contradicts the continuity of $|\Phi_\B |$.

If  $(\Re(z_n))_{n\in \N}$ admits a subsequence going to $-\infty$ (resp.\ $\infty$), then the image of this subsequence by $| \Phi_\B |$ should still tend to $\infty$, which contradicts Proposition \ref{prop:quant:dilog} (4).
\end{proof}

\begin{proof}[Proof of Theorem \ref{thm:part:func:Htrig:odd}]
Let $n \geqslant 3$ be an odd integer and $p=\frac{n-3}{2}$. The proof will consist in three steps: computing the partition function $\mathcal{Z}_{\hbar}(Y_n,\alpha)$, applying the dominated convergence theorem in $\alpha \to \tau$ and finally retrieving the value $J_{X_n}(\hbar,0)$ in $\alpha =\tau$.

\textit{Step 1. Computing the partition function $\mathcal{Z}_{\hbar}(Y_n,\alpha)$.}

Like in the proof of Theorem \ref{thm:part:func} we start by computing the kinematical kernel. We  denote 
\[
\widehat{\mathbf{t}}=(t_1,\ldots,t_{p-1},t_p,t_U,t_V,t_W,t_Z) \in \mathbb{R}^{Y_n^3}
\]
the vector whose coordinates are associated to the tetrahedra ($t_j$ for $T_j$). The generic vector in $\mathbb{R}^{Y_n^2}$ which corresponds to the faces variables will be denoted
\[
\widehat{\mathbf{x}}=(e_1,\ldots,e_{p+1},f_1,\ldots,f_p,v,r,s,s',g,u,m) \in \mathbb{R}^{Y_n^2}. 
\]

By definition, the kinematical kernel is:
$$\mathcal{K}_{Y_n}\left (\mathbf{\widehat{t}}\right ) = \int_{\widehat{\mathbf{x}} \in \R^{Y_n^{2}}} d\widehat{\mathbf{x}} \prod_{T \in Y_n^3} e^{2 i \pi \varepsilon(T) x_0(T) \mathsf{t}(T)}
\delta( x_0(T)- x_1(T)+ x_2(T))
\delta( x_2(T)- x_3(T)+ \mathsf{t}(T)).
$$

Following Lemma \ref{lem:dirac},
we compute from Figure \ref{fig:H:trig:odd} that:
$$
\mathcal{K}_{Y_n}\left (\mathbf{\widehat{t}}\right ) =
\int_{\widehat{\mathbf{x}} \in \R^{Y_n^{2}}} d\widehat{\mathbf{x}} 
\int_{\widehat{\mathbf{w}} \in \R^{2 (p+4)}} d\widehat{\mathbf{w}} \
e^{ 2 i \pi \mathbf{\widehat{t}}^{\!\top} \widehat{S} \widehat{\mathbf{x}}}
e^{ -2 i \pi \widehat{\mathbf{w}}^{\!\top} \widehat{H} \widehat{\mathbf{x}}}
e^{ -2 i \pi \widehat{\mathbf{w}}^{\!\top} \widehat{D} \mathbf{\widehat{t}}},
$$
where the matrices $\widehat{S}, \widehat{H}, \widehat{D}$ are given by:
$$ 
\widehat{S}=\kbordermatrix{
	\mbox{} 	& e_1 	& \dots 	& e_p	&e_{p+1} 	& \omit\vrule	&f_1	& \ldots 	& f_p	&\omit\vrule	&  v 	& r 	& s 	& s'	& g	& u 	& m 	\\
	t_1 		& 1	 	 & 		&\push{\low{0}} 	& \omit\vrule	& 	&  		&  	&\omit\vrule	&   	&  	&   	&	&	&  	&  	\\
	\vdots 	& 	    	&\ddots 	& 		&	  	& \omit\vrule	& 	&  0		&  	&\omit\vrule	& 	&   	&  	& 0	&	&	&  	\\
	t_{p}    	&\push{0}		       	& 1        	& 		& \omit\vrule	&  	&  		&	&\omit\vrule	&	&	&	&	&	&	&  	\\ 
\cline{1-1} \cline{2-17}
	t_U	 	& 		& 		& 		&  		& \omit\vrule	&	&	 	&	&\omit\vrule	&0  	&-1 	&0  	& 0	& 0	&0 	& 0 	\\
	t_V	 	& 		&\push{0}	  		&  		& \omit\vrule	&	& 	0	& 	&\omit\vrule	&0  	&0 	&0  	& 0	& -1	& 0 	& 0 	\\
	t_W	 	& 		&  		& 		& 	 	& \omit\vrule	&	&	  	&  	&\omit\vrule	&0 	&0 	&0  	& 0	& 0	& -1 	& 0	\\
	t_Z	 	& 		&  		& 		& 	 	& \omit\vrule	&	&	  	&  	&\omit\vrule	&0 	&0 	&0  	& 0	& 0	& 0 	& 1
},$$

$$
\widehat{H}=\kbordermatrix{
\mbox{} 		& e_1 	& e_2 	& \dots 	&e_p		& e_{p+1}	& \omit\vrule	&f_1	&f_{2}	&  \ldots 	& f_p	&\omit\vrule	&  v 	& r 	& s 	& s'		& g	& u 	& m 	\\
	w_1 		& 1		&   -1 	& 		&		&		& \omit\vrule	& 1	&		&  		&  	&\omit\vrule	&   	&  	&  	&		&	&  	&  	\\
	\vdots 	&     		&   \ddots	& \ddots 	&	\push{0}		& \omit\vrule	& 	&\ddots	&\push{0}  	&\omit\vrule	& 	&   	&  	& \low{0}	&	&	&  	\\
	\vdots 	&     	\push{0}   		& \ddots 	&\ddots	&		& \omit\vrule	& \push{0}  	&\ddots	&  	&\omit\vrule	& 	&   	&  	&		&	&	&  	\\
	w_{p} 	&		&		&		&	1	& -1		& \omit\vrule 	&	&		&		&1	&\omit\vrule	&	&	&	&		&	&	&  	\\
\cline{1-1} \cline{2-19}
	w_{U} 	&		&		&		&		&  		& \omit\vrule 	&	&		&		&	&\omit\vrule	& -1	& 1	& 1	&	0	& 0	& 0	& 0 	\\
	w_{V} 	&		&		& \low{0}	&		&  		& \omit\vrule 	&	&		&		&1	&\omit\vrule	& 0	& 0	& 0	&	-1	& 1	& 0	& 0  	\\
	w_{W} 	&		&		&		&		&  		& \omit\vrule 	&	&		&		&	&\omit\vrule	& 1	& -1	& 0	&	0	& 0	& 1	& 0 	\\
	w_{Z} 	&		&		&		&		&  		& \omit\vrule 	&	&		&		&	&\omit\vrule	& 0	& 0	& 1	&	0	& 0	& 0	& 0 	\\
\cline{1-1} \cline{2-19}
	w'_1 		& -1		&   	 	& 		&		&		& \omit\vrule	& 1	&  		&  		&	&\omit\vrule	&   	&  	&   	&		&	&  	&  	\\
	\vdots 	&     		&   		& 	 	&		&		& \omit\vrule	& -1	&  \ddots	&  \push{0}	&\omit\vrule	& 	&   	&  	& \low{0}	&	&	&  	\\
	\vdots 	&     		&   		& 	 	&		&		& \omit\vrule	& 	&  \ddots	& \ddots 	&	&\omit\vrule	& 	&   	&  	&		&	&	&  	\\
	w'_{p} 	&		&		&		&		&  		& \omit\vrule 	&	\push{0}	&-1		&1	&\omit\vrule	&	&	&	&		&	&	&  	\\
\cline{1-1} \cline{2-19}
	w'_{U} 	&		&		&		&		&  		& \omit\vrule 	&	&		&		&	&\omit\vrule	& 0	& 0	& 1	&	0	& -1	& 0	& 0  	\\
	w'_{V} 	&		&		&		&		&  		& \omit\vrule 	&	&		&		&1	&\omit\vrule	& 0	& 0	& 0	&	0	& 0	& -1	& 0 	\\
	w'_{W} 	&		&		&		&		&  -1		& \omit\vrule 	&	&		&		&	&\omit\vrule	& 1	& 0	& 0	&	0	& 0	& 0	& 0 	\\
	w'_{Z} 	&		&		&		&		&  		& \omit\vrule 	&	&		&		&	&\omit\vrule	& 0	& 0	& 1	& 	-1	& 0	& 0	& 0 
},$$
$$
\widehat{D}=\kbordermatrix{
	\mbox{} 	& t_1 	& \dots 	& t_{p} 	& \omit\vrule	& t_{U}	& t_{V}	& t_{W}	& t_{Z}	\\ 
	w_1 		& 		&   	 	& 		&			&		&		& 		&	\\
	\vdots 	&     		&   		& 	 	&			&		&		& 		&	\\
	w_{p} 	&		&		&	    	&			& 		&		&	\\  \cline{1-1} 
	w_{U} 	&		&		&		&			&	0	&  		&		&	\\
	w_{V} 	&		&		& 		&			&		&  		&		&	\\
	w_{W} 	&		&		&		&			&		&  		&		&	\\
	w_{Z} 	&		&		&		&			&		&  		&		&	\\
\cline{1-1} \cline{2-9}
	w'_1 		& 1		&   	 	& 		&			&		&		&  		&	\\
	\vdots 	&     		&\ddots   	& 	 	&			&		&		& 	0	&	\\
	w'_{p} 	&		&		& 		&			&		&  		&		&	\\  \cline{1-1}  
	w'_{U} 	&		&		&		&			&\ddots	&  		&		&	\\
	w'_{V} 	&		&		&		&			&		&  		& 		&	\\
	w'_{W} 	&		&	0	&		&			&		&  		&\ddots	& 	\\
	w'_{Z} 	&		&		&		&			&		&  		&		& 1
	}.
	$$

Let us define $S$ the submatrix of $\widehat{S}$ without the $m$-column, $H$ the submatrix of $\widehat{H}$ without the $m$-column and the $w_V$-row, $R_V$ this very $w_V$-row of $\widehat{H}$, $D$ the submatrix of $\widehat{D}$ without the $w_V$-row, $\mathbf{x}$ the subvector of $\widehat{\mathbf{x}}$ without the variable $m$ and $\mathbf{w}$ the subvector of $\widehat{\mathbf{w}}$ without the variable $w_V$. Finally let us denote
$f_{\widehat{\mathbf{t}},w_V}(\mathbf{x}):=e^{2i \pi (\widehat{\mathbf{t}}^{\!\top} S - w_V R_V)\mathbf{x}}$. We remark that $H$ is invertible (whereas $\widehat{H}$ was not) and $\det(H)=-1$.  Observe that since $H$ and its inverse $H^{-1}$ (which is written a little further) have integer coefficients, their determinants are $1$ or $-1$, which is enough for concluding that $|\det(H)|=1$ in the following computations.

 Hence, by using multi-dimensional Fourier transform and the integral definition of the Dirac delta function, we compute:

\begin{align*}
&\mathcal{K}_{Y_n}\left (\mathbf{\widehat{t}}\right ) =
\int_{\widehat{\mathbf{x}} \in \R^{Y_n^{2}}} d\widehat{\mathbf{x}} 
\int_{\widehat{\mathbf{w}} \in \R^{2 (p+4)}} d\widehat{\mathbf{w}} \
e^{ 2 i \pi \mathbf{\widehat{t}}^{\!\top} \widehat{S} \widehat{\mathbf{x}}}
e^{ -2 i \pi \widehat{\mathbf{w}}^{\!\top} \widehat{H} \widehat{\mathbf{x}}}
e^{ -2 i \pi \widehat{\mathbf{w}}^{\!\top} \widehat{D} \mathbf{\widehat{t}}}\\
&= \int_{m \in \R} d m
\int_{w_V \in \R} d w_V
\int_{\mathbf{x} \in \R^{2p+7}} d\mathbf{x}
\int_{\mathbf{w} \in \R^{2p+7}} d\mathbf{w} \
e^{ 2 i \pi t_Z m}
e^{ -2 i \pi w_V R_V \mathbf{x}}
e^{ 2 i \pi \mathbf{\widehat{t}}^{\!\top} S \mathbf{x}}
e^{ -2 i \pi \mathbf{w}^{\!\top} H \mathbf{x}}
e^{ -2 i \pi \mathbf{w}^{\!\top} D \mathbf{\widehat{t}}}\\
&=\int_{m \in \R} d m \ e^{ 2 i \pi t_Z m}
\int_{w_V \in \R} d w_V
\int_{\mathbf{w} \in \R^{2p+7}} d\mathbf{w} \
e^{ -2 i \pi \mathbf{w}^{\!\top} D \mathbf{\widehat{t}}}
\int_{\mathbf{x} \in \R^{2p+7}} d\mathbf{x} \
f_{\widehat{\mathbf{t}},w_V}(\mathbf{x})
e^{ -2 i \pi \mathbf{w}^{\!\top} H \mathbf{x}}\\
&= \delta(-t_Z)
\int_{w_V \in \R} d w_V
\int_{\mathbf{w} \in \R^{2p+7}} d\mathbf{w} \
e^{ -2 i \pi \mathbf{w}^{\!\top} D \mathbf{\widehat{t}}} \
\mathcal{F}\left (f_{\widehat{\mathbf{t}},w_V}\right ) (H^{\!\top} \mathbf{w})\\
&= \delta(-t_Z)
\int_{w_V \in \R} d w_V \
\frac{1}{| \det(H)|}
\mathcal{F}\left (\mathcal{F}\left (f_{\widehat{\mathbf{t}},w_V}\right )\right ) (H^{-1} D \mathbf{\widehat{t}}) \\
&=  \delta(-t_Z)
\int_{w_V \in \R} d w_V \
f_{\widehat{\mathbf{t}},w_V} (-H^{-1} D \mathbf{\widehat{t}})\\
&=  \delta(-t_Z)
\int_{w_V \in \R} d w_V \
e^{2i \pi (\widehat{\mathbf{t}}^{\!\top} S - w_V R_V) (-H^{-1} D \mathbf{\widehat{t}})}\\
&=  \delta(-t_Z)
e^{2i \pi \widehat{\mathbf{t}}^{\!\top} (-S H^{-1} D) \mathbf{\widehat{t}}}
\int_{w_V \in \R} d w_V \
e^{-2i \pi  w_V (-R_V H^{-1} D \mathbf{\widehat{t}})}\\
&=  \delta(-t_Z)
e^{2i \pi \widehat{\mathbf{t}}^{\!\top} (-S H^{-1} D) \mathbf{\widehat{t}}}
\delta  (-R_V H^{-1} D \mathbf{\widehat{t}}).
\end{align*}

We can now compute $H^{-1}=$
$$ 
\kbordermatrix{
	\mbox{} 	& w_{1}	& w_{2} 	& \ldots	& w_{p-1}	& w_{p}	& \omit\vrule	& w_{U}			& w_{W}			& w_{Z}			& \omit\vrule	&  w'_{1}	& w'_{2}	& \ldots	& w'_{p-1}			& w'_{p}			& \omit\vrule 	& w'_{U}	& w'_{V}			& w'_{W}	& w'_{Z}	\\
	e_{1} 	& 	0	&		& \cdots	& 		&	0	& \omit\vrule 	& 	1			& 	1			& 	-1			& \omit\vrule	&  	-1	& 	-1	&	\push{\cdots}			& 	-1			& \omit\vrule 	& 	0	& 	1			& 	0	&	0	\\
	e_{2} 	& 	-1	& 	0	& 		& 		& 		& \omit\vrule 	& 	2			& 	2			& 	-2			& \omit\vrule	&  	-1	& 	-2	&	\push{\cdots}			& 	-2			& \omit\vrule 	& 	0	& 	2			& 	0	&	0	\\
\low{\vdots} 	& 	-1	& 	-1	& \ddots	& 		& \vdots	& \omit\vrule 	& 				& 				& 				& \omit\vrule	& 	 	&		& \ddots	& 				& 	\vdots		& \omit\vrule 	& 		& 				& 		&		\\
		 	& \vdots	& 		& \ddots	& 	0	& 	0	& \omit\vrule 	& \vdots			& \vdots			&  \vdots			& \omit\vrule	&  \vdots	& \vdots	&		&\text{\tiny {\(1-p\)}}	& \text{\tiny {\(1-p\)}}	& \omit\vrule 	& \vdots	& \vdots			& \vdots	& \vdots	\\
	e_{p} 	& 		& 		& 		& 	-1	& 	0	& \omit\vrule 	& 				& 				& 				& \omit\vrule	&  		& 		&		& \text{\tiny {\(1-p\)}}	& 	-p			& \omit\vrule 	& 		& 				& 		&		\\
	e_{p+1} 	& -1		& 		& \cdots	& 		& 	-1	& \omit\vrule 	& \text{\tiny {\(p+1\)}}	& \text{\tiny {\(p+1\)}}	& \text{\tiny {\(-p-1\)}}	& \omit\vrule	&  	-1	&	-2	&\cdots	& \text{\tiny {\(1-p\)}}	& 	-p			& \omit\vrule 	& 	0	& \text{\tiny {\(p+1\)}}	& 	0	&	0	\\
\cline{1-1} \cline{2-21}
	f_{1} 	& 		& 		& 		& 		& 		& \omit\vrule 	& 	1			& 	1			& 	-1			& \omit\vrule	&  	0	& -1		& 	\push{\cdots}			&	-1			& \omit\vrule 	& 	0	& 	1			& 	0	&	0	\\
	f_{2} 	& 		& 		& 		& 		& 		& \omit\vrule 	& 				& 				& 				& \omit\vrule	&  	0	&0		& \ddots	&				&	\low{\vdots}	& \omit\vrule 	& 		& 				& 		&		\\
	\vdots 	& 		& 		& 	0	& 		& 		& \omit\vrule 	& \vdots			& \vdots			& \vdots			& \omit\vrule	& \vdots 	& 		& \ddots	&	-1			&	-1			& \omit\vrule 	& \vdots	& \vdots			& \vdots	& \vdots	\\
	f_{p-1} 	& 		& 		& 		& 		& 		& \omit\vrule 	& 				& 				& 				& \omit\vrule	&  	0	& 		& 		&	0			&	-1			& \omit\vrule 	& 		& 				& 		&		\\
	f_{p} 	& 		& 		& 		& 		& 		& \omit\vrule 	& 	1			& 	1			& 	-1			& \omit\vrule	&  	0	& 		& \cdots	&				&	0			& \omit\vrule 	& 	0	& 	1			& 	0	&	0	\\
\cline{1-1} \cline{2-21}
	v	 	& 	-1	& 		& \cdots	& 		& 	-1	& \omit\vrule 	& \text{\tiny {\(p+1\)}}	& \text{\tiny {\(p+1\)}}	&\text{\tiny {\(-p-1\)}}	& \omit\vrule	&  	-1	& 	-2	& \push{\cdots}				&	-p			& \omit\vrule 	& 	0	& \text{\tiny {\(p+1\)}}	& 	1	&	0	\\
	r	 	& 	-1	& 		& \cdots	& 		& 	-1	& \omit\vrule 	& \text{\tiny {\(p+2\)}}	& \text{\tiny {\(p+1\)}}	&\text{\tiny {\(-p-2\)}}	& \omit\vrule	&  	-1	& 	-2	& \push{\cdots}				&	-p			& \omit\vrule 	& 	0	& \text{\tiny {\(p+1\)}}	& 	1	&	0	\\
	s	 	& 		& 		& 		& 		& 		& \omit\vrule 	& 	0			& 	0			& 	1			& \omit\vrule	&  		& 		& 		&				&				& \omit\vrule 	& 	0	& 	0			& 	0	&	0	\\
	s'	 	& 		& 		& \low{0}	& 		& 		& \omit\vrule 	& 	0			& 	0			& 	1			& \omit\vrule	&  		& 		& \low{0}	&				&				& \omit\vrule 	& 	0	& 	0			& 	0	&	-1	\\
	g	 	& 		& 		& 		& 		& 		& \omit\vrule 	& 	0			& 	0			& 	1			& \omit\vrule	&  		& 		& 		&				&				& \omit\vrule 	& 	-1	& 	0			& 	0	&	0	\\
	u	 	& 		& 		& 		& 		& 		& \omit\vrule 	& 	1			& 	1			& 	-1			& \omit\vrule	&  		& 		& 		&				&				& \omit\vrule 	& 	0	& 	0			& 	0	&	0		
},$$
and thus compute that $-R_V H^{-1} D \mathbf{\widehat{t}}=t_U-t_V-t_Z$ and
$$-S H^{-1} D= \kbordermatrix{
	\mbox{}	&t_1		&t_2		& \cdots 	& t_{p-1} 	& t_p 	& \omit\vrule	& t_U 	& t_V 	& t_W 	&t_Z		\\
	t_1 		& 1 		& 1 		& \cdots 	& 1  		& 1 		& \omit\vrule	& 0 		& -1 		& 0  		& 0 		\\
	t_2 		& 1 		& 2 		& \cdots 	& 2  		& 2 		& \omit\vrule	& 0 		& -2 		& 0  		& 0 		\\
	\vdots 	& \vdots 	& \vdots 	& \ddots 	& \vdots  	& \vdots 	& \omit\vrule	& \vdots 	& \vdots 	& \vdots 	& \vdots	\\
	t_{p-1} 	& 1 		& 2 		& \cdots 	& p-1  	& p-1 	& \omit\vrule	& 0		& -(p-1)  	& 0  		& 0 		\\
	t_p 		& 1 		&  2 		& \cdots 	&  p-1 	& p 		& \omit\vrule	& 0		& -p 		& 0  		& 0 		\\
\cline{1-1} \cline{2-11}	
	t_U 		& -1 		& -2 		& \cdots 	& -(p-1)  	& -p  	& \omit\vrule	& 0 		& p+1 	& 1  		& 0 		\\
	t_V 		& 0 		& 0 		& \cdots 	&  0 		& 0 		& \omit\vrule	& -1 		& 0 		& 0  		& 0 		\\
	t_W 		& 0 		& 0 		& \cdots 	&  0 		& 0 		& \omit\vrule	& 0 		& 0 		& 0 		& 0 		\\
	t_Z 		& 0 		& 0 		& \cdots 	&  0 		& 0 		& \omit\vrule	& 0 		& 0 		& 0 		& 0 		}.$$

Since $\widehat{\mathbf{t}}^{\!\top} (-S H^{-1} D) \mathbf{\widehat{t}} = \mathbf{t}^{\!\top} Q_n \mathbf{t}
+ (t_V-t_U)(t_1+\ldots+p t_p-pt_U)$, where $\mathbf{t}=(t_1,\ldots,t_p,t_U,t_W)$ and $Q_n$ is defined in Theorem \ref{thm:part:func}, we conclude that the kinematical kernel can be written as  
\[
\mathcal{K}_{Y_n}(\mathbf{\widehat{t}})= 
 e^{2 i \pi \left( \mathbf{t}^{\!\top} Q_n \mathbf{t} +(t_V - t_U)(t_1 + \cdots + p t_p - p t_U) \right)}
\delta(-t_Z) \delta(t_U - t_V - t_Z).
\]

We now compute the dynamical content. We denote
$\alpha=(a_1,b_1,c_1,\ldots,a_W,b_W,c_W,a_Z,b_Z,c_Z)$ a general 
vector in $\mathcal{S}_{Y_n}$.
Let $\hbar>0$. The dynamical content $\mathcal{D}_{\hbar,Y_n}(\mathbf{\widehat{t}},\alpha)$ is equal to:
\[
e^{\frac{1}{\sqrt{\hbar}} \widehat{C}(\alpha)^{\!\top} \mathbf{\widehat{t}}}
\dfrac{
\Phi_\B\left (t_U + \frac{i}{2 \pi \sqrt{\hbar}} (\pi-a_U)\right )
\Phi_\B\left (t_V + \frac{i}{2 \pi \sqrt{\hbar}}(\pi-a_V)\right )
\Phi_\B\left (t_W + \frac{i}{2 \pi \sqrt{\hbar}} (\pi-a_W)\right )
}{
\Phi_\B\left (t_1 - \frac{i}{2 \pi \sqrt{\hbar}} (\pi-a_1)\right )
\cdots
\Phi_\B\left (t_p - \frac{i}{2 \pi \sqrt{\hbar}} (\pi-a_p)\right )
\Phi_\B\left (t_Z - \frac{i}{2 \pi \sqrt{\hbar}} (\pi-a_Z)\right )
},
\]
where $\widehat{C}(\alpha) = (c_1, \ldots, c_p,c_U,c_V, c_W, c_Z)^{\!\top}$.

Let us come back to the computation of the partition function of the Teichm\"uller TQFT. By definition, 
\[
\mathcal{Z}_{\hbar}(Y_n,\alpha)= \int_{\mathbf{\widehat{t}}\in\R^{Y_n^{3}}} d\mathbf{\widehat{t}} \mathcal{K}_{Y_n}(\mathbf{\widehat{t}}) \mathcal{D}_{\hbar,Y_n}(\mathbf{\widehat{t}},\alpha).
\]
We begin by integrating over the variables $t_V$ and $t_Z$, which consists in removing the two Dirac delta functions $\delta(t_Z)$ and $\delta(t_U - t_V - t_Z)$ in the kinematical kernel and replacing $t_Z$ by $0$ and $t_V$ by $t_U$ in the other terms. Therefore, we have
$$
\Phi_{\B}\left( \frac{\pi-a_Z}{2\pi i \sqrt{\hbar}} \right)  \mathcal{Z}_{\hbar}(Y_n,\alpha) 
= \int_{\mathbf{t}\in\R^{p+2}} d\mathbf{t} 
e^{2 i \pi \mathbf{t}^{\!\top} Q_n\mathbf{t}} e^{\frac{1}{\sqrt{\hbar}} (c_1 t_1 + \cdots + c_p t_p + (c_U + c_V)t_U + c_W t_W)} 
\Pi(\mathbf{t},\alpha,\hbar),$$
where $\mathbf{t} =(t_1, \ldots,t_p,t_U,t_W)$ and
$$\Pi(\mathbf{t},\alpha,\hbar) :=
\frac{
\Phi_\B\left (t_U + \frac{i}{2 \pi \sqrt{\hbar}} (\pi-a_U)\right )
\Phi_\B\left (t_U + \frac{i}{2 \pi \sqrt{\hbar}}(\pi-a_V)\right )
\Phi_\B\left (t_W + \frac{i}{2 \pi \sqrt{\hbar}} (\pi-a_W)\right )
}{
\Phi_\B\left (t_1 - \frac{i}{2 \pi \sqrt{\hbar}} (\pi-a_1)\right )
\cdots
\Phi_\B\left (t_p - \frac{i}{2 \pi \sqrt{\hbar}} (\pi-a_p)\right )
}.
$$

\textit{Step 2. Applying the dominated convergence theorem for $\alpha \to \tau$.}

For the rest of the proof, let 
$$\tau = (a^\tau_1,b^\tau_1,c^\tau_1,\ldots,a^\tau_Z,b^\tau_Z,c^\tau_Z) \in \mathcal{S}_{Y_n \setminus Z} \times \overline{\mathcal{S}_Z}$$
be such that $\omega_j(\tau) = 2 \pi$ for all $j \in \{0,1, \ldots, p-1,p\}$,
$\widehat{\omega}_j(\tau) = 2 \pi$ for all $j \in \{s,d,p+1\}$
and $\widehat{\omega}_{\overrightarrow{K_n}}(\tau)=a^\tau_Z=0$.

Let $\delta>0$ such that there exists a {neighbourhood} $\mathfrak{U}$ of $\tau$ in $\mathcal{S}_{Y_n \setminus Z} \times \overline{\mathcal{S}_Z}$ such that for each $\alpha \in \mathfrak{U} \cap \mathcal{S}_{Y_n}$ the $3p+9$ first coordinates $a_1, \ldots, c_W$ of $\alpha$ live in $(\delta,\pi-\delta)$.

Then for all $\alpha \in \mathfrak{U}\cap \mathcal{S}_{Y_n}$, for any $j \in \{1, \ldots,p,U,V,W\}$, and for any $t \in \R$, we have
$$\left |e^{\frac{1}{ \sqrt{\hbar}} c_j t} \Phi_\B\left (t \pm \frac{i}{ 2 \pi \sqrt{\hbar}}(b_j+c_j)\right )^{\pm 1}\right | \leqslant M_{\delta,\hbar} \ e^{-\frac{1}{\sqrt{\hbar}} \delta |t| }.$$
Indeed, this is immediate for $t \leqslant 0$ by Lemma \ref{lem:PhiB:bounded} and the fact that $c_j>\delta$. For $t \geqslant 0$, one has to use that $b_j > \delta$ but also Proposition \ref{prop:quant:dilog} (1) and (2) to remark that :
$$ \left |\Phi_\B\left (t + \frac{i}{ 2 \pi \sqrt{\hbar}}(b_j+c_j)\right )\right | =
\left |\Phi_\B\left (-t + \frac{i}{ 2 \pi \sqrt{\hbar}}(b_j+c_j)\right )\right | \left |e^{i \pi \left ( \frac{i}{ 2 \pi \sqrt{\hbar}}(b_j+c_j)\right )^2}\right | \leqslant M_{\delta,\hbar} e^{-\frac{1}{\sqrt{\hbar}} (b_j+c_j) t}.$$

Consequently, we have a domination of the previous integrand uniformly over $\mathfrak{U}\cap \mathcal{S}_{Y_n}$, i.e.\ 
$$\left |e^{2 i \pi \mathbf{y}^{\!\top} Q_n\mathbf{y}} e^{\frac{1}{\sqrt{\hbar}} (c_1 t_1 + \cdots + c_p t_p + (c_U + c_V)t_U + c_W t_W)} 
\Pi(\mathbf{t},\alpha,\hbar)\right | \leqslant  \left (M_{\delta,\hbar}\right )^{p+3}
e^{-\frac{1}{\sqrt{\hbar}} \delta \left ( |t_1|+\ldots |t_p|+2|t_U|+|t_W| \right )}
$$
for  all $\alpha \in \mathfrak{U}\cap \mathcal{S}_{Y_n}$ and  for all $\mathbf{t} \in \R^{p+2}$.

Since the right hand side of this inequality is integrable over $\R^{p+2}$, we can then apply the dominated convergence theorem to conclude that $\Phi_{\B}\left( \frac{\pi-a_Z}{2\pi i \sqrt{\hbar}} \right)  \mathcal{Z}_{\hbar}(Y_n,\alpha) $ tends to 
$$
\int_{\mathbf{t}\in\R^{p+2}} d\mathbf{t} 
e^{2 i \pi \mathbf{t}^{\!\top} Q_n\mathbf{t}} e^{\frac{1}{\sqrt{\hbar}} (c^\tau_1 t_1 + \cdots + c^\tau_p t_p + (c^\tau_U + c^\tau_V)t_U + c^\tau_W t_W)} \Pi(\mathbf{t},\tau,\hbar)$$
as $\alpha \in \mathcal{S}_{Y_n}, \alpha \to \tau$ (recall that $c_j^\tau$ denotes the $c_j$ coordinate of $\tau$).

\textit{Step 3. Retrieving the value $J_{X_n}(\hbar,0)$ in $\alpha =\tau$.}

Let us now prove that
$$
\int_{\mathbf{t}\in\R^{p+2}} d\mathbf{t} 
e^{2 i \pi \mathbf{t}^{\!\top} Q_n\mathbf{t}} e^{\frac{1}{\sqrt{\hbar}} (c^\tau_1 t_1 + \cdots + c^\tau_p t_p + (c^\tau_U + c^\tau_V)t_U + c^\tau_W t_W)} \Pi(\mathbf{t},\tau,\hbar) = J_{X_n}(\hbar,0).$$

We first do the following  change of variables:
\begin{itemize}
\item $y'_k = t_k - \frac{i}{2 \pi \sqrt{\hbar}} (\pi-a^\tau_k)$ for $1 \leqslant k \leqslant p$,
\item $y'_l = t_l + \frac{i}{2 \pi \sqrt{\hbar}} (\pi-a^\tau_l)$ for $l\in\{U,W\}$,
\end{itemize}
and we denote $\mathbf{y'}=\left (y'_1, \ldots, y'_{p}, y'_U, y'_W\right )^{\!\top}$. 
Note that the term $\Phi_\B\left (t_U + \frac{i}{2 \pi \sqrt{\hbar}}(\pi-a^\tau_V)\right )$ will become 
$\Phi_\B\left (y'_U + \frac{i}{2 \pi \sqrt{\hbar}}(a^\tau_U-a^\tau_V)\right )=
\Phi_\B\left (y'_U \right ),$
since $a^\tau_U-a^\tau_V = (\widehat{\omega}_{s}(\tau)- 2\pi)+(\widehat{\omega}_{d}(\tau)- 2\pi) = 0$.

We also denote 
$$\mathcal{Y}'_{\hbar,\tau} :=
\prod_{k=1}^p\left (\R - \frac{i}{2 \pi \sqrt{\hbar}} (\pi-a^\tau_k)\right ) 
\times 
\prod_{l=U,W}
\left (\R + \frac{i}{2 \pi \sqrt{\hbar}} (\pi-a^\tau_l)\right ),$$
the subset of $\C^{p+2}$ {in} which the variables in $\mathbf{y'}$  reside. 

By a similar computation as in the proof of Theorem \ref{thm:part:func}, we obtain
\begin{align*}
&\int_{\mathbf{t}\in\R^{p+2}} d\mathbf{t} 
e^{2 i \pi \mathbf{t}^{\!\top} Q_n\mathbf{t}} e^{\frac{1}{\sqrt{\hbar}} (c^\tau_1 t_1 + \cdots + c^\tau_p t_p + (c^\tau_U + c^\tau_V) t_U + c^\tau_W t_W)} \Pi(\mathbf{t},\tau,\hbar)\\
&\stareq
\int_{\mathbf{y'} \in \mathcal{Y}'_{\hbar,\tau}}  d\mathbf{y'}
e^{
2 i \pi \mathbf{y}^{\prime T} Q_n  \mathbf{y'} + \frac{1}{\sqrt{\hbar}} \mathcal{W}(\tau)^{\!\top} \mathbf{y'} 
}
\dfrac{
\Phi_\B\left (y'_U\right )
\Phi_\B\left (y'_U \right )
\Phi_\B\left (y'_W\right )
}{
\Phi_\B\left (y'_1\right )
\cdots
\Phi_\B\left (y'_p\right )
},
\end{align*}
where for any $\alpha \in \mathcal{S}_{Y_n \setminus Z}$, $\mathcal{W}(\alpha)$ is defined as
$$\mathcal{W}(\alpha):= 2 Q_n \Gamma(\alpha)+C(\alpha)+(0,\ldots,0,c_V,0)^{\!\top},$$
following the definitions of $\Gamma(\alpha)$ and $C(\alpha)$ in the proof of Theorem \ref{thm:part:func}.

Hence, from the value of $J_{X_n}(\hbar,0)$, it remains only to prove that $\mathcal{W}(\tau) = \mathcal{W}_n$.
 
Let us denote $\Lambda: (u_1,\ldots,u_p,u_U,u_V,u_W) \mapsto (u_1,\ldots,u_p,u_U,u_W)$ the process of forgetting the second-to-last coordinate. then obviously $C(\alpha) = \Lambda (\widetilde{C}(\alpha))$. Recall from Lemma \ref{lem:2QGamma+C} 
that $\widetilde{\mathcal{W}}(\alpha) = 2 \widetilde{Q}_n \widetilde{\Gamma}(\alpha) + \widetilde{C}(\alpha)$ depends almost only on edge weights of the angles in $X_n$.

Thus, a direct calculation shows that for any $\alpha \in \mathcal{S}_{Y_n \setminus Z}$, we have
\begin{equation*} \label{eqn:v:Alpha:Odd}
\mathcal{W}(\alpha) =
\Lambda(\widetilde{\mathcal{W}}(\alpha)) + 
\begin{bmatrix}
0 \\ \vdots \\ 0 \\ c_V - 4(\pi-a_U)+3(\pi-a_V)-(\pi-a_W) \\ a_U-a_V
\end{bmatrix}.
\end{equation*}

Now, if we specify $\alpha=\tau$, then the weights $\omega_{X_n,j}(\alpha)$ appearing in $\Lambda(\widetilde{\mathcal{W}}(\alpha))$ will all be equal to $2\pi$, since 
$\omega_s(\tau) =\widehat{\omega}_s(\tau)-\widehat{\omega}_{\overrightarrow{K_n}}(\tau) = 2 \pi$
and
$$\omega_{p+1}(\tau) =\widehat{\omega}_d(\tau)+\widehat{\omega}_{p+1}(\tau) - 2\left (\pi-\widehat{\omega}_{\overrightarrow{K_n}}(\tau)\right ) = 2\pi.$$
Hence 
$$\mathcal{W}(\tau)= \mathcal{W}_{n} + 
\begin{bmatrix}
0 \\ \vdots \\ 0 \\ 
\pi - \frac{1}{2}\lambda_{X_n}(\tau)+
c^\tau_V - 4(\pi-a^\tau_U)+3(\pi-a^\tau_V)-(\pi-a^\tau_W) \\ a^\tau_U-a^\tau_V
\end{bmatrix}.
$$

Recall that $a^\tau_U-a^\tau_V=0$, and remark finally that
\begin{align*}
&\pi - \frac{1}{2}\lambda_{X_n}(\tau)+
c^\tau_V - 4(\pi-a^\tau_U)+3(\pi-a^\tau_V)-(\pi-a^\tau_W) \\
&= 3a^\tau_U-2a^\tau_V+a^\tau_W+b^\tau_W-\pi\\
&= 2(a^\tau_U-a^\tau_V)+(a^\tau_U-c^\tau_W)\\
&= -(\widehat{\omega}_{d}(\tau)- 2\pi) -\widehat{\omega}_{\overrightarrow{K_n}}(\tau)=0.
\end{align*}
Hence $\mathcal{W}(\tau) = \mathcal{W}_n$ and the theorem is proven.
\end{proof}


\section{Proving the volume conjecture (odd case)}\label{sec:vol:conj}

We now arrive to the final and most technical part of this paper, that is to say the proof of the volume conjecture using detailed analytical methods. We advise the reader to be familiar with the proofs and {notation} of Section \ref{sec:part:odd} before reading this section. Having read section \ref{sec:part:H:odd} is not as essential, but can nevertheless help understanding some arguments in the following first three subsections. The main result is as follows:

\begin{theorem} \label{thm:vol:conj}
Let $n$ be an odd integer greater or equal to $3$. Let $J_{X_n}$ and $\mathfrak{J}_{X_n}$ be the functions defined in Theorem \ref{thm:part:func} and Corollary \ref{cor:part:func}. Then we have:
$$
\lim_{\hbar \to 0^+} 2\pi \hbar \log \vert J_{X_n}(\hbar,0) \vert
= \lim_{\hbar \to 0^+} 2\pi \hbar \log \vert \mathfrak{J}_{X_n}(\hbar,0) \vert
 = -\emph{Vol}(S^3\backslash K_n).$$
 In other words, the Teichm\"uller TQFT volume conjecture of Andersen--Kashaev is proved for the infinite family of odd twist knots.
\end{theorem}

The proof of Theorem \ref{thm:vol:conj} will be split into several lemmas. The general idea is to translate the expressions in Theorem \ref{thm:vol:conj} into asymptotics of the form of Theorem \ref{thm:SPM}, and check that the assumptions of Theorem \ref{thm:SPM} are satisfied one by one, i.e.\ that we are allowed to apply the saddle point method. Technical analytical lemmas are required for the asymptotics and error bounds, notably due to the fact that we work with \textit{unbounded} integration contours.

More precisely, here is an overview of Section \ref{sec:vol:conj}:
\begin{itemize}
	\item \underline{Sections \ref{sub:S:U}, \ref{sub:ReS:Yalpha} and \ref{sub:ReS:Y0}:} For the ``classical'' potential $S$, we check the prerequisites for the saddle point method, notably that $\Re(S)$ attains a maximum of $-\Vol(S^3\setminus K_n)$ at the complete angle structure (from Lemma \ref{lem:complex:sym} to Lemma \ref{lem:-vol}). This part refers to Thurston's gluing equations and the properties of the classical dilogarithm.
	\item \underline{Section \ref{sub:asym:Y0}:} We apply the saddle point method to the classical potential $S$ on a compact integration contour (Proposition \ref{prop:compact:contour:S:SPM}) and we then deduce asymptotics when the contour is unbounded (Lemma \ref{lem:unbounded:contour} and Proposition \ref{prop:all:contour:S}). This part is where the analytical arguments start.
	\item \underline{Section \ref{sub:asym:PhiB}:} We compare the classical and quantum dilogarithms $\Li$ and $\Phi_{\B}$ in the asymptotic $\B \to 0^+$ (Lemmas \ref{lem:parity}, \ref{lem:unif:bound}, \ref{lem:unif:bound:neg}) and deduce asymptotics for the quantum potential $S_\B$ (Proposition \ref{prop:all:contour:Sb}). This part, and Lemma \ref{lem:unif:bound} in particular, contains the heart of the proof, and needs several new analytical arguments to establish uniform bounds on an unbounded integration contour.
	\item \underline{Section \ref{sub:asym:hbar}:} In order to get back to the functions $J_{X_n}$ and $\mathfrak{J}_{X_n}$ of Theorem \ref{thm:vol:conj}, we compare the two previous potentials with a second quantum potential $S'_\B$ related to $J_{X_n}$ (Remark \ref{rem:J':S'b}) and we deduce the corresponding asymptotics for $S'_\B$ (Lemma \ref{lem:unif:bound:hbar} and Proposition \ref{prop:all:contour:S'b}). This part uses similar analytical arguments as the previous one, and is needed because of the particular construction of the Teichm\"uller TQFT partition function and the subtle difference between $\frac{1}{\B^2}$ and $\frac{1}{\hbar}$.
	\item \underline{Section \ref{sub:conjvol:conclusion}:} We conclude with the (now short) proof of Theorem \ref{thm:vol:conj} and we offer comments on how our techniques could be re-used for further works.
\end{itemize}

Let us finish this introduction by establishing some {notation}. For the remainder of this section, $n$ will be an odd integer greater or equal to $3$ and $p=\frac{n-3}{2}$.

Let us now recall and define some {notation}:
\begin{itemize}
\item We denote the following product of open ``horizontal  bands" in $\C$, and
$$\mathcal{U}:= \prod_{k=1}^p\left (\R + i (-\pi,0) \right ) 
\times 
\prod_{l=U,W}
\left (\R + i (0,\pi)\right ),$$
 an open subset of $\C^{p+2}$.
\item For any angle structure $\alpha = (a_1, \ldots, c_W) \in \mathcal{A}_{X_n}$, we denote 
$$\mathcal{Y}_\alpha :=
\prod_{k=1}^p\left (\R - i (\pi - a_k)\right ) 
\times 
\prod_{l=U,W}
\left (\R + i (\pi - a_l)\right ),$$
an affine real plane of real dimension $p+2$ in $\C^{p+2}$, contained in the band $\mathcal{U}$.
\item For the complete angle structure $\alpha^0 = (a^0_1, \ldots, c^0_W) \in \mathcal{A}_{X_n}$ (which exists because of Theorem \ref{thm:geometric}), we denote 
$$\mathcal{Y}^0 := \mathcal{Y}_{\alpha^0}.$$
\item We define the potential function $S\colon \mathcal{U} \to \C$, an holomorphic  function on $p+2$ complex variables, by:
$$S(\mathbf{y}) =
i \mathbf{y}^{\!\top} Q_n \mathbf{y} +  
  \mathbf{y}^{\!\top} \mathcal{W}_n
  + i \Li\left (-e^{y_1}\right ) + \cdots
  + i \Li\left (-e^{y_p}\right )
  - 2 i \Li\left (-e^{y_U}\right )
  - i \Li\left (-e^{y_W}\right ), 
   $$
   where $Q_n$ and $\mathcal{W}_n$ are like in Theorem \ref{thm:part:func}.
\end{itemize}

\subsection{Properties of the potential function $S$ on the open band $\mathcal{U}$}\label{sub:S:U}

The following lemma will be very useful to prove the invertibility of the holomorphic hessian of the potential $S$.

\begin{lemma}\label{lem:complex:sym}
Let $m\geqslant 1$ an integer, and $S_1, S_2 \in M_m(\R)$ such that $S_1$ is symmetric positive definite and $S_2$ is symmetric. Then the complex symmetric matrix $S_1 + i S_2$ is invertible.
\end{lemma}

\begin{proof}
Let $v \in \C^m$ such that $(S_1 + i S_2) v = 0$. Let us prove that $v=0$.

Since $S_1$ and $S_2$ are real symmetric (hence hermitian), we have $\overline{v}^{\!\top} S_1 v, \overline{v}^{\!\top} S_2 v \in \R$.

Now, since $(S_1 + i S_2) v = 0$, then 
$$0 = \overline{v}^{\!\top} (S_1 + i S_2) v  = \overline{v}^{\!\top} S_1 v +i \overline{v}^{\!\top} S_2 v,$$
thus, by taking the real part, we get  $0=\overline{v}^{\!\top} S_1 v$, which implies $v=0$ since $S_1$ is positive definite.
\end{proof}

We can now prove that the holomorphic hessian is non-degenerate at each point.

\begin{lemma}\label{lem:hess}
For every $\mathbf{y}\in \mathcal{U}$, the holomorphic hessian of $S$ is given by:
$$ \mathrm{Hess}(S)(\mathbf{y}) = \left (\dfrac{\partial^2 S}{\partial y_j \partial y_k}\right )_{j,k\in\{1,\ldots,p,U,W\}}
(\mathbf{y}) =
2 i  Q_n + i
\begin{pmatrix}
\frac{-1}{1+e^{-y_1}} & \ & 0 &0 &0 \\
\ & \ddots & \ & \vdots & \vdots \\
0 & \ & \frac{-1}{1+e^{-y_p}} &0 &0 \\
0 & \cdots & 0 &\frac{2}{1+e^{-y_U}} &0 \\
0 & \cdots & 0 &0 &\frac{1}{1+e^{-y_W}} 
\end{pmatrix}
.$$
Furthermore, $\mathrm{Hess}(S)(\mathbf{y})$ has non-zero determinant for every $\mathbf{y}\in \mathcal{U}$.
\end{lemma}

\begin{proof}
The first part follows from the double differentiation of $S$ and the fact that 
$$\dfrac{\partial \Li(-e^y)}{\partial y} = - \Log(1+e^y)$$
 for $y \in \R \pm i(0,\pi)$ (note that $y \in \R \pm i(0,\pi)$ implies $-e^y \in \C \setminus \R$).

Let us prove the second part. Let $\mathbf{y}\in \mathcal{U}$. Then $\Im(\mathrm{Hess}(S)(\mathbf{y}))$ is a symmetric matrix (as the sum of $Q_n$ and a diagonal matrix), and
$$
\Re(\mathrm{Hess}(S)(\mathbf{y}))=
\begin{pmatrix}
-\Im\left (\frac{-1}{1+e^{-y_1}}\right ) & \ & 0 &0 &0 \\
\ & \ddots & \ & \vdots & \vdots \\
0 & \ & -\Im\left (\frac{-1}{1+e^{-y_p}}\right ) &0 &0 \\
0 & \cdots & 0 &-\Im\left (\frac{2}{1+e^{-y_U}}\right ) &0 \\
0 & \cdots & 0 &0 &-\Im\left (\frac{1}{1+e^{-y_W}}\right )
\end{pmatrix}$$
is diagonal with negative coefficients (because $\Im(y_1), \ldots, \Im(y_p) \in(-\pi,0)$ and $\Im(y_U),\Im(y_W)\in (0,\pi)$). Hence it follows from Lemma \ref{lem:complex:sym} that $\mathrm{Hess}(S)(\mathbf{y})$ is invertible for every $\mathbf{y}\in \mathcal{U}$.
\end{proof}

The following lemma establishes an equivalence between critical points of the potential $S$ and complex shape structures that solve the balancing and completeness equations.

\begin{lemma}\label{lem:grad:thurston}
Let us consider the diffeomorphism
$$\psi := \left (\prod_{T \in \{T_1,\ldots,T_p,U,W\} } \psi_T \right )\colon (\R+i\R_{>0})^{p+2} \to \mathcal{U},$$
where $\psi_T$ was defined in Section \ref{sub:thurston}.
Then $\psi$ induces a bijective mapping between\\
$\{\mathbf{y} \in \mathcal{U};
\nabla S(\mathbf{y}) = 0\}$ and
$$\left \{\mathbf{z}=(z_1,\ldots,z_p,z_U,z_W) \in (\R+i\R_{>0})^{p+2} |
\mathcal{E}_{X_n,0}(\mathbf{z}) \wedge \ldots \wedge \mathcal{E}_{X_n,p-1}(\mathbf{z}) \wedge
\mathcal{E}^{co}_{X_n,p+1}(\mathbf{z}) \wedge \mathcal{E}^{co}_{X_n,s}(\mathbf{z})
\right \},$$
 where the equations $\mathcal{E}_{X_n,0}(\mathbf{z}), \ldots , \mathcal{E}_{X_n,p-1}(\mathbf{z}),
 \mathcal{E}^{co}_{X_n,p+1}(\mathbf{z}), \mathcal{E}^{co}_{X_n,s}(\mathbf{z})$ were defined at the end of Section \ref{sec:geom}.
 
In particular, $S$ admits only one critical point
$\mathbf{y^0}$ on $\mathcal{U}$, corresponding to the complete hyperbolic structure $\mathbf{z^0}$ on the geometric ideal triangulation $X_n$ (adding $z^0_V$ equal to $z^0_U$).
\end{lemma}

\begin{proof}
First we compute, for every $\mathbf{y} \in \mathcal{U}$,
$$
\nabla S(\mathbf{y}) =
\begin{pmatrix}
\partial_1 S(\mathbf{y})\\
\vdots\\
\partial_p S(\mathbf{y})\\
\partial_U S(\mathbf{y})\\
\partial_W S(\mathbf{y})
\end{pmatrix} =
2 i Q_n \mathbf{y} + \mathcal{W}_n + i
\begin{pmatrix}
-\Log (1+e^{y_1})\\
\vdots\\
-\Log (1+e^{y_p})\\
2\Log (1+e^{y_U})\\
\Log (1+e^{y_W})
\end{pmatrix}.
$$

Then, we define a  lower triangular matrix $A=\kbordermatrix{
	\mbox{}	&y_1	&y_2		&y_3		&\cdots  	& y_p	& \omit\vrule	& y_U	& y_W 	\\
	y_1 		& 1 	& 		& 		& 		& 		& \omit\vrule	& 		& 		\\
	y_2 		&-2 	& 1 		& 		& 		&0 		& \omit\vrule	& 		&		\\
	y_3 		&1 	& -2 		& 1 		& 		& 		& \omit\vrule	& 		&		\\
	\vdots 	& 	& \ddots 	& \ddots 	& \ddots 	& 		& \omit\vrule	& 		&		\\
	y_p		& 	& 		& 1 		& -2 		& 1 		& \omit\vrule	&0 		&0		\\
\cline{1-1} \cline{1-9}
	y_U 		& 	& 		& 		& 		& 1 		& \omit\vrule	& 1 		& 0 		\\
	y_W 		& 	&0 		& 		& 		&0 		& \omit\vrule	& 0		&1		}
\in GL_{p+2}(\Z)$, and 
we compute
$$
A \cdot  \nabla S(\mathbf{y}) =
\begin{pmatrix}
2i(y_1+\cdots+y_p-y_U)-2\pi p -i \Log (1+e^{y_1})\\
-2i y_1 + 2 \pi +2 i \Log (1+e^{y_1}) - i \Log (1+e^{y_2}) \\
2\pi -i \Log (1+e^{y_1}) +2 i \Log (1+e^{y_2}) - 2i y_2
-i \Log (1+e^{y_3})\\
\vdots\\
2\pi -i \Log (1+e^{y_{k-1}}) +2 i \Log (1+e^{y_k}) - 2i y_k
-i \Log (1+e^{y_{k+1}})\\
\vdots\\
2\pi -i \Log (1+e^{y_{p-2}}) +2 i \Log (1+e^{y_{p-1}}) - 2i y_{p-1}
-i \Log (1+e^{y_p})\\
\pi -i \Log (1+e^{y_p}) +2 i \Log (1+e^{y_U})+i y_W\\
\pi + i y_U +i\Log(1+e^{y_W})
\end{pmatrix}.
$$

For $1 \leqslant k \leqslant p$, by denoting $y_k=\psi_{T_k}(z_k)$, we have
$$ \Log(z_k) = y_k + i \pi, \
\Log(z'_k) = -\Log(1+e^{y_k}), \
\Log(z''_k) = \Log(1+e^{-y_k}),$$
and for $l = {U,W}$, by denoting $y_l=\psi_{T_l}(z_l)$, we have
$$ \Log(z_l) = -y_l + i \pi, \
\Log(z'_l) = -\Log(1+e^{-y_l}), \
\Log(z''_l) = \Log(1+e^{y_l}).$$

Hence we compute, for all $\mathbf{z} \in (\R+i\R_{>0})^{p+2}$,
$$
A \cdot  (\nabla S)(\psi(\mathbf{z})) =
i \begin{pmatrix}
\Log(z'_1) + 2 \Log(z_1)+\cdots + 2\Log(z_p)+2\Log(z_U)-2i\pi\\
2\Log(z''_1)+\Log(z'_2)-2i\pi \\
\Log(z'_{1})+2\Log(z''_2)+\Log(z'_{3})-2i\pi \\
\vdots\\
\Log(z'_{k-1})+2\Log(z''_k)+\Log(z'_{k+1})-2i\pi \\
\vdots\\
\Log(z'_{p-2})+2\Log(z''_{p-1})+\Log(z'_{p})-2i\pi \\
\Log(z'_{p}) +2 \Log(z''_{U})-\Log(z_{W}) \\
\Log(z''_{W}) -\Log(z_{U})
\end{pmatrix}.
$$
This last vector is zero if and only if one has
$$\mathcal{E}_{X_n,0}(\mathbf{z}) \wedge \ldots \wedge \mathcal{E}_{X_n,p-1}(\mathbf{z}) \wedge
\mathcal{E}^{co}_{X_n,p+1}(\mathbf{z}) \wedge \mathcal{E}^{co}_{X_n,s}(\mathbf{z}).$$
Since $A$ is invertible, we thus have 
$$ 
\mathbf{z} \in (\R+i\R_{>0})^{p+2} \text{ and } 
\mathcal{E}_{X_n,0}(\mathbf{z}) \wedge \ldots \wedge \mathcal{E}_{X_n,p-1}(\mathbf{z}) \wedge
\mathcal{E}^{co}_{X_n,p+1}(\mathbf{z}) \wedge \mathcal{E}^{co}_{X_n,s}(\mathbf{z})$$
$$\Updownarrow$$
$$\psi(\mathbf{z}) \in \mathcal{U} \text{ and }
(\nabla S)(\psi(\mathbf{z})) = 0.$$
\end{proof}

Let us now consider the multi-contour
$$\mathcal{Y}^0= \mathcal{Y}_{\alpha^0} =
\prod_{k=1}^p\left (\R - i (\pi - a^0_k)\right ) 
\times 
\prod_{l=U,W}
\left (\R + i (\pi - a^0_l)\right ),$$
 where $\alpha^0 \in \mathcal{A}_{X_n}$ is the complete hyperbolic angle structure corresponding to the complete hyperbolic complex shape structure $\mathbf{z^0}$. Notice that $\mathbf{y^0} \in \mathcal{Y}^0 \subset \mathcal{U}$. 

We will parametrise $\mathbf{y}  \in \mathcal{Y}^0$ as
$$\mathbf{y}= 
\begin{pmatrix}
y_1 \\ \vdots \\ y_W
\end{pmatrix} =
\begin{pmatrix}
x_1 + i d^0_1 \\ \vdots \\ x_W + i d^0_W
\end{pmatrix} =
\mathbf{x}+ i \mathbf{d^0},
$$
where $d^0_k = -(\pi - a^0_k) <0$ for $k=1, \ldots, p$ and $d^0_l = \pi - a^0_l>0$ for $l=U,W$. For the scrupulous readers, this means that $\mathbf{d^0}$ is a new notation for $\Gamma(\alpha^0)$, where $\Gamma(\alpha)$ was defined in Section \ref{sec:part:odd}.
Notice that $\mathcal{Y}^0 = \R^{p+2} + i \mathbf{d^0} \subset \C^{p+2}$ is an $\R$-affine subspace of $\C^{p+2}$.

\subsection{Concavity of $\Re S$ on each contour $\mathcal{Y}_{\alpha}$} \label{sub:ReS:Yalpha}
Now we focus on the behaviour of the real part $\Re S$ of the classical potential, on each horizontal contour $\mathcal{Y}_{\alpha}$.

\begin{lemma}\label{lem:concave}
For any $\alpha \in \mathcal{A}_{X_n}$, the function $\Re S \colon \mathcal{Y}_{\alpha} \to \R$ is strictly concave on $\mathcal{Y}_{\alpha}$.
\end{lemma}

\begin{proof} Let $\alpha \in \mathcal{A}_{X_n}$.
Since $\Re S \colon \mathcal{Y}_{\alpha} \to \R$ is twice continuously differentiable (as a function on $p+2$ real variables), we only need to check that its (real) hessian matrix $\left (\Re S\vert_{\mathcal{Y}_{\alpha}}\right )''$ is negative definite on every point $\mathbf{x}+ i \mathbf{d} \in \mathcal{Y}_{\alpha}$.

Now, since this real hessian is equal to the real part of the holomorphic hessian of $S$, it follows from Lemma \ref{lem:hess} that for all $\mathbf{x} \in \R^{p+2}$, this real hessian is:
\begin{align*}
&\left (\Re S\vert_{\mathcal{Y}_{\alpha}}\right )''(\mathbf{x}+ i \mathbf{d})=
\Re(\mathrm{Hess}(S)(\mathbf{x}+ i \mathbf{d}))\\
&= 
\begin{pmatrix}
-\Im\left (\frac{-1}{1+e^{-x_1-i d_1}}\right ) & \ & 0 &0 &0 \\
\ & \ddots & \ & \vdots & \vdots \\
0 & \ & -\Im\left (\frac{-1}{1+e^{-x_p-i d_p}}\right ) &0 &0 \\
0 & \cdots & 0 &-\Im\left (\frac{2}{1+e^{-x_U-i d_U}}\right ) &0 \\
0 & \cdots & 0 &0 &-\Im\left (\frac{1}{1+e^{-x_W-i d_W}}\right )
\end{pmatrix},
\end{align*}
which is diagonal  with negative coefficients, since $d_1,\ldots,d_p \in (-\pi,0)$ and $d_U,d_W\in (0,\pi)$. 

In particular $\left (\Re S\vert_{\mathcal{Y}_{\alpha}}\right )''$ is negative definite everywhere, thus $\Re S\vert_{\mathcal{Y}_{\alpha}}$ is strictly concave.
\end{proof}

\subsection{Properties of $\Re S$ on the complete contour $\mathcal{Y}^0$}\label{sub:ReS:Y0}

On the complete contour $\mathcal{Y}^0$, the function $\Re S$ is not only strictly concave but also admits a strict global maximum, at the complete structure  $\mathbf{y^0}$.

\begin{lemma}\label{lem:maximum}
	The function  $\Re S \colon \mathcal{Y}^0 \to \R$ admits a strict global maximum on $\mathbf{y^0} \in \mathcal{Y}^0$.
\end{lemma}

\begin{proof}
	Since the holomorphic gradient of $S\colon \mathcal{U}\to \C$ vanishes on $\mathbf{y^0}$ by Lemma \ref{lem:grad:thurston}, the (real) gradient of $\Re S\vert_{\mathcal{Y}^0}$ (which is the real part of the holomorphic gradient of $S$) then vanishes as well on $\mathbf{y^0}$, thus $\mathbf{y^0}$ is a critical point of 
	$\Re S\vert_{\mathcal{Y}^0}$.
	
	Besides, $\Re S\vert_{\mathcal{Y}^0}$ is strictly concave by Lemma \ref{lem:concave}, thus $\mathbf{y^0}$ is a global maximum of $\Re S\vert_{\mathcal{Y}^0}$.
\end{proof}

Before computing the value $\Re S (\mathbf{y^0})$, we establish a useful formula for the potential $S$:

\begin{lemma}\label{lem:rewriteS}
	The function $S\colon \mathcal{U} \to \C $ can be re-written 
	\begin{multline*}
		S(\mathbf{y}) =  i \Li\left (-e^{y_1}\right ) + \cdots
		+ i \Li\left (-e^{y_p}\right )
		+ 2 i \Li\left (-e^{-y_U}\right )
		+ i \Li\left (-e^{-y_W}\right ) \\
		+ i \mathbf{y}^{\!\top} Q_n \mathbf{y} +  
		i y_U^2 + i \frac{y_W^2}{2} +
		\mathbf{y}^{\!\top} \mathcal{W}_n
		+   i \frac{\pi^2}{2}.
	\end{multline*}
\end{lemma}

\begin{proof}
	We first recall the well-known formula for the dilogarithm (see Proposition \ref{prop:dilog} (1)):
	$$ \forall z \in \C \setminus [1,+\infty), \
	\Li\left (\frac{1}{z}\right ) = - \Li(z) - \frac{\pi^2}{6} - \frac{1}{2}\Log(-z)^2.
	$$
	We then apply this formula for $z=-e^{y_l}$ for $l\in\{U,W\}$ to conclude the proof.
\end{proof}

We can now use this formula to prove that the hyperbolic volume appears at the complete structure $\mathbf{y^0}$, in the following lemma.

\begin{lemma}\label{lem:-vol} We have
$$ \Re(S)(\mathbf{y^0}) = - \mathrm{Vol}(S^3 \setminus K_n).$$
\end{lemma}

\begin{proof}
From Lemma \ref{lem:rewriteS},  for all $\mathbf{y} \in \mathcal{U}$ we have
\begin{multline*}
S(\mathbf{y}) =  i \Li\left (-e^{y_1}\right ) + \cdots
  + i \Li\left (-e^{y_p}\right )
  + 2 i \Li\left (-e^{-y_U}\right )
  + i \Li\left (-e^{-y_W}\right ) \\
   + i \mathbf{y}^{\!\top} Q_n \mathbf{y} +  
i y_U^2 + i \frac{y_W^2}{2} +
  \mathbf{y}^{\!\top} \mathcal{W}_n
  +   i \frac{\pi^2}{2}, 
\end{multline*}
thus
\begin{multline*}
\Re(S)(\mathbf{y}) =  - \Im\left ( \Li\left (-e^{y_1}\right )\right ) - \cdots
  - \Im\left ( \Li\left (-e^{y_p}\right )\right )
- 2 \Im\left ( \Li\left (-e^{-y_U}\right )\right )
- \Im\left ( \Li\left (-e^{-y_W}\right )\right )
   \\
   -\Im\left ( \mathbf{y}^{\!\top} Q_n \mathbf{y} +  
 y_U^2 +  \frac{y_W^2}{2}\right ) +
 \Re\left ( \mathbf{y}^{\!\top} \mathcal{W}_n\right ).
\end{multline*}
Recall that for $z \in \R + i  \R_{>0}$, the ideal hyperbolic tetrahedron of complex shape $z$ has  hyperbolic volume $D(z) =\Im(\Li(z)) +\arg(1-z) \log|z|$ (where $D$ is the Bloch--Wigner function). 
Note that for $z=z_k = -e^{y_k}$ (with $1 \leqslant k \leqslant p$), we have 
$\arg(1-z) \log|z| = - c_k x_k$
and for $z=z_l = -e^{-y_l}$ (with $l\in\{U,W\}$), 
we have
$\arg(1-z) \log|z| = b_l x_l$. Thus we have for  $\mathbf{y} \in \mathcal{U}$:
\begin{multline*}
\Re(S)(\mathbf{y}) =  - D(z_1) - \cdots - D(z_p) - 2 D(z_U) - D(z_W)   
-c_1 x_1 - \cdots -c_p x_p + 2b_U x_U + b_W x_W\\
-2 \mathbf{x}^{\!\top} Q_n \mathbf{d} - 2 d_U x_U - d_W x_W
+ \mathbf{x}^{\!\top} \mathcal{W}_n.
\end{multline*}
Recall that $\mathbf{z^0}$ is the complex shape structure corresponding to the complete hyperbolic structure on the ideal triangulation $X_n$ where $z^0_U$ is the complex shape of both tetrahedra $U$ and $V$ (because of the completeness equation $z_U=z_V$). Thus 
\begin{align*}
-  \mathrm{Vol}(S^3 \setminus K_n) &= - D(z_1^0) - \cdots - D(z_p^0) -  D(z_U^0) - D(z_V^0) - D(z_W^0) \\
&= - D(z_1^0) - \cdots - D(z_p^0) - 2 D(z_U^0) - D(z_W^0).
\end{align*}
Hence we only need to prove that $(\mathbf{x^0})^{\!\top} \cdot \mathcal{T} = 0$, where
$$
\mathcal{T} :=
\begin{pmatrix}
-c_1^0 \\ \vdots \\ -c_p^0 \\ 2 b_U^0 \\ b_W^0
\end{pmatrix} + \mathcal{W}_n
-2 Q_n \mathbf{d^0} +
\begin{pmatrix}
0 \\ \vdots \\ 0 \\ -2 d_U^0 \\ -d_W^0
\end{pmatrix}.
$$
Since $d_l^0 = \pi - a_l^0 = b_l^0 +c_l^0$ for $l=U,W$, we have 
$\mathcal{T} = 
- \begin{pmatrix}
c_1^0 \\ \vdots \\ c_p^0 \\ 2 c_U^0 \\ c_W^0
\end{pmatrix} + \mathcal{W}_n
-2 Q_n \mathbf{d^0}.$

It then follows from the definitions of $\mathcal{W}, \mathcal{W}_n, \widetilde{\Gamma}, \widetilde{C}, \mathbf{d^0}$ and their {relations} established in Sections \ref{sec:part:odd} and \ref{sec:part:H:odd} that $\mathcal{T}=0$.
More precisely, define for instance 
$$\tau^0 := \alpha^0 \oplus (0,0,\pi) \in \mathcal{S}_{Y_n \setminus Z} \times  \overline{\mathcal{S}_Z},$$
which satisfies the assumptions on $\tau$ in Theorem \ref{thm:part:func:Htrig:odd} (as can be checked by computing the weights listed at the beginning of Section \ref{sec:part:H:odd}). Then recall from the end of the proof of Theorem \ref{thm:part:func:Htrig:odd} and the fact that $(a^0_U,b^0_U,c^0_U)=(a^0_V,b^0_V,c^0_V)$ that
$$\mathcal{W}_n = \mathcal{W}(\tau^0):= 2 Q_n \Gamma(\tau^0)+C(\tau^0)+(0,\ldots,0,c^{\tau^0}_V,0)^{\!\top}
= 2 Q_n \mathbf{d^0} + (c_1^0, \ldots, c_p^0,2 c_U^0, c_W^0)^{\!\top},
$$
and thus $\mathcal{T}=0$.
The readers having skipped Section \ref{sec:part:H:odd} can instead use the identity $\widetilde{\mathcal{W}}(\alpha)= 2 \widetilde{Q}_n  \widetilde{\Gamma}(\alpha)+\widetilde{C}(\alpha)$ at the end of Section \ref{sec:part:odd} to arrive at the same conclusion.
\end{proof}

\subsection{Asymptotics of integrals on $\mathcal{Y}^0$} \label{sub:asym:Y0}

For the remainder of the section, let $r_0>0$ and $\gamma=\{ \mathbf{y}\in \mathcal{Y}^0 \ \vert \ \parallel \mathbf{y}-\mathbf{y^0} \parallel \ \leqslant r_0 \}$ a $p+2$-dimensional ball inside $\mathcal{Y}^0$ containing $\mathbf{y^0}$. We start with asymptotics of an integral on this compact contour $\gamma$.

\begin{proposition}\label{prop:compact:contour:S:SPM}
There exists a constant $\rho \in \C^*$ such that, as $\lambda \to \infty$,
$$
\int_{\gamma} d\mathbf{y} \ e^{\lambda S(\mathbf{y})}
= \rho \lambda^{-\frac{p+2}{2}} 
\exp\left (\lambda S(\mathbf{y^0})\right )
\left (
1 + 
o_{\lambda \to \infty}\left (1\right )
\right ).
$$
In particular, 
$$\dfrac{1}{\lambda} \log \left \vert
\int_{\gamma} d\mathbf{y} \ e^{\lambda S(\mathbf{y})}
\right \vert \underset{\lambda \to \infty}{\longrightarrow}  \Re S(\mathbf{y^0}) = - \mathrm{Vol}(S^3 \setminus K_n).
$$
   \end{proposition}
   
   \begin{proof}
We apply the saddle point method as in Theorem \ref{thm:SPM}, with $m=p+2$, $\gamma^m=\gamma$, $z=\mathbf{y}$, $z^0=\mathbf{y^0}$, $D=\mathcal{U}$, $f=1$ and $S$ as defined in the beginning of this section. Let us check the technical requirements:
\begin{itemize}
\item $\mathbf{y^0}$ is an interior point of $\gamma$ by construction.
\item $\max_\gamma \Re S$ is attained only at $\mathbf{y^0}$ by Lemma \ref{lem:maximum}.
\item $\nabla S (\mathbf{y^0})=0$ by Lemma \ref{lem:grad:thurston}.
\item $\det \mathrm{Hess}(S)(\mathbf{y^0}) \neq 0$ by Lemma \ref{lem:hess}.
\end{itemize}
Thus the first statement follows from Theorem \ref{thm:SPM}, with $\rho := \dfrac{(2\pi)^{\frac{p+2}{2}}}{\sqrt{\det \mathrm{Hess}(S)(\mathbf{y^0})}} \in \C^*$.

The second statement then follows from immediate computation and Lemma \ref{lem:-vol}. 
   \end{proof}

Now we compute an upper bound on the remainder term, i.e.\ the integral on $\mathcal{Y}^0 \setminus \gamma$ the whole unbounded contour minus the compact ball.

\begin{lemma}\label{lem:unbounded:contour}
There exists constants $A,B>0$ such that for all $\lambda > A$,
$$
\left \vert
\int_{\mathcal{Y}^0 \setminus \gamma} d\mathbf{y} \ e^{\lambda S(\mathbf{y})}
\right \vert
 \leqslant
B e^{\lambda M},
$$
where $M := \max_{\partial \gamma} \Re S$.
   \end{lemma}

\begin{proof}
First we apply a change of variables to $p+2$-dimensional spherical coordinates
$$\mathbf{y} \in \mathcal{Y}^0 \setminus \gamma \Longleftrightarrow
r \overrightarrow{\eta}
\in (r_0,\infty) \times \mathbb{S}^{p+1},$$
which yields:
$$
\int_{\mathcal{Y}^0 \setminus \gamma} d\mathbf{y} \ e^{\lambda S(\mathbf{y})}
= 
\int_{\mathbb{S}^{p+1}} d\,vol_{\mathbb{S}^{p+1}} \int_{r_0}^\infty r^{p+1} e^{\lambda S(r \overrightarrow{\eta})} dr
$$
for all $\lambda>0$.

Consequently, we have for all $\lambda>0$:
$$
\left \vert
\int_{\mathcal{Y}^0 \setminus \gamma} d\mathbf{y} \ e^{\lambda S(\mathbf{y})}
\right \vert
 \leqslant
 \mathrm{vol}(\mathbb{S}^{p+1})
 \sup_{\overrightarrow{\eta} \in \mathbb{S}^{p+1}} 
 \int_{r_0}^\infty r^{p+1} e^{\lambda \Re(S)(r \overrightarrow{\eta})} dr.$$

Let us fix $\overrightarrow{\eta} \in \mathbb{S}^{p+1}$ and denote $f=f_{\overrightarrow{\eta}}:= (r \mapsto \Re(S)(r \overrightarrow{\eta}))$ the restriction of $\Re(S)$ on the ray $(r_0,\infty)\overrightarrow{\eta}$. 
Let $\lambda>0$.
Let us find an upper bound on $  \int_{r_0}^\infty r^{p+1} e^{\lambda f(r)} dr$.

Since  $\Re(S)$ is strictly concave by Lemma \ref{lem:concave} and $f$ is its restriction on a convex set, $f$ is strictly concave as well on $(r_0,+\infty)$ (and even on $[0,+\infty)$). Now let us consider the slope function $N\colon [r_0,+\infty) \to \R$ defined by $N(r) := \dfrac{f(r)- f(r_0)}{r-r_0}$ for $r>r_0$ and $N(r_0):=f'(r_0)$. The function $N$ is $C^1$ and satisfies $N'(r) = \frac{f'(r)-N(r)}{r-r_0}$ for $r>r_0$. Now, since $f$ is strictly concave, we have $f'(r )< N(r)$ for any $r \in (r_0,\infty)$, thus $N$ is decreasing on this same interval.
Hence
$$
 \int_{r_0}^\infty r^{p+1} e^{\lambda f(r)} dr
 = e^{\lambda f(r_0)}
  \int_{r_0}^\infty r^{p+1} e^{\lambda N(r)(r-r_0)} dr
  \leqslant
  e^{\lambda f(r_0)}
  \int_{r_0}^\infty r^{p+1} e^{\lambda N(r_0)(r-r_0)} dr.
 $$
Note that $N(r_0)=f'(r_0)<0$ by Lemmas \ref{lem:concave} and \ref{lem:maximum}. Using integration by parts, we can prove by induction that
$$
\int_{r_0}^\infty r^{p+1} e^{\lambda N(r_0)(r-r_0)} dr = \frac{1}{(\lambda N(r_0))^{p+2}}
\sum_{k=0}^{p+1} (-1)^{p+1-k} \frac{(p+1)!}{k!}(\lambda N(r_0))^k r_0^k.
$$ 

Moreover, $N(r_0)=f'(r_0)= \langle (\nabla \Re(S))(r_0 \overrightarrow{\eta}) ; \overrightarrow{\eta} \rangle$, and since $S$ is holomorphic, we conclude that 
$( \overrightarrow{\eta} \mapsto N(r_0) = f_{\overrightarrow{\eta}}'(r_0))$ is a continous map from $\mathbb{S}^{p+1}$ to $\R_{<0}$. Hence there exist $m_1, m_2 >0$ such that $0 < m_1 \leqslant |N(r_0)| \leqslant m_2$ for all vectors $\overrightarrow{\eta} \in \mathbb{S}^{p+1}$.

We thus conclude that for all $\lambda>\frac{1}{m_1 r_0}$, we have the (somewhat unoptimal) upper bound:
\begin{align*}
\int_{r_0}^\infty r^{p+1} e^{\lambda f(r)} dr 
& \leqslant 
e^{\lambda f(r_0)}
\frac{1}{(\lambda N(r_0))^{p+2}}
\sum_{k=0}^{p+1} (-1)^{p+1-k} \frac{(p+1)!}{k!}(\lambda N(r_0))^k r_0^k \\
& \leqslant
e^{\lambda f(r_0)}
\left  \vert \frac{1}{(\lambda N(r_0))^{p+2}}
\sum_{k=0}^{p+1} (-1)^{p+1-k} \frac{(p+1)!}{k!}(\lambda N(r_0))^k r_0^k \right  \vert \\
& \leqslant
e^{\lambda f(r_0)}
 \frac{1}{\vert\lambda N(r_0)\vert^{p+2}}
\sum_{k=0}^{p+1}  (p+1)! \ \vert\lambda N(r_0) r_0\vert^k  \\
& \leqslant
e^{\lambda f(r_0)}
 \frac{(p+2)! \ \vert\lambda N(r_0) r_0\vert^{p+2}}{\vert\lambda N(r_0)\vert^{p+2}}  = (p+2)! \ r_0^{p+2} e^{\lambda f(r_0)}.  \\
\end{align*}

Now, since  
$\int_{r_0}^\infty r^{p+1} e^{\lambda f_{\overrightarrow{\eta}}(r)} dr  \leqslant C e^{\lambda f_{\overrightarrow{\eta}}(r_0)}$ for all $\lambda>\frac{1}{m_1 r_0}$, for all $\overrightarrow{\eta} \in \mathbb{S}^{p+1}$ and with the constant $C>0$ independent of $\lambda$ and $\overrightarrow{\eta}$, we can finally conclude that:
$$
\left \vert
\int_{\mathcal{Y}^0 \setminus \gamma} d\mathbf{y} \ e^{\lambda S(\mathbf{y})}
\right \vert
 \leqslant
 \mathrm{vol}(\mathbb{S}^{p+1})
 \sup_{\overrightarrow{\eta} \in \mathbb{S}^{p+1}} 
 \int_{r_0}^\infty r^{p+1} e^{\lambda \Re(S)(r \overrightarrow{\eta})} dr
\leqslant
C \mathrm{vol}(\mathbb{S}^{p+1}) e^{\lambda M} 
 $$
for all $\lambda>\frac{1}{m_1 r_0}$, where $M= \max_{\partial \gamma} \Re S$. This concludes the proof, by putting $A:= \frac{1}{m_1 r_0}$ and $B:=C  \mathrm{vol}(\mathbb{S}^{p+1})$.
\end{proof}

Finally we obtain the asymptotics for the integral on the whole contour $\mathcal{Y}^0$:

\begin{proposition}\label{prop:all:contour:S}
For the same constant $\rho \in \C^*$ as in Proposition \ref{prop:compact:contour:S:SPM}, we have, as $\lambda \to \infty$,
$$
\int_{\mathcal{Y}^0} d\mathbf{y} \ e^{\lambda S(\mathbf{y})}
= \rho \lambda^{-\frac{p+2}{2}} 
\exp\left (\lambda S(\mathbf{y^0})\right )
\left (
1 + 
o_{\lambda \to \infty}\left (1\right )
\right ).
$$
In particular, 
$$\dfrac{1}{\lambda} \log \left \vert
\int_{\mathcal{Y}^0} d\mathbf{y} \ e^{\lambda S(\mathbf{y})}
\right \vert \underset{\lambda \to \infty}{\longrightarrow}  \Re S(\mathbf{y^0}) = - \mathrm{Vol}(S^3 \setminus K_n).
$$
   \end{proposition}
   
   \begin{proof}
As for Proposition \ref{prop:compact:contour:S:SPM}, the second statement {immediately} follows from the first one. Let us prove the first statement.   
   
   From  Lemma \ref{lem:unbounded:contour}, for all $\lambda>A$, we have $\left \vert
\int_{\mathcal{Y}^0 \setminus \gamma} d\mathbf{y} \ e^{\lambda S(\mathbf{y})}
\right \vert
 \leqslant
B e^{\lambda M}$. Then, since $M < \Re(S)(\mathbf{y^0})$ by Lemmas \ref{lem:concave} and \ref{lem:maximum}, we have
$$
\int_{\mathcal{Y}^0 \setminus \gamma} d\mathbf{y} \ e^{\lambda S(\mathbf{y})} = o_{\lambda \to \infty}\left (
\lambda^{-\frac{p+2}{2}} 
\exp\left (\lambda S(\mathbf{y^0})\right )
\right )
.$$
 The first statement then follows from Proposition \ref{prop:compact:contour:S:SPM} and the equality 
 $$\int_{\mathcal{Y}^0} d\mathbf{y} \ e^{\lambda S(\mathbf{y})} = 
 \int_{\gamma} d\mathbf{y} \ e^{\lambda S(\mathbf{y})} + 
 \int_{\mathcal{Y}^0 \setminus \gamma} d\mathbf{y} \ e^{\lambda S(\mathbf{y})}.$$
   \end{proof}

\subsection{Extending the asymptotics to the quantum dilogarithm}\label{sub:asym:PhiB}

Let us now introduce some new {notation}:
\begin{itemize}
\item We let $R$ denote any positive number in $(0,\pi)$, for example $\pi/2$. Its exact value will not be relevant.
\item We denote $I_R^+ := (R,\infty)$, $I^-_R := (-\infty,-R)$, $\Lambda_R$ the closed upper half circle of radius $R$ in the complex plane, and $\Omega_R := I_R^- \cup \Lambda_R \cup I^+_R$. Remark that we can replace the contour $\R + i 0^+$ with $\Omega_R$ in the definition of $\Phi_\B$, by the Cauchy theorem.
\item For $\delta>0$, we define the product of closed ``horizontal  bands" in $\C$
$$\mathcal{U}_{\delta}:= \prod_{k=1}^p\left (\R + i [-\pi+\delta,-\delta] \right ) 
\times 
\prod_{l=U,W}
\left (\R + i [\delta,\pi-\delta]\right )$$ a closed subset of $\mathcal{U}$.
\item For $\B>0$, we define a new potential function 
$S_{\B}\colon \mathcal{U} \to \C$, an holomorphic  function on $p+2$ complex variables, by:
$$S_{\B}(\mathbf{y}) =
i \mathbf{y}^{\!\top} Q_n \mathbf{y} +  
  \mathbf{y}^{\!\top} \mathcal{W}_n
+ 2 \pi \B^2 \ \Log\left (
\dfrac{
\Phi_\B\left ( \frac{y_U}{2 \pi \B}  \right )^2
\Phi_\B\left ( \frac{y_W}{2 \pi \B} \right )
}{
\Phi_\B\left (\frac{y_1}{2 \pi \B}\right )
\cdots 
\Phi_\B\left (\frac{y_p}{2 \pi \B}\right )
},
\right ) 
   $$
   where $Q_n$ and $\mathcal{W}_n$ are like in Theorem \ref{thm:part:func}.
\end{itemize}

The following lemma establishes a ``parity property'' for the difference between classical and quantum dilogarithms on the horizontal band $\R + i (0,\pi)$.

\begin{lemma}\label{lem:parity}
For all $\B \in (0,1)$ and all $y \in \R + i (0,\pi)$, 
$$
\Re\left (
\Log\left ( \Phi_\B\left ( \frac{-\overline{y}}{2 \pi \B}  \right )\right )
-
\left (
\frac{-i}{2 \pi \B^2} \Li(-e^{-\overline{y}})
\right )
\right )
= 
\Re\left (
\Log\left ( \Phi_\B\left ( \frac{y}{2 \pi \B}  \right )\right )
-
\left (
\frac{-i}{2 \pi \B^2} \Li(-e^y)
\right )
\right )
.$$
\end{lemma}

\begin{proof}
Let $\B \in (0,1)$ and $y \in \R + i(0,\pi)$.

From the fact that $\Li$ is real-analytic and Proposition \ref{prop:dilog} (1) applied to $z=-e^y$, we have 
\begin{align*}
\overline{\frac{-i}{2 \pi \B^2} \Li(-e^{-\overline{y}})} 
&= \frac{i}{2 \pi \B^2} \Li(-e^{-y}) \\
&= \frac{i}{2 \pi \B^2}\left ( -\Li(-e^y)
- \frac{\pi^2}{6} - \frac{y^2}{2}
\right ) \\
&= 
\frac{-i}{2 \pi \B^2} \Li(-e^{y})
 -
\frac{i \pi}{12 \B^2}
-
\frac{i y^2}{4 \pi \B^2}.
\end{align*}

Moreover, from Proposition \ref{prop:quant:dilog} (1) and (2), we have 
$$
\overline{\Phi_\B\left ( \frac{-\overline{y}}{2 \pi \B}  \right )} = \dfrac{1}{\Phi_\B\left ( \frac{-y}{2 \pi \B}  \right )} =
\Phi_\B\left ( \frac{y}{2 \pi \B}  \right ) 
\exp \left ( -i\frac{\pi}{12}(\B^2 + \B^{-2})\right ) \exp\left (-i \pi \left (\frac{y}{2 \pi \B}\right )^2\right ).
$$

Therefore
$$
\overline{
\Log\left ( \Phi_\B\left ( \frac{-\overline{y}}{2 \pi \B}  \right )\right )
-
\left (
\frac{-i}{2 \pi \B^2} \Li(-e^{-\overline{y}})
\right )
}
= 
\Log\left ( \Phi_\B\left ( \frac{y}{2 \pi \B}  \right )\right )
-
\left (
\frac{-i}{2 \pi \B^2} \Li(-e^y)
\right )
- \frac{i \pi}{12} \B^2,$$
and the statement follows.
\end{proof}

As a consequence, we can bound uniformly the difference between classical and quantum dilogarithms on compact horizontal bands above the horizontal axis.

\begin{lemma}\label{lem:unif:bound}
For all $\delta>0$, there exists a constant $B_{\delta}>0$ such that for all $\B \in (0,1)$ and all $y \in \R + i [\delta,\pi-\delta]$, 
$$
\left \vert
\Re\left (
\Log\left ( \Phi_\B\left ( \frac{y}{2 \pi \B}  \right )\right )
-
\left (
\frac{-i}{2 \pi \B^2} \Li(-e^y)
\right )
\right )
\right \vert \leqslant B_{\delta} \B^2
.$$
Moreover, $B_{\delta}$ {can be chosen} of the form $B_\delta = C/\delta + C'$ with $C,C'>0$.
\end{lemma}

The proof of Lemma \ref{lem:unif:bound} is quite lengthy, but contains relatively classical calculus arguments. The key points are the fact that $\Im(y)$ is uniformly  upper bounded by a quantity \textit{strictly smaller} than $\pi$, and that we can restrict ourselves to $y \in (-\infty,0] + i [\delta,\pi-\delta]$ (thanks to Lemma \ref{lem:parity}) which implies that $\Re(y)$ is uniformly upper bounded by $0$. The necessity of this last remark stems from the fact that the state variable $y$ must be integrated on an contour with \textit{unbounded real part} in the definition of the Teichm\"uller TQFT, whereas the contour is usually bounded when studying the volume conjecture for the colored Jones polynomials. Compare with \cite[Lemma 3]{AH}. The parity trick of Lemma \ref{lem:parity} and its application to an unbounded contour are the main technical novelties compared with the methods of \cite{AH}.

\begin{proof}
Let $\delta >0$.
In the following proof, $y = x+ i d$ will denote a generic element in $(-\infty,0] + i [\delta,\pi-\delta]$, with $x \in (-\infty,0], d \in [\delta,\pi-\delta]$.
We  remark that we only need to prove the statement for $y \in (-\infty,0] + i [\delta,\pi-\delta]$, thanks to Lemma \ref{lem:parity}.

We first compute, for any $\B \in (0,1)$ and  $y \in \R + i [\delta,\pi-\delta]$:
\begin{align*}
\Log \ \Phi_\B \left ( \frac{y}{2 \pi \B}  \right ) 
&= \int_{w \in \Omega_{R \B}}
\dfrac{\exp\left (-i \frac{y w}{\pi \B}\right ) dw}{4 w\sinh(\B w) \sinh({\B}^{-1}w)} \\
&= \int_{v \in \Omega_{R}}
\dfrac{\exp\left (-i \frac{y v}{\pi}\right ) dv}{4 v\sinh(\B^2 v) \sinh(v)} \\
&= \dfrac{1}{\B^2} \int_{v \in \Omega_{R}}
\dfrac{\exp\left (-i \frac{y v}{\pi}\right )}{4 v^2 \sinh(v)}  \dfrac{(v \B^2)}{\sinh(v \B^2) } dv,
\end{align*}
where the first equality comes from the definition of $\Phi_\B$ (choosing the integration contour $\Omega_{R \B}$), the second one comes from the change of variables $ v=\frac{w}{\B}$ and the last one is a simple re-writing.

Next, we remark that there exists a constant $\sigma_R > 0$ such that $\vert (\frac{v}{\sinh(v)})'' \vert \leq \sigma_R$ for all $v \in \mathbb{R} \cup D_R$, where $D_R$ is the upper half disk of radius $R$.
Indeed, note first that $\sinh$ is nonzero everywhere on $\R \cup D_R$. Then a quick computation yields $\left (\frac{v}{\sinh(v)}\right )'' = \dfrac{v(1+\cosh(v)^2)-2 \sinh(v)\cosh(v)}{\sinh(v)^3}$, which is well-defined and continous on $\R \cup D$, has a limit of $-1/3$ at $v=0$ and has a zero limit in $v\in \R, v \to \pm \infty$. The boundedness on $\R \cup D_R$ follows.

Now, it follows from Taylor's theorem that for every $\B \in (0,1)$ and every $v \in \Omega_R$,
$$ \dfrac{(v \B^2)}{\sinh(v \B^2)} = 1 + (v \B^2)^2 \epsilon(v\B^2),$$
where $\epsilon(v\B^2) := \int_0^1 (1-t)\left (\frac{z}{\sinh(z)}\right )''(v \B^2 t) \ dt$. It then follows from the previous paragraph that $\vert \epsilon(v\B^2) \vert \leqslant \sigma_R$ for every $\B \in (0,1)$ and every $v \in \Omega_R$.

Recall from Proposition \ref{prop:dilog} (2) that for all $\B \in (0,1)$ and all $y \in \R + i [\delta,\pi-\delta]$,
$$\dfrac{1}{\B^2} \int_{v \in \Omega_{R}}
\dfrac{\exp\left (-i \frac{y v}{\pi}\right )}{4 v^2 \sinh(v)}   dv
= \frac{-i}{2 \pi \B^2} \Li(-e^y).
$$

Therefore we can write for all $\B \in (0,1)$ and all $y \in \R + i [\delta,\pi-\delta]$:
\begin{align*}
\Log\left ( \Phi_\B\left ( \frac{y}{2 \pi \B}  \right )\right )
-
\left (
\frac{-i}{2 \pi \B^2} \Li(-e^y)
\right ) 
&= \dfrac{1}{\B^2} \int_{v \in \Omega_{R}}
\dfrac{\exp\left (-i \frac{y v}{\pi}\right )}{4 v^2 \sinh(v)}  \left (\dfrac{(v \B^2)}{\sinh(v \B^2) }-1\right ) dv \\
&= \dfrac{1}{\B^2} \int_{v \in \Omega_{R}}
\dfrac{\exp\left (-i \frac{y v}{\pi}\right )}{4 v^2 \sinh(v)}  (v \B^2)^2 \epsilon(v\B^2) dv \\
&= \B^2 \int_{v \in \Omega_{R}}
\epsilon(v\B^2) \dfrac{\exp\left (-i \frac{y v}{\pi}\right )}{4  \sinh(v)}   dv.
\end{align*}

Now it suffices to prove that the quantity $$\Re\left ( \int_{v \in \Omega_{R}}
\epsilon(v\B^2) \dfrac{\exp\left (-i \frac{y v}{\pi}\right )}{4  \sinh(v)}   dv\right )$$ is uniformly bounded on $y \in (-\infty,0] + i [\delta,\pi-\delta], \B \in (0,1)$. We will split this integral into three parts and prove that each part is uniformly bounded in this way.

Firstly, on the contour $I^+_R$, we have for all $\B \in (0,1)$ and all $y \in \R + i [\delta,\pi-\delta]$:
\begin{align*}
\left \vert 
\Re\left ( 
\int_{v \in I^+_{R}}
\epsilon(v\B^2) 
\dfrac{\exp\left (-i \frac{y v}{\pi}\right )}
{4  \sinh(v)}   dv
\right )
\right \vert 
&\leqslant
\left \vert  
\int_{v \in I^+_{R}}
\epsilon(v\B^2) 
\dfrac{\exp\left (-i \frac{y v}{\pi}\right )}
{4  \sinh(v)}   dv 
\right \vert \\
&\leqslant
\int_{R}^\infty
\vert \epsilon(v\B^2) \vert 
\dfrac{\left \vert 
\exp\left (-i \frac{y v}{\pi}\right )
\right \vert}{4  \sinh(v)}   dv \\
&\leqslant \frac{\sigma_R}{4}
\int_{R}^\infty
 \dfrac{\exp\left (
 \frac{\Im(y) v}{\pi}
 \right )}{\sinh(v)}   dv \\
 &\leqslant \frac{\sigma_R}{4}
\int_{R}^\infty
 \dfrac{\exp\left (
 \frac{(\pi-\delta) v}{\pi}
 \right )}{ \frac{1-e^{-2R}}{2} e^v }   dv \\
 &= \dfrac{\pi \sigma_R e^{-\frac{\delta R}{\pi}}}{2 \delta (1-e^{-2R})},
\end{align*}
where in the last inequality we used the fact that 
$\frac{1-e^{-2R}}{2} e^v \leqslant \sinh(v)$ for all $v\geqslant R$.

Secondly, on the contour $I^-_R$, we have similarly for all $\B \in (0,1)$ and all $y \in \R + i [\delta,\pi-\delta]$:
\begin{align*}
\left \vert 
\Re\left ( 
\int_{v \in I^-_{R}}
\epsilon(v\B^2) 
\dfrac{\exp\left (-i \frac{y v}{\pi}\right )}
{4  \sinh(v)}   dv
\right )
\right \vert 
&\leqslant
\left \vert  
\int_{v \in I^-_{R}}
\epsilon(v\B^2) 
\dfrac{\exp\left (-i \frac{y v}{\pi}\right )}
{4  \sinh(v)}   dv 
\right \vert \\
&\leqslant
\int_{-\infty}^{-R}
\vert \epsilon(v\B^2) \vert 
\dfrac{\left \vert 
\exp\left (-i \frac{y v}{\pi}\right )
\right \vert}{4  \vert \sinh(v) \vert}   dv \\
&=
\int_{R}^{\infty}
\vert \epsilon(-v\B^2) \vert 
\dfrac{\left \vert 
\exp\left (i \frac{y v}{\pi}\right )
\right \vert}{4  \sinh(v) }   dv \\
&\leqslant \frac{\sigma_R}{4}
\int_{R}^\infty
 \dfrac{\exp\left (
 \frac{- \Im(y) v}{\pi}
 \right )}{\sinh(v)}   dv \\
 &\leqslant \frac{\sigma_R}{4}
\int_{R}^\infty
 \dfrac{1}{ \frac{1-e^{-2R}}{2} e^v }   dv \\
 &= \dfrac{\sigma_R e^{-R}}{2 (1-e^{-2R})} = \dfrac{\sigma_R}{4 \sinh(R)}.
\end{align*}

Finally, to obtain the bound on the contour $\Lambda_R$, we will need the assumption that $y \in (-\infty,0] + i [\delta,\pi-\delta]$, since the upper bound will depend on $\Re(y)$. Moreover, we will use the fact that since $\vert \sinh \vert$ is a continous nonzero function on the contour $\Lambda_R$, it is lower bounded by a constant $s_R>0$ on this countour.
We then obtain, for all $\B \in (0,1)$ and all $y \in (-\infty,0] + i [\delta,\pi-\delta]$:
\begin{align*}
\left \vert 
\Re\left ( 
\int_{v \in \Lambda_R}
\epsilon(v\B^2) 
\dfrac{\exp\left (-i \frac{y v}{\pi}\right )}
{4  \sinh(v)}   dv
\right )
\right \vert 
&\leqslant
\left \vert  
\int_{v \in \Lambda_R}
\epsilon(v\B^2) 
\dfrac{\exp\left (-i \frac{y v}{\pi}\right )}
{4  \sinh(v)}   dv 
\right \vert \\
&\leqslant
\int_{v \in \Lambda_R}
\vert \epsilon(v\B^2) \vert 
\dfrac{\left \vert 
\exp\left (-i \frac{y v}{\pi}\right )
\right \vert}{4  \vert \sinh(v) \vert}   dv \\
&\leqslant \frac{\sigma_R}{4 s_R}
\int_{v \in \Lambda_R}
\exp\left (
\Re\left (
-i \frac{y v}{\pi}
\right )
 \right )   dv \\
 &= \frac{\sigma_R}{4 s_R}
\int_{v \in \Lambda_R}
\exp\left (
\frac{\Re(y) \Im(v) + \Im(y) \Re(v)}{\pi}
 \right )   dv \\
 &\leqslant \frac{\sigma_R}{4 s_R} (\pi R) \exp\left (
\dfrac{0 + (\pi-\delta) R}{\pi} 
 \right )
 \leqslant 
 \dfrac{\sigma_R \pi R e^R}{4 s_R},
\end{align*}
where the fourth inequality is due to the fact that $\Re(y) \leqslant 0$, $\Im(v) \geqslant 0$, $ 0 < \Im(y) \leqslant \pi - \delta$ and $ \Re(v) \leqslant R$.

The lemma follows, by taking for example the constant
$$ B_\delta:= 
\dfrac{\pi \sigma_R e^{-\frac{\delta R}{\pi}}}{2 \delta (1-e^{-2R})} + \dfrac{\sigma_R}{4 \sinh(R)} + \dfrac{\sigma_R \pi R e^R}{4 s_R}
.$$
\end{proof}

The following lemma is simply a variant of Lemma \ref{lem:unif:bound} for compact horizontal bands with negative imaginary part.

\begin{lemma}\label{lem:unif:bound:neg}
For all $\delta>0$, there exists a constant $B_{\delta}>0$ (the same as in Lemma \ref{lem:unif:bound}) such that for all $\B \in (0,1)$ and all $y \in \R - i [\delta,\pi-\delta]$, 
$$
\left \vert
\Re\left (
\Log\left ( \Phi_\B\left ( \frac{y}{2 \pi \B}  \right )\right )
-
\left (
\frac{-i}{2 \pi \B^2} \Li(-e^y)
\right )
\right )
\right \vert \leqslant B_{\delta} \B^2
.$$
\end{lemma}

\begin{proof}
The result follows immediately from the fact that $\Li (\overline{\cdot}) = \overline{\Li(\cdot)}$, Proposition \ref{prop:quant:dilog} (2) and Lemma \ref{lem:unif:bound}.
\end{proof}

The following Proposition \ref{prop:all:contour:Sb} will not actually be used in the proof of Theorem \ref{thm:vol:conj}, but  fits naturally in the current discussion.

\begin{proposition}\label{prop:all:contour:Sb}
For some constant $\rho' \in \C^*$, we have, as $\B \to 0^+$,
\begin{align*}
\int_{\mathcal{Y}^0} 
d\mathbf{y} e^{\frac{1}{2 \pi \B^2}S_\B(\mathbf{y})}
&=
\int_{\mathcal{Y}^0} 
d\mathbf{y} \
e^{\frac
{
i \mathbf{y}^{\!\top} Q_n \mathbf{y}  + 
  \mathbf{y}^{\!\top} \mathcal{W}_n
}
 {2 \pi \B^2} 
}
\dfrac{
\Phi_\B\left ( \frac{y_U}{2 \pi \B}  \right )^2
\Phi_\B\left ( \frac{y_W}{2 \pi \B} \right )
}{
\Phi_\B\left (\frac{y_1}{2 \pi \B}\right )
\cdots 
\Phi_\B\left (\frac{y_p}{2 \pi \B}\right )
} \\
&=  
e^{\frac{1}{2 \pi \B^2} S(\mathbf{y^0})}
\left (
\rho' \B^{p+2} \left (
1 + 
o_{\B \to 0^+}\left (1\right )
\right )
+ O_{\B \to 0^+}(1)
 \right ).
\end{align*}
In particular, 
$$2 \pi \B^2 \log \left \vert
\int_{\mathcal{Y}^0} d\mathbf{y} \ e^{\frac{1}{2 \pi \B^2}S_\B(\mathbf{y})}
\right \vert \underset{\B \to 0^+}{\longrightarrow}  \Re S(\mathbf{y^0}) = - \mathrm{Vol}(S^3 \setminus K_n).
$$
\end{proposition}

\begin{proof}
The second statement follows from the first one from the fact that the behaviour of
$$\left (
\rho' \B^{p+2} \left (
1 + 
o_{\B \to 0^+}\left (1\right )
\right )
+ O_{\B \to 0^+}(1)
 \right )$$
 is polynomial in $\B$ as $\B \to 0^+$.
 
 To prove the first statement, we will split the integral on $\mathcal{Y}^0$ into two parts, one on the compact contour $\gamma$ from before and the other on the unbounded contour $\mathcal{Y}^0\setminus \gamma$.
 
First we notice that there exists a $\delta>0$ such that for all $\mathbf{y}=(y_1, \ldots,y_p,y_U,y_W)$ in $\mathcal{Y}^0$, $\Im(y_1), \ldots \Im(y_p) \in [-(\pi-\delta),-\delta]$ and $\Im(y_U),\Im(y_W) \in [\delta,\pi-\delta]$. From Lemmas \ref{lem:unif:bound} and  \ref{lem:unif:bound:neg}, if we denote $(\zeta_1, \ldots, \zeta_p,\zeta_U,\zeta_W) := (-1, \ldots,-1,2,1)$, it then follows that:
\begin{align*}
\left \vert
\Re\left (
\frac{1}{2 \pi \B^2}S_\B(\mathbf{y}) - \frac{1}{2 \pi \B^2}S(\mathbf{y})
\right )
\right \vert 
&=
\left \vert
\Re\left (
\sum_{j=1}^W
{\zeta}_j \left (\Log\left ( \Phi_\B\left ( \frac{y_j}{2 \pi \B}  \right )\right )
-
\left (
\frac{-i}{2 \pi \B^2} \Li(-e^{y_j})
\right )
\right )
\right )
\right \vert \\
&\leqslant
\sum_{j=1}^W \vert {\zeta}_j \vert
\left \vert
\Re\left (
 \left (\Log\left ( \Phi_\B\left ( \frac{y_j}{2 \pi \B}  \right )\right )
-
\left (
\frac{-i}{2 \pi \B^2} \Li(-e^{y_j})
\right )
\right )
\right )
\right \vert \\
& \leqslant (p+3)B_{\delta} \B^2.
\end{align*}

Let us now focus on the compact contour $\gamma$ and prove that
$$
\int_{\gamma} 
d\mathbf{y} \ e^{\frac{1}{2 \pi \B^2}S_\B(\mathbf{y})}
=  
e^{\frac{1}{2 \pi \B^2} S(\mathbf{y^0})}
\left (
\rho' \B^{p+2} \left (
1 + 
o_{\B \to 0^+}\left (1\right )
\right )
+ {O}_{\B \to 0^+}(1)
 \right ).$$
 From Proposition \ref{prop:compact:contour:S:SPM}, by identifying $\lambda = \frac{1}{2\pi \B^2}$ and $\rho' := \rho (2\pi)^{\frac{p+2}{2}}$ it suffices to prove that
 $$
 \int_{\gamma} 
d\mathbf{y} \ e^{\frac{1}{2 \pi \B^2}S(\mathbf{y})}
\left (
e^{\frac{1}{2 \pi \B^2}(S_\B(\mathbf{y})-S(\mathbf{y}))} -1
\right )
=  
e^{\frac{1}{2 \pi \B^2} S(\mathbf{y^0})}
{O}_{\B \to 0^+}(1)
.$$
This last equality follows from the upper bound $(p+3)B_{\delta} \B^2$ of the previous paragraph, the compactness of $\gamma$, and Lemma \ref{lem:maximum}.

Finally, let us prove that on the unbounded contour, we have
$$
\int_{\mathcal{Y}^0 \setminus \gamma} 
d\mathbf{y} \ e^{\frac{1}{2 \pi \B^2}S_\B(\mathbf{y})}
=  
e^{\frac{1}{2 \pi \B^2} S(\mathbf{y^0})}
 {O}_{\B \to 0^+}(1).$$
Let $A,B$ be the constants from Lemma \ref{lem:unbounded:contour}.
From the proof of Lemma \ref{lem:unbounded:contour}, we have that for all $\B < (2 \pi A)^{-1/2}$:
$$ \int_{\mathcal{Y}^0\setminus \gamma} d\mathbf{y} \ e^{\frac{1}{2 \pi \B^2} \Re(S) (\mathbf{y})} \leqslant B e^{\frac{1}{2 \pi \B^2} M}.$$
Moreover, for all $\B \in (0,1)$ and $\mathbf{y} \in \mathcal{Y}^0 \setminus \gamma$, we have
$e^{\frac{1}{2 \pi \B^2}\Re\left (S_\B(\mathbf{y})- S(\mathbf{y})\right )} \leqslant e^{(p+3)B_{\delta} \B^2}.$\\
 Let us denote $\upsilon := \frac{\Re(S)(\mathbf{y^0})- M}{2}$.
Thus, for all  $b>0$ smaller than both $(2 \pi A)^{-1/2}$ and $\left ( \dfrac{\upsilon}{2 \pi (p+3) B_{\delta}}
\right )^{1/4}$, we have:
\begin{align*}
\left  \vert
\int_{\mathcal{Y}^0 \setminus \gamma} 
d\mathbf{y} \ e^{\frac{1}{2 \pi \B^2}S_\B(\mathbf{y})}
\right \vert
&= 
\left  \vert
\int_{\mathcal{Y}^0 \setminus \gamma} 
d\mathbf{y} \ e^{\frac{1}{2 \pi \B^2}S(\mathbf{y})}
e^{\frac{1}{2 \pi \B^2}(S_\B(\mathbf{y})-S(\mathbf{y}))}
\right \vert \\
& \leqslant
\int_{\mathcal{Y}^0 \setminus \gamma} 
d\mathbf{y} \ e^{\frac{1}{2 \pi \B^2}\Re(S)(\mathbf{y})}
e^{\frac{1}{2 \pi \B^2}\Re(S_\B(\mathbf{y})-S(\mathbf{y}))} \\
& \leqslant
 B e^{\frac{1}{2 \pi \B^2} M} e^{(p+3)B_{\delta} \B^2} \leqslant B e^{\frac{1}{2 \pi \B^2} (M+\upsilon)} \\
 &=  
e^{\frac{1}{2 \pi \B^2} S(\mathbf{y^0})}
 {O}_{\B \to 0^+}(1),
\end{align*}
which concludes the proof.
\end{proof}

\subsection{Going from $\B$ to $\hbar$}\label{sub:asym:hbar}

Recall that for every $\B>0$, we associate a corresponding parameter $\hbar := \B^2 (1+\B^2)^{-2} >0$.

For $\B>0$, we define a new potential function 
$S'_{\B}\colon \mathcal{U} \to \C$, a holomorphic  function on $p+2$ complex variables, by:
$$S'_{\B}(\mathbf{y}) =
i \mathbf{y}^{\!\top} Q_n \mathbf{y} +  
  \mathbf{y}^{\!\top} \mathcal{W}_n
+ 2 \pi \hbar \ \Log\left (
\dfrac{
\Phi_\B\left ( \frac{y_U}{2 \pi \sqrt{\hbar}}  \right )^2
\Phi_\B\left ( \frac{y_W}{2 \pi \sqrt{\hbar}} \right )
}{
\Phi_\B\left (\frac{y_1}{2 \pi \sqrt{\hbar}}\right )
\cdots 
\Phi_\B\left (\frac{y_p}{2 \pi \sqrt{\hbar}}\right )
},
\right ) 
   $$
   where $Q_n$ and $\mathcal{W}_n$ are like in Theorem \ref{thm:part:func}.
   
\begin{remark}\label{rem:J':S'b}  
   Notice that 
   $$\vert \mathfrak{J}_{X_n}(\hbar,0) \vert =\left \vert \left (\dfrac{1}{2\pi \sqrt{\hbar}}\right )^{p+3} \int_{\mathcal{Y}^0} d \mathbf{y} \ e^{\frac{1}{2\pi \hbar} S'_{\B}(\mathbf{y})} \right \vert.$$
   Indeed, this follows from taking $\tau=\tau^0$ in Theorem \ref{thm:part:func:Htrig:odd}, where $\tau^0$ is defined at the end of the proof of Lemma \ref{lem:-vol}.
   \end{remark}

The following Lemma \ref{lem:unif:bound:hbar} will play a similar role as Lemmas \ref{lem:unif:bound} and \ref{lem:unif:bound:neg}, but its proof is fortunately shorter.

\begin{lemma}\label{lem:unif:bound:hbar}
For all $\delta \in (0,\frac{\pi}{2})$, there exists constants $c_{\delta}, C_{\delta}>0$ such that for all  $\B \in (0,c_{\delta})$ and all $y \in \R+i \left ([-(\pi-\delta),-\delta]\cup[\delta,\pi-\delta]\right )$,  we have:
$$
\left \vert
\Re\left (
\left (
\frac{-i}{2 \pi \B^2} \Li\left (-e^{y(1+\B^2)}\right )
\right )
-
\left (
\frac{-i}{2 \pi \B^2} (1+\B^2)^2 \Li(-e^y)
\right )
\right )
\right \vert \leqslant C_{\delta}
.$$
\end{lemma}

\begin{proof}
Let $\delta \in (0,\frac{\pi}{2}) $. Let us define $c_{\delta} := \sqrt{\dfrac{\delta}{2(\pi-\delta)}}$, so that $(\pi-\delta)(1+c_\delta^2) = \pi-\delta/2$.

We  consider the function
$$(x,d,u,\B) \mapsto \left \vert \Log \left ( 1 + e^{(x+id)(1+u\B^2)}\right )\right  \vert,$$
which is continous and well-defined on $[-1,0]\times [\delta,\pi-\delta]\times[0,1]\times[0,c_{\delta}]$; indeed,  since 
$$d(1+u \B^2) \leqslant (\pi-\delta)(1+c_\delta^2) = \pi-\delta/2 < \pi,$$ the exponential will then never be $-1$. Let us denote $L_{\delta}>0$ the maximum of this function.

Let us define
$$\Delta(\B,y):= \Im\left (
\Li\left (-e^{y(1+\B^2)}\right )
-
 (1+\B^2)^2 \Li(-e^y)
\right )$$
for all $\B \in (0,1)$ and all $y \in \R+i \left ([-(\pi-\delta),-\delta]\cup[\delta,\pi-\delta]\right )$.

We first remark a parity property like in Lemma \ref{lem:parity}. Indeed, it similarly follows from Proposition \ref{prop:dilog} (1) that
$\Delta(\B,y) = -\Delta(\B,-y) = -\Delta(\B,\overline{y}) = \Delta(\B,-\overline{y})$
for all $\B \in (0,1)$ and all $y \in \R+i \left ([-(\pi-\delta),-\delta]\cup[\delta,\pi-\delta]\right )$.
Thus we can consider that $y \in \R_{\leqslant 0}+i [\delta,\pi-\delta]$ in the remainder of the proof.

It then follows from Taylor's theorem that for all $\B \in (0,1)$ and all $y \in \R_{\leqslant 0}+i [\delta,\pi-\delta]$,
\begin{align*}
\Delta(\B,y) &= \Im\left (
-\left (\int_{0}^1 
\Log \left ( 1 + e^{y(1+u\B^2)}\right )
 (-y\B^2) du\right )
- (2\B^2 + \B^4) \Li(-e^y)
\right ) \\
&= - \B^2 \Im\left (
y \left (\int_{0}^1 
\Log \left ( 1 + e^{y(1+u\B^2)}\right )
 du\right )
+ (2 + \B^2) \Li(-e^y)
\right ).
\end{align*}

We will bound $\left \vert \dfrac{\Delta(\B,y)}{-\B^2} \right \vert$ separately for $\Re(y) \in [-1,0]$ and then for $\Re(y) \in (-\infty,-1)$.

Firstly, we have for all $y \in [-1,0] + i [\delta,\pi-\delta]$ and all $\B \in (0,c_{\delta})$:
\begin{align*}
\left \vert \dfrac{\Delta(\B,y)}{-\B^2} \right \vert 
& \leqslant \vert y \vert 
\left (\int_{0}^1 
\left \vert \Log \left ( 1 + e^{y(1+u\B^2)}\right )
\right \vert
 du\right )
 +  (2 + \B^2) \vert \Li(-e^y) \vert \\
 & \leqslant \sqrt{1 + (\pi-\delta)^2} L_\delta + 3 L'_\delta,
\end{align*}
where $L'_\delta$ is the maximum of $(x,d) \mapsto \vert \Li(-e^y) \vert$ on $(-\infty,0]\times[\delta,\pi-\delta]$.

Secondly, let $y = x+id \in (-\infty,-1] + i [\delta,\pi-\delta]$ and $\B \in (0,c_{\delta})$. For all $u \in [0,1]$, we have $\left \vert e^{y(1+u\B^2)} \right \vert <1$, therefore (from the triangle inequality on the Taylor expansion):
$$ \left \vert
 \Log \left ( 
 1 + e^{y(1+u\B^2)}
 \right )
\right \vert 
\leqslant
- \log\left (
1-
\left \vert e^{y(1+u\B^2)}
\right \vert 
\right ) 
= \log\left ( 1 + \dfrac{e^{x(1+u\B^2)}}{1-e^{x(1+u\B^2)}}
\right ) \leqslant
\dfrac{e^{2 x}}{1-e^{2 x}},
$$
hence 
\begin{align*}
\left \vert \dfrac{\Delta(\B,y)}{-\B^2} \right \vert 
& \leqslant \vert y \vert 
\left (\int_{0}^1 
\left \vert \Log \left ( 1 + e^{y(1+u\B^2)}\right )
\right \vert
 du\right )
 +  (2 + \B^2) \vert \Li(-e^y) \vert \\
 & \leqslant \sqrt{x^2 + (\pi-\delta)^2} \dfrac{e^{2 x}}{1-e^{2 x}} + 3 L'_\delta \\
 & \leqslant E_\delta + 3 L'_\delta,
\end{align*}
where $E_\delta$ is the maximum of the function $ x \in (-\infty,-1] \mapsto \sqrt{x^2 + (\pi-\delta)^2} \dfrac{e^{2 x}}{1-e^{2 x}}$.

We now conclude the proof by defining $C_{\delta} := \frac{1}{2\pi} \max\{\sqrt{1 + (\pi-\delta)^2} L_\delta + 3 L'_\delta, E_\delta + 3 L'_\delta\}$.
\end{proof}

We can now state and prove the final piece of the proof of Theorem \ref{thm:vol:conj}.

\begin{proposition}\label{prop:all:contour:S'b}
For the constant $\rho' \in \C^*$ defined in Proposition \ref{prop:all:contour:Sb}, we have, as $\hbar \to 0^+$,
\begin{align*}
\int_{\mathcal{Y}^0} 
d\mathbf{y} e^{\frac{1}{2 \pi \hbar}S'_\B(\mathbf{y})}
&=
\int_{\mathcal{Y}^0} 
d\mathbf{y} \
e^{\frac
{
i \mathbf{y}^{\!\top} Q_n\mathbf{y} + 
  \mathbf{y}^{\!\top} \mathcal{W}_n
}
 {2 \pi \hbar} 
}
\dfrac{
\Phi_\B\left ( \frac{y_U}{2 \pi \sqrt{\hbar}}  \right )^2
\Phi_\B\left ( \frac{y_W}{2 \pi \sqrt{\hbar}} \right )
}{
\Phi_\B\left (\frac{y_1}{2 \pi \sqrt{\hbar}}\right )
\cdots 
\Phi_\B\left (\frac{y_p}{2 \pi \sqrt{\hbar}}\right )
} \\
&=  
e^{\frac{1}{2 \pi \hbar} S(\mathbf{y^0})}
\left (
\rho' \hbar^{\frac{p+2}{2}} \left (
1 + 
o_{\hbar \to 0^+}\left (1\right )
\right )
+ {O}_{\hbar \to 0^+}(1)
 \right ).
\end{align*}
In particular, 
$$(2 \pi \hbar) \log \left \vert
\int_{\mathcal{Y}^0} d\mathbf{y} \ e^{\frac{1}{2 \pi \hbar}S'_\B(\mathbf{y})}
\right \vert \underset{\hbar \to 0^+}{\longrightarrow}  \Re S(\mathbf{y^0}) = - \mathrm{Vol}(S^3 \setminus K_n).
$$
\end{proposition}

\begin{proof}
The proof will be similar to the one of Proposition \ref{prop:all:contour:Sb} (notably, the second statement follows from the first one in the exact same way), but will need also Lemma \ref{lem:unif:bound:hbar} to bound an extra term. 
Let us prove the first statement.

Let $\delta>0$ such that the absolute value of the imaginary parts of the coordinates of any $\mathbf{y} \in \mathcal{Y}^0$ {lying} in $[\delta,\pi-\delta]$. Let us again denote $({\zeta}_1, \ldots, {\zeta}_p,{\zeta}_U,{\zeta}_W) := (-1,\ldots,-1,2,1)$. Then for all $\mathbf{y} \in \mathcal{Y}^0$ and all $\B \in (0,c_{\delta})$, it follows from Lemmas \ref{lem:unif:bound}, \ref{lem:unif:bound:neg} and \ref{lem:unif:bound:hbar} that
\begin{align*}
&\left \vert
\Re\left (
\frac{1}{2 \pi \hbar}S'_\B(\mathbf{y}) 
- \frac{1}{2 \pi \hbar}S(\mathbf{y})
\right )
\right \vert 
=
\left \vert
\Re\left (
\sum_{j=1}^W
{\zeta}_j \left (\Log\left ( \Phi_\B\left ( \frac{y_j}{2 \pi \sqrt{\hbar}}  \right )\right )
-
\left (
\frac{-i}{2 \pi \hbar} \Li(-e^{y_j})
\right )
\right )
\right )
\right \vert \\
	& \hspace*{1.5cm} \leqslant
\sum_{j=1}^W \vert {\zeta}_j \vert
\left \vert
\Re\left (
 \left (\Log\left ( \Phi_\B\left ( \frac{y_j(1+\B^2)}{2 \pi \B}  \right )\right )
-
\left (
\frac{-i}{2 \pi \B^2} \Li(-e^{y_j(1+\B^2)})
\right )
\right )
\right )
\right \vert \\
& \hspace*{1.5cm} \ \	+\sum_{j=1}^W \vert {\zeta}_j \vert
\left \vert
	\Re\left (
		 \left (
\frac{-i}{2 \pi \B^2} \Li(-e^{y_j(1+\B^2)})
		\right )
-
		\left (
\frac{-i}{2 \pi \B^2}(1+\B^2)^2 \Li(-e^{y_j})
		\right )
	\right )
\right \vert 
\\
& \hspace*{1.5cm} \leqslant (p+3)\left (B_{\frac{\delta}{2}} \B^2 + C_{\delta}\right )
\leqslant (p+3)\left (B_{\frac{\delta}{2}}  + C_{\delta}\right ).
\end{align*}

The remainder of the proof is now the same as for Proposition \ref{prop:all:contour:Sb}, by identifying $\lambda= \frac{1}{2 \pi \hbar}$ and taking
$\hbar$ small enough so that the associated
 $\B$ satisfies 
 $$0 < \B < \min\left \{c_\delta, (2\pi A)^{-1/2}, \left (
\dfrac{\upsilon}{2 \pi (p+3) (B_{\delta/2}+C_{\delta})}
\right )^{1/2}\right \}.$$
\end{proof}

\subsection{Conclusion and comments}\label{sub:conjvol:conclusion}

\begin{proof}[Proof of Theorem \ref{thm:vol:conj}]
The second equality follows from Remark \ref{rem:J':S'b} and Proposition \ref{prop:all:contour:S'b}, and the first equality follows from the identity
$$J_{X_n}(\hbar,x) = 2 \pi \sqrt{\hbar} \ \mathfrak{J}_{X_n}(\hbar, (2\pi \sqrt{\hbar})x).$$
\end{proof}

Some comments are in order. 
\begin{itemize}
\item The various upper bounds we constructed were far from optimal, since we were mostly interested {in proving} that the \textit{exponential decrease rate} yielded the hyperbolic volume. Anyone interested in computing a more detailed asymptotic expansion of $\mathfrak{J}_{X_n}(\hbar,0)$ (looking for the \textit{complex volume}, the \textit{Reidemeister torsions} or potential deeper terms such as the $n$-loop invariants of \cite{DG}) would probably need to develop the estimations of Lemmas \ref{lem:unbounded:contour}, \ref{lem:unif:bound} and \ref{lem:unif:bound:hbar} at higher order and with sharper precision, as well as carefully study  the coefficients appearing in Theorem \ref{thm:SPM}.
\item In this theory, the integration variables $y_j$ in $\mathfrak{J}_{X_n}(\hbar,0)$ lie in an \textit{unbounded} part of $\C$, contrary to what happens for Kashaev's invariant or the colored Jones polynomials. This is why uniform bounds such as the ones of Lemmas \ref{lem:unbounded:contour}, \ref{lem:unif:bound} and \ref{lem:unif:bound:hbar} were new but absolutely necessary technical difficulties to overcome to obtain the desired asymptotics. Since these results do not depend of the knot, triangulation or potential function $S$ (assuming it has the same general form as in here), we hope that they can be of use to further studies of asymptotics of quantum invariants such as the Teichm\"uller TQFT.
\end{itemize}

\section{The case of even twist knots}\label{sec:appendix}

When the twist knot $K_n$ has an even number of crossings, we can prove the same results as for the odd twist knots, which are:
\begin{itemize}
 \item the construction of convenient H-triangulations and ideal triangulations (Section \ref{sub:even:trig}),
 \item the geometricity of the ideal triangulations  (Section \ref{sub:even:geom}),
 \item the computation of the partition functions of the Teichm\"uller TQFT (Section \ref{sub:even:tqft}),
 \item  the volume conjecture as a consequence of geometricity (Section \ref{sub:even:vol:conj}).
\end{itemize}

We tried to provide details of only the parts of proofs that differ from the case of odd twist knots. As the reader will see, most of these differences lie in explicit values and not in general processes of proof. As such, we expect that the techniques developed in the previous sections and adapted in this one can be generalised to several other families of knots in $3$-manifolds.

\subsection{Construction of triangulations}\label{sub:even:trig}

In the rest of this section we consider a twist knot $K_n$ with $n$ even, $n\geqslant 4$ (the case $n=2$ will be treated in Remark \ref{rem:K2}). We proceed as in Section \ref{sec:trig}, and build an H-triangulation of $(S^3,K_n)$ from a diagram of $K_n$. The first step is described in Figure \ref{fig:diagram:htriang:even}. Note that $D$ is once again an $(n+1)$-gon, and $E$ is an $(n+2)$-gon.

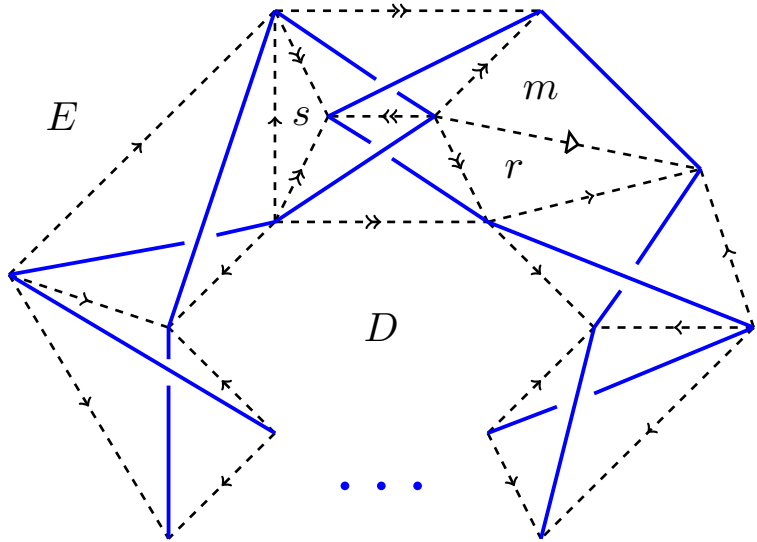
\begin{figure}[!h]
\centering
\begin{tikzpicture}[every path/.style={string ,black} , every node/.style={transform shape , knot crossing , inner sep=1.5 pt } ]

\begin{scope}[scale=0.7]

\begin{scope}[dashed,decoration={
    markings,
    mark=at position 0.5 with {\arrow{>}}}
    ] 
    \draw[postaction={decorate}] (-2,-2)--(-4,-4);
    \draw[postaction={decorate}] (-7,1)--(-4,-4);
    \draw[postaction={decorate}] (-2,-2)--(-4,0);
    \draw[postaction={decorate}] (-7,1)--(-4,0);
    \draw[postaction={decorate}] (-7,1)--(-2,6);
    \draw[postaction={decorate}] (-2,2)--(-4,0);
    \draw[postaction={decorate}] (-2,2)--(-2,6);
    \draw[postaction={decorate}] (2,2)--(6,3);
    \draw[postaction={decorate}] (2,2)--(4,0);
    \draw[postaction={decorate}] (7,0)--(6,3);
    \draw[postaction={decorate}] (7,0)--(4,0);
    \draw[postaction={decorate}] (7,0)--(3,-4);
    \draw[postaction={decorate}] (2,-2)--(3,-4);
    \draw[postaction={decorate}] (2,-2)--(4,0);
\end{scope}

\begin{scope}[dashed,decoration={
    markings,
    mark=at position 0.5 with {\arrow{>>}}}
    ] 
    \draw[postaction={decorate}] (-2,2)--(2,2);
    \draw[postaction={decorate}] (-2,2)--(-1,4);
    \draw[postaction={decorate}] (1,4)--(2,2);
    \draw[postaction={decorate}] (1,4)--(-1,4);
    \draw[postaction={decorate}] (1,4)--(3,6);
    \draw[postaction={decorate}] (-2,6)--(-1,4);
    \draw[postaction={decorate}] (-2,6)--(3,6);
\end{scope}

\draw[style=dashed] (1,4) -- (6,3);

\begin{scope}[xshift=3.5cm, yshift=3.5cm, rotate=-100, scale=0.2]
\draw (1,0) -- (0,1) -- (-1,0) -- (1,0);
\end{scope}

\draw[color=blue, line width=0.5mm] (3,6) -- (6,3);
\draw[color=blue, line width=0.5mm] (2,2) -- (7,0);
\draw[color=blue, line width=0.5mm] (4,0) -- (3,-4);
\draw[color=blue, line width=0.5mm] (-7,1) -- (-2,-2);
\draw[color=blue, line width=0.5mm] (-4,0) -- (-2,6);
\draw[color=blue, line width=0.5mm] (3,6) -- (-1,4);
\draw[color=blue, line width=0.5mm] (-2,2) -- (1,4);

\draw[color=blue, line width=0.5mm] (2,-2) -- (3.3,-1.5);
\draw[color=blue, line width=0.5mm] (4,-1.25) -- (7,0);

\draw[color=blue, line width=0.5mm] (4,0) -- (4.5,0.7);
\draw[color=blue, line width=0.5mm] (4.8,1.2) -- (6,3);

\draw[color=blue, line width=0.5mm] (-1,4) -- (-0.2,3.5);
\draw[color=blue, line width=0.5mm] (0.2,3.2) -- (2,2);

\draw[color=blue, line width=0.5mm] (1,4) -- (0.3,4.45);
\draw[color=blue, line width=0.5mm] (-0.1,4.7) -- (-2,6);

\draw[color=blue, line width=0.5mm] (-7,1) -- (-3.7,1.6);
\draw[color=blue, line width=0.5mm] (-3.1,1.75) -- (-2,2);

\draw[color=blue, line width=0.5mm] (-4,0) -- (-4,-0.6);
\draw[color=blue, line width=0.5mm] (-4,-1.1) -- (-4,-4);

\draw[scale=4,color=blue] (0,-3/4) node {$\ldots$};

\draw[scale=2] (0,0) node {$D$};
\draw[scale=2] (3/2,4.5/2) node {$m$};
\draw[scale=2] (2.5/2,3/2) node {$r$};
\draw[scale=2] (-1.5/2,4/2) node {$s$};
\draw[scale=2] (-6/2,4/2) node {$E$};

\end{scope}
\end{tikzpicture}
\caption{Building an H-triangulation from a diagram of $K_n$} \label{fig:diagram:htriang:even}
\end{figure}

From Figure \ref{fig:diagram:htriang:even} we go to Figure \ref{fig:boundary:balls:even} and Figure \ref{fig:Htriang:polyhedron:even} exactly as in Section \ref{sec:trig}.

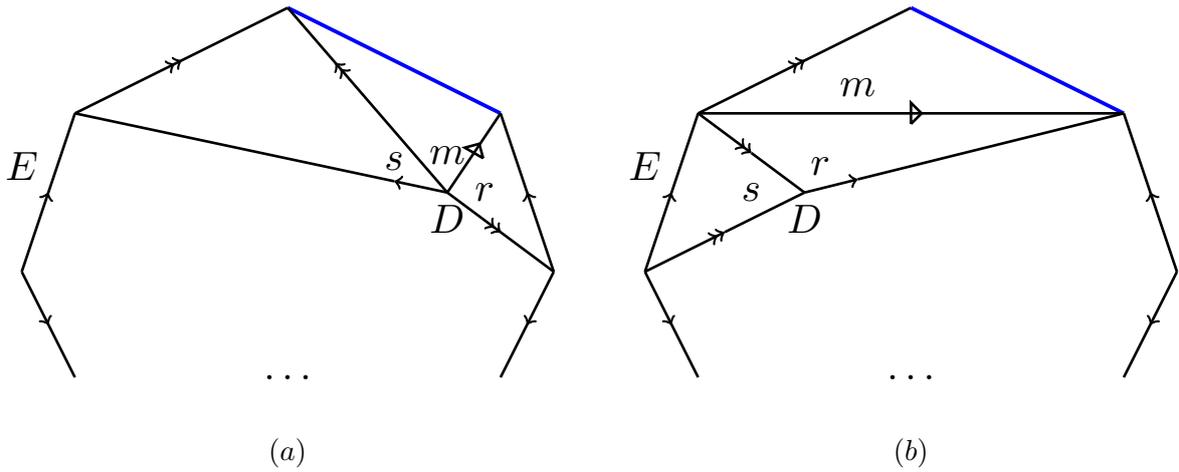
\begin{figure}[!h]
\centering
\begin{tikzpicture}[every path/.style={string ,black} , every node/.style={transform shape , knot crossing , inner sep=1.5 pt } ]

\draw[scale=1] (0,-1) node {$(a)$};
\draw[scale=1] (8.2,-1) node {$(b)$};


\begin{scope}[scale=0.7]

\begin{scope}[decoration={
    markings,
    mark=at position 0.5 with {\arrow{>}}}
    ] 
    \draw[postaction={decorate}] (-5,2)--(-4,0);
    \draw[postaction={decorate}] (-5,2)--(-4,5);
    \draw[postaction={decorate}] (5,2)--(4,5);
    \draw[postaction={decorate}] (5,2)--(4,0);
\end{scope}

\begin{scope}[decoration={
    markings,
    mark=at position 0.5 with {\arrow{>>}}}
    ] 
    \draw[postaction={decorate}] (-4,5)--(0,7);
\end{scope}

\draw[color=blue, line width=0.5mm] (0,7) -- (4,5);

\draw[scale=2] (0,0) node {$\ldots$};

\draw[->] (3,3.5) -- (2,3.5 +1.5/7);
\draw (2,3.5 +1.5/7) -- (-4,5);

\begin{scope}[xshift=3.5cm, yshift=4.25cm, rotate=-30, scale=0.2]
\draw (1,0) -- (0,1) -- (-1,0) -- (1,0);
\end{scope}
\draw (3,3.5) -- (4,5);

\begin{scope}[decoration={
    markings,
    mark=at position 0.5 with {\arrow{>>}}}
    ] 
    \draw[postaction={decorate}] (3,3.5) -- (5,2);
\end{scope}

\begin{scope}[decoration={
    markings,
    mark=at position 0.7 with {\arrow{>>}}}
    ] 
    \draw[postaction={decorate}] (3,3.5) -- (0,7);
\end{scope}

\draw[scale=2] (3/2,3/2) node {$D$};
\draw[scale=2] (3/2,4.2/2) node {$m$};
\draw[scale=2] (2/2,4.1/2) node {$s$};
\draw[scale=2] (3.7/2,3.5/2) node {$r$};

\draw[scale=2] (-5/2,4/2) node {$E$};

\end{scope}


\begin{scope}[xshift=8.2cm,scale=0.7]

\begin{scope}[decoration={
    markings,
    mark=at position 0.5 with {\arrow{>}}}
    ] 
    \draw[postaction={decorate}] (-5,2)--(-4,0);
    \draw[postaction={decorate}] (-5,2)--(-4,5);
    \draw[postaction={decorate}] (5,2)--(4,5);
    \draw[postaction={decorate}] (5,2)--(4,0);
\end{scope}

\begin{scope}[decoration={
    markings,
    mark=at position 0.5 with {\arrow{>>}}}
    ] 
    \draw[postaction={decorate}] (-4,5)--(0,7);
\end{scope}

\draw[color=blue, line width=0.5mm] (0,7) -- (4,5);

\draw[scale=2] (0,0) node {$\ldots$};

\begin{scope}[xshift=0cm, yshift=5cm, rotate=-90, scale=0.2]
\draw (1,0) -- (0,1) -- (-1,0) -- (1,0);
\end{scope}
\draw (-4,5) -- (4,5);

\draw[->>] (-5,2) -- (-3.5,2.75);
\draw (-3.5,2.75) -- (-2,3.5);

\draw[->>] (-4,5) -- (-3,4.25);
\draw (-3,4.25) -- (-2,3.5);

\draw (4,5) -- (-1,5-1.5*5/6);
\draw[<-] (-1,5-1.5*5/6) -- (-2,3.5);
\draw[scale=2] (-2/2,3/2) node {$D$};
\draw[scale=2] (-1/2,5.5/2) node {$m$};
\draw[scale=2] (-3/2,3.5/2) node {$s$};
\draw[scale=2] (-1.7/2,4/2) node {$r$};

\draw[scale=2] (-5/2,4/2) node {$E$};

\end{scope}
\end{tikzpicture}
\caption{Boundaries of $B_+$ and $B_-$}\label{fig:boundary:balls:even}
\end{figure}

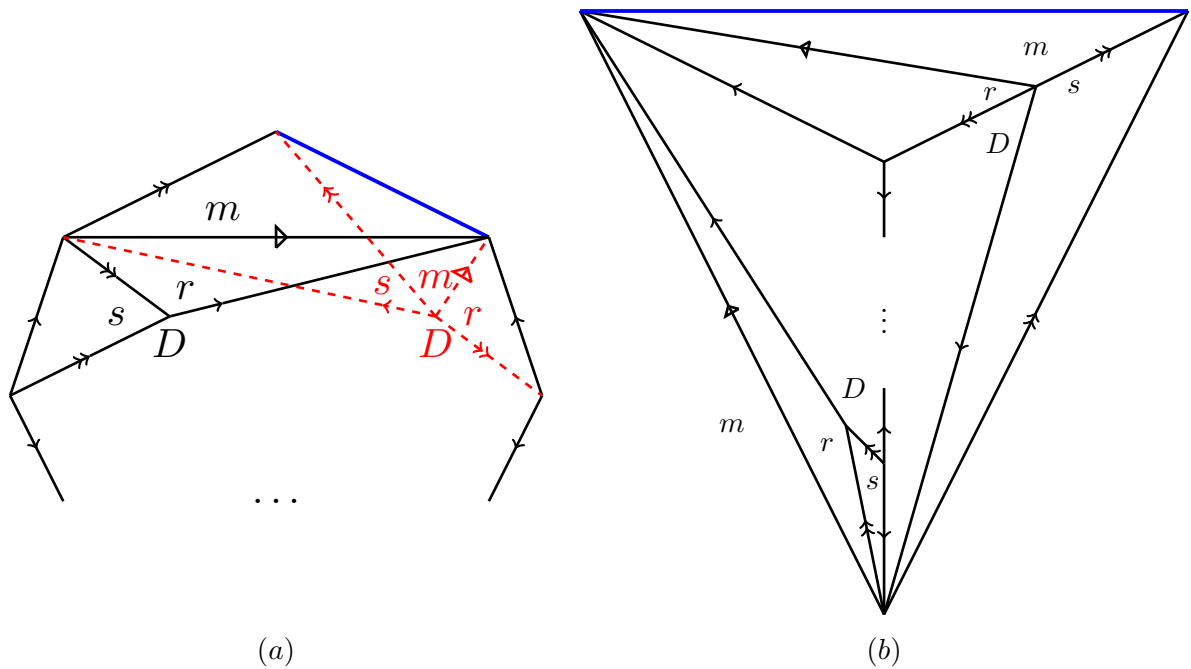
\begin{figure}[!h]
\centering
\begin{tikzpicture}[every path/.style={string ,black} , every node/.style={transform shape , knot crossing , inner sep=1.5 pt } ]

\draw[scale=1] (0,-2) node {$(a)$};
\draw[scale=1] (8,-2) node {$(b)$};


\begin{scope}[scale=0.7]

\begin{scope}[decoration={
    markings,
    mark=at position 0.5 with {\arrow{>}}}
    ] 
    \draw[postaction={decorate}] (-5,2)--(-4,0);
    \draw[postaction={decorate}] (-5,2)--(-4,5);
    \draw[postaction={decorate}] (5,2)--(4,5);
    \draw[postaction={decorate}] (5,2)--(4,0);
\end{scope}

\begin{scope}[decoration={
    markings,
    mark=at position 0.5 with {\arrow{>>}}}
    ] 
    \draw[postaction={decorate}] (-4,5)--(0,7);
\end{scope}

\draw[color=blue, line width=0.5mm] (0,7) -- (4,5);

\draw[scale=2] (0,0) node {$\ldots$};

\begin{scope}[xshift=0cm, yshift=5cm, rotate=-90, scale=0.2]
\draw (1,0) -- (0,1) -- (-1,0) -- (1,0);
\end{scope}
\draw (-4,5) -- (4,5);

\draw[->>] (-5,2) -- (-3.5,2.75);
\draw (-3.5,2.75) -- (-2,3.5);

\draw[->>] (-4,5) -- (-3,4.25);
\draw (-3,4.25) -- (-2,3.5);

\draw (4,5) -- (-1,5-1.5*5/6);
\draw[<-] (-1,5-1.5*5/6) -- (-2,3.5);
\draw[scale=2] (-2/2,3/2) node {$D$};
\draw[scale=2] (-1/2,5.5/2) node {$m$};
\draw[scale=2] (-3/2,3.5/2) node {$s$};
\draw[scale=2] (-1.7/2,4/2) node {$r$};


\begin{scope}[xshift=3.5cm, yshift=4.25cm, rotate=-30, scale=0.2]
\draw[color=red] (1,0) -- (0,1) -- (-1,0) -- (1,0);
\end{scope}

\begin{scope}[style=dashed]

\draw[color=red][->] (3,3.5) -- (2,3.5 +1.5/7);
\draw[color=red] (2,3.5 +1.5/7) -- (-4,5);

\draw[color=red] (3,3.5) -- (4,5);

\begin{scope}[decoration={
    markings,
    mark=at position 0.5 with {\arrow{>>}}}
    ] 
    \draw[postaction={decorate}][color=red] (3,3.5) -- (5,2);
\end{scope}

\begin{scope}[decoration={
    markings,
    mark=at position 0.7 with {\arrow{>>}}}
    ] 
    \draw[postaction={decorate}][color=red] (3,3.5) -- (0,7);
\end{scope}

\draw[scale=2][color=red] (3/2,3/2) node {$D$};
\draw[scale=2][color=red] (3/2,4.2/2) node {$m$};
\draw[scale=2][color=red] (2/2,4.1/2) node {$s$};
\draw[scale=2][color=red] (3.7/2,3.5/2) node {$r$};

\end{scope}

\end{scope}

\begin{scope}[xshift=8cm,yshift=-1.5cm,scale=0.5]

\begin{scope}[decoration={
    markings,
    mark=at position 0.5 with {\arrow{>}}}
    ] 
    \draw[postaction={decorate}] (0,4)--(0,0);
    \draw[postaction={decorate}] (0,4)--(0,6);
    \draw[postaction={decorate}] (0,12) -- (0,10);
    \draw[postaction={decorate}] (0,12)--(-8,16);
    \draw[postaction={decorate}, color=red] (4,14) -- (0,0);
    \draw[postaction={decorate}] (-1,5)--(-8,16);
\end{scope}

\begin{scope}[decoration={
    markings,
    mark=at position 0.5 with {\arrow{>>}}}
    ] 
    \draw[postaction={decorate}] (0,4)--(-1,5);
    \draw[postaction={decorate}] (0,0)--(-1,5);
    \draw[postaction={decorate}, color=red] (4,14) -- (8,16);
    \draw[postaction={decorate}] (0,0) -- (8,16);
    \draw[postaction={decorate}, color=red] (4,14) -- (0,12);
\end{scope}

\draw (0,0) -- (-8,16);
\begin{scope}[xshift=-4cm, yshift=8cm, rotate=30, scale=0.2]
\draw (1,0) -- (0,1) -- (-1,0) -- (1,0);
\end{scope}

\draw[color=red] (4,14) -- (-8,16);
\begin{scope}[xshift=-2cm, yshift=15cm, rotate=75, scale=0.2]
\draw[color=red] (1,0) -- (0,1) -- (-1,0) -- (1,0);
\end{scope}

\draw[color=blue, line width=0.5mm] (-8,16) -- (8,16);

\draw[scale=2] (0,8/2) node {$\vdots$};

\draw[scale=2, color=red] (4/2,15/2) node {$m$};
\draw[scale=2, color=red] (2.8/2,13.8/2) node {$r$};
\draw[scale=2, color=red] (5/2,14/2) node {$s$};
\draw[scale=2, color=red] (3/2,12.5/2) node {$D$};

\draw[scale=2] (-4/2,5/2) node {$m$};
\draw[scale=2] (-1.5/2,4.5/2) node {$r$};
\draw[scale=2] (-0.3/2,3.5/2) node {$s$};
\draw[scale=2] (-0.8/2,6/2) node {$D$};

\end{scope}

\end{tikzpicture}
\caption{A cellular decomposition of $(S^3,K_n)$ as a polyhedron glued to itself}\label{fig:Htriang:polyhedron:even}
\end{figure}

Then we add a new edge (with simple full arrow) and cut $D$ into $u$ and $D'$ (see Figure \ref{fig:even:Htriang:bigon:trick} (a)), and then we apply the bigon trick $p$ times, where $p:= \frac{n-2}{2}$. We finally obtain the polyhedron in Figure \ref{fig:even:Htriang:bigon:trick} (b).

\begin{figure}[!h]
\centering
\begin{tikzpicture}[every path/.style={string ,black}]

\draw[scale=1] (0,-2) node {$(a)$};
\draw[scale=1] (8,-2) node {$(b)$};

\begin{scope}[xshift=0cm,yshift=-1.5cm,scale=0.45]

\draw (0,6) node[shape=circle,fill=black,scale=0.3] {};
\draw (0,8) node[shape=circle,fill=black,scale=0.3] {};
\draw (0,10) node[shape=circle,fill=black,scale=0.3] {};

\draw[->,>=latex] (0,4) .. controls +(-0.5,0) and +(0,-0.5) .. (-1,7);
\draw (-1,7) .. controls +(0,0.5) and +(-0.5,0) .. (0,8);

\draw[scale=2] (-0.5/2,7/2) node {$u$};

\draw[->,>=latex] (4,14) -- (2,9);
\draw (2,9) -- (0,4);

\draw[scale=2] (0.5/2,3.5/2) node {$u$};

\begin{scope}[decoration={
    markings,
    mark=at position 0.5 with {\arrow{>}}}
    ] 
    \draw[postaction={decorate}] (0,4)--(0,0);
    \draw[postaction={decorate}] (0,4)--(0,6);
    \draw[postaction={decorate}] (0,8)--(0,6);
    \draw[postaction={decorate}] (0,12) -- (0,10);
    \draw[postaction={decorate}] (0,12)--(-8,16);
    \draw[postaction={decorate}] (4,14) -- (0,0);
    \draw[postaction={decorate}] (-1,5)--(-8,16);
\end{scope}

\begin{scope}[decoration={
    markings,
    mark=at position 0.5 with {\arrow{>>}}}
    ] 
    \draw[postaction={decorate}] (0,4)--(-1,5);
    \draw[postaction={decorate}] (0,0)--(-1,5);
    \draw[postaction={decorate}] (4,14) -- (8,16);
    \draw[postaction={decorate}] (0,0) -- (8,16);
    \draw[postaction={decorate}] (4,14) -- (0,12);
\end{scope}

\draw (0,0) -- (-8,16);
\begin{scope}[xshift=-4cm, yshift=8cm, rotate=30, scale=0.2]
\draw (1,0) -- (0,1) -- (-1,0) -- (1,0);
\end{scope}

\draw (4,14) -- (-8,16);
\begin{scope}[xshift=-2cm, yshift=15cm, rotate=75, scale=0.2]
\draw (1,0) -- (0,1) -- (-1,0) -- (1,0);
\end{scope}

\draw[color=blue, line width=0.5mm] (-8,16) -- (8,16);

\draw[scale=2] (0,9/2) node {$\vdots$};

\draw[scale=2] (4/2,15/2) node {$m$};
\draw[scale=2] (2.8/2,13.8/2) node {$r$};
\draw[scale=2] (5/2,14/2) node {$s$};
\draw[scale=2] (2/2,11.5/2) node {$D'$};

\draw[scale=2] (-4/2,5/2) node {$m$};
\draw[scale=2] (-1.5/2,4.5/2) node {$r$};
\draw[scale=2] (-0.3/2,3.5/2) node {$s$};
\draw[scale=2] (-4/2,12/2) node {$D'$};

\end{scope}

\begin{scope}[xshift=8cm,yshift=-1.5cm,scale=0.45]

\draw (0,6) node[shape=circle,fill=black,scale=0.3] {};
\draw (0,10) node[shape=circle,fill=black,scale=0.3] {};

\draw[->,>=latex] (0,12) -- (-1,9);
\draw (-1,9) -- (-2,6);

\draw[scale=2] (-2.5/2,11/2) node {$u$};

\draw[->,>=latex] (4,14) -- (2,9);
\draw (2,9) -- (0,4);

\draw[scale=2] (0.5/2,3.5/2) node {$u$};

\begin{scope}[decoration={
    markings,
    mark=at position 0.5 with {\arrow{>}}}
    ] 
    \draw[postaction={decorate}] (0,4)--(0,0);
    \draw[postaction={decorate}] (0,12)--(-8,16);
    \draw[postaction={decorate}] (4,14) -- (0,0);
    \draw[postaction={decorate}] (-2,6)--(-8,16);
\end{scope}

\draw[->,>=latex] (0,4) -- (0,5);
\draw (0,5) -- (0,6);

\draw[->,>=latex] (0,10) -- (0,11);
\draw (0,11) -- (0,12);

\begin{scope}[decoration={
    markings,
    mark=at position 0.5 with {\arrow{>>}}}
    ] 
    \draw[postaction={decorate}] (0,4)--(-2,6);
    \draw[postaction={decorate}] (0,0)--(-2,6);
    \draw[postaction={decorate}] (4,14) -- (8,16);
    \draw[postaction={decorate}] (0,0) -- (8,16);
    \draw[postaction={decorate}] (4,14) -- (0,12);
\end{scope}

\draw (0,0) -- (-8,16);
\begin{scope}[xshift=-4cm, yshift=8cm, rotate=30, scale=0.2]
\draw (1,0) -- (0,1) -- (-1,0) -- (1,0);
\end{scope}

\draw (4,14) -- (-8,16);
\begin{scope}[xshift=-2cm, yshift=15cm, rotate=75, scale=0.2]
\draw (1,0) -- (0,1) -- (-1,0) -- (1,0);
\end{scope}

\draw[color=blue, line width=0.5mm] (-8,16) -- (8,16);

\draw[scale=2] (0,8/2) node {$\vdots$};

\draw[scale=2] (4/2,15/2) node {$m$};
\draw[scale=2] (2.8/2,13.8/2) node {$r$};
\draw[scale=2] (5/2,14/2) node {$s$};
\draw[scale=2] (2/2,11/2) node {$G$};

\draw[scale=2] (-4/2,5/2) node {$m$};
\draw[scale=2] (-2.3/2,5.5/2) node {$r$};
\draw[scale=2] (-0.3/2,3.5/2) node {$s$};
\draw[scale=2] (-1/2,6/2) node {$G$};

\end{scope}

\end{tikzpicture}
\caption{A cellular decomposition of $(S^3,K_n)$ before and after the bigon trick}\label{fig:even:Htriang:bigon:trick}
\end{figure}
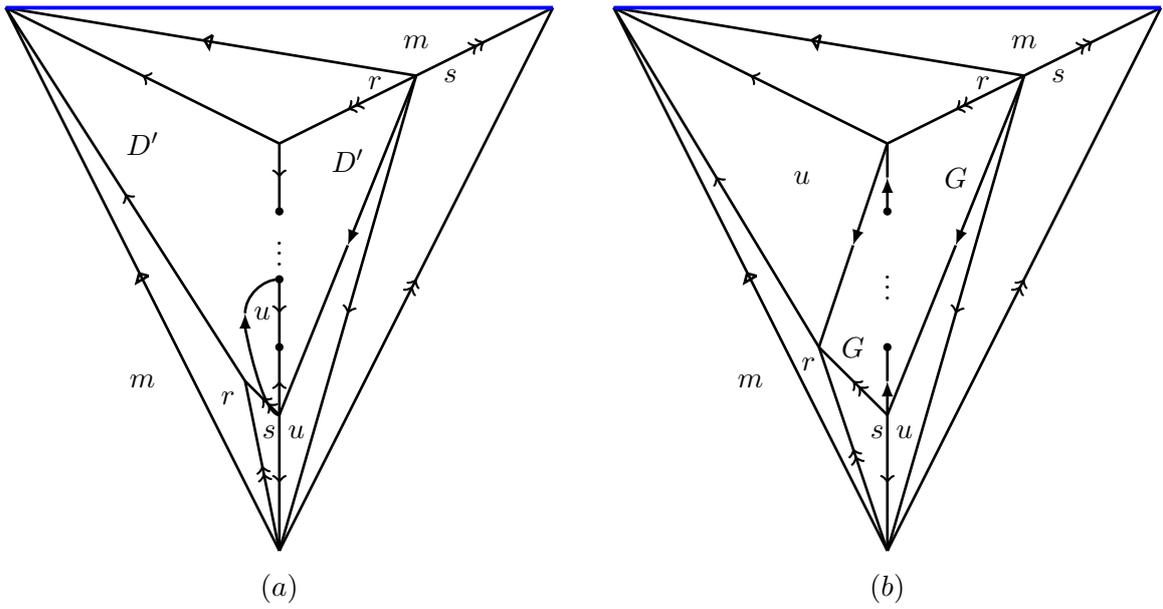

We now chop off the quadrilateral made up of the two adjacent faces $G$ (which are $(p+2)$-gons) and we add a new edge (double full arrow) and two new faces $e_{p+1},f_p$. We triangulate the previous quadrilateral as in Figure \ref{fig:GG:tower} and we finally obtain a decomposition of $S^3$ in three polyhedra glued to one another, as described in Figure \ref{fig:even:Htriang:flip:tower}. Note that if $p=1$, then $G=e_1=e_p=f_0=f_{p-1}$ and there is no tower.

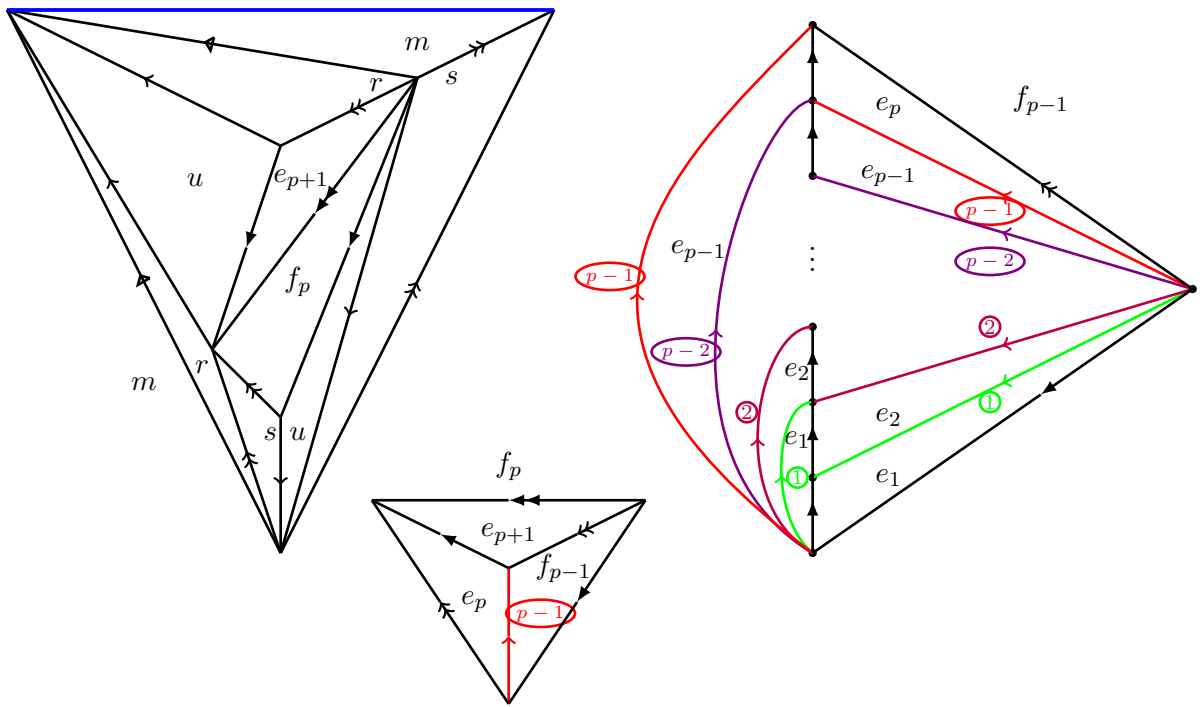
\begin{figure}[!h]
\centering
\begin{tikzpicture}[every path/.style={string ,black}]

\begin{scope}[xshift=0cm,yshift=0cm,scale=0.45]

\draw[->,>=latex] (0,12) -- (-1,9);
\draw (-1,9) -- (-2,6);

\draw[scale=2] (-2.5/2,11/2) node {$u$};

\draw[->,>=latex] (4,14) -- (2,9);
\draw (2,9) -- (0,4);

\draw[scale=2] (0.5/2,3.5/2) node {$u$};

\begin{scope}[decoration={
    markings,
    mark=at position 0.5 with {\arrow{>}}}
    ] 
    \draw[postaction={decorate}] (0,4)--(0,0);
    \draw[postaction={decorate}] (0,12)--(-8,16);
    \draw[postaction={decorate}] (4,14) -- (0,0);
    \draw[postaction={decorate}] (-2,6)--(-8,16);
\end{scope}

\draw[->>,>=latex] (4,14) -- (1,10);
\draw (1,10) -- (-2,6);

\begin{scope}[decoration={
    markings,
    mark=at position 0.5 with {\arrow{>>}}}
    ] 
    \draw[postaction={decorate}] (0,4)--(-2,6);
    \draw[postaction={decorate}] (0,0)--(-2,6);
    \draw[postaction={decorate}] (4,14) -- (8,16);
    \draw[postaction={decorate}] (0,0) -- (8,16);
    \draw[postaction={decorate}] (4,14) -- (0,12);
\end{scope}

\draw (0,0) -- (-8,16);
\begin{scope}[xshift=-4cm, yshift=8cm, rotate=30, scale=0.2]
\draw (1,0) -- (0,1) -- (-1,0) -- (1,0);
\end{scope}

\draw (4,14) -- (-8,16);
\begin{scope}[xshift=-2cm, yshift=15cm, rotate=75, scale=0.2]
\draw (1,0) -- (0,1) -- (-1,0) -- (1,0);
\end{scope}

\draw[color=blue, line width=0.5mm] (-8,16) -- (8,16);

\draw[scale=2] (4/2,15/2) node {$m$};
\draw[scale=2] (2.8/2,13.8/2) node {$r$};
\draw[scale=2] (5/2,14/2) node {$s$};
\draw[scale=2] (0.5/2,8/2) node {$f_p$};

\draw[scale=2] (-4/2,5/2) node {$m$};
\draw[scale=2] (-2.3/2,5.5/2) node {$r$};
\draw[scale=2] (-0.3/2,3.5/2) node {$s$};
\draw[scale=2] (0.6/2,11/2) node {$e_{p+1}$};

\end{scope}

\begin{scope}[xshift=3cm,yshift=-2cm,scale=0.9]

\draw[color=red,->] (0,0)--(0,1);
\draw[color=red] (0,1)--(0,2);

\begin{scope}[xshift=0cm,yshift=0cm,rotate=0,scale=1/1.5]
\draw[color=red] (0.7,2)  node{\tiny $p-1$};
\node[draw,ellipse,color=red] (S) at(0.7,2) {\ \ \ \ };
\end{scope}

\draw[->,>=latex] (2,3)--(1,1.5);
\draw (1,1.5)--(0,0);

\draw[->,>=latex] (0,2)--(-1,2.5);
\draw (-1,2.5)--(-2,3);

\draw[->>,>=latex] (2,3)--(0,3);
\draw (0,3)--(-2,3);

\draw[->>] (2,3)--(1,2.5);
\draw (1,2.5)--(0,2);

\draw[->>] (0,0)--(-1,1.5);
\draw (-1,1.5)--(-2,3);

\draw[scale=2] (0.8/2,2/2)  node{$f_{p-1}$};
\draw[scale=2] (-0.5/2,1.5/2)  node{$e_{p}$};
\draw[scale=2] (0/2,2.5/2)  node{$e_{p+1}$};
\draw[scale=2] (0/2,3.5/2)  node{$f_p$};

\end{scope}

\begin{scope}[xshift=7cm,yshift=0cm,scale=1]

\draw (0,0) node[shape=circle,fill=black,scale=0.3] {};
\draw (0,1) node[shape=circle,fill=black,scale=0.3] {};
\draw (0,2) node[shape=circle,fill=black,scale=0.3] {};
\draw (0,3) node[shape=circle,fill=black,scale=0.3] {};
\draw (0,5) node[shape=circle,fill=black,scale=0.3] {};
\draw (0,6) node[shape=circle,fill=black,scale=0.3] {};
\draw (0,7) node[shape=circle,fill=black,scale=0.3] {};
\draw (5,3.5) node[shape=circle,fill=black,scale=0.3] {};

\draw[scale=1] (0,4) node {$\vdots$};

\draw[scale=1,color=black] (1,1) node {$e_1$};
\draw[scale=1,color=black] (-0.2,1.5) node {$e_1$};
\draw[scale=1,color=black] (1,1.8) node {$e_2$};
\draw[scale=1,color=black] (-0.2,2.4) node {$e_2$};
\draw[scale=1,color=black] (1,5) node {$e_{p-1}$};
\draw[scale=1,color=black] (-1.5,4) node {$e_{p-1}$};
\draw[scale=1,color=black] (1,5.9) node {$e_p$};
\draw[scale=1,color=black] (3,6) node {$f_{p-1}$};

\begin{scope}[xshift=0cm,yshift=0cm,rotate=0,scale=1/1.5]
\draw[color=green] (3.5,3) circle (0.2) node{\scriptsize $1$};
\end{scope}

\begin{scope}[xshift=0cm,yshift=0cm,rotate=0,scale=1/1.5]
\draw[color=green] (-0.3,1.5) circle (0.2) node{\scriptsize $1$};
\end{scope}

\begin{scope}[xshift=0cm,yshift=0cm,rotate=0,scale=1/1.5,color=purple]
\draw[color=purple] (3.5,4.5) circle (0.2) node{\scriptsize $2$};
\end{scope}

\begin{scope}[xshift=0cm,yshift=0cm,rotate=0,scale=1/1.5,color=purple]
\draw[color=purple] (-1.3,2.8) circle (0.2) node{\scriptsize $2$};
\end{scope}

\begin{scope}[xshift=0cm,yshift=0cm,rotate=0,scale=1/1.5]
\draw[color=violet] (3.5,5.8)  node{\tiny $p-2$};
\node[draw,ellipse,color=violet] (S) at(3.5,5.8) {\ \ \ \ };
\end{scope}

\begin{scope}[xshift=0cm,yshift=0cm,rotate=0,scale=1/1.5]
\draw[color=violet] (-2.5,4)  node{\tiny $p-2$};
\node[draw,ellipse,color=violet] (S) at(-2.5,4) {\ \ \ \ };
\end{scope}

\begin{scope}[xshift=0cm,yshift=0cm,rotate=0,scale=1/1.5]
\draw[color=red] (3.5,6.8)  node{\tiny $p-1$};
\node[draw,ellipse,color=red] (S) at(3.5,6.8) {\ \ \ \ };
\end{scope}

\begin{scope}[xshift=0cm,yshift=0cm,rotate=0,scale=1/1.5]
\draw[color=red] (-4,5.5)  node{\tiny $p-1$};
\node[draw,ellipse,color=red] (S) at(-4,5.5) {\ \ \ \ };
\end{scope}

\begin{scope}[decoration={markings,mark=at position 0.5 with {\arrow{>}}}] 
    \draw[postaction={decorate},color=green] (5,3.5)--(0,1);
    \draw[postaction={decorate},color=purple] (5,3.5)--(0,2);
    \draw[postaction={decorate},color=violet] (5,3.5)--(0,5);
    \draw[postaction={decorate},color=red] (5,3.5)--(0,6); 
   	\draw[postaction={decorate},color=green] (0,0) .. controls +(-0.6,0.5) and +(-0.5,0) .. (0,2);  
   	\draw[postaction={decorate},color=purple] (0,0) .. controls +(-1.2,0.7) and +(-0.7,0) .. (0,3);  
   	\draw[postaction={decorate},color=violet] (0,0) .. controls +(-2.5,1.5) and +(-0.7,0) .. (0,6);  
   	\draw[postaction={decorate},color=red] (0,0) .. controls +(-4,3) and +(-2,-2) .. (0,7); 
\end{scope}

\draw (0,0) -- (3,2.1);
\draw[<-,>=latex] (3,2.1) -- (5,3.5);
\draw (0,7) -- (3,7-2.1);
\draw[<<-] (3,7-2.1) -- (5,3.5);

\begin{scope}[yshift=0cm]
\draw[->,>=latex] (0,0) -- (0,0.7);
\draw (0,0.6)--(0,1);
\end{scope}

\begin{scope}[yshift=1cm]
\draw[->,>=latex] (0,0) -- (0,0.7);
\draw (0,0.6)--(0,1);
\end{scope}

\begin{scope}[yshift=2cm]
\draw[->,>=latex] (0,0) -- (0,0.7);
\draw (0,0.6)--(0,1);
\end{scope}

\begin{scope}[yshift=5cm]
\draw[->,>=latex] (0,0) -- (0,0.7);
\draw (0,0.6)--(0,1);
\end{scope}

\begin{scope}[yshift=6cm]
\draw[->,>=latex] (0,0) -- (0,0.7);
\draw (0,0.6)--(0,1);
\end{scope}

\end{scope}

\end{tikzpicture}
\caption{A flip move and a tower of tetrahedra}\label{fig:even:Htriang:flip:tower}
\end{figure}

We can then decompose the polyhedra in Figure \ref{fig:even:Htriang:flip:tower} into ordered tetrahedra and obtain the H-triangulation of Figure \ref{fig:H:trig:even}. Along the way, in order to harmonize the notation with the small cases ($p=0,1$), we did the following arrow replacements:
\begin{itemize}
\item {Replace} full black simple arrow by simple arrow with circled $0$,
\item {replace} full black double arrow by simple arrow with circled $p+1$,
\item {replace} double arrow by simple arrow with circled $p$,
\item {replace} full white arrow by double full white arrow. 
\end{itemize}
Moreover, we cut the previous {polyhe}dron into $p+4$ tetrahedra, introducing new triangular faces $v$ (behind $e_{p+1},r,u$), $g$ (behind $f_p,s,u$), $s'$ (completing $m,m,s$), and $f_1, \ldots, f_{p-1}$ at each of the $p-1$ floors of the tower of Figure \ref{fig:even:Htriang:flip:tower}. We add the convention $f_0=e_1$ to account for the case $p=0$.

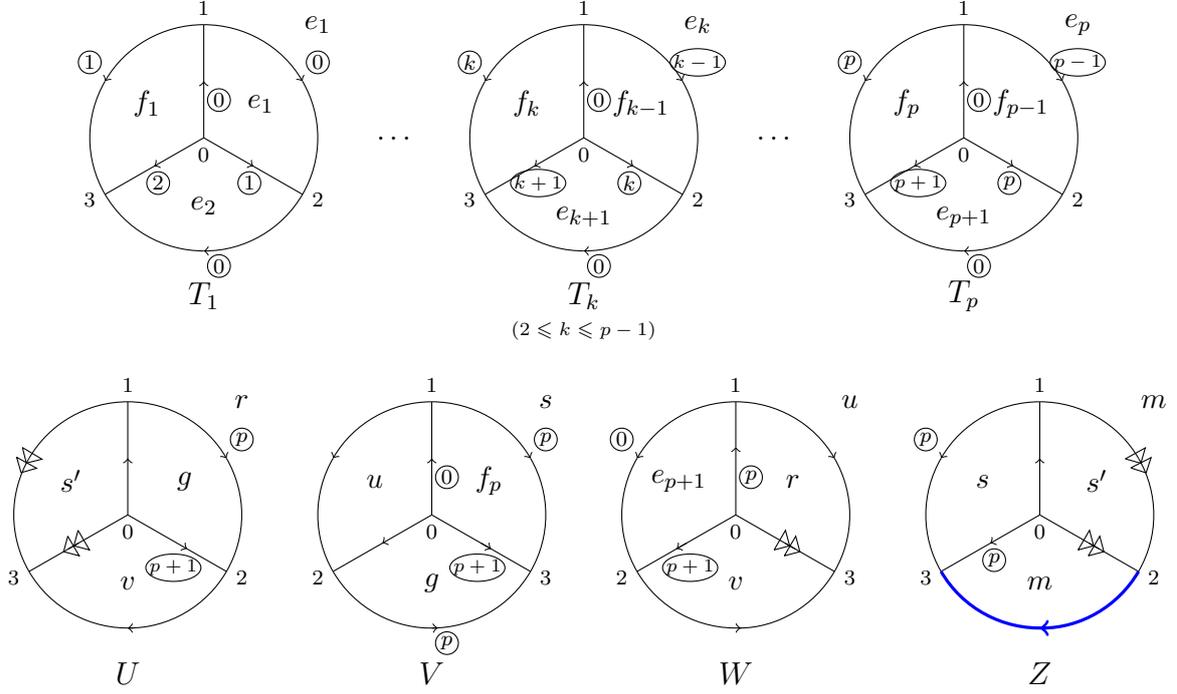
\begin{figure}[!h]
\begin{tikzpicture}

\begin{scope}[xshift=1cm,yshift=0cm,rotate=0,scale=1.5]

\draw (0,-0.15) node{\scriptsize $0$} ;
\draw (0,1.15) node{\scriptsize $1$} ;
\draw (1,-0.55) node{\scriptsize $2$} ;
\draw (-1,-0.55) node{\scriptsize $3$} ;
\draw (1,1) node{$e_1$} ;
\draw (0,-0.6) node{$e_2$} ;
\draw (-0.5,0.3) node{$f_1$} ;
\draw (0.5,0.3) node{$e_1$} ;

\draw (0,-1.4) node{\large $T_1$} ;

\path [draw=black,postaction={on each segment={mid arrow=black}}]
(0,0)--(-1.732/2,-0.5);

\path [draw=black,postaction={on each segment={mid arrow =black}}]
(0,0)--(0,1);

\path [draw=black,postaction={on each segment={mid arrow =black}}]
(0,0)--(1.732/2,-0.5);

\draw[->](1.732/2,-0.5) arc (-30:-90:1);
\draw (0,-1) arc (-90:-150:1);

\draw[->](0,1) arc (90:30:1);
\draw (1.732/2,0.5) arc (30:-30:1);

\draw[->](0,1) arc (90:150:1);
\draw (-1.732/2,0.5) arc (150:210:1);

%
\begin{scope}[xshift=0cm,yshift=0cm,rotate=0,scale=1/1.5]
\draw (1.5,1) circle (0.15) node{\scriptsize $0$};
\end{scope}

%
\begin{scope}[xshift=0cm,yshift=0cm,rotate=0,scale=1/1.5]
\draw (0.2,0.5) circle (0.15) node{\scriptsize $0$};
\end{scope}

%
\begin{scope}[xshift=0cm,yshift=0cm,rotate=0,scale=1/1.5]
\draw (0.2,-1.7) circle (0.15) node{\scriptsize $0$};
\end{scope}

\begin{scope}[xshift=0cm,yshift=0cm,rotate=0,scale=1/1.5]
\draw (-0.6,-0.6) circle (0.15) node{\scriptsize $2$};
\end{scope}

\begin{scope}[xshift=0cm,yshift=0cm,rotate=0,scale=1/1.5]
\draw (-1.5,1) circle (0.15) node{\scriptsize $1$};
\end{scope}

\begin{scope}[xshift=0cm,yshift=0cm,rotate=0,scale=1/1.5]
\draw (0.6,-0.6) circle (0.15) node{\scriptsize $1$};
\end{scope}

\end{scope}

\draw (3.5,0) node{$\ldots$} ;
\draw (8.5,0) node{$\ldots$} ;

\begin{scope}[xshift=6cm,yshift=0cm,rotate=0,scale=1.5]

\draw (0,-0.15) node{\scriptsize $0$} ;
\draw (0,1.15) node{\scriptsize $1$} ;
\draw (1,-0.55) node{\scriptsize $2$} ;
\draw (-1,-0.55) node{\scriptsize $3$} ;
\draw (1,1) node{$e_k$} ;
\draw (0,-0.7) node{$e_{k+1}$} ;
\draw (-0.5,0.3) node{$f_k$} ;
\draw (0.5,0.3) node{$f_{k-1}$} ;

\draw (0,-1.4) node{\large $T_k$} ;
\draw (0,-1.7) node{\tiny $(2\leqslant k \leqslant p-1)$} ;

\path [draw=black,postaction={on each segment={mid arrow=black}}]
(0,0)--(-1.732/2,-0.5);

\path [draw=black,postaction={on each segment={mid arrow =black}}]
(0,0)--(0,1);

\path [draw=black,postaction={on each segment={mid arrow =black}}]
(0,0)--(1.732/2,-0.5);

\draw[->](1.732/2,-0.5) arc (-30:-90:1);
\draw (0,-1) arc (-90:-150:1);

\draw[->](0,1) arc (90:30:1);
\draw (1.732/2,0.5) arc (30:-30:1);

\draw[->](0,1) arc (90:150:1);
\draw (-1.732/2,0.5) arc (150:210:1);

%
\begin{scope}[xshift=0cm,yshift=0cm,rotate=0,scale=1/1.5]
\draw (0.2,0.5) circle (0.15) node{\scriptsize $0$};
\end{scope}

%
\begin{scope}[xshift=0cm,yshift=0cm,rotate=0,scale=1/1.5]
\draw (0.2,-1.7) circle (0.15) node{\scriptsize $0$};
\end{scope}

\begin{scope}[xshift=0cm,yshift=0cm,rotate=0,scale=1/1.5]
\draw (-1.5,1) circle (0.15) node{\scriptsize $k$};
\end{scope}

\begin{scope}[xshift=0cm,yshift=0cm,rotate=0,scale=1/1.5]
\draw (1.5,1)  node{\tiny $k-1$};
\node[draw,ellipse] (S) at(1.5,1) {\ \ \ };
\end{scope}

\begin{scope}[xshift=0cm,yshift=0cm,rotate=0,scale=1/1.5]
\draw (-0.6,-0.6)  node{\tiny $k+1$};
\node[draw,ellipse] (S) at(-0.6,-0.6) {\ \ \ };
\end{scope}

\begin{scope}[xshift=0cm,yshift=0cm,rotate=0,scale=1/1.5]
\draw (0.6,-0.6) circle (0.15) node{\scriptsize $k$};
\end{scope}

\end{scope}

\begin{scope}[xshift=11cm,yshift=0cm,rotate=0,scale=1.5]

\draw (0,-0.15) node{\scriptsize $0$} ;
\draw (0,1.15) node{\scriptsize $1$} ;
\draw (1,-0.55) node{\scriptsize $2$} ;
\draw (-1,-0.55) node{\scriptsize $3$} ;
\draw (1,1) node{$e_p$} ;
\draw (0,-0.7) node{$e_{p+1}$} ;
\draw (-0.5,0.3) node{$f_p$} ;
\draw (0.5,0.3) node{$f_{p-1}$} ;

\draw (0,-1.4) node{\large $T_p$} ;

\path [draw=black,postaction={on each segment={mid arrow=black}}]
(0,0)--(-1.732/2,-0.5);

\path [draw=black,postaction={on each segment={mid arrow =black}}]
(0,0)--(0,1);

\path [draw=black,postaction={on each segment={mid arrow =black}}]
(0,0)--(1.732/2,-0.5);

\draw[->](1.732/2,-0.5) arc (-30:-90:1);
\draw (0,-1) arc (-90:-150:1);

\draw[->](0,1) arc (90:30:1);
\draw (1.732/2,0.5) arc (30:-30:1);

\draw[->](0,1) arc (90:150:1);
\draw (-1.732/2,0.5) arc (150:210:1);

%
\begin{scope}[xshift=0cm,yshift=0cm,rotate=0,scale=1/1.5]
\draw (0.2,0.5) circle (0.15) node{\scriptsize $0$};
\end{scope}

\begin{scope}[xshift=0cm,yshift=0cm,rotate=0,scale=1/1.5]
\draw (0.2,-1.7) circle (0.15) node{\scriptsize $0$};
\end{scope}

\begin{scope}[xshift=0cm,yshift=0cm,rotate=0,scale=1/1.5]
\draw (-1.5,1) circle (0.15) node{\scriptsize $p$};
\end{scope}

\begin{scope}[xshift=0cm,yshift=0cm,rotate=0,scale=1/1.5]
\draw (1.5,1)  node{\tiny $p-1$};
\node[draw,ellipse] (S) at(1.5,1) {\ \ \ };
\end{scope}

\begin{scope}[xshift=0cm,yshift=0cm,rotate=0,scale=1/1.5]
\draw (-0.6,-0.6)  node{\tiny $p+1$};
\node[draw,ellipse] (S) at(-0.6,-0.6) {\ \ \ };
\end{scope}

\begin{scope}[xshift=0cm,yshift=0cm,rotate=0,scale=1/1.5]
\draw (0.6,-0.6) circle (0.15) node{\scriptsize $p$};
\end{scope}

\end{scope}


\begin{scope}[xshift=0cm,yshift=-5cm,rotate=0,scale=1.5]

\draw (0,-0.15) node{\scriptsize $0$} ;
\draw (0,1.15) node{\scriptsize $1$} ;
\draw (1,-0.55) node{\scriptsize $2$} ;
\draw (-1,-0.55) node{\scriptsize $3$} ;
\draw (1,1) node{$r$} ;
\draw (0,-0.6) node{$v$} ;
\draw (-0.5,0.3) node{$s'$} ;
\draw (0.5,0.3) node{$g$} ;

\draw (0,-1.4) node{\large $U$} ;

\path [draw=black,postaction={on each segment={mid arrow =black}}]
(0,0)--(0,1);

\draw[->] (0,0)--(1.732*0.3,-1*0.3);
\draw (1.732*0.3,-1*0.3)--(1.732/2,-1/2);




\begin{scope}[xshift=0cm,yshift=0cm,rotate=0,scale=1/1.5]
\draw (0.6,-0.7)  node{\tiny $p+1$};
\node[draw,ellipse] (S) at(0.6,-0.7) {\ \ \ };
\end{scope}

\begin{scope}[xshift=0cm,yshift=0cm,rotate=0,scale=1/1.5]
\draw (1.5,1) circle (0.15) node{\scriptsize $p$};
\end{scope}

\draw[color=black][->](1.732/2,-0.5) arc (-30:-90:1);
\draw[color=black] (-1.732/2,-0.5) arc (-150:-87:1);

\draw[color=black][->](0,1) arc (90:30:1);
\draw[color=black] (1.732/2,0.5) arc (30:-30:1);




\draw[color=black](-1.732/2*0.6,-0.5*0.6) -- (0,0);
\draw[color=black](-1.732/4,-0.25) -- (-1.732/2,-0.5);

\begin{scope}[xshift=-0.386cm, yshift=-0.223cm, rotate=120, scale=0.1]
\draw (1,0) -- (0,1) -- (-1,0) -- (1,0);
\end{scope}

\begin{scope}[xshift=-0.476cm, yshift=-0.275cm, rotate=120, scale=0.1]
\draw (1,0) -- (0,1) -- (-1,0) -- (1,0);
\end{scope}

\draw[color=black](0,1) arc (90:150:1);
\draw[color=black] (-1.732/2,0.5) arc (150:210:1);


\begin{scope}[xshift=-0.838cm, yshift=0.544cm, rotate=149, scale=0.1]
\draw (1,0) -- (0,1) -- (-1,0) -- (1,0);
\end{scope}

\begin{scope}[xshift=-0.887cm, yshift=0.461cm, rotate=151, scale=0.1]
\draw (1,0) -- (0,1) -- (-1,0) -- (1,0);
\end{scope}

\end{scope}

\begin{scope}[xshift=4cm,yshift=-5cm,rotate=0,scale=1.5]

\draw (0,-0.15) node{\scriptsize $0$} ;
\draw (0,1.15) node{\scriptsize $1$} ;
\draw (1,-0.55) node{\scriptsize $3$} ;
\draw (-1,-0.55) node{\scriptsize $2$} ;
\draw (1,1) node{$s$} ;
\draw (0,-0.6) node{$g$} ;
\draw (-0.5,0.3) node{$u$} ;
\draw (0.5,0.3) node{$f_p$} ;

\draw (0,-1.4) node{\large $V$} ;

\begin{scope}[xshift=0cm,yshift=0cm,rotate=0,scale=1/1.5]
\draw (0.6,-0.7)  node{\tiny $p+1$};
\node[draw,ellipse] (S) at(0.6,-0.7) {\ \ \ };
\end{scope}

\begin{scope}[xshift=0cm,yshift=0cm,rotate=0,scale=1/1.5]
\draw (1.5,1) circle (0.15) node{\scriptsize $p$};
\end{scope}

%
\begin{scope}[xshift=0cm,yshift=0cm,rotate=0,scale=1/1.5]
\draw (0.2,0.5) circle (0.15) node{\scriptsize $0$};
\end{scope}

%
\begin{scope}[xshift=0cm,yshift=0cm,rotate=0,scale=1/1.5]
\draw (0.2,-1.7) circle (0.15) node{\scriptsize $p$};
\end{scope}

\path [draw=black,postaction={on each segment={mid arrow =black}}]
(0,0)--(-1.732/2,-0.5);

\draw[color=black][->](0,0)--(1.732/2*0.6,-0.5*0.6);
\draw[color=black](1.732/4,-0.25) -- (1.732/2,-0.5);

\draw[color=black](1.732/2,-0.5) arc (-30:-90:1);
\draw[color=black][->] (-1.732/2,-0.5) arc (-150:-87:1);

\path [draw=black,postaction={on each segment={mid arrow  =black}}]
(0,0)--(0,1);

\draw[color=black][->](0,1) arc (90:30:1);
\draw[color=black] (1.732/2,0.5) arc (30:-30:1);

\draw[color=black][->](0,1) arc (90:150:1);
\draw[color=black] (-1.732/2,0.5) arc (150:210:1);
\end{scope}

\begin{scope}[xshift=8cm,yshift=-5cm,rotate=0,scale=1.5]

\draw (0,-0.15) node{\scriptsize $0$} ;
\draw (0,1.15) node{\scriptsize $1$} ;
\draw (1,-0.55) node{\scriptsize $3$} ;
\draw (-1,-0.55) node{\scriptsize $2$} ;
\draw (1,1) node{$u$} ;
\draw (0,-0.6) node{$v$} ;
\draw (-0.5,0.3) node{$e_{p+1}$} ;
\draw (0.5,0.3) node{$r$} ;

\draw (0,-1.4) node{\large $W$} ;

\draw[->](0,0)--(0,0.6);
\draw(0,0.6)--(0,1);

\begin{scope}[xshift=0cm,yshift=0cm,rotate=0,scale=1/1.5]
\draw (-0.6,-0.7)  node{\tiny $p+1$};
\node[draw,ellipse] (S) at(-0.6,-0.7) {\ \ \ };
\end{scope}

\begin{scope}[xshift=0cm,yshift=0cm,rotate=0,scale=1/1.5]
\draw (-1.5,1) circle (0.15) node{\scriptsize $0$};
\end{scope}

%
\begin{scope}[xshift=0cm,yshift=0cm,rotate=0,scale=1/1.5]
\draw (0.2,0.5) circle (0.15) node{\scriptsize $p$};
\end{scope}

\draw[color=black][->](0,0)--(-1.732/2*0.6,-0.5*0.6);
\draw[color=black](-1.732/4,-0.25) -- (-1.732/2,-0.5);


\draw[color=black][-<](1.732/2,-0.5) arc (-30:-90:1);
\draw[color=black] (-1.732/2,-0.5) arc (-150:-87:1);

\draw[color=black][->](0,1) arc (90:30:1);
\draw[color=black] (1.732/2,0.5) arc (30:-30:1);

\draw[color=black][->](0,1) arc (90:150:1);
\draw[color=black] (-1.732/2,0.5) arc (150:210:1);


\draw[color=black](1.732/2*0.6,-0.5*0.6) -- (0,0);
\draw[color=black](1.732/4,-0.25) -- (1.732/2,-0.5);

\begin{scope}[xshift=0.386cm, yshift=-0.223cm, rotate=-120, scale=0.1]
\draw (1,0) -- (0,1) -- (-1,0) -- (1,0);
\end{scope}

\begin{scope}[xshift=0.476cm, yshift=-0.275cm, rotate=-120, scale=0.1]
\draw (1,0) -- (0,1) -- (-1,0) -- (1,0);
\end{scope}

\end{scope}

\begin{scope}[xshift=12cm,yshift=-5cm,rotate=0,scale=1.5]

\draw (0,-0.15) node{\scriptsize $0$} ;
\draw (0,1.15) node{\scriptsize $1$} ;
\draw (1,-0.55) node{\scriptsize $2$} ;
\draw (-1,-0.55) node{\scriptsize $3$} ;
\draw (1,1) node{$m$} ;
\draw (0,-0.6) node{$m$} ;
\draw (-0.5,0.3) node{$s$} ;
\draw (0.5,0.3) node{$s'$} ;

\draw (0,-1.4) node{\large $Z$} ;

\path [draw=black,postaction={on each segment={mid arrow =black}}]
(0,0)--(0,1);

\path [draw=black, postaction={on each segment={mid arrow =black}}]
(0,0)--(-1.732/2,-0.5);


\begin{scope}[xshift=0cm,yshift=0cm,rotate=0,scale=1/1.5]
\draw (-0.6,-0.6) circle (0.15) node{\scriptsize $p$};
\end{scope}

\begin{scope}[xshift=0cm,yshift=0cm,rotate=0,scale=1/1.5]
\draw (-1.5,1) circle (0.15) node{\scriptsize $p$};
\end{scope}

\draw[very thick,color=blue][->](1.732/2,-0.5) arc (-30:-90:1);
\draw[very thick,color=blue] (-1.732/2,-0.5) arc (-150:-87:1);


\draw[color=black][->](0,1) arc (90:150:1);
\draw[color=black] (-1.732/2,0.5) arc (150:210:1);


\draw[color=black](1.732/2*0.6,-0.5*0.6) -- (0,0);
\draw[color=black](1.732/4,-0.25) -- (1.732/2,-0.5);

\begin{scope}[xshift=0.386cm, yshift=-0.223cm, rotate=-120, scale=0.1]
\draw (1,0) -- (0,1) -- (-1,0) -- (1,0);
\end{scope}

\begin{scope}[xshift=0.476cm, yshift=-0.275cm, rotate=-120, scale=0.1]
\draw (1,0) -- (0,1) -- (-1,0) -- (1,0);
\end{scope}

\draw[color=black](0,1) arc (90:30:1);
\draw[color=black] (1.732/2,0.5) arc (30:-30:1);


\begin{scope}[xshift=0.838cm, yshift=0.544cm, rotate=-151, scale=0.1]
\draw (1,0) -- (0,1) -- (-1,0) -- (1,0);
\end{scope}

\begin{scope}[xshift=0.887cm, yshift=0.461cm, rotate=-153, scale=0.1]
\draw (1,0) -- (0,1) -- (-1,0) -- (1,0);
\end{scope}

\end{scope}
\end{tikzpicture}
\caption{An H-triangulation for $(S^3,K_n)$, $n$ even, $n \geqslant 4$, with $p=\frac{n-2}{2}$} \label{fig:H:trig:even}
\end{figure}

In the H-triangulation of Figure \ref{fig:H:trig:even} there are
\begin{itemize}
\item $1$ common vertex,
\item $p+5 = \frac{n+8}{2}$ edges 
(simple arrow $\overrightarrow{\eta_s}$, double white triangle arrow $\overrightarrow{\eta_d}$, blue simple arrow $\overrightarrow{K_n}$, and the simple arrows $\overrightarrow{\eta_0}, \ldots, \overrightarrow{\eta_{p+1}}$ indexed by $0, \ldots p+1$ in circles)
\item $2p+8 = n+6$ faces ($e_1, \ldots, e_{p+1}, f_1, \ldots, f_{p},g, m,r,s,s',u,v  $),
\item $p+4 = \frac{n+6}{2}$ tetrahedra ($T_1, \ldots, T_{p},  U,  V, W, Z$) .
\end{itemize}

Finally, by collapsing the tetrahedron $Z$ (like in the previous section) we obtain the ideal triangulation of $S^3 \setminus K_n$ described in Figure \ref{fig:id:trig:even}. We identified the face $s'$ with $s$ and the white triangle arrow with the arrow circled by $p$.

\begin{figure}[!h]
\begin{tikzpicture}

\begin{scope}[xshift=1cm,yshift=0cm,rotate=0,scale=1.5]

\draw (0,-0.15) node{\scriptsize $0$} ;
\draw (0,1.15) node{\scriptsize $1$} ;
\draw (1,-0.55) node{\scriptsize $2$} ;
\draw (-1,-0.55) node{\scriptsize $3$} ;
\draw (1,1) node{$e_1$} ;
\draw (0,-0.6) node{$e_2$} ;
\draw (-0.5,0.3) node{$f_1$} ;
\draw (0.5,0.3) node{$e_1$} ;

\draw (0,-1.4) node{\large $T_1$} ;

\path [draw=black,postaction={on each segment={mid arrow=black}}]
(0,0)--(-1.732/2,-0.5);

\path [draw=black,postaction={on each segment={mid arrow =black}}]
(0,0)--(0,1);

\path [draw=black,postaction={on each segment={mid arrow =black}}]
(0,0)--(1.732/2,-0.5);

\draw[->](1.732/2,-0.5) arc (-30:-90:1);
\draw (0,-1) arc (-90:-150:1);

\draw[->](0,1) arc (90:30:1);
\draw (1.732/2,0.5) arc (30:-30:1);

\draw[->](0,1) arc (90:150:1);
\draw (-1.732/2,0.5) arc (150:210:1);

%
\begin{scope}[xshift=0cm,yshift=0cm,rotate=0,scale=1/1.5]
\draw (1.5,1) circle (0.15) node{\scriptsize $0$};
\end{scope}

%
\begin{scope}[xshift=0cm,yshift=0cm,rotate=0,scale=1/1.5]
\draw (0.2,0.5) circle (0.15) node{\scriptsize $0$};
\end{scope}

%
\begin{scope}[xshift=0cm,yshift=0cm,rotate=0,scale=1/1.5]
\draw (0.2,-1.7) circle (0.15) node{\scriptsize $0$};
\end{scope}

\begin{scope}[xshift=0cm,yshift=0cm,rotate=0,scale=1/1.5]
\draw (-0.6,-0.6) circle (0.15) node{\scriptsize $2$};
\end{scope}

\begin{scope}[xshift=0cm,yshift=0cm,rotate=0,scale=1/1.5]
\draw (-1.5,1) circle (0.15) node{\scriptsize $1$};
\end{scope}

\begin{scope}[xshift=0cm,yshift=0cm,rotate=0,scale=1/1.5]
\draw (0.6,-0.6) circle (0.15) node{\scriptsize $1$};
\end{scope}

\end{scope}

\draw (3.5,0) node{$\ldots$} ;
\draw (8.5,0) node{$\ldots$} ;

\begin{scope}[xshift=6cm,yshift=0cm,rotate=0,scale=1.5]

\draw (0,-0.15) node{\scriptsize $0$} ;
\draw (0,1.15) node{\scriptsize $1$} ;
\draw (1,-0.55) node{\scriptsize $2$} ;
\draw (-1,-0.55) node{\scriptsize $3$} ;
\draw (1,1) node{$e_k$} ;
\draw (0,-0.7) node{$e_{k+1}$} ;
\draw (-0.5,0.3) node{$f_k$} ;
\draw (0.5,0.3) node{$f_{k-1}$} ;

\draw (0,-1.4) node{\large $T_k$} ;
\draw (0,-1.7) node{\tiny $(2\leqslant k \leqslant p-1)$} ;

\path [draw=black,postaction={on each segment={mid arrow=black}}]
(0,0)--(-1.732/2,-0.5);

\path [draw=black,postaction={on each segment={mid arrow =black}}]
(0,0)--(0,1);

\path [draw=black,postaction={on each segment={mid arrow =black}}]
(0,0)--(1.732/2,-0.5);

\draw[->](1.732/2,-0.5) arc (-30:-90:1);
\draw (0,-1) arc (-90:-150:1);

\draw[->](0,1) arc (90:30:1);
\draw (1.732/2,0.5) arc (30:-30:1);

\draw[->](0,1) arc (90:150:1);
\draw (-1.732/2,0.5) arc (150:210:1);

%
\begin{scope}[xshift=0cm,yshift=0cm,rotate=0,scale=1/1.5]
\draw (0.2,0.5) circle (0.15) node{\scriptsize $0$};
\end{scope}

%
\begin{scope}[xshift=0cm,yshift=0cm,rotate=0,scale=1/1.5]
\draw (0.2,-1.7) circle (0.15) node{\scriptsize $0$};
\end{scope}

\begin{scope}[xshift=0cm,yshift=0cm,rotate=0,scale=1/1.5]
\draw (-1.5,1) circle (0.15) node{\scriptsize $k$};
\end{scope}

\begin{scope}[xshift=0cm,yshift=0cm,rotate=0,scale=1/1.5]
\draw (1.5,1)  node{\tiny $k-1$};
\node[draw,ellipse] (S) at(1.5,1) {\ \ \ };
\end{scope}

\begin{scope}[xshift=0cm,yshift=0cm,rotate=0,scale=1/1.5]
\draw (-0.6,-0.6)  node{\tiny $k+1$};
\node[draw,ellipse] (S) at(-0.6,-0.6) {\ \ \ };
\end{scope}

\begin{scope}[xshift=0cm,yshift=0cm,rotate=0,scale=1/1.5]
\draw (0.6,-0.6) circle (0.15) node{\scriptsize $k$};
\end{scope}

\end{scope}

\begin{scope}[xshift=11cm,yshift=0cm,rotate=0,scale=1.5]

\draw (0,-0.15) node{\scriptsize $0$} ;
\draw (0,1.15) node{\scriptsize $1$} ;
\draw (1,-0.55) node{\scriptsize $2$} ;
\draw (-1,-0.55) node{\scriptsize $3$} ;
\draw (1,1) node{$e_p$} ;
\draw (0,-0.7) node{$e_{p+1}$} ;
\draw (-0.5,0.3) node{$f_p$} ;
\draw (0.5,0.3) node{$f_{p-1}$} ;

\draw (0,-1.4) node{\large $T_p$} ;

\path [draw=black,postaction={on each segment={mid arrow=black}}]
(0,0)--(-1.732/2,-0.5);

\path [draw=black,postaction={on each segment={mid arrow =black}}]
(0,0)--(0,1);

\path [draw=black,postaction={on each segment={mid arrow =black}}]
(0,0)--(1.732/2,-0.5);

\draw[->](1.732/2,-0.5) arc (-30:-90:1);
\draw (0,-1) arc (-90:-150:1);

\draw[->](0,1) arc (90:30:1);
\draw (1.732/2,0.5) arc (30:-30:1);

\draw[->](0,1) arc (90:150:1);
\draw (-1.732/2,0.5) arc (150:210:1);

%
\begin{scope}[xshift=0cm,yshift=0cm,rotate=0,scale=1/1.5]
\draw (0.2,0.5) circle (0.15) node{\scriptsize $0$};
\end{scope}

\begin{scope}[xshift=0cm,yshift=0cm,rotate=0,scale=1/1.5]
\draw (0.2,-1.7) circle (0.15) node{\scriptsize $0$};
\end{scope}

\begin{scope}[xshift=0cm,yshift=0cm,rotate=0,scale=1/1.5]
\draw (-1.5,1) circle (0.15) node{\scriptsize $p$};
\end{scope}

\begin{scope}[xshift=0cm,yshift=0cm,rotate=0,scale=1/1.5]
\draw (1.5,1)  node{\tiny $p-1$};
\node[draw,ellipse] (S) at(1.5,1) {\ \ \ };
\end{scope}

\begin{scope}[xshift=0cm,yshift=0cm,rotate=0,scale=1/1.5]
\draw (-0.6,-0.6)  node{\tiny $p+1$};
\node[draw,ellipse] (S) at(-0.6,-0.6) {\ \ \ };
\end{scope}

\begin{scope}[xshift=0cm,yshift=0cm,rotate=0,scale=1/1.5]
\draw (0.6,-0.6) circle (0.15) node{\scriptsize $p$};
\end{scope}

\end{scope}


\begin{scope}[xshift=1cm,yshift=-5cm,rotate=0,scale=1.5]

\draw (0,-0.15) node{\scriptsize $0$} ;
\draw (0,1.15) node{\scriptsize $1$} ;
\draw (1,-0.55) node{\scriptsize $2$} ;
\draw (-1,-0.55) node{\scriptsize $3$} ;
\draw (1,1) node{$r$} ;
\draw (0,-0.6) node{$v$} ;
\draw (-0.5,0.3) node{$s$} ;
\draw (0.5,0.3) node{$g$} ;

\draw (0,-1.4) node{\large $U$} ;

\path [draw=black,postaction={on each segment={mid arrow =black}}]
(0,0)--(0,1);

\draw[->] (0,0)--(1.732*0.3,-1*0.3);
\draw (1.732*0.3,-1*0.3)--(1.732/2,-1/2);

\draw[->] (0,0)--(-1.732*0.3,-1*0.3);
\draw (-1.732*0.3,-1*0.3)--(-1.732/2,-1/2);

\begin{scope}[xshift=0cm,yshift=0cm,rotate=0,scale=1/1.5]
\draw (0.6,-0.7)  node{\tiny $p+1$};
\node[draw,ellipse] (S) at(0.6,-0.7) {\ \ \ };
\end{scope}

\begin{scope}[xshift=0cm,yshift=0cm,rotate=0,scale=1/1.5]
\draw (1.5,1) circle (0.15) node{\scriptsize $p$};
\end{scope}

\begin{scope}[xshift=0cm,yshift=0cm,rotate=0,scale=1/1.5]
\draw (-0.6,-0.6) circle (0.15) node{\scriptsize $p$};
\end{scope}

\begin{scope}[xshift=0cm,yshift=0cm,rotate=0,scale=1/1.5]
\draw (-1.5,1) circle (0.15) node{\scriptsize $p$};
\end{scope}

\draw[color=black][->](1.732/2,-0.5) arc (-30:-90:1);
\draw[color=black] (-1.732/2,-0.5) arc (-150:-87:1);

\draw[color=black][->](0,1) arc (90:30:1);
\draw[color=black] (1.732/2,0.5) arc (30:-30:1);

\draw[color=black][->](0,1) arc (90:150:1);
\draw[color=black] (-1.732/2,0.5) arc (150:210:1);

\end{scope}

\begin{scope}[xshift=6cm,yshift=-5cm,rotate=0,scale=1.5]

\draw (0,-0.15) node{\scriptsize $0$} ;
\draw (0,1.15) node{\scriptsize $1$} ;
\draw (1,-0.55) node{\scriptsize $3$} ;
\draw (-1,-0.55) node{\scriptsize $2$} ;
\draw (1,1) node{$s$} ;
\draw (0,-0.6) node{$g$} ;
\draw (-0.5,0.3) node{$u$} ;
\draw (0.5,0.3) node{$f_p$} ;

\draw (0,-1.4) node{\large $V$} ;

\begin{scope}[xshift=0cm,yshift=0cm,rotate=0,scale=1/1.5]
\draw (0.6,-0.7)  node{\tiny $p+1$};
\node[draw,ellipse] (S) at(0.6,-0.7) {\ \ \ };
\end{scope}

\begin{scope}[xshift=0cm,yshift=0cm,rotate=0,scale=1/1.5]
\draw (1.5,1) circle (0.15) node{\scriptsize $p$};
\end{scope}

%
\begin{scope}[xshift=0cm,yshift=0cm,rotate=0,scale=1/1.5]
\draw (0.2,0.5) circle (0.15) node{\scriptsize $0$};
\end{scope}

%
\begin{scope}[xshift=0cm,yshift=0cm,rotate=0,scale=1/1.5]
\draw (0.2,-1.7) circle (0.15) node{\scriptsize $p$};
\end{scope}

\path [draw=black,postaction={on each segment={mid arrow =black}}]
(0,0)--(-1.732/2,-0.5);

\draw[color=black][->](0,0)--(1.732/2*0.6,-0.5*0.6);
\draw[color=black](1.732/4,-0.25) -- (1.732/2,-0.5);

\draw[color=black](1.732/2,-0.5) arc (-30:-90:1);
\draw[color=black][->] (-1.732/2,-0.5) arc (-150:-87:1);

\path [draw=black,postaction={on each segment={mid arrow  =black}}]
(0,0)--(0,1);

\draw[color=black][->](0,1) arc (90:30:1);
\draw[color=black] (1.732/2,0.5) arc (30:-30:1);

\draw[color=black][->](0,1) arc (90:150:1);
\draw[color=black] (-1.732/2,0.5) arc (150:210:1);
\end{scope}

\begin{scope}[xshift=11cm,yshift=-5cm,rotate=0,scale=1.5]

\draw (0,-0.15) node{\scriptsize $0$} ;
\draw (0,1.15) node{\scriptsize $1$} ;
\draw (1,-0.55) node{\scriptsize $3$} ;
\draw (-1,-0.55) node{\scriptsize $2$} ;
\draw (1,1) node{$u$} ;
\draw (0,-0.6) node{$v$} ;
\draw (-0.5,0.3) node{$e_{p+1}$} ;
\draw (0.5,0.3) node{$r$} ;

\draw (0,-1.4) node{\large $W$} ;

\draw[->](0,0)--(0,0.6);
\draw(0,0.6)--(0,1);

\begin{scope}[xshift=0cm,yshift=0cm,rotate=0,scale=1/1.5]
\draw (0.6,-0.6) circle (0.15) node{\scriptsize $p$};
\end{scope}

\begin{scope}[xshift=0cm,yshift=0cm,rotate=0,scale=1/1.5]
\draw (-0.6,-0.7)  node{\tiny $p+1$};
\node[draw,ellipse] (S) at(-0.6,-0.7) {\ \ \ };
\end{scope}

\begin{scope}[xshift=0cm,yshift=0cm,rotate=0,scale=1/1.5]
\draw (-1.5,1) circle (0.15) node{\scriptsize $0$};
\end{scope}

%
\begin{scope}[xshift=0cm,yshift=0cm,rotate=0,scale=1/1.5]
\draw (0.2,0.5) circle (0.15) node{\scriptsize $p$};
\end{scope}

\draw[color=black][->](0,0)--(-1.732/2*0.6,-0.5*0.6);
\draw[color=black](-1.732/4,-0.25) -- (-1.732/2,-0.5);

\draw[color=black,<-](1.732/2*0.6,-0.5*0.6) -- (0,0);
\draw[color=black](1.732/4,-0.25) -- (1.732/2,-0.5);

\draw[color=black][-<](1.732/2,-0.5) arc (-30:-90:1);
\draw[color=black] (-1.732/2,-0.5) arc (-150:-87:1);

\draw[color=black][->](0,1) arc (90:30:1);
\draw[color=black] (1.732/2,0.5) arc (30:-30:1);

\draw[color=black][->](0,1) arc (90:150:1);
\draw[color=black] (-1.732/2,0.5) arc (150:210:1);
\end{scope}
\end{tikzpicture}
\caption{An ideal triangulation for $S^3 \setminus K_n$, $n$ even, $n \geqslant 4$, with $p=\frac{n-2}{2}$} \label{fig:id:trig:even}
\end{figure}
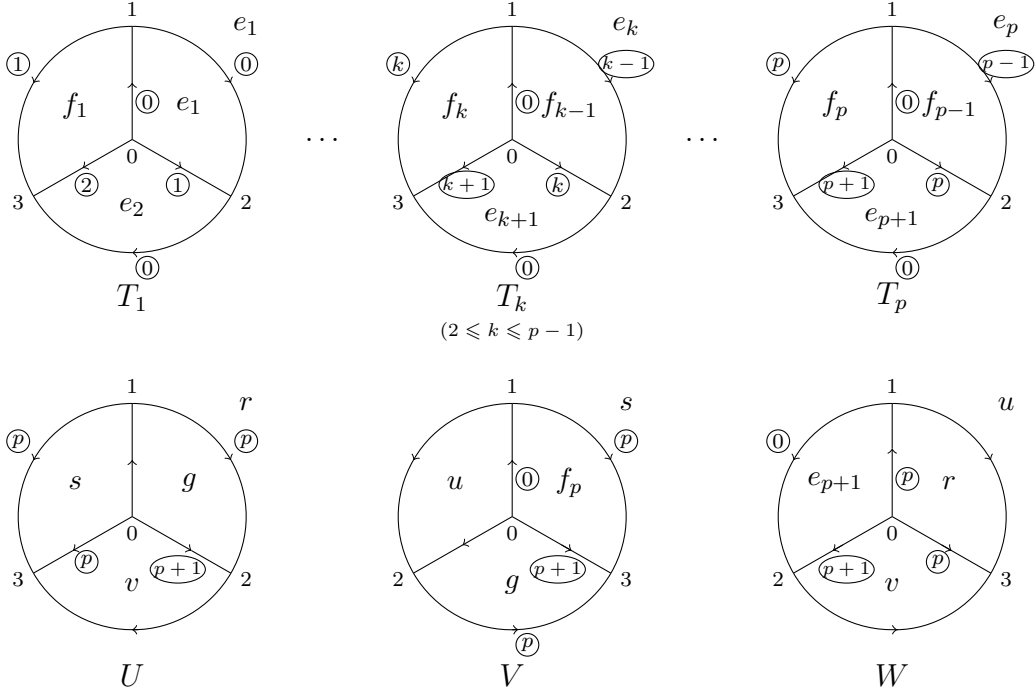

In Figure \ref{fig:id:trig:even} there are
\begin{itemize}
\item $1$ common vertex,
\item $p+3 = \frac{n+4}{2}$ edges 
(simple arrow $\overrightarrow{\eta_s}$ and the simple arrows $\overrightarrow{\eta_0}, \ldots, \overrightarrow{\eta_{p+1}}$ indexed by $0, \ldots p+1$ in circles),
\item $2p+6 = n+4$ faces ($e_1, \ldots, e_{p+1}, f_1, \ldots, f_{p}, g,r,s,u,v  $),
\item $p+3 = \frac{n+4}{2}$ tetrahedra ($T_1, \ldots, T_{p}, U,  V, W$).
\end{itemize}

\begin{remark}\label{rem:K2}
When $n=2$, i.e.\ $p=0$ here, the triangulations of Figures \ref{fig:H:trig:even} and \ref{fig:id:trig:even} are still correct (with the convention $f_0=e_1$), one just needs to stop the previous reasoning at Figure \ref{fig:even:Htriang:bigon:trick} (b) and collapse the bigon $G$ into a segment.

In this case, the ideal triangulation $X_2$ of the figure-eight knot complement $S^3 \setminus K_2$ described in Figure \ref{fig:id:trig:even} has three tetrahedra, although it is well-known that this knot complement has Matveev complexity $2$. The ideal triangulations of Figures \ref{fig:id:tri:41:complement} and \ref{fig:id:trig:even} are actually related by a Pachner $3-2$ move. 
\end{remark}

\subsection{Gluing equations and proving geometricity}\label{sub:even:geom}

\begin{figure}[!h]
	\centering
	\begin{tikzpicture}[every path/.style={string ,black}]
	
	\begin{scope}[scale=0.7]
	
	\draw[color=blue] (-7.5,7.5) node {$3_U$};
	\draw (-8,9.6) node {$r$};
	\draw (-6.3,6) node {$s$};
	\draw (-7.2,4) node {$v$};
	\draw[color=brown] (-9.5,9.6) node {$a$};
	\draw[color=brown] (-6.3,9.6) node {$b$};
	\draw[color=brown] (-6.3,1.5) node {$c$};

	\draw[color=blue] (-8.5,3) node {$3_W$};
	\draw (-8,0.4) node {$r$};
	\draw (-9.8,4.5) node {$u$};
	\draw (-8.2,4) node {$v$};
	\draw[color=brown] (-9.8,8.5) node {$a$};
	\draw[color=brown] (-9.8,0.4) node {$b$};
	\draw[color=brown] (-6.4,0.3) node {$c$};

	\draw[color=blue] (-4.5,7) node {$3_V$};
	\draw (-4.5,8.6) node {$g$};
	\draw (-5.7,6) node {$s$};
	\draw (-4.6,5) node {$f_p$};
	\draw[color=brown] (-5.8,9.5) node {$a$};
	\draw[color=brown] (-5.7,1.5) node {$b$};
	\draw[color=brown] (-3.5,8) node {$c$};

	\draw[color=blue] (-2.6,6.5) node {$3_p$};
	\draw (-4.2,4) node {$f_p$};
	\draw (-1.6,8) node {$e_{p+1}$};
	\draw (-3,5) node {$e_p$};
	\draw[color=brown] (-0.7,8.5) node {$a$};
	
	\draw[color=blue] (-0.7,9.05) node {\tiny $2_W$};
	\draw (-1.3,8.6) node {\tiny $e_{p+1}$};
	\draw (-0.18,9.4) node {\tiny $u$};
	\draw (-1.2,9) node {\tiny $v$};
	\draw[color=brown] (-0.15,9.65) node {\tiny $a$};
	\draw[color=brown] (-0.15,9.1) node {\tiny $b$};
	\draw[color=brown] (-2,8.5) node {\tiny $c$};

	\draw[color=blue] (-2.7,9) node {$2_U$};
	\draw (-4,9) node {$g$};
	\draw (-3,9.7) node {$r$};
	\draw (-1.7,9.1) node {$v$};
	\draw[color=brown] (-0.7,9.8) node {$a$};
	\draw[color=brown] (-3,8.4) node {$b$};
	\draw[color=brown] (-4.9,9.6) node {$c$};

	\draw[color=blue] (-3.9,2.3) node {\tiny $3_{p-1}$};
	\draw (-3,4) node {\tiny $e_p$};
	\draw (-3.45,3) node {\tiny $e_{p-1}$};
	\draw (-4.35,1.7) node {\tiny $f_{p-1}$};
	\draw[color=brown] (-1.2,6.95) node {\tiny $a$};
	
	\draw[color=blue] (-3.4,1.2) node {$2_{p}$};
	\draw (-2,1.4) node {\tiny $e_p$};
	\draw (-3,0.7) node {\tiny $e_{p+1}$};
	\draw (-4.2,0.9) node {\tiny $f_{p-1}$};
	\draw[color=brown] (-1,1.05) node {\tiny $a$};
	
	\draw[color=blue] (-1,0.4) node {$1_W$};
	\draw (-2,0.15) node {\tiny $r$};
	\draw (-2.8,0.3) node {\tiny $e_{p+1}$};
	\draw (-0.3,0.5) node {\tiny $u$};
	\draw[color=brown] (-4.2,0.15) node {\tiny $a$};
	\draw[color=brown] (-.2,.2) node {\tiny $b$};
	\draw[color=brown] (-.2,.75) node {\tiny $c$};

	\draw[color=blue] (-0.7,4) node {$2_{1}$};
	\draw (-0.5,5.9) node {\tiny $e_1$};
	\draw (-0.3,4.7) node {\tiny $e_1$};
	\draw (-0.5,3) node {\tiny $e_2$};
	\draw[color=brown] (-0.2,2) node {\tiny $a$};
	
	\draw[color=blue] (-1.4,4.5) node {$3_1$};
	\draw (-1,5.7) node {\tiny $e_1$};
	\draw (-1.35,5) node {\tiny $e_2$};
	\draw (-1.65,3.9) node {\tiny $f_{1}$};
	\draw[color=brown] (-0.65,7) node {\tiny $a$};
	
	\draw[color=blue] (-1.7,2.9) node {$2_2$};
	\draw (-1.3,2.3) node {\tiny $e_3$};
	\draw (-1,2.8) node {\tiny $e_2$};
	\draw (-1.5,3.45) node {\tiny $f_{1}$};
	\draw[color=brown] (-0.6,1.8) node {\tiny $a$};
	
	\draw[color=blue] (8.5,7.5) node {$2_V$};
	\draw (16-8,9.6) node {$s$};
	\draw (16-6.3,6) node {$u$};
	\draw (16-7.2,4) node {$g$};
	\draw[color=brown] (6.5,9.6) node {$a$};
	\draw[color=brown] (9.7,1.5) node {$b$};
	\draw[color=brown] (9.7,9.6) node {$c$};
	
	\draw[color=blue] (8,1.8) node {$1_U$};
	\draw (16-8,0.4) node {$s$};
	\draw (16-9.8,4) node {$r$};
	\draw (16-8.2,4) node {$g$};
	\draw[color=brown] (9.5,0.4) node {$a$};
	\draw[color=brown] (6.3,.4) node {$b$};
	\draw[color=brown] (6.3,8.5) node {$c$};

	\draw[color=blue] (5.1,3.2) node {$0_W$};
	\draw (5,1.4) node {$v$};
	\draw (5.7,4) node {$r$};
	\draw (5.3,5) node {$e_{p+1}$};
	\draw[color=brown] (5.8,8.5) node {$a$};
	\draw[color=brown] (4.4,2.1) node {$b$};
	\draw[color=brown] (5.7,.7) node {$c$};

	\draw[color=blue] (3.5,1) node {$0_U$};
	\draw (4.5,1) node {$v$};
	\draw (3,0.3) node {$s$};
	\draw (2.5,0.9) node {$g$};
	\draw[color=brown] (1,0.2) node {$a$};
	\draw[color=brown] (4,1.6) node {$b$};
	\draw[color=brown] (5.4,.2) node {$c$};

	\draw[color=blue] (3.5,3.5) node {$0_p$};
	\draw (4.2,5) node {$e_{p+1}$};
	\draw (2,2) node {$f_p$};
	\draw (3,4.5) node {$f_{p-1}$};
	\draw[color=brown] (1,1.7) node {$a$};
	
	\draw[color=blue] (3.3,8.5) node {$1_p$};
	\draw (4.5,9.1) node {$e_p$};
	\draw (3.4,9.3) node {$f_p$};
	\draw (2.1,8.8) node {$f_{p-1}$};
	\draw[color=brown] (1,8.9) node {$a$};
	
	\draw[color=blue] (4,7.7) node {\tiny $0_{p-1}$};
	\draw (4.5,8.45) node {\tiny $e_p$};
	\draw (3,6.5) node {\tiny $f_{p-2}$};
	\draw (4,7.2) node {\tiny $f_{p-1}$};
	\draw[color=brown] (1.8,4) node {\tiny $a$};
	
	\draw[color=blue] (0.7,0.8) node {$0_{V}$};
	\draw (0.2,0.6) node {\tiny $u$};
	\draw (2,1.3) node {\tiny $f_p$};
	\draw (1.35,0.85) node {\tiny $g$};
	\draw[color=brown] (0.2,0.9) node {\tiny $a$};
	\draw[color=brown] (.15,.25) node {\tiny $b$};
	\draw[color=brown] (2.6,1.5) node {\tiny $c$};

	\draw[color=blue] (1,9.6) node {$1_{V}$};
	\draw (0.2,9.5) node {\tiny $u$};
	\draw (1.8,9.5) node {\tiny $f_p$};
	\draw (2,9.85) node {\tiny $s$};
	\draw[color=brown] (0.2,9.2) node {\tiny $a$};
	\draw[color=brown] (4,9.85) node {\tiny $b$};
	\draw[color=brown] (.2,9.8) node {\tiny $c$};
	
	\draw[color=blue] (0.5,6) node {$1_{1}$};
	\draw (0.25,4.8) node {\tiny $e_1$};
	\draw (0.5,4) node {\tiny $e_1$};
	\draw (0.5,7) node {\tiny $f_1$};
	\draw[color=brown] (0.15,8.1) node {\tiny $a$};
	
	\draw[color=blue] (1.8,7.1) node {$1_{2}$};
	\draw (1.6,6.6) node {\tiny $e_2$};
	\draw (1.3,7.6) node {\tiny $f_2$};
	\draw (1.1,7) node {\tiny $f_1$};
	\draw[color=brown] (0.5,8.3) node {\tiny $a$};
	
	\draw[color=blue] (1.45,5.5) node {$0_1$};
	\draw (1.8,6.2) node {\tiny $e_2$};
	\draw (1.1,4.5) node {\tiny $e_1$};
	\draw (1.4,5) node {\tiny $f_1$};
	\draw[color=brown] (0.65,3) node {\tiny $a$};
	
	\draw (0,0)--(6,0)--(10,0)--(10,10)--(6,10)--(0,10)--(-3,8)--(-6,10)--(-10,10)--(-10,0)--(-6,0)--(0,0)--(4,2)--(6,0)--(6,10)--(10,0);
	\draw (-10,10)--(-6,0)--(-6,10);
	\draw (-6,0)--(-3,8)--(0,9)--(0,10);
	\draw (-6,0)--(0,9)--(6,10)--(4,2)--(0,1)--(0,0)--(-6,0)--(0,1)--(6,10);
	\draw (-6,0)--(-3.6,2)--(0,1)--(-2.4,3)--(-1.2,4)--(0,1)--(0,9)--(-1.2,4);
	\draw (-3.6,2)--(0,9)--(-2.4,3);
	\draw (0,9)--(1.2,6)--(0,1)--(2.4,7)--(0,9)--(3.6,8)--(0,1);
	\draw (1.2,6)--(2.4,7);
	\draw (3.6,8)--(6,10);
	\draw (0,10)--(-6,10);
	
	\draw (3-0.12,7.5-0.1) node[shape=circle,fill=black,scale=0.2] {};
	\draw (3,7.5) node[shape=circle,fill=black,scale=0.2] {};
	\draw (3+0.12,7.5+0.1) node[shape=circle,fill=black,scale=0.2] {};
	
	\draw (-3-0.12,2.5-0.1) node[shape=circle,fill=black,scale=0.2] {};
	\draw (-3,2.5) node[shape=circle,fill=black,scale=0.2] {};
	\draw (-3+0.12,2.5+0.1) node[shape=circle,fill=black,scale=0.2] {};
	
	\draw[color=violet,style=dashed,->] (7,0)--(7,5);
	\draw[color=violet,style=dashed, very thick] (7,5)--(7,10);
	\draw[color=violet] (7.7,1.1) node {\scriptsize $m_{X_n}$};

	\draw[color=teal,style=dashed, very thick,->] (-10,4)--(-8.5,2);	
	\draw[color=teal,style=dashed, very thick] (-8.5,2)--(-7,0);	
	\draw[color=teal,style=dashed, very thick,->] (-7,10)--(-6,9)--(-5.5,9);	
	\draw[color=teal,style=dashed, very thick]	(-5.5,9)--(-5,9.3)--(1.3,9.3)--(1.5,10);	
	\draw[color=teal,style=dashed, very thick,->] (1.5,0)--(2,.4)--(6,1)--(8,2.5)	;
	\draw[color=teal,style=dashed, very thick] (8,2.5)--(10,4)	;
	\draw[color=teal] (-8.5,1.5) node {\scriptsize $l_{X_n}$};

	\draw[color=red, very thick,->] (-10,0)--(-8,0);
	\draw[color=red, very thick] (-8,0)--(-6,0);
	\draw[color=red] (-8,-.5) node {(i)};
	
	\draw[color=red, very thick,->] (-6,0)--(-6,5);
	\draw[color=red, very thick] (-6,5)--(-6,10);
	\draw[color=red] (-5.4,5) node {(ii)};
	
	\draw[color=red, very thick,->] (-6,10)--(-3,10);
	\draw[color=red, very thick] (-3,10)--(0,10);
	\draw[color=red] (-3,10.5) node {(iii)};
	
	\draw[color=red, very thick,->] (0,0)--(3,0);
	\draw[color=red, very thick] (3,0)--(6,0);
	\draw[color=red] (3,-.5) node {(iv)};
	
	\draw[color=red, very thick,->] (6,0)--(6,6);
	\draw[color=red, very thick] (6,6)--(6,10);
	\draw[color=red] (5.5,6) node {(v)};
	
	\draw[color=red, very thick,->] (6,10)--(8,10);
	\draw[color=red, very thick] (8,10)--(10,10);
	\draw[color=red] (8,10.5) node {(vi)};

	\end{scope}
	
	\end{tikzpicture}
	\caption{Triangulation of the boundary torus for the truncation of $X_n$, $n$ even, with angles (brown), meridian curve $m_{X_n}$ (violet, dashed), longitude curve $l_{X_n}$ (green, dashed) and preferred longitude curve {$l_{X_n}^0=$} (i)$\cup \ldots \cup$(vi) (red).}\label{fig:trig:cusp:even}
\end{figure}
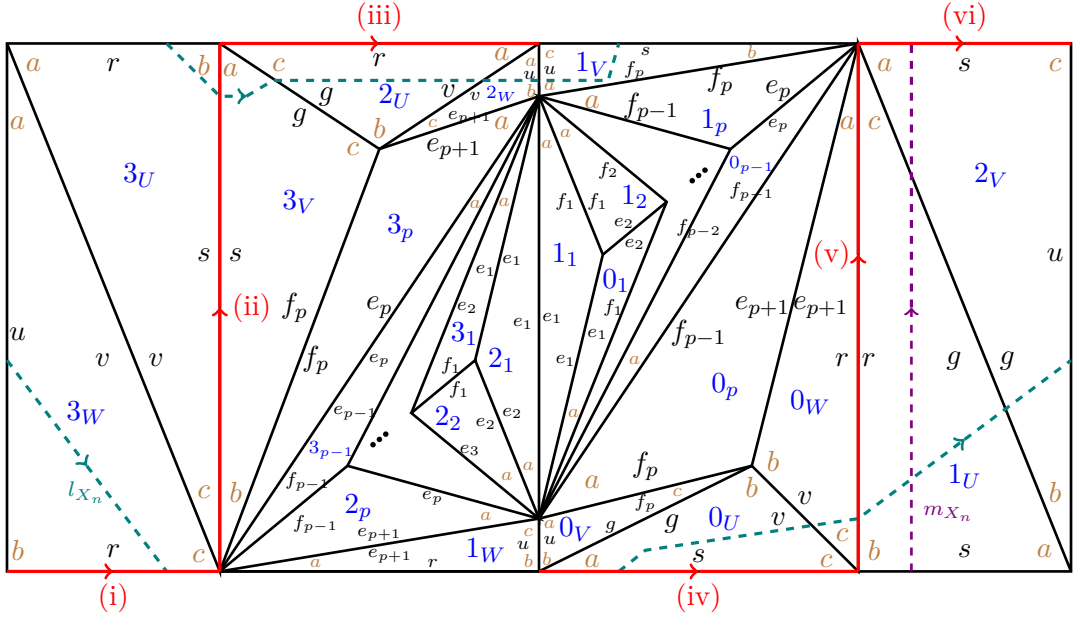

As in Section \ref{sub:complete:odd}, we constructed in Figure \ref{fig:trig:cusp:even} a triangulation of the boundary torus
$\partial \nu(K_n)$ from the datum in Figure \ref{fig:id:trig:even}. Here for the positive tetrahedra $T_1, \ldots, T_p$ we only indicated the brown $a$ angles for readability (the $b$ and $c$ follow clockwise). We also drew a meridian curve $m_{X_n}$ in violet and dashed, a longitude curve $l_{X_n}$ in green and dashed, and a preferred longitude curve {$l_{X_n}^0=$} (i)$\cup \ldots \cup$(vi) in red (one can check it is indeed a preferred longitude in Figure \ref{fig:longitude:even}).

\begin{figure}[!h]
	\includegraphics[scale=1.5]{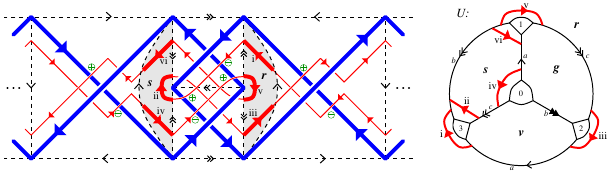}
	\caption{A preferred longitude {$l_{X_n}^0=$} (i) $\cup \ldots \cup$ (vi) (in red) for the even twist knot $K_n$, seen in $S^3 \setminus K_n$ (left) and on the truncated tetrahedron $U$ (right).}\label{fig:longitude:even}
\end{figure}

Let us now list the angular and complex weight functions associated to edges of $X_n$.
For $\alpha=(a_1,b_1,c_1,\ldots,a_p,b_p,c_p,a_U,b_U,c_U,a_V,b_V,c_V,a_W,b_W,c_W) \in \mathcal{S}_{X_n}$ a shape structure on $X_n$, we compute the weights of each edge:
\begin{itemize}
	\item $\omega_s(\alpha):= \omega_{X_n,\alpha}(\overrightarrow{\eta_s})=
	2 a_U+b_V+c_V+a_W+b_W
	$
	\item $\omega_0(\alpha):= \omega_{X_n,\alpha}(\overrightarrow{\eta_0})=
	2 a_1 + c_1 + 2 a_2 + \ldots + 2 a_p + a_V+c_W
	$
	\item $\omega_1(\alpha):= \omega_{X_n,\alpha}(\overrightarrow{\eta_1})=
	2b_1+c_2
	$
	\\
	\vspace*{-2mm}
	\item $\omega_k(\alpha):= \omega_{X_n,\alpha}(\overrightarrow{\eta_k})=
	c_{k-1}+2b_k+c_{k+1}
	$
	\ \
	(for $2\leqslant k \leqslant p-1$)
	\\
	\vspace*{-2mm}
	\item $\omega_p(\alpha):= \omega_{X_n,\alpha}(\overrightarrow{\eta_p})=
	c_{p-1}+2b_p+b_U+2c_U+a_V+b_V+a_W+c_W$
	\item $\omega_{p+1}(\alpha):= \omega_{X_n,\alpha}(\overrightarrow{\eta_{p+1}})=
	c_p+b_U+c_V+b_W$
\end{itemize}

For a complex shape structure $\widetilde{\mathbf{z}}=(z_1,\ldots,z_p,z_U,z_V,z_W) \in (\R+i\R_{>0})^{p+3}$, its complex weight functions are:

\begin{itemize}
	\item $\omega^{\C}_s(\widetilde{\mathbf{z}}):= \omega^{\C}_{X_n,\alpha}(\overrightarrow{\eta_s})=
	2\Log(z_U) + \Log(z'_V) + \Log(z''_V) + \Log(z_W) + \Log(z'_W)
	$
	\item $\omega^{\C}_0(\widetilde{\mathbf{z}}):= \omega^{\C}_{X_n,\alpha}(\overrightarrow{\eta_0})=
	2\Log(z_1) + \Log(z'_1) + 2\Log(z_2) + \cdots + 2\Log(z_p) + \Log(z_V) + \Log(z''_W)
	$
	\item $\omega^{\C}_1(\widetilde{\mathbf{z}}):= \omega^{\C}_{X_n,\alpha}(\overrightarrow{\eta_1})=
	2\Log(z''_1) + \Log(z'_2)
	$
	\\
	\vspace*{-2mm}
	\item $\omega^{\C}_k(\widetilde{\mathbf{z}}):= \omega^{\C}_{X_n,\alpha}(\overrightarrow{\eta_k})=
	\Log(z'_{k-1}) + 2\Log(z''_k) + \Log(z'_{k+1})
	$
	\ \
	(for $2\leqslant k \leqslant p-1$)
	\\
	\vspace*{-2mm}
	\item $\omega^{\C}_p(\widetilde{\mathbf{z}}):= \omega^{\C}_{X_n,\alpha}(\overrightarrow{\eta_p})=
		\Log(z'_{p-1}) + 2\Log(z''_p) + 2\Log(z'_U) + \Log(z''_U) + \Log(z_V) + \Log(z'_V) + \Log(z_W) + \Log(z''_W)$
	\item $\omega^{\C} _{p+1}(\widetilde{\mathbf{z}}):= \omega^{\C}_{X_n,\alpha}(\overrightarrow{\eta_{p+1}})=
	\Log(z'_p) + \Log(z''_U) + \Log(z''_V) + \Log(z'_W)
	$
\end{itemize}

To the meridian curve $m_{X_n}$ and the longitude curve $l_{X_n}$ are associated  angular holonomies
$$m_{X_n}(\alpha):=a_V-a_U, \ \ \ l_{X_n}(\alpha):=2(a_W-b_V),$$
and one possible complex completeness equation is once again (from the meridian curve):
$$\Log(z_U)-\Log(z_V)=0.$$
Furthermore,  one can again see in Figure \ref{fig:trig:cusp:even} that in the homology group of the boundary torus, we have the relation
$$ {l_{X_n}^0= } \mathrm{(i)} \cup \ldots \cup \mathrm{(vi)} = l_{X_n} + 2 m_{X_n}.$$

Using properties of shape structures, we see that the balancing conditions are equivalent to the following $p+2$ equations: 

\begin{itemize}
\item $E_s(\alpha): \ 2a_U + b_V + c_V + a_W + b_W = 2\pi$
\item $E_1(\alpha): \ 2b_1+c_2 = 2 \pi$
\\ \vspace*{-2mm}
\item $E_k(\alpha): \ c_{k-1}+2b_k+c_{k+1} = 2 \pi $ \quad (for $2\leqslant k \leqslant p-1$)
\\ \vspace*{-2mm}
\item $E_p(\alpha): \ c_{p-1}+2b_p + b_U + 2c_U + a_V + b_V + a_W + c_W = 2\pi$
\item $E_{p+1} (\alpha): \ c_p + b_U + c_V + b_W = 2\pi$
\end{itemize}
The missing $(p+3)$-rd equation, stating that the angles around the vertices of degree $2p+3$ in Figure~\ref{fig:trig:cusp:even} add up to $2\pi$, is redundant: summed with all of the above, it becomes simply that the sum of all angles is $(p+3)\pi$.


\begin{theorem} \label{thm:appendix:geom:even}
	$X_n$ is geometric for $n \geq 2$ even.
\end{theorem}

\begin{proof}
We begin by treating the case of $n\geq 6$, i.e.\ $p\geq 2$.
First we show that the space of positive angle structures is nonempty. For small enough $\epsilon>0$, the values
$$
\begin{pmatrix} a_j \\ b_j \\ c_j \end{pmatrix} := \begin{pmatrix} \epsilon \\ \pi - \epsilon(j^2+1) \\ \epsilon j^2  \end{pmatrix}\text{for } 1\leq j \leq p-1, \quad
\begin{pmatrix} a_p \\ b_p \\ c_p \end{pmatrix} := \begin{pmatrix} 3\pi/4 - \epsilon (p^2+2p-1)/2 \\ ~\,\pi/4 - \epsilon (p^2-2p+1)/2 \\ \epsilon p^2 \end{pmatrix}, $$ $$
\begin{pmatrix} a_U \\ b_U \\ c_U \end{pmatrix} = 
\begin{pmatrix} a_V \\ c_V \\ b_V \end{pmatrix} = 
\begin{pmatrix} c_W \\ b_W \\ a_W \end{pmatrix} := 
\begin{pmatrix} \pi/4 + \epsilon p^2/2 \\ 2\pi/3 - \epsilon p^2/3 \\ \pi/12 - \epsilon p^2/6 \end{pmatrix}
$$
give a positive solution to $E_s,E_1, \dots, E_{p+1}$.

Next, we claim that among the volume maximizers, there is one such that $U,V,W$ have identical angles modulo the permutation used in the formula above.
Let $F_j$ denote the constraint $a_j+b_j+c_j=\pi$.
The angles of $U,V,W$ appear only in equations $E_s, E_p, E_{p+1}$. These can be rewritten 
$$\begin{array}{r|l}
E_{p+1}                                & c_p + (b_U+ c_V+b_W ) =2 \pi \\
3E_p + 2 E_s - (3F_U+2F_V+2F_W) & 3 c_{p-1} + 6 b_p + (a_U + a_V + c_W) + 3 (c_U+b_V+a_W)= 3 \pi\\
E_s - (F_V + F_W)                  & 2a_U =  a_V + c_W.   \end{array}$$
The involution $(a_V, b_V, c_V) \leftrightarrow (c_W,a_W,b_W)$ preserves these equations, so by concavity of the volume function, there is a maximizer such that $(a_V, b_V, c_V)=(c_W,a_W,b_W)$. 
The last of the $3$ equations above then gives $a_U=a_V=c_W$. The order-3 substitution of variables $$(a_U, b_U, c_U) \rightarrow  (a_V, c_V, b_V) \rightarrow (c_W, b_W, a_W) \rightarrow (a_U, b_U, c_U)$$
then clearly leaves the other two equations unchanged, so by concavity we may average out and find a maximizer such that $(a_U, b_U, c_U)=(a_V, c_V, b_V)=(c_W, b_W, a_W)$, as desired.

These identifications make $E_s$ redundant. Moreover, dropping the angles of $V$ and $W$ as variables, we may now rewrite the system of constraints as 
\begin{itemize}
\item $E_1 : \ 2b_1+c_2 = 2 \pi$
\vspace{2mm}
\item $E_k : \ c_{k-1}+2b_k+c_{k+1} = 2 \pi $ \quad (for $2\leqslant k \leqslant p-1$)
\vspace{2mm}
\item $E'_p : \ c_{p-1}+2b_p + a_U + 3c_U  = \pi$ \quad (not $2\pi$!)
\item $E'_{p+1} : \ c_p + 3b_U = 2\pi$
\end{itemize}
Recall from Lemma \ref{lem:flat:taut}
 that at a volume maximizer, if $a_j b_j c_j=0$ then $a_j, b_j, c_j$ are $0,0,\pi$ up to order.
\begin{lemma} \label{lem:girafe}
 At a volume maximizer, if $a_k b_k c_k=0$ then $k=p$ and $(a_p, b_p, c_p)=(0,0,\pi)$.
\end{lemma}
\begin{proof}
First, $E'_{p+1}$ gives $b_U=(2\pi-c_p)/3\in[\pi/3, 2\pi/3]$ so the tetrahedron $U$ is {not flat}.

\noindent $\bullet$ 
Let us show by induction on $1 \leq k \leq p-1$ that $b_k > 0$. 
By $E_1$ we have $b_1=\pi-c_2/2 \geq \pi/2$, giving the case $k=1$.
For the induction step, suppose $2\leq k \leq p-1$ and $b_{k-1}>0$.
Then $c_{k-1}<\pi$, which by $E_k$ implies that $b_k>0$.

\noindent $\bullet$ 
Let us now show by \emph{descending} induction on $p-1 \geq k \geq 1$ that $b_k < \pi$. 
For the initialisation, suppose $(a_{p-1},b_{p-1}, c_{p-1})=(0,\pi,0)$ and aim for a contradiction.
Recall that $p\geq 2$: by $E_{p-1}$ we have $c_p=0$, hence $b_U=2\pi/3$ by $E'_{p+1}$.
But $c_p=0$ also implies $b_p\in \{0,\pi\}$, hence $b_p=0$ by $E'_p$.
Together with $c_{p-1}=0$, by $E'_p$ this yields $a_U+3c_U=\pi$.
But we showed  that $b_U=2\pi/3$, hence $(a_U, b_U, c_U)=(0,2\pi/3, \pi/3)$, a forbidden configuration.
This contradiction shows $b_{p-1}<\pi$.

For the (downward) induction step, suppose $p-2 \geq k \geq 1$ and $b_{k+1}<\pi$.
Actually $0<b_{k+1}<\pi$ (previous bullet-point), hence $0<c_{k+1}$: by $E_k$, this implies $b_k<\pi$.

\noindent $\bullet$ 
It only remains to rule out $c_p=0$. 
Note that the non-negative sequence $(0, c_1, \dots, c_p)$ is convex, because $E_k$ can be rewritten $c_{k-1} - 2c_k + c_{k+1} = 2 a_k \geq 0$ (agreeing that ``$c_0$'' stands for $0$).
But we showed $0<b_{p-1}<\pi$: hence, $c_{p-1}>0$ which entails $c_p\geq \frac{p}{p-1} c_{p-1} > 0$.
\end{proof}

We can now prove that the volume maximizer has only positive angles.
By the above lemma, if not, then we may assume $(a_p, b_p, c_p)=(0,0,\pi)$ and that all other tetrahedra are nondegenerate.
We will exhibit a smooth path of deformations of the angles, along which the derivative of the volume is positive. 
(As a function of the angles, the volume of an ideal tetrahedron is not smooth near the point $(0,0,\pi)$, but it has a well-defined derivative in the direction of any segment.)

Using $E_{p-1}, E'_p, E'_{p+1}$, it is straigthforward to check that the angles satisfy
\begin{equation} \label{zebre} \begin{pmatrix} 
 a_{p-1} & a_{p} & a_U \\
 b_{p-1} & b_{p} & b_U \\
 c_{p-1} & c_{p} & c_U \end{pmatrix}
= \begin{pmatrix} 
 (\pi+c_{p-2}-2c_{p-1})/2 & 0 & (\pi+c_{p-1})/2 \\
 (\pi-c_{p-2})/2 & 0 & \pi/3 \\
 c_{p-1} & \pi & \pi/6-c_{p-1}/2 \end{pmatrix}.
\end{equation}
For small $t>0$, the $t$-deformation given by $(a^t_k,b^t_k,c^t_k)=(a_k,b_k,c_k)$ for $1\leq k \leq p-2$ and 
$$ \begin{pmatrix} 
 a^t_{p-1} & a^t_{p} & a^t_U \\[1mm]
 b^t_{p-1} & b^t_{p} & b^t_U \\[1mm]
 c^t_{p-1} & c^t_{p} & c^t_U \end{pmatrix}
= \begin{pmatrix} 
 a_{p-1} & 0 & a_U \\
 b_{p-1} & 0 & b_U \\
 c_{p-1} & \pi & c_U \end{pmatrix}
+t  \begin{pmatrix} 
 -1 & 2 & - 1  \\
  1 & 0  & 2/3  \\
  0 & - 2 &   1/3 \end{pmatrix}
 $$
is still an angle structure, i.e.\ satisfies $E_1,\dots, E_{p-1}, E'_p, E'_{p+1}$.
By definition of the volume functional $\mathcal{V}$ (Section~\ref{sub:volume}), we have for this deformation
\begin{equation} \label{eq:slope} 
\left .\mathrm{exp} \left ( \frac{-\partial \mathcal{V}} {\partial t} \right )\right |_{t=0} =
\frac{\sin (b_{p-1})}{\sin (a_{p-1})} \frac{\sin^2(b_U) \sin(c_U)}{\sin^3(a_U)}.
 \end{equation}
Each factor $\sin(\theta)$ appears to the power $\partial \theta/\partial t$, but tripled for $\theta=a_U, b_U, c_U$ because there are 3 isometric copies of the tetrahedron~$U$.
The $p$-th tetrahedron stays flat, hence does not contribute volume.
The formula for $c_U$ in~\eqref{zebre} gives $0\leq c_{p-1} \leq \pi/3$.
We proved in the lemma above that $(0,c_1,\dots, c_p)$ is convex, 
hence nondecreasing: thus~\eqref{zebre} 
also yields $a_{p-1} \in [\pi/6, \pi/2]$.
Therefore,
$$\frac{\sin (b_{p-1})}{\sin (a_{p-1})} \leq \frac {1}{\sin (\pi/6)}=2.$$
On the other hand, still using~\eqref{zebre}, 
$$\frac{\sin^2(b_U) \sin(c_U)}{\sin^3(a_U)} = \frac{3}{4} \frac{\sin (\pi/6-c_{p-1}/2)}{\sin^3(\pi/2+c_{p-1}/2)} \leq \frac{3}{4} \frac{\sin(\pi/6)}{\sin^3(\pi/2)} = \frac{3}{8}$$
by an easy monotonicity argument for $c_{p-1}$ ranging over $[0,\pi/3]$.
In conclusion,~\eqref{eq:slope} is bounded above by $2\cdot 3/8<1$, hence $(\partial \mathcal{V} / \partial t)_{t=0^+}>0$ as desired. 

Thus, the volume maximizer is interior to the space of angle structures. 
By Theorem~\ref{thm:casson:rivin}, this implies Theorem~\ref{thm:appendix:geom:even} for $p\geq 2$. 
It only remains to discuss $p=0,1$.

\noindent $\bullet$ For $p=1$ we find the initial gluing equations
$$ \begin{array}{lrcl}
E_s : & 2a_U + b_V + c_V + a_W + b_W &=& 2\pi \\
E_1 : & 2b_1 + b_U + 2c_U + a_V + b_V + a_W + c_W &=& 2\pi \\
E_2 : & c_1 + b_U + c_V + b_W &=& 2\pi 
\end{array} $$
(only the term ``$c_{p-1}$'' has disappeared from $E_1$). 
Symmetry between $U,V,W$ can be argued as in the $p\geq 2$ case, reducing the above to
$$ \begin{array}{lrcl}
E'_1 : &  2b_1 + a_U + 3c_U  &=& \pi \\
E'_2 : &   c_1 + 3b_U &=& 2\pi. 
\end{array}$$
The tetrahedron $U$ is not flat, as $b_U=(2\pi-c_1)/3\in[\pi/3, 2\pi/3]$.
If $c_1=0$ then $b_1\in\{0,\pi\}$ must be $0$ by $E'_1$, hence $(a_U, b_U, c_U)=(0,2\pi/3, \pi/3)$ which is prohibited.
If $c_1=\pi$ then  
$$ \begin{pmatrix} 
 a_1  & a_U \\
 b_1  & b_U \\
 c_1  & c_U \end{pmatrix}
= \begin{pmatrix} 
  0   & \pi/2 \\
  0   & \pi/3 \\
  \pi & \pi/6  \end{pmatrix}
\text{ can be perturbed by adding } ~
t \begin{pmatrix} 
   2 &  -1  \\
   0  &  2/3  \\
 - 2 &  1/3  \end{pmatrix} $$ 
(where $0<t\ll 1$) to produce a path of angle structures, yielding as before
$$\left . \mathrm{exp} \left ( \frac{-\partial \mathcal{V}} {\partial t} \right ) \right |_{t=0} =
 \frac{\sin^2(b_U) \sin(c_U)}{\sin^3(a_U)} = \frac{3}{8}<1. $$
\noindent $\bullet$ For $p=0$ it is straightforward to check that $(a_U, b_U, c_U)=(a_V, c_V, b_V)=(c_W, b_W, a_W)=(\pi/6, 2\pi/3, \pi/6)$ yields the complete hyperbolic metric 
(this is actually the result of a $2\rightarrow 3$ Pachner move on the standard triangulation of the figure eight knot complement into two regular ideal tetrahedra).
Theorem~\ref{thm:appendix:geom:even} is proved.
\end{proof}

\subsection{Computation of the partition functions}\label{sub:even:tqft}

The following theorem is the version of Theorem \ref{thm:part:func} for even $n$. Note that here $\mu_{X_n}(\alpha)=-m_{X_n}(\alpha)$ and once again
$\lambda_{X_n}(\alpha)=l_{X_n}(\alpha)+2m_{X_n}(\alpha)$ corresponds to a preferred longitude.

\begin{theorem}\label{thm:even:part:func}
Let $n$ be a positive even integer and $p=\frac{n-2}{2}$. Consider the ideal triangulation $X_n$ of $S^3\setminus K_n$ described in Figure \ref{fig:id:trig:even}. Then for all angle structures $\alpha =(a_1,\ldots,c_W)\in \mathcal{A}_{X_n}$ and all $\hbar>0$, we have:
\begin{equation*}
\mathcal{Z}_{\hbar}(X_n,\alpha) 
\stareq 
\int_{\mathbb{R}+i \frac{\mu_{X_n}(\alpha) }{2\pi \sqrt{\hbar}}  } 
J_{X_n}(\hbar,x)
e^{\frac{1}{2 \sqrt{\hbar}}  x  \lambda_{X_n}(\alpha)} 
dx,
\end{equation*}
with 
\begin{itemize}
\item the degree one angle polynomial $\mu_{X_n}\colon\alpha\mapsto  a_U- a_V$,
\item the degree one angle polynomial $\lambda_{X_n}\colon\alpha\mapsto 
2(a_V-a_U+a_W-b_V)$,
\item the map
\begin{equation*}
J_{X_n}\colon(\hbar,x)\mapsto
\int_{\mathcal{Y}'} 
d\mathbf{y}' \
e^{2 i \pi \mathbf{y'}^{\!\top} Q_n \mathbf{y'}}
e^{2 i \pi x(x- y'_U-y'_W)}
e^{
\frac{1}{\sqrt{\hbar}} (\mathbf{y'}^{\!\top} \mathcal{W}_n - \pi x)
}
\dfrac{
\Phi_\B\left (x - y'_U \right )
\Phi_\B\left (y'_W\right )
}{
\Phi_\B\left (y'_1\right )
\cdots
\Phi_\B\left (y'_p\right )\Phi_\B\left (y'_U\right )
} 
,
\end{equation*}
where 
$$\mathcal{Y}'=\mathcal{Y}'_{\hbar,\alpha} =
\left (\prod_{k=1,\ldots,p,U}\left (\R - \frac{i}{2 \pi \sqrt{\hbar}} (\pi-a_k)\right ) \right )
\times 
\left (\R + \frac{i}{2 \pi \sqrt{\hbar}} (\pi-a_W)\right ),$$
\begin{equation*}
\mathbf{y'}=\begin{bmatrix}
y'_1 \\ \vdots \\ y'_p \\ y'_U \\y'_W
\end{bmatrix}, 
\quad 
\mathcal{W}_n=\begin{bmatrix}-2p\pi \\ \vdots \\ -2 \pi \left ( k p - \frac{k(k-1)}{2}\right ) \\ \vdots \\ -p(p+1)\pi \\ -(p^2+p+3)\pi \\ \pi\end{bmatrix} 
\quad 
\text{ and }
\quad
Q_n=\begin{bmatrix}
1 & 1 & \cdots & 1 & 1 & 0 \\ 
1 & 2 & \cdots & 2 & 2 & 0 \\ 
\vdots & \vdots & \ddots & \vdots & \vdots & \vdots \\ 
1 & 2 & \cdots & p & p & 0 \\ 
1 & 2 & \cdots & p & p+1 & -\frac{1}{2} \\ 
0 & 0 & \cdots & 0 & -\frac{1}{2} & 0 
\end{bmatrix}.
\end{equation*}
\end{itemize}
\end{theorem}

\begin{proof}
Since the computations are very similar to those of the proof of Theorem \ref{thm:part:func} we will not give all the details. 
Let $n \geq 2$ be an even integer and set $p=\frac{n-2}{2}$. As before, we denote $\widetilde{\mathbf{t}}=(t_1,\ldots,t_{p-1},t_p,t_U,t_V,t_W)^{\!\top} \in \R^{X^3}$  the vector whose coordinates are associated to the tetrahedra, and $\mathbf{x}=(e_1,\dots,e_p,e_{p+1},f_1,\ldots, f_p,v,r,s,g,u)^{\!\top} \in \R^{X^2}$  the face variables vector. 

Like in Lemma \ref{lem:kin:odd}, we compute 
$
\mathcal{K}_{X_n}\left (\mathbf{\widetilde{t}}\right ) = \frac{1}{|\det(A_e)|} e^{ 2 i \pi \mathbf{\widetilde{t}}^{\!\top} (-R_e A_e^{-1} B) \mathbf{\widetilde{t}}}$, where $B$ is like in the proof of  Lemma \ref{lem:kin:odd}, but $R_e, A_e$ (\textit{e} standing for \textit{even}) are given by 
$$
R_e=\kbordermatrix{
	\mbox{} 	& e_1 	& \dots 	& e_p	&e_{p+1} 	& \omit\vrule	&f_1	& \ldots 	& f_p	&\omit\vrule	&  v 	& r 	& s 		& g 	& u \\
	t_1 		& 1	 	 & 		&\push{\low{0}} 	& \omit\vrule	& 	&  		&  	&\omit\vrule	&   	&  	&   		&  	&  	\\
	\vdots 	& 	    	&\ddots 	& 		&	  	& \omit\vrule	& 	&  0		&  	&\omit\vrule	& 	&   	&  0		&	&  	\\
	t_{p}    	&\push{0}		       	& 1        	& 		& \omit\vrule	&  	&  		&	&\omit\vrule	&	&	&		&	&  	\\ 
\cline{1-1} \cline{2-15}
	t_U	 	& 		& 		& 		&  		& \omit\vrule	&	&	 	&	&\omit\vrule	&0  	&1 	&0  		& 0 	& 0	\\
	t_V	 	& 		&\push{0}	  		&  		& \omit\vrule	&	& 	0	& 	&\omit\vrule	&0  	&0 	&-1  		& 0 	& 0	\\
	t_W	 	& 		&  		& 		& 	 	& \omit\vrule	&	&	  	&  	&\omit\vrule	&0 	&0 	&0  		& 0 	&-1
},
$$

$$
A_e=\kbordermatrix{
	\mbox{} 	& e_1 	& e_2 	& \dots 	&e_p		& e_{p+1}	& \omit\vrule	&f_1	&f_{2}	&  \ldots 	& f_p	&\omit\vrule	&  v 	& r 	& s 		& g 	& u 	\\
	w_1 		& 1		&   -1 	& 		&		&		& \omit\vrule	& 1	&		&  		&  	&\omit\vrule	&   	&  	&  	 	&  	&  	\\
	\vdots 	&     		&   \ddots	& \ddots 	&	\push{0}		& \omit\vrule	& 	&\ddots	&\push{0}  	&\omit\vrule	& 	&   	&  \low{0}	&	&  	\\
	\vdots 	&     	\push{0}   		& \ddots 	&\ddots	&		& \omit\vrule	& \push{0}  	&\ddots	&  	&\omit\vrule	& 	&   	&  		&	&  	\\
	w_{p} 	&		&		&		&	1	& -1		& \omit\vrule 	&	&		&		&1	&\omit\vrule	&	&	&		&	&  	\\
\cline{1-1} \cline{2-17}
	w_{U} 	&		&		&		&		&  		& \omit\vrule 	&	&		&		&0	&\omit\vrule	& -1	& 1	& 1		& 0	& 0	\\
	w_{V} 	&		&		& 0		&		&  		& \omit\vrule 	&	&		&		&1	&\omit\vrule	& 0	& 0	& 1		& -1	& 0	\\
	w_{W} 	&		&		&		&		&  		& \omit\vrule 	&	&		&		&0	&\omit\vrule	& -1	& 1	& 0		& 0	& 1	\\
\cline{1-1} \cline{2-17}
	w'_1 		& -1		&   	 	& 		&		&		& \omit\vrule	& 1	&  		&  		&	&\omit\vrule	&   	&  	&   		&  	& 	\\
	\vdots 	&     		&   		& 	 	&		&		& \omit\vrule	& -1	&  \ddots	&  \push{0}	&\omit\vrule	& 	&   	&  \low{0}	&	& 	\\
	\vdots 	&     		&   		& 	 	&		&		& \omit\vrule	& 	&  \ddots	& \ddots 	&	&\omit\vrule	& 	&   	&  		&	& 	\\
	w'_{p} 	&		&		&		&		&  		& \omit\vrule 	&	\push{0}	&-1		&1	&\omit\vrule	&	&	&		&	& 	\\
\cline{1-1} \cline{2-17}
	w'_{U} 	&		&		&		&		&  0		& \omit\vrule 	&	&		&		&0	&\omit\vrule	& 0	& 0	& 1		& -1	& 0 	\\
	w'_{V} 	&		&		&		&		&  0		& \omit\vrule 	&	&		&		&1	&\omit\vrule	& 0	& 0	& 0		& 0	& -1	\\
	w'_{W} 	&		&		&		&		&  -1		& \omit\vrule 	&	&		&		&0	&\omit\vrule	& 0	& 1	& 0		& 0	& 0
}.$$
Careful computation yields that $\det(A_e)=-1$ and that $A_e^{-1}$ is equal to

$$
A_e^{-1}=\kbordermatrix{
	\mbox{} 	& w_{1}	& w_{2} 	& \ldots	& w_{p-1}	& w_{p}	& \omit\vrule	& w_{U}	& w_{V}			& w_{W}	& \omit\vrule	&  w'_{1}	& w'_{2}	& \ldots	& w'_{p-1}			& w'_{p}			& \omit\vrule 	& w'_{U}			& w'_{V}	& w'_{W}	\\ 
	e_{1} 	& 	0	&		& \cdots	& 		&	0	& \omit\vrule 	& 	0	& 	1			& 	0	& \omit\vrule	&  	-1	& 	-1	&	\push{\cdots}			& 	-1			& \omit\vrule 	& 	-1			& 	0	& 	0	\\ 
	e_{2} 	& 	-1	& 	0	& 		& 		& 		& \omit\vrule 	& 	0	& 	2			& 	0	& \omit\vrule	&  	-1	& 	-2	&	\push{\cdots}			& 	-2			& \omit\vrule 	& 	-2			& 	0	& 	0	\\ 
\low{\vdots} 	& 	-1	& 	-1	& \ddots	& 		& \vdots	& \omit\vrule 	& 		& 				& 		& \omit\vrule	& 	 	&		& \ddots	& 				& 	\vdots		& \omit\vrule 	& 				& 		& 		\\ 
		 	& \vdots	& 		& \ddots	& 	0	& 	0	& \omit\vrule 	& \vdots	& \vdots			& \vdots	& \omit\vrule	&  \vdots	& \vdots	&		&\text{\tiny {\(1-p\)}}	& \text{\tiny {\(1-p\)}}	& \omit\vrule 	& \vdots			& \vdots	& \vdots	\\ 
	e_{p} 	& 		& 		& 		& 	-1	& 	0	& \omit\vrule 	& 		& 				& 		& \omit\vrule	&  		& 		&		& \text{\tiny {\(1-p\)}}	& 	-p			& \omit\vrule 	& 				& 		& 		\\ 
	e_{p+1} 	& -1		& 		& \cdots	& 		& 	-1	& \omit\vrule 	& 	0	& \text{\tiny {\(p+1\)}}	& 	0	& \omit\vrule	&  	-1	&	-2	&\cdots	& \text{\tiny {\(1-p\)}}	& 	-p			& \omit\vrule 	& \text{\tiny {\(-p-1\)}}	& 	0	& 	0	\\ 
\cline{1-1} \cline{2-20}
	f_{1} 	& 		& 		& 		& 		& 		& \omit\vrule 	& 	0	& 	1			& 	0	& \omit\vrule	&  	0	& -1		& 	\push{\cdots}			&	-1			& \omit\vrule 	& 	-1			& 	0	& 	0	\\ 
	f_{2} 	& 		& 		& 		& 		& 		& \omit\vrule 	& 		& 	1			& 		& \omit\vrule	&  	0	&0		& \ddots	&				&	\low{\vdots}	& \omit\vrule 	& 	-1			& 		& 		\\ 
	\vdots 	& 		& 		& 	0	& 		& 		& \omit\vrule 	& \vdots	& \vdots			& \vdots	& \omit\vrule	& \vdots 	& 		& \ddots	&	-1			&	-1			& \omit\vrule 	& \vdots			& \vdots	& \vdots	\\ 
	f_{p-1} 	& 		& 		& 		& 		& 		& \omit\vrule 	& 		& 				& 		& \omit\vrule	&  	0	& 		& 		&	0			&	-1			& \omit\vrule 	& 				& 		& 		\\ 
	f_{p} 	& 		& 		& 		& 		& 		& \omit\vrule 	& 	0	& 	1			& 	0	& \omit\vrule	&  	0	& 		& \cdots	&				&	0			& \omit\vrule 	& 	-1			& 	0	& 	0	\\ 
\cline{1-1} \cline{2-20}
	v	 	& 	-1	& 		& \cdots	& 		& 	-1	& \omit\vrule 	& 	0	& \text{\tiny {\(p+2\)}}	& 	-1	& \omit\vrule	&  	-1	& 	-2	& \push{\cdots}				&	-p			& \omit\vrule 	& \text{\tiny {\(-p-2\)}}	& 	-1	& 	1	\\ 
	r	 	& 	-1	& 		& \cdots	& 		& 	-1	& \omit\vrule 	& 	0	& \text{\tiny {\(p+1\)}}	& 	0	& \omit\vrule	&  	-1	& 	-2	& \push{\cdots}				&	-p			& \omit\vrule 	& \text{\tiny {\(-p-1\)}}	& 	0	& 	1	\\ 
	s	 	& 		& 		& 		& 		& 		& \omit\vrule 	& 	1	& 	1			& 	-1	& \omit\vrule	&  		& 		& 		&				&				& \omit\vrule 	& 	-1			& 	-1	& 	0	\\ 
	g	 	& 		& 		& 	0	& 		& 		& \omit\vrule 	& 	1	& 	1			& 	-1	& \omit\vrule	&  		& 		& 	0	&				&				& \omit\vrule 	& 	-2			& 	-1	& 	0	\\ 
	u	 	& 		& 		& 		& 		& 		& \omit\vrule 	& 	0	& 	1			& 	0	& \omit\vrule	&  		& 		& 		&				&				& \omit\vrule 	& 	-1			& 	-1	& 	0	
}.
$$
Hence $\mathcal{K}_{X_n}(\widetilde{\mathbf{t}})=  \exp\left (2 i \pi \widetilde{\mathbf{t}}^{\!\top} \widetilde{Q}_n \widetilde{\mathbf{t}} \right )$,
where
$$
\widetilde{Q}_n:=
\frac{(-R_e A_e^{-1} B)+(-R_e A_e^{-1} B)^{\!\top}}{2}=
\kbordermatrix{
	\mbox{}	&t_1		&t_2		&\cdots 	& t_{p-1} 	& t_p 	& \omit\vrule	& t_U 	& t_V 	& t_W 	\\
	t_1 		& 1 		& 1 		& \cdots 	& 1  		& 1 		& \omit\vrule	& 1 		& 0 		& 0 		\\
	t_2 		& 1 		& 2 		& \cdots 	& 2  		& 2 		& \omit\vrule	& 2 		& 0 		& 0 		\\
	\vdots 	& \vdots 	& \vdots 	& \ddots 	& \vdots  	& \vdots 	& \omit\vrule	& \vdots 	& \vdots 	& \vdots 	\\
	t_{p-1} 	& 1 		& 2 		& \cdots 	& p-1  	& p-1 	& \omit\vrule	& p-1 	& 0 		& 0 		\\
	t_p 		& 1 		&  2 		& \cdots 	&  p-1 	& p 		& \omit\vrule	& p 		& 0 		& 0 		\\
\cline{1-1} \cline{2-10}
	t_U 		& 1 		& 2 		& \cdots 	& p-1  	& p  		& \omit\vrule	& p+1 	& -1/2 	& -1 		\\
	t_V 		& 0 		& 0 		& \cdots 	&  0 		& 0 		& \omit\vrule	& -1/2 	& -1 		& -1/2 	\\
	t_W 		& 0 		& 0 		& \cdots 	&  0 		& 0 		& \omit\vrule	& -1 		& -1/2 	& 0 }.
$$

Now, like in Lemma \ref{lem:2QGamma+C}, if we denote  $\widetilde{C}(\alpha) = (c_1,\ldots,c_W)^{\!\top}$, and $\widetilde{\Gamma}(\alpha) := (a_1-\pi,\ldots, a_p-\pi,a_U-\pi,\pi-a_V,\pi-a_W)^{\!\top}$, then
(indexing entries by  $k\in\{1,\ldots,p+3\}$) we can compute:
$
2\widetilde{Q}_n \widetilde{\Gamma}(\alpha) + \widetilde{C}(\alpha) =$ 

$$
\renewcommand{\kbldelim}{(}
\renewcommand{\kbrdelim}{)}
 \kbordermatrix{
	\mbox{}	& \\
	 k=1	& \vdots\\
\vdots	& \hspace{6mm} k(\omega_{s}(\alpha) -2(p+2) \pi) + \sum_{j=1}^{k}j \omega_{k-j}(\alpha) \\
	k=p 	& \vdots \\ \cline{1-1} 
\cline{2-2}
	 	& \omega_{s}(\alpha)- \omega_{p+1}(\alpha) + \left ( p(\omega_{s}(\alpha) -2(p+2) \pi) + \sum_{j=1}^{p}j \omega_{p-j}(\alpha) \right )  -4 \pi + \frac{1}{2} \lambda_{X_n}(\alpha) \\
		& \frac{1}{2}\lambda_{X_n}(\alpha) - \pi \\
		& 3 \pi - \omega_{s}(\alpha)
},
$$
where $\lambda_{X_n}(\alpha)=2(-a_U+a_V-b_V+a_W)$.
Notably we have for all angle structures $\alpha \in \mathcal{A}_{X_n}$:
\renewcommand{\arraystretch}{1.2} $$
2\widetilde{Q}_n \widetilde{\Gamma}(\alpha) + \widetilde{C}(\alpha) = \:
\renewcommand{\kbldelim}{(}
\renewcommand{\kbrdelim}{)}
 \kbordermatrix{
	\mbox{}	& \\
	 k=1	& \vdots\\
\vdots	& -2 \pi \left ( k p - \dfrac{k(k-1)}{2}\right ) \\
	k=p 	& \vdots \\
	\cline{1-1}   \cline{2-2}
		& -(p^2+p+4)\pi + \frac{1}{2}\lambda_{X_n}(\alpha) \\
		& \frac{1}{2}\lambda_{X_n}(\alpha) - \pi \\
		& \pi }.
$$ \renewcommand{\arraystretch}{1.0}
The above computations are fairly quick consequences of the similarities between the matrices $\widetilde{Q}_n$ and the weights $\omega_j(\alpha)$ whether $n$ is odd or even.

Denote again $\alpha=(a_1,b_1,c_1,\ldots,a_W,b_W,c_W)$ a general 
vector of dihedral angles in $\mathcal{A}_{X_n}$.
Let $\hbar>0$. Since the tetrahedron $T_U$ is of positive sign here, the dynamical content $\mathcal{D}_{\hbar,X_n}(\widetilde{\mathbf{t}},\alpha)$ thus becomes
\[
e^{\frac{1}{\sqrt{\hbar}} \widetilde{C}(\alpha)^{\!\top} \widetilde{\mathbf{t}}}
\dfrac{
\Phi_\B\left (t_V + \frac{i}{2 \pi \sqrt{\hbar}}(\pi-a_V)\right )
\Phi_\B\left (t_W + \frac{i}{2 \pi \sqrt{\hbar}} (\pi-a_W)\right )
}{
\Phi_\B\left (t_1 - \frac{i}{2 \pi \sqrt{\hbar}} (\pi-a_1)\right )
\cdots
\Phi_\B\left (t_p - \frac{i}{2 \pi \sqrt{\hbar}} (\pi-a_p)\right )
\Phi_\B\left (t_U - \frac{i}{2 \pi \sqrt{\hbar}} (\pi-a_U)\right )}.\]
According to tetrahedra signs, we do the following change of variables:
\begin{itemize}
\item $y'_k = t_k - \frac{i}{2 \pi \sqrt{\hbar}} (\pi-a_k)$ for $k \in \{1,\ldots,p,U\}$,
\item $y'_l = t_l + \frac{i}{2 \pi \sqrt{\hbar}} (\pi-a_l)$ for $l\in\{V,W\}$,
\end{itemize}
and we define $\widetilde{\mathbf{y'}}:=\left (y'_1, \ldots, y'_{p}, y'_U, y'_V, y'_W\right )^{\!\top}$. 
We also denote 
\[
\widetilde{\mathcal{Y}}'_{\hbar, \alpha} :=
\prod_{k=1, \ldots, p, U}\left (\R - \frac{i}{2 \pi \sqrt{\hbar}} (\pi-a_k)\right ) 
\times 
\prod_{l=V,W}
\left (\R + \frac{i}{2 \pi \sqrt{\hbar}} (\pi-a_l)\right ).
\]
After computations similar to the ones in the proof of Theorem \ref{thm:part:func}, we obtain:
\begin{equation*}
\mathcal{Z}_{\hbar}(X_n,\alpha)
\stareq
\int_{\mathbf{\widetilde{y}'} \in \widetilde{\mathcal{Y}}'_{\hbar,\alpha}} d\mathbf{\widetilde{y}'}
e^{
2 i \pi \mathbf{\widetilde{y}}^{\prime T} \widetilde{Q}_n  \mathbf{\widetilde{y}'} + \frac{1}{\sqrt{\hbar}} \left ( 2\widetilde{Q}_n \widetilde{\Gamma}(\alpha) + \widetilde{C}(\alpha)\right )^{\!\top} \mathbf{\widetilde{y}'} 
}
\dfrac{
\Phi_\B\left (y'_V\right )
\Phi_\B\left (y'_W\right )
}{
\Phi_\B\left (y'_1\right )
\cdots
\Phi_\B\left (y'_p\right )
\Phi_\B\left (y'_U\right )
},
\end{equation*}

We define a new variable $x:= y'_U + y'_V$ living in the set 
$$\mathcal{Y}^0_{\hbar,\alpha}=\R + \frac{i}{2 \pi \sqrt{\hbar}} (a_U-a_V),$$
 and we also define $
\mathbf{y'}$ (respectively  $\mathcal{Y}'
_{\hbar,\alpha}$) exactly like $
\widetilde{\mathbf{y}'}$ (respectively $
\widetilde{\mathcal{Y}}'_{\hbar,\alpha}$) but with the second-to-last coordinate (corresponding to $y_V$) taken out. 
We also define 
\begin{equation*} 
\mathcal{W}_{n}=
\begin{bmatrix}\mathcal{W}_{n,1} \\ \vdots \\ \mathcal{W}_{n,k} \\ \vdots \\ \mathcal{W}_{n,p} \\ \mathcal{W}_{n,U} \\ \mathcal{W}_{n,W} \end{bmatrix}
:=
\begin{bmatrix}-2p\pi \\ \vdots \\ -2 \pi \left ( k p - \frac{k(k-1)}{2}\right ) \\ \vdots \\ -p(p+1)\pi \\ -(p^2+p+3)\pi \\ \pi\end{bmatrix}
\qquad 
\text{ and }
\qquad
Q_n=\begin{bmatrix}
1 & 1 & \cdots & 1 & 1 & 0 \\ 
1 & 2 & \cdots & 2 & 2 & 0 \\ 
\vdots & \vdots & \ddots & \vdots & \vdots & \vdots \\ 
1 & 2 & \cdots & p & p & 0 \\ 
1 & 2 & \cdots & p & p+1 & -\frac{1}{2} \\ 
0 & 0 & \cdots & 0 & -\frac{1}{2} & 0 
\end{bmatrix}.
\end{equation*}
This time, $Q_n$ is obtained from $\widetilde{Q}_n$ by replacing the two rows corresponding to $y_U$ and $y_V$ with their difference (row of $y_U$ minus row of $y_V$), and by replacing the two  columns corresponding to $y_U$ and $y_V$ with their difference.
We now use the substitution $y'_V = x - y'_U$ and we compute
\begin{align*}
2 i \pi \widetilde{\mathbf{y}}^{\prime T} \widetilde{Q}_n \widetilde{\mathbf{y}'} 
&= 
2 i \pi \left ( 
(\mathbf{y'}^{\!\top} Q_n \mathbf{y'} 
- (p+1) {y'_U}^2 + y'_U y'_W) 
+ (p+1){y'_U}^2 -  y'_U y'_V - 2 y'_U y'_W
 - {y'_V}^2 - y'_V y'_W
\right ) \\
&= 
2 i \pi \left (\mathbf{y'}^{\!\top} Q_n \mathbf{y'} + x y'_U -x y'_W -x^2\right ),
\end{align*}
and 
$\frac{1}{\sqrt{\hbar}} \left ( 2\widetilde{Q}_n \widetilde{\Gamma}(\alpha) + \widetilde{C}(\alpha)\right )^{\!\top} \widetilde{\mathbf{y}'}
= \frac{1}{\sqrt{\hbar}} \left (\mathcal{W}_n^{\!\top} \mathbf{y'}
+x (\frac{1}{2}\lambda_{X_n}(\alpha)-\pi)\right )$,
thus

\begin{align*}
&\mathcal{Z}_{\hbar}(X_n,\alpha) \stareq
\int_{\mathbf{\widetilde{y}'} \in \widetilde{\mathcal{Y}}'_{\hbar,\alpha}} d\mathbf{\widetilde{y}'}
e^{
2 i \pi \mathbf{\widetilde{y}}^{\prime T} \widetilde{Q}_n  \mathbf{\widetilde{y}'} + \frac{1}{\sqrt{\hbar}}\left ( 2\widetilde{Q}_n \widetilde{\Gamma}(\alpha) + \widetilde{C}(\alpha)\right )^{\!\top} \mathbf{\widetilde{y}'} 
}
\dfrac{
\Phi_\B\left (y'_V\right )
\Phi_\B\left (y'_W\right )
}{
\Phi_\B\left (y'_1\right )
\cdots
\Phi_\B\left (y'_p\right )
\Phi_\B\left (y'_U\right )
}
\\
&\stareq
\int dx d\mathbf{y'} 
e^{2 i \pi \left (\mathbf{y'}^{\!\top} Q_n \mathbf{y'}+x(y'_U -y'_W-x)\right )+
\frac{1}{\sqrt{\hbar}} \left (\mathcal{W}_n^{\!\top} \mathbf{y'}
+x (\frac{1}{2}\lambda_{X_n}(\alpha)-\pi)\right )
}
\dfrac{
\Phi_\B\left (x-y'_U\right )
\Phi_\B\left (y'_W\right )
}{
\Phi_\B\left (y'_1\right )
\cdots
\Phi_\B\left (y'_p\right )
\Phi_\B\left (y'_U\right )
}  ,
\end{align*}
where the variables $(\mathbf{y'},x)$ in the last integral lie in $\mathcal{Y}'_{\hbar,\alpha} \times \mathcal{Y}^0_{\hbar,\alpha}$.
The theorem follows.
\end{proof}

We now state the counterpart of Corollary \ref{cor:part:func}, which is proven in exactly the same way.

\begin{corollary}\label{cor:even:part:func}
Let $n$ be a positive even integer, $p=\frac{n-2}{2}$ and $X_n$ the ideal triangulation of $S^3\setminus K_n$ from Figure \ref{fig:id:trig:even}. Then for all angle structures $\alpha \in \mathcal{A}_{X_n}$ and all $\hbar>0$, we have:
\begin{equation*}
\mathcal{Z}_{\hbar}(X_n,\alpha) 
\stareq 
\int_{\mathbb{R}+i \mu_{X_n}(\alpha)  } 
\mathfrak{J}_{X_n}(\hbar,\mathsf{x})
e^{\frac{1}{4 \pi \hbar}  \mathsf{x}  \lambda_{X_n}(\alpha)} 
d\mathsf{x},
\end{equation*}
with the map
\begin{equation*}
\mathfrak{J}_{X_n}\colon(\hbar,\mathsf{x})\mapsto
\left (
\frac{1}{2 \pi \sqrt{\hbar}}
\right )^{p+3}
\int_{\mathcal{Y}_\alpha} 
d\mathbf{y} \
e^{\frac
{
i \mathbf{y}^{\!\top} Q_n \mathbf{y} + 
 i \mathsf{x}(y_U-y_W-\mathsf{x}) + 
  \mathbf{y}^{\!\top} \mathcal{W}_n - \pi \mathsf{x}
}
 {2 \pi \hbar} 
}
\dfrac{
\Phi_\B\left ( \frac{\mathsf{x}-y_U}{2 \pi \sqrt{\hbar}} \right )
\Phi_\B\left ( \frac{y_W}{2 \pi \sqrt{\hbar}} \right )
}{
\Phi_\B\left (\frac{y_1}{2 \pi \sqrt{\hbar}}\right )
\cdots 
\Phi_\B\left (\frac{y_p}{2 \pi \sqrt{\hbar}}\right )
\Phi_\B\left ( \frac{y_U}{2 \pi \sqrt{\hbar}}  \right )
} 
,
\end{equation*}
where $\mu_{X_n},\lambda_{X_n}, \mathcal{W}_n, Q_n$ are the same as in Theorem \ref{thm:even:part:func}, and 
$$\mathcal{Y}_\alpha = \left (
\prod_{k=1, \ldots,p,U}\left (\R - i (\pi - a_k)\right ) 
\right )
\times 
\left (\R + i (\pi - a_W)\right ).$$
\end{corollary}

\begin{proof}
Exactly similar to the proof of Corollary \ref{cor:part:func}.
\end{proof}

We finally come to H-triangulations for even twists knots.
Again, before stating Theorem \ref{thm:part:func:Htrig:even}, we compute the weights on each edge of the H-triangulation $Y_n$ given in Figure \ref{fig:H:trig:even} (for $n \geqslant 3$ even).  

We use exactly the same {notation} as the odd case. We denoted $\overrightarrow{\eta_0}, \ldots, \overrightarrow{\eta_{p+1}}, \overrightarrow{\eta_s}, \overrightarrow{\eta_d}, \overrightarrow{K_n} \in (Y_n)^1$ the $p+5$ edges in $Y_n$ respectively represented in Figure \ref{fig:H:trig:even} by arrows with circled $0$, \ldots, circled $p+1$, simple arrow, double arrow and blue simple arrow.

For $\alpha=(a_1,b_1,c_1,\ldots,a_p,b_p,c_p,a_U,b_U,c_U,a_V,b_V,c_V,a_W,b_W,c_W,a_Z,b_Z,c_Z) \in \mathcal{S}_{Y_n}$ a shape structure on $Y_n$, the weights of each edge are given by:
\begin{itemize}
\item $\widehat{\omega}_s(\alpha):= \omega_{Y_n,\alpha}(\overrightarrow{\eta_s})=
2 a_U+b_V+c_V+a_W+b_W+a_Z
$
\item $\widehat{\omega}_d(\alpha):= \omega_{Y_n,\alpha}(\overrightarrow{\eta_d})=
b_U+c_U+c_W+b_Z+c_Z
$
\item $\omega_0(\alpha):= \omega_{Y_n,\alpha}(\overrightarrow{\eta_0})=
2 a_1 + c_1 + 2 a_2 + \ldots + 2 a_p + a_V+c_W
$
\item $\omega_1(\alpha):= \omega_{Y_n,\alpha}(\overrightarrow{\eta_1})=
2b_1+c_2
$
\\
\vspace*{-2mm}
\item $\omega_k(\alpha):= \omega_{Y_n,\alpha}(\overrightarrow{\eta_k})=
c_{k-1}+2b_k+c_{k+1}
$
 \ \
(for $2\leqslant k \leqslant p-1$)
\\
\vspace*{-2mm}
\item $\widehat{\omega}_p(\alpha):= \omega_{Y_n,\alpha}(\overrightarrow{\eta_p})=
c_{p-1}+2b_p+c_U+a_V+b_V+a_W+b_Z+c_Z$
\item $\omega_{p+1}(\alpha):= \omega_{Y_n,\alpha}(\overrightarrow{\eta_{p+1}})=
c_p+b_U+c_V+b_W
$
\item $\widehat{\omega}_{\overrightarrow{K_n}}(\alpha):= \omega_{Y_n,\alpha}(\overrightarrow{K_n})=
a_Z
$
\end{itemize}

We can now compute the partition function for the H-triangulations $Y_n$ ($n$ even), and prove the following theorem. As for the odd case, we will denote $\mathcal{S}_{Y_n \backslash Z}$ the space of shape structures on every tetrahedron of $Y_n$ except for $Z$.

\begin{theorem}\label{thm:part:func:Htrig:even}
Let $n$ be a positive even integer and $p=\frac{n-2}{2}$. Consider the one-vertex H-triangulation $Y_n$ of the pair $(S^3,K_n)$ described in Figure \ref{fig:H:trig:even}. Then for every $\hbar>0$ and for every 
$\tau\in \mathcal{S}_{Y_n \setminus Z} \times \overline{\mathcal{S}_Z}$
 such that $\omega_{Y_n,\tau}$ vanishes on $\overrightarrow{K_n}$ and is equal to $2\pi$ on every other edge, one has
\begin{equation*}
\underset{\tiny \begin{matrix}\alpha \to \tau \\ \alpha \in \mathcal{S}_{Y_n}\end{matrix}}{\lim}
	 \Phi_{\B}\left( \frac{\pi-\omega_{Y_n,\alpha}\left (\overrightarrow{K_n}\right )}{2\pi i \sqrt{\hbar}} \right)  \mathcal{Z}_{\hbar}(Y_n,\alpha) \stareq J_{X_n}(\hbar,0),
\end{equation*}
where $J_{X_n}$ is defined in Theorem \ref{thm:even:part:func}.
\end{theorem}

\begin{proof}
Let $n$ be an even integer and $p=\frac{n-2}{2}$. The proof is similar to the odd case and will be separated in three steps: computing the partition function $\mathcal{Z}_{\hbar}(Y_n,\alpha)$, applying the dominated convergence theorem in $\alpha \to \tau$ and finally retrieving the value $J_{X_n}(\hbar,0)$ in $\alpha =\tau$.

\textit{Step 1. Computing the partition function $\mathcal{Z}_{\hbar}(Y_n,\alpha)$.}

Like in the proof of Theorem \ref{thm:even:part:func} we start by computing the kinematical kernel. We denote 
$\widehat{\mathbf{t}}=(t_1,\ldots,t_p,t_U,t_V,t_W,t_Z) \in \mathbb{R}^{Y_n^3}$ and $
{\widehat{\mathbf{x}}=(e_1,\ldots,e_p,e_{p+1},f_1,\ldots,f_p,v,r,s,s',g,u,m) \in \mathbb{R}^{Y_n^2}}$.

Like in the proof of Theorem \ref{thm:part:func:Htrig:odd}, using Figure \ref{fig:H:trig:even}, we compute
$$
\mathcal{K}_{Y_n}\left (\mathbf{\widehat{t}}\right ) =
\int_{\widehat{\mathbf{x}} \in \R^{Y_n^{2}}} d\widehat{\mathbf{x}} 
\int_{\widehat{\mathbf{w}} \in \R^{2 (p+4)}} d\widehat{\mathbf{w}} \
e^{ 2 i \pi \mathbf{\widehat{t}}^{\!\top} \widehat{S_e} \widehat{\mathbf{x}}}
e^{ -2 i \pi \widehat{\mathbf{w}}^{\!\top} \widehat{H_e} \widehat{\mathbf{x}}}
e^{ -2 i \pi \widehat{\mathbf{w}}^{\!\top} \widehat{D} \mathbf{\widehat{t}}},
$$
where $\widehat{D}$ is like in proof of Theorem \ref{thm:part:func:Htrig:odd}, whereas the matrices $\widehat{S_e}$ and $\widehat{H_e}$ are given by:

$$
\widehat{S_e}=\kbordermatrix{
	\mbox{} 	& e_1 	& \dots 	& e_p	&e_{p+1} 	& \omit\vrule	&f_1	& \ldots 	& f_p	&\omit\vrule	&  v 	& r 	& s 	& s'	& g	& u 	& m 	\\
	t_1 		& 1	 	 & 		&\push{\low{0}} 	& \omit\vrule	& 	&  		&  	&\omit\vrule	&   	&  	&   	&	&	&  	&  	\\
	\vdots 	& 	    	&\ddots 	& 		&	  	& \omit\vrule	& 	&  0		&  	&\omit\vrule	& 	&   	&  	& 0	&	&	&  	\\
	t_{p}    	&\push{0}		       	& 1        	& 		& \omit\vrule	&  	&  		&	&\omit\vrule	&	&	&	&	&	&	&  	\\ 
\cline{1-1} \cline{2-17}
	t_U	 	& 		& 		& 		&  		& \omit\vrule	&	&	 	&	&\omit\vrule	&0  	&1 	&0  	& 0	& 0	&0 	& 0 	\\
	t_V	 	& 		&\push{0}	  		&  		& \omit\vrule	&	& 	0	& 	&\omit\vrule	&0  	&0 	&-1  	& 0	& 0	& 0 	& 0 	\\
	t_W	 	& 		&  		& 		& 	 	& \omit\vrule	&	&	  	&  	&\omit\vrule	&0 	&0 	&0  	& 0	& 0	& -1 	& 0	\\
	t_Z	 	& 		&  		& 		& 	 	& \omit\vrule	&	&	  	&  	&\omit\vrule	&0 	&0 	&0  	& 0	& 0	& 0 	& 1
},
$$

$$
\widehat{H_e}=\kbordermatrix{
	\mbox{} 	& e_1 	& e_2 	& \dots 	&e_p		& e_{p+1}	& \omit\vrule	&f_1	&f_{2}	&  \ldots 	& f_p	&\omit\vrule	&  v 	& r 	& s 	& s'		& g	& u 	& m 	\\
	w_1 		& 1		&   -1 	& 		&		&		& \omit\vrule	& 1	&		&  		&  	&\omit\vrule	&   	&  	&  	&		&	&  	&  	\\
	\vdots 	&     		&   \ddots	& \ddots 	&	\push{0}		& \omit\vrule	& 	&\ddots	&\push{0}  	&\omit\vrule	& 	&   	&  	& \low{0}	&	&	&  	\\
	\vdots 	&     	\push{0}   		& \ddots 	&\ddots	&		& \omit\vrule	& \push{0}  	&\ddots	&  	&\omit\vrule	& 	&   	&  	&		&	&	&  	\\
	w_{p} 	&		&		&		&	1	& -1		& \omit\vrule 	&	&		&		&1	&\omit\vrule	&	&	&	&		&	&	&  	\\
\cline{1-1} \cline{2-19}
	w_{U} 	&		&		&		&		&  		& \omit\vrule 	&	&		&		&0	&\omit\vrule	& -1	& 1	& 0	&	1	& 0	& 0	& 0 	\\
	w_{V} 	&		&		& \low{0}	&		&  		& \omit\vrule 	&	&		&		&1	&\omit\vrule	& 0	& 0	& 1	&	0	& -1	& 0	& 0  	\\
	w_{W} 	&		&		&		&		&  		& \omit\vrule 	&	&		&		&0	&\omit\vrule	& -1	& 1	& 0	&	0	& 0	& 1	& 0 	\\
	w_{Z} 	&		&		&		&		&  		& \omit\vrule 	&	&		&		&0	&\omit\vrule	& 0	& 0	& 1	&	0	& 0	& 0	& 0 	\\
\cline{1-1} \cline{2-19}
	w'_1 		& -1		&   	 	& 		&		&		& \omit\vrule	& 1	&  		&  		&	&\omit\vrule	&   	&  	&   	&		&	&  	&  	\\
	\vdots 	&     		&   		& 	 	&		&		& \omit\vrule	& -1	&  \ddots	&  \push{0}	&\omit\vrule	& 	&   	&  	& \low{0}	&	&	&  	\\
	\vdots 	&     		&   		& 	 	&		&		& \omit\vrule	& 	&  \ddots	& \ddots 	&	&\omit\vrule	& 	&   	&  	&		&	&	&  	\\
	w'_{p} 	&		&		&		&		&  		& \omit\vrule 	&	\push{0}	&-1		&1	&\omit\vrule	&	&	&	&		&	&	&  	\\
\cline{1-1} \cline{2-19}
	w'_{U} 	&		&		&		&		&  0		& \omit\vrule 	&	&		&		& 0	&\omit\vrule	& 0	& 0	& 0	&	1	& -1	& 0	& 0  	\\
	w'_{V} 	&		&		&		&		&  0		& \omit\vrule 	&	&		&		& 1	&\omit\vrule	& 0	& 0	& 0	&	0	& 0	& -1	& 0 	\\
	w'_{W} 	&		&		&		&		&  -1		& \omit\vrule 	&	&		&		& 0	&\omit\vrule	& 0	& 1	& 0	&	0	& 0	& 0	& 0 	\\
	w'_{Z} 	&		&		&		&		&  0		& \omit\vrule 	&	&		&		& 0	&\omit\vrule	& 0	& 0	& 1	& 	-1	& 0	& 0	& 0 
}.$$

Like in the odd case, let us define $S_e$ the submatrix of $\widehat{S_e}$ without the $m$-column, $H_e$ the submatrix of $\widehat{H_e}$ without the $m$-column and the $w_V$-row, $R_{e,V}$ this very $w_V$-row of $\widehat{H_e}$, $D$ the submatrix of $\widehat{D}$ without the $w_V$-row, $\mathbf{x}$ the subvector of $\widehat{\mathbf{x}}$ without the variable $m$ and $\mathbf{w}$ the subvector of $\widehat{\mathbf{w}}$ without the variable $w_V$. 
We remark that $H_e$ is invertible and $\det(H_e)=-1$. Hence, by using multi-dimensional Fourier transform and the integral definition of the Dirac delta function like in the odd case, we compute 

$$
\mathcal{K}_{Y_n}\left (\mathbf{\widehat{t}}\right ) 
= \delta(-t_Z)
e^{2i \pi \widehat{\mathbf{t}}^{\!\top} (-S_e H_e^{-1} D) \mathbf{\widehat{t}}}
\delta  (-R_{e,V} H_e^{-1} D \mathbf{\widehat{t}}).$$

We can now compute $H_e^{-1}=$
$$ 
\kbordermatrix{
	\mbox{} 	& w_{1}	& w_{2} 	& \ldots	& w_{p-1}	& w_{p}	& \omit\vrule	& w_{U}			& w_{W}			& w_{Z}			& \omit\vrule	&  w'_{1}	& w'_{2}	& \ldots	& w'_{p-1}			& w'_{p}			& \omit\vrule 	& w'_{U}	& w'_{V}			& w'_{W}	& w'_{Z}			\\
	e_{1} 	& 	0	&		& \cdots	& 		&	0	& \omit\vrule 	& 	-1			& 	1			& 	1			& \omit\vrule	&  	-1	& 	-1	&	\push{\cdots}			& 	-1			& \omit\vrule 	& 	0	& 	1			& 	0	&	-1			\\
	e_{2} 	& 	-1	& 	0	& 		& 		& 		& \omit\vrule 	& 	-2			& 	2			& 	2			& \omit\vrule	&  	-1	& 	-2	&	\push{\cdots}			& 	-2			& \omit\vrule 	& 	0	& 	2			& 	0	&	-2			\\
\low{\vdots} 	& 	-1	& 	-1	& \ddots	& 		& \vdots	& \omit\vrule 	& 				& 				& 				& \omit\vrule	& 	 	&		& \ddots	& 				& 	\vdots		& \omit\vrule 	& 		& 				& 		&				\\
		 	& \vdots	& 		& \ddots	& 	0	& 	0	& \omit\vrule 	& \vdots			& \vdots			&  \vdots			& \omit\vrule	&  \vdots	& \vdots	&		&\text{\tiny {\(1-p\)}}	& \text{\tiny {\(1-p\)}}	& \omit\vrule 	& \vdots	& \vdots			& \vdots	& \vdots			\\
	e_{p} 	& 		& 		& 		& 	-1	& 	0	& \omit\vrule 	& 				& 				& 				& \omit\vrule	&  		& 		&		& \text{\tiny {\(1-p\)}}	& 	-p			& \omit\vrule 	& 		& 				& 		&				\\
	e_{p+1} 	& -1		& 		& \cdots	& 		& 	-1	& \omit\vrule 	& \text{\tiny {\(-p-1\)}}	& \text{\tiny {\(p+1\)}}	& \text{\tiny {\(p+1\)}}	& \omit\vrule	&  	-1	&	-2	&\cdots	& \text{\tiny {\(1-p\)}}	& 	-p			& \omit\vrule 	& 	0	& \text{\tiny {\(p+1\)}}	& 	0	& \text{\tiny {\(-p-1\)}}	\\
\cline{1-1} \cline{2-21}
	f_{1} 	& 		& 		& 		& 		& 		& \omit\vrule 	& 	-1			& 	1			& 	1			& \omit\vrule	&  	0	& -1		& 	\push{\cdots}			&	-1			& \omit\vrule 	& 	0	& 	1			& 	0	&	-1			\\
	f_{2} 	& 		& 		& 		& 		& 		& \omit\vrule 	& 				& 				& 				& \omit\vrule	&  	0	&0		& \ddots	&				&	\low{\vdots}	& \omit\vrule 	& 		& 				& 		&				\\
	\vdots 	& 		& 		& 	0	& 		& 		& \omit\vrule 	& \vdots			& \vdots			& \vdots			& \omit\vrule	& \vdots 	& 		& \ddots	&	-1			&	-1			& \omit\vrule 	& \vdots	& \vdots			& \vdots	& \vdots			\\
	f_{p-1} 	& 		& 		& 		& 		& 		& \omit\vrule 	& 				& 				& 				& \omit\vrule	&  	0	& 		& 		&	0			&	-1			& \omit\vrule 	& 		& 				& 		&				\\
	f_{p} 	& 		& 		& 		& 		& 		& \omit\vrule 	& 	-1			& 	1			& 	1			& \omit\vrule	&  	0	& 		& \cdots	&				&	0			& \omit\vrule 	& 	0	& 	1			& 	0	&	-1			\\
\cline{1-1} \cline{2-21}
	v	 	& 	-1	& 		& \cdots	& 		& 	-1	& \omit\vrule 	& \text{\tiny {\(-p-2\)}}	& \text{\tiny {\(p+1\)}}	&\text{\tiny {\(p+2\)}}	& \omit\vrule	&  	-1	& 	-2	& \push{\cdots}				&	-p			& \omit\vrule 	& 	0	& \text{\tiny {\(p+1\)}}	& 	1	& \text{\tiny {\(-p-2\)}}	\\
	r	 	& 	-1	& 		& \cdots	& 		& 	-1	& \omit\vrule 	& \text{\tiny {\(-p-1\)}}	& \text{\tiny {\(p+1\)}}	&\text{\tiny {\(p+1\)}}	& \omit\vrule	&  	-1	& 	-2	& \push{\cdots}				&	-p			& \omit\vrule 	& 	0	& \text{\tiny {\(p+1\)}}	& 	1	& \text{\tiny {\(-p-1\)}}	\\
	s	 	& 		& 		& 		& 		& 		& \omit\vrule 	& 	0			& 	0			& 	1			& \omit\vrule	&  		& 		& 		&				&				& \omit\vrule 	& 	0	& 	0			& 	0	&	0			\\
	s'	 	& 		& 		& \low{0}	& 		& 		& \omit\vrule 	& 	0			& 	0			& 	1			& \omit\vrule	&  		& 		& \low{0}	&				&				& \omit\vrule 	& 	0	& 	0			& 	0	&	-1			\\
	g	 	& 		& 		& 		& 		& 		& \omit\vrule 	& 	0			& 	0			& 	1			& \omit\vrule	&  		& 		& 		&				&				& \omit\vrule 	& 	-1	& 	0			& 	0	&	-1			\\
	u	 	& 		& 		& 		& 		& 		& \omit\vrule 	& 	-1			& 	1			& 	1			& \omit\vrule	&  		& 		& 		&				&				& \omit\vrule 	& 	0	& 	0			& 	0	&	-1		
},$$
and thus find that $-R_{e,V} H_e^{-1} D \mathbf{\widehat{t}}=-t_U-t_V$ and
$$-S_e H_e^{-1} D= \kbordermatrix{
	\mbox{}	&t_1		&t_2		&\cdots 	& t_{p-1} 	& t_p 	& \omit\vrule	& t_U 	& t_V 	& t_W 	&t_Z		\\
	t_1 		& 1 		& 1 		& \cdots 	& 1  		& 1 		& \omit\vrule	& 0 		& -1 		& 0  		& 1 		\\
	t_2 		& 1 		& 2 		& \cdots 	& 2  		& 2 		& \omit\vrule	& 0 		& -2 		& 0  		& 2 		\\
	\vdots 	& \vdots 	& \vdots 	& \ddots 	& \vdots  	& \vdots 	& \omit\vrule	& \vdots 	& \vdots 	& \vdots 	& \vdots	\\
	t_{p-1} 	& 1 		& 2 		& \cdots 	& p-1  	& p-1 	& \omit\vrule	& 0		& -(p-1)  	& 0  		& p-1 	\\
	t_p 		& 1 		&  2 		& \cdots 	&  p-1 	& p 		& \omit\vrule	& 0		& -p 		& 0  		& p 		\\
\cline{1-1} \cline{2-11}
	t_U 		& 1 		& 2 		& \cdots 	& p-1  	& p  		& \omit\vrule	& 0 		& -(p+1) 	& -1  	& p+1 	\\
	t_V 		& 0 		& 0 		& \cdots 	&  0 		& 0 		& \omit\vrule	& 0 		& 0 		& 0  		& 0 		\\
	t_W 		& 0 		& 0 		& \cdots 	&  0 		& 0 		& \omit\vrule	& 0 		& 0 		& 0 		& -1 		\\
	t_Z 		& 0 		& 0 		& \cdots 	&  0 		& 0 		& \omit\vrule	& 0 		& 0 		& 0 		& 0
}.$$

Since 
$$\widehat{\mathbf{t}}^{\!\top} (-S_e H_e^{-1} D) \mathbf{\widehat{t}} = \mathbf{t}^{\!\top} Q_n \mathbf{t}
+ (-t_U-t_V)(t_1+\ldots+p t_p+(p+1)t_U)
+t_Z(t_1+\ldots+pt_p+(p+1)t_U-t_W),$$ 
where $\mathbf{t}=(t_1,\ldots,t_p,t_U,t_W)$ and $Q_n$ is defined in Theorem \ref{thm:even:part:func}, we conclude that the kinematical kernel can be written as  
\[
\mathcal{K}_{Y_n}(\mathbf{\widehat{t}})= 
e^{2 i \pi \left( \mathbf{t}^{\!\top} Q_n \mathbf{t}-t_W t_Z +(t_Z - t_U-t_V)(t_1 + \cdots + p t_p + (p+1) t_U) \right)}
\delta(t_Z) \delta(-t_U - t_V).
\]

We now compute the dynamical content. We denote
$\alpha=(a_1,b_1,c_1,\ldots,a_W,b_W,c_W,a_Z,b_Z,c_Z)$ a general 
vector in $\mathcal{S}_{Y_n}$.
Let $\hbar>0$. The dynamical content $\mathcal{D}_{\hbar,Y_n}(\mathbf{\widehat{t}},\alpha)$ is equal to:
\[
e^{\frac{1}{\sqrt{\hbar}} \widehat{C}(\alpha)^{\!\top} \mathbf{\widehat{t}}}
\dfrac{
\Phi_\B\left (t_V + \frac{i}{2 \pi \sqrt{\hbar}}(\pi-a_V)\right )
\Phi_\B\left (t_W + \frac{i}{2 \pi \sqrt{\hbar}} (\pi-a_W)\right )
}{
\prod_{k=1}^p
\Phi_\B\left (t_k - \frac{i}{2 \pi \sqrt{\hbar}} (\pi-a_k)\right )
\Phi_\B\left (t_U + \frac{i}{2 \pi \sqrt{\hbar}} (\pi-a_U)\right )
\Phi_\B\left (t_Z - \frac{i}{2 \pi \sqrt{\hbar}} (\pi-a_Z)\right )
},
\]
where $\widehat{C}(\alpha) = (c_1, \ldots, c_p,c_U,c_V, c_W, c_Z)^{\!\top}$.

Let us come back to the computation of the partition function of the Teichm\"uller TQFT. 
We begin by integrating over the variables $t_V$ and $t_Z$, which consists in removing the two Dirac delta functions $\delta(-t_Z)$ and $\delta(-t_U -t_V)$ in the kinematical kernel and replacing $t_Z$ by $0$ and $t_V$ by $-t_U$ in the other terms. Therefore, we have
$$
\Phi_{\B}\left( \frac{\pi-a_Z}{2\pi i \sqrt{\hbar}} \right)  \mathcal{Z}_{\hbar}(Y_n,\alpha) 
\stareq \int_{\mathbf{t}\in\R^{p+2}} d\mathbf{t} \
e^{2 i \pi \mathbf{t}^{\!\top} Q_n\mathbf{t}} e^{\frac{1}{\sqrt{\hbar}} (c_1 t_1 + \cdots + c_p t_p + (c_U - c_V)t_U + c_W t_W)} 
\Pi(\mathbf{t},\alpha,\hbar),
$$
where $\mathbf{t} =(t_1, \ldots,t_p,t_U,t_W)$ and
$$
\Pi(\mathbf{t},\alpha,\hbar) :=
\frac{
\Phi_\B\left (-t_U + \frac{i}{2 \pi \sqrt{\hbar}}(\pi-a_V)\right )
\Phi_\B\left (t_W + \frac{i}{2 \pi \sqrt{\hbar}} (\pi-a_W)\right )
}{
\Phi_\B\left (t_1 - \frac{i}{2 \pi \sqrt{\hbar}} (\pi-a_1)\right )
\cdots
\Phi_\B\left (t_p - \frac{i}{2 \pi \sqrt{\hbar}} (\pi-a_p)\right )
\Phi_\B\left (t_U - \frac{i}{2 \pi \sqrt{\hbar}} (\pi-a_U)\right )
}.
$$

\textit{Step 2. Applying the dominated convergence theorem for $\alpha \to \tau$.}

This step is exactly as in the proof of Theorem \ref{thm:part:func:Htrig:odd}.
As for the odd case, for the rest of the proof, set 
$$\tau = (a^\tau_1,b^\tau_1,c^\tau_1,\ldots,a^\tau_Z,b^\tau_Z,c^\tau_Z) \in \mathcal{S}_{Y_n \setminus Z} \times \overline{\mathcal{S}_Z}$$
be such that $\omega_j(\tau) = 2 \pi$ for all $j \in \{0,1, \ldots, p-1,p+1\}$,
$\widehat{\omega}_j(\tau) = 2 \pi$ for all $j \in \{s,d,p\}$
 and $\widehat{\omega}_{\overrightarrow{K_n}}(\tau)=a^\tau_Z=0$.

\textit{Step 3. Retrieving the value $J_{X_n}(\hbar,0)$ in $\alpha =\tau$.}

Similarly as in the odd case, we do the following change of variables:
\begin{itemize}
\item $y'_k = t_k - \frac{i}{2 \pi \sqrt{\hbar}} (\pi-a_k)$ for $k \in \{1,\ldots,p,U\}$,
\item $y'_W = t_W + \frac{i}{2 \pi \sqrt{\hbar}} (\pi-a_W)$,
\end{itemize}
and we denote $\mathbf{y'}=\left (y'_1, \ldots, y'_{p}, y'_U, y'_W\right )^{\!\top}$. 
Again $a^\tau_U-a^\tau_V = (\widehat{\omega}_{s}(\tau)- 2\pi)+(\widehat{\omega}_{d}(\tau)- 2\pi) = 0$.

We also denote 
$$
\widetilde{\mathcal{Y}}'_{\hbar,\tau} :=
\prod_{k=1,\ldots,p,U}\left (\R - \frac{i}{2 \pi \sqrt{\hbar}} (\pi-a^\tau_k)\right ) 
\times 
\left (\R + \frac{i}{2 \pi \sqrt{\hbar}} (\pi-a^\tau_W)\right ),
$$
the subset of $\C^{p+2}$ on which the variables in $\mathbf{y'}$ reside. 

By a similar computation as in the proof of Theorem \ref{thm:even:part:func}, we obtain
\begin{align*}
&\int_{\mathbf{t}\in\R^{p+2}} d\mathbf{t} 
e^{2 i \pi \mathbf{t}^{\!\top} Q_n\mathbf{t}} e^{\frac{1}{\sqrt{\hbar}} (c^\tau_1 t_1 + \cdots + c^\tau_p t_p + (c^\tau_U - c^\tau_V)t_U + c^\tau_W t_W)} \Pi(\mathbf{t},\tau,\hbar)\\
&\stareq
\int_{\mathbf{y'} \in \mathcal{Y}'_{\hbar,\tau}}  d\mathbf{y'}
e^{
2 i \pi \mathbf{y}^{\prime T} Q_n  \mathbf{y'} + \frac{1}{\sqrt{\hbar}} \mathcal{W}(\tau)^{\!\top} \mathbf{y'} 
}
\dfrac{
\Phi_\B\left (-y'_U \right )
\Phi_\B\left (y'_W\right )
}{
\Phi_\B\left (y'_1\right )
\cdots
\Phi_\B\left (y'_p\right )
\Phi_\B\left (y'_U\right )
},
\end{align*}

where for any $\alpha \in \mathcal{S}_{Y_n \setminus Z}$, $\mathcal{W}(\alpha)$ is defined as
$$\mathcal{W}(\alpha):= 2 Q_n \Gamma(\alpha)+C(\alpha)+(0,\ldots,0,-c_V,0)^{\!\top},$$
with  $\Gamma(\alpha)=(a_1-\pi,\ldots,a_p-\pi,a_U-\pi, \pi-a_W)^{\!\top}$ and $C(\alpha)=(c_1,\ldots,c_p,c_U,c_W)$.
Hence, from the value of $J_{X_n}(\hbar,0)$, it remains only to prove that $\mathcal{W}(\tau) = \mathcal{W}_n$.

Let us denote $\Lambda: (u_1,\ldots,u_p,u_U,u_V,u_W) \mapsto (u_1,\ldots,u_p,u_U,u_W)$ the process of forgetting the second-to-last coordinate. then obviously $C(\alpha) = \Lambda (\widetilde{C}(\alpha))$. Recall from the proof of Theorem \ref{thm:even:part:func}
that $\widetilde{\mathcal{W}}(\alpha) = 2 \widetilde{Q}_n \widetilde{\Gamma}(\alpha) + \widetilde{C}(\alpha)$ depends almost only on edge weights of the angles in $X_n$.

Thus, a direct calculation shows that for any $\alpha \in \mathcal{S}_{Y_n \setminus Z}$, we have
\begin{equation*} \label{eqn:v:Alpha:Odd}
\mathcal{W}(\alpha) =
\Lambda(\widetilde{\mathcal{W}}(\alpha)) + 
\begin{bmatrix}
0 \\ \vdots \\ 0 \\ 
-c_V +(\pi-a_V)+(\pi-a_W) 
\\ a_U-a_V
\end{bmatrix}.
\end{equation*}

Now, if we specify $\alpha=\tau$, then the weights $\omega_{X_n,j}(\alpha)$ appearing in $\Lambda(\widetilde{\mathcal{W}}(\alpha))$  all become $2\pi$, since 
$\omega_s(\tau) =\widehat{\omega}_s(\tau)-\widehat{\omega}_{\overrightarrow{K_n}}(\tau) = 2 \pi$
and
$\omega_{p}(\tau) =\widehat{\omega}_d(\tau)+\widehat{\omega}_{p}(\tau) - 2\left (\pi-\widehat{\omega}_{\overrightarrow{K_n}}(\tau)\right ) = 2\pi.$
Hence 
$$\mathcal{W}(\tau)= \mathcal{W}_{n} + 
\begin{bmatrix}
0 \\ \vdots \\ 0 \\ 
\frac{1}{2}\lambda_{X_n}(\tau)- \pi -c^\tau_V +(\pi-a^\tau_V)+(\pi-a^\tau_W) \\ a^\tau_U-a^\tau_V
\end{bmatrix}.
$$
Finally, since
$\frac{1}{2}\lambda_{X_n}(\tau)=a^\tau_V-a^\tau_U+a^\tau_W-b^\tau_V$ and $a^\tau_U-a^\tau_V=0$, we conclude that $\mathcal{W}(\tau) = \mathcal{W}_n$ and the theorem is proven.
\end{proof}

\subsection{Geometricity implies the volume conjecture} \label{sub:even:vol:conj}

In this section we will prove the following theorem, which can be compared with Theorem \ref{thm:vol:conj}.

\begin{theorem}\label{thm:even:vol:conj}
Let $n$ be a positive even integer, and $J_{X_n}, \mathfrak{J}_{X_n}$ the functions defined in Theorem \ref{thm:even:part:func} and Corollary \ref{cor:even:part:func}. If the ideal triangulation $X_n$ is geometric, then
$$
\lim_{\hbar \to 0^+} 2\pi \hbar \log \vert J_{X_n}(\hbar,0) \vert
= \lim_{\hbar \to 0^+} 2\pi \hbar \log \vert \mathfrak{J}_{X_n}(\hbar,0) \vert
 = -\emph{Vol}(S^3\backslash K_n).$$
\end{theorem}

The following Corollary \ref{cor:vol:conj:even} is  an immediate consequence of Theorem \ref{thm:even:vol:conj} and Theorem \ref{thm:appendix:geom:even}.

\begin{corollary}\label{cor:vol:conj:even}
The Teichm\"uller TQFT volume conjecture of Andersen--Kashaev is proven for the even twist knots.
\end{corollary}

\begin{proof}[Proof of Theorem \ref{thm:even:vol:conj}]
To prove Theorem \ref{thm:even:vol:conj}, we will follow exactly the same general path as in Section \ref{sec:vol:conj}. For the sake of brevity, we will thus only state the modifications that are due to the fact that $n$ is even instead of odd. For the remainder of the section, let $n$ be a positive even integer such that $X_n$ is geometric.
Let us first list the changes in {notation}:
\begin{itemize}
\item The open ``multi-band'' is now
$\mathcal{U} := \left (
\prod_{k=1, \ldots,p,U}\left (\R + i (-\pi,0)\right ) 
\right )
\times 
\left (\R + i (0,\pi)\right ),$
and the closed one $\mathcal{U}_{\delta}$ (for $\delta>0$) is
$\mathcal{U}_{\delta}:= \prod_{k=1,\ldots,p,U}\left (\R + i [-\pi+\delta,-\delta] \right ) 
\times 
\left (\R + i [\delta,\pi-\delta]\right ).$
\item As said in Corollary \ref{cor:even:part:func}, $\mathcal{Y}_\alpha := \left (
\prod_{k=1, \ldots,p,U}\left (\R - i (\pi - a_k)\right ) 
\right )
\times 
\left (\R + i (\pi - a_W)\right ).$
\item The potential function $S\colon \mathcal{U} \to \C$ is now
$S := \mathbf{y} \mapsto$
$$i \mathbf{y}^{\!\top} Q_n \mathbf{y} +  
  \mathbf{y}^{\!\top} \mathcal{W}_n
  + i \Li\left (-e^{y_1}\right ) + \cdots
  + i \Li\left (-e^{y_p}\right )
  + i \Li\left (-e^{y_U}\right )
  - i \Li\left (-e^{-y_U}\right )
  - i \Li\left (-e^{y_W}\right ).$$
  The expressions of its quantum deformations $S_{\B}$ and $S'_{\B}$ (for $\B>0$) should be obvious.
  \item The vector ${\zeta}$, first appearing in Proposition \ref{prop:all:contour:Sb}, is now 
  ${\zeta} := (-1, \ldots,-1,-2,1)$.
\end{itemize}

We will state and prove several facts, which are variants of statements in Section \ref{sec:vol:conj}.

Before all, let us remark that the non-degeneracy of the holomorphic hessian of $S$ (Lemma \ref{lem:hess}) and the strict concavity of $\Re(S)$ (Lemma \ref{lem:concave}) are obtained immediately by arguments and computations similar with the ones in Section \ref{sec:vol:conj}.

However, relating the vanishing of $\nabla S$ to Thurston's gluing equations (Lemma \ref{lem:grad:thurston}) needs a little more detail:\\
\textit{Fact 1. The diffeomorphism $\psi$ induces a bijective mapping between
$\{\mathbf{y} \in \mathcal{U};
\nabla S(\mathbf{y}) = 0\}$ and
$\{\mathbf{z} \in (\R+i\R_{>0})^{p+2}; \mathcal{E}_{X_n}^{co}(\mathbf{z})\}$.}

The system $\mathcal{E}^{co}_{X_n}(\mathbf{z})$ of equations (satisfied by the complete hyperbolic structure) is: 
\begin{itemize}
	\item $\mathcal{E}_{X_n,0}(\mathbf{z}) \colon \Log(z'_1) + 2 \Log(z_1)+\cdots + 2\Log(z_p)+2\Log(z_U) = 2i\pi$
	\item $\mathcal{E}_{X_n,1}(\mathbf{z}) \colon 2\Log(z''_1)+\Log(z'_2)=2i\pi$\\
	\vspace*{-2mm}
	\item $\mathcal{E}_{X_n,k}(\mathbf{z}) \colon \Log(z'_{k-1})+2\Log(z''_k)+\Log(z'_{k+1})=2i\pi$\ \
	(for $2\leqslant k \leqslant p-1$)\\
	\vspace*{-2mm}
	\item $\mathcal{E}_{X_n,p+1}^{co}(\mathbf{z}) \colon \Log(z'_{p}) +2 \Log(z''_{U})+\Log(z_{W})=2i\pi$
	\item $\mathcal{E}_{X_n,s}^{co}(\mathbf{z}) \colon \Log(z''_{W}) -\Log(z_{U})=0$
\end{itemize}

To prove Fact 1, let us first compute, for  $\mathbf{y} \in \mathcal{U}$:
$$
\nabla S(\mathbf{y}) =
2 i Q_n \mathbf{y} + \mathcal{W}_n + i
\begin{pmatrix}
-\Log (1+e^{y_1})\\
\vdots\\
-\Log (1+e^{y_p})\\
-\Log (1+e^{y_U})-\Log (1+e^{-y_U})\\
\Log (1+e^{y_W})
\end{pmatrix}.$$

Then, we define the  matrix $A=\kbordermatrix{
	\mbox{}	&y_1	&y_2		&y_3		&\cdots 	&y_p	& \omit\vrule	&y_U	& y_W 	\\
	y_1 		& 1 	& 		& 		& 		& 	& \omit\vrule	& 		&  		\\
	y_2 		&-2 	& 1 		& 		& 		&0 	& \omit\vrule	&  		&		\\
	y_3 		&1 	& -2 		& 1 		& 		& 	& \omit\vrule	& 		&		\\
	\vdots 	& 	& \ddots 	& \ddots 	& \ddots 	& 	& \omit\vrule	& 		&		\\
	y_p		& 	& 		& 1 		& -2 		& 1 	& \omit\vrule	&0 		&0		\\
\cline{1-1} \cline{2-9}
	y_U 		& 	& 		& 		& 		& -1 	& \omit\vrule	& 1 		& 1		\\
	y_W 		& 	&0 		& 		& 		&0 	& \omit\vrule	& 0		&1}
\in GL_{p+2}(\Z)$, and 
we compute $A \cdot  \nabla S(\mathbf{y})=$
$$
\begin{pmatrix}
2i(y_1+\cdots+y_p-y_U)-2\pi p -i \Log (1+e^{y_1})\\
-2i y_1 + 2 \pi +2 i \Log (1+e^{y_1}) - i \Log (1+e^{y_2}) \\
- 2i y_2 + 2\pi -i \Log (1+e^{y_1}) +2 i \Log (1+e^{y_2}) 
-i \Log (1+e^{y_3})\\
\vdots\\
- 2i y_k+2\pi -i \Log (1+e^{y_{k-1}}) +2 i \Log (1+e^{y_k}) 
-i \Log (1+e^{y_{k+1}})\\
\vdots\\
- 2i y_{p-1}+2\pi -i \Log (1+e^{y_{p-2}}) +2 i \Log (1+e^{y_{p-1}}) 
-i \Log (1+e^{y_p})\\
i y_U-i y_W -2 \pi -i \Log (1+e^{y_p}) - i \Log (1+e^{y_U})- i \Log (1+e^{-y_U})+ i \Log (1+e^{y_W})\\
- i y_U +i \pi +i\Log(1+e^{y_W})
\end{pmatrix}.
$$
Hence we compute, for all $\mathbf{z} \in (\R+i\R_{>0})^{p+2}$,
$$
A \cdot  (\nabla S)(\psi(\mathbf{z})) =
i \begin{pmatrix}
\Log(z'_1) + 2 \Log(z_1)+\cdots + 2\Log(z_p)+2\Log(z_U)-2i\pi\\
2\Log(z''_1)+\Log(z'_2)-2i\pi \\
\Log(z'_{1})+2\Log(z''_2)+\Log(z'_{3})-2i\pi \\
\vdots\\
\Log(z'_{k-1})+2\Log(z''_k)+\Log(z'_{k+1})-2i\pi \\
\vdots\\
\Log(z'_{p-2})+2\Log(z''_{p-1})+\Log(z'_{p})-2i\pi \\
-\Log(z'_{p}) -2 \Log(z''_{U})-\Log(z'_{W}) + 2i\pi \\
\Log(z''_{W}) -\Log(z_{U})
\end{pmatrix},
$$
which is zero if and only if the system $\mathcal{E}^{co}_{X_n}(\mathbf{z})$ is satisfied. Fact 1 then follows from the invertibility of $A$.

The second fact, a variant of Lemma \ref{lem:rewriteS}, is proven similarly, using Proposition \ref{prop:dilog}:\\
\textit{Fact 2.
	The function $S\colon \mathcal{U} \to \C $ can be re-written}
\begin{multline*}
S(\mathbf{y}) =  i \Li\left (-e^{y_1}\right ) + \cdots
+ i \Li\left (-e^{y_p}\right )
+ 2 i \Li\left (-e^{y_U}\right )
+ i \Li\left (-e^{-y_W}\right ) \\
+ i \mathbf{y}^{\!\top} Q_n \mathbf{y} +  
i \frac{y_U^2}{2} + i \frac{y_W^2}{2} +
\mathbf{y}^{\!\top} \mathcal{W}_n
+   i \frac{\pi^2}{3}.
\end{multline*}

Consequently, the fact that $\Re(S)(\mathbf{y^0}) = - \mathrm{Vol}(S^3 \setminus K_n)$ is proven like in the proof of Lemma \ref{lem:-vol}, using the particular form of $S$ stated in Fact 2, and the fact that at the complete angle structure, 
$ -e^{y^0_U}=z^0_U = z^0_V = -e^{-y^0_V}$
is the complex shape of both tetrahedra $U$ and $V$.

The rest of the statements in Section \ref{sec:vol:conj} (Lemma \ref{lem:maximum} and Proposition \ref{prop:compact:contour:S:SPM} to Proposition \ref{prop:all:contour:S'b}) are proven in exactly the same way, using the new {notation} defined at the beginning of this proof.

Notably, we obtain the following asymptotic behaviour for $\mathfrak{J}_{X_n}(\hbar,0)$:
$$\mathfrak{J}_{X_n}(\hbar,0) =\left (\dfrac{1}{2\pi \sqrt{\hbar}}\right )^{p+3} e^{\frac{1}{2 \pi \hbar} S(\mathbf{y^0})}
\left (
\rho' \hbar^{\frac{p+2}{2}} \left (
1 + 
o_{\hbar \to 0^+}\left (1\right )
\right )
+ {O}_{\hbar \to 0^+}(1)
 \right ).$$
\end{proof}

\end{document}